\documentclass[10pt, a4paper]{amsart}
\usepackage{amsmath, amssymb,amsthm}
\usepackage[extra]{tipa}
\usepackage{lettrine}
\usepackage[Sonny]{fncychap}
\usepackage[english]{babel}
\usepackage{amsmath, amssymb,amsthm}
\usepackage[all]{xy}
\usepackage{tabularx}
\usepackage[applemac]{inputenc}
\usepackage{graphicx} 
\usepackage[dvipsnames]{xcolor}
\usepackage{hyperref}
\usepackage[applemac]{inputenc}
\usepackage{enumitem}
\usepackage{geometry}
\usepackage{prerex}
\usetikzlibrary{fit}
\usepackage{pdflscape}

\makeatletter
\def\oversortoftilde#1{\mathop{\vbox{\m@th\ialign{##\crcr\noalign{\kern3\p@}%
      \sortoftildefill\crcr\noalign{\kern3\p@\nointerlineskip}%
      $\hfil\displaystyle{#1}\hfil$\crcr}}}\limits}

\def\sortoftildefill{$\m@th \setbox\z@\hbox{$\braceld$}%
  \braceld\leaders\vrule \@height\ht\z@ \@depth\z@\hfill\braceru$}

\makeatother

\usepackage{mathrsfs}
\DeclareMathOperator{\B}{\overline{B}}

\DeclareMathOperator{\Norm}{Norm}
\DeclareMathOperator{\Ker}{Ker}
\DeclareMathOperator{\Pic}{Pic}
\DeclareMathOperator{\Gal}{Gal}
\DeclareMathOperator{\Rep}{Rep}
\DeclareMathOperator{\Res}{Res}
\DeclareMathOperator{\Hom}{Hom}
\DeclareMathOperator{\Ann}{Ann}
\DeclareMathOperator{\Aut}{Aut}
\DeclareMathOperator{\Card}{Card}
\DeclareMathOperator{\Disc}{Disc}
\DeclareMathOperator{\Tr}{Tr}
\DeclareMathOperator{\cusps}{cusps}
\DeclareMathOperator{\Frob}{Frob}
\DeclareMathOperator{\SL}{SL}
\DeclareMathOperator{\GL}{GL}
\DeclareMathOperator{\PSL}{PSL}
\DeclareMathOperator{\rk}{rk}
\DeclareMathOperator{\NP}{NP}
\newcommand{\cf}{\textit{cf. }}
\newcommand{\ie}{\textit{i.e. }}

\theoremstyle{definition}

\newtheorem{rem}{Remark}
\newtheorem{rems}{Remarks}
\theoremstyle{plain}
\newtheorem{thm}{Theorem}[section]

\newtheorem{conj}[thm]{Conjecture}
\newtheorem{lem}[thm]{Lemma}
\newtheorem{corr}[thm]{Corollary}

\newtheorem{prop}[thm]{Proposition}

\newcommand{\legendre}[2]{\genfrac{(}{)}{}{}{#1}{#2}}

\title{Higher Eisenstein elements, higher Eichler formulas and rank of Hecke algebras}
\author{Emmanuel Lecouturier}
\date{04/01/18}

\begin{document}
\maketitle

\begin{abstract} 
Let $N$ and $p$ be primes such that $p$ divides the numerator of $\frac{N-1}{12}$. In this paper, we study the rank $g_p$ of the completion of the Hecke algebra acting on cuspidal modular forms of weight $2$ and level $\Gamma_0(N)$ at the $p$-maximal Eisenstein ideal. We give in particular an explicit criterion to know if $g_p \geq 3$, thus answering partially a question of Mazur.

In order to study $g_p$, we develop the theory of \textit{higher Eisenstein elements}, and compute the first few such elements in four different Hecke modules. This has applications such as generalizations of the Eichler mass formula in characteristic $p$.

\end{abstract}

\section{Introduction and results}\label{Section_introduction}
Let $p\geq 2$ and $N$ be two prime numbers such that $p$ divides the numerator of $\frac{N-1}{12}$, whose $p$-adic valuation is denoted by $t\geq 1$. This is the situation of an Eisenstein prime extensively studied in \cite{Mazur_Eisenstein}. 

Let $\nu = \gcd(N-1,12)$. We fix in all the article a surjective group homomorphism $\log : (\mathbf{Z}/N\mathbf{Z})^{\times} \rightarrow \mathbf{Z}/p^t\mathbf{Z}$. Let $\tilde{\mathbb{T}}$ (resp. $\mathbb{T}$) be the $\mathbf{Z}_p$-Hecke algebra acting on the space of modular forms (resp. cuspidal modular forms) of weight $2$ and level $\Gamma_0(N)$. Let $\tilde{I}$ (resp. $I$) be the ideal of $\tilde{\mathbb{T}}$ (resp. $\mathbb{T}$) generated by the Hecke operators $T_n - \sum_{d} d$, where the sum is over the divisors of $n$ prime to $N$.

Let $\tilde{\mathfrak{P}} = \tilde{I}+(p)$ and $\mathfrak{P} = I+(p)$; these are maximal ideals.
The kernel of the natural map $\tilde{\mathbb{T}} \rightarrow \mathbb{T}$
is $\mathbf{Z}_p\cdot T_0$ for some $T_0 \in \tilde{\mathbb{T}}$. A particular choice of $T_0$ will be made later using modular forms.  Let $\tilde{\mathbf{T}}$ (resp. $\mathbf{T}$) be the $\tilde{\mathfrak{P}}$-adic (resp. $\mathfrak{P}$-adic) completion of $\tilde{\mathbb{T}}$ (resp. $\mathbb{T}$). Let $g_p \geq 1$ be the rank of $\mathbf{T}$ as a $\mathbf{Z}_p$-module. Barry Mazur asked what can be said about $g_p$, and more generally about the Newton polygon of $\mathbf{T}$ \cite[p. 140]{Mazur_Eisenstein}. This is one of the main motivation of this paper, and we provide a partial answer to Mazur's question. 

Lo\"ic Merel was the first to give explicit information about $g_p$. For simplicity, in the rest of the introduction, we assume that $p \geq 5$.

\begin{thm}\cite[Th\'eor\`eme $2$]{Merel_accouplement}\label{thm_Introduction_g_p>1}
We have $g_p>1$ if and only if 
$$\sum_{k=1}^{\frac{N-1}{2}} k \cdot \log(k) \equiv 0 \text{ (modulo }p\text{).}$$
\end{thm}

We prove the following deceptively simple generalization.
\begin{thm}\label{thm_Introduction_g_p>2}
We have $g_p>2$ if and only if 
$$\sum_{k=1}^{\frac{N-1}{2}} k \cdot \log(k) \equiv \sum_{k=1}^{\frac{N-1}{2}} k \cdot \log(k)^2 \equiv 0 \text{ (modulo }p\text{).}$$
\end{thm}
The obvious generalization does not hold. More precisely, there seems to be no link between the vanishing of $\sum_{k=1}^{\frac{N-1}{2}} k \cdot \log(k)^3$ and the fact that $g_p>3$. For instance, if $p=5$ and $N=3671$, we have $g_p=5$ but $\sum_{k=1}^{\frac{N-1}{2}} k \cdot \log(k)^3 \not\equiv 0 \text{ (modulo }p\text{)}$, and if $p=7$ and $N=4229$, we have $g_p=3$ and $\sum_{k=1}^{\frac{N-1}{2}} k \cdot \log(k)^3 \equiv 0 \text{ (modulo }p\text{)}$.

Frank Calegari and Matthew Emerton have identified $\tilde{\mathbf{T}}$ with a universal deformation ring for the residual representation $\overline{\rho} = \begin{pmatrix}\overline{\chi}_p& 0 \\ 0 & 1\end{pmatrix}$, where $\overline{\chi}_p$ is the reduction modulo $p$ of the $p$th cyclotomic character.  They deduce a characterization of $g_p$ in terms of the existence of certain Galois deformations of $\rho$. Using class field theory, they were able to prove the following result.
\begin{thm}\label{thm_Introduction_g_p>2_CE_class_gp}
If $g_p \geq 2$ then the $p$-Sylow subgroup of the class group of $\mathbf{Q}(N^{\frac{1}{p}})$ is not cyclic.
\end{thm}
The converse of Theorem \ref{thm_Introduction_g_p>2_CE_class_gp} happens to be false in general if $p > 5$.
Recently, Preston Wake and Carl Wang--Erickson \cite{Wake} have built on the work of Calegari and Emerton which tackled the determination of $g_p$ through the theory of deformations of Galois representations. It would be interesting to compare their results to ours. In particular, they give another proof of Theorem \ref{thm_Introduction_g_p>1} and they proposed our Theorem \ref{thm_Introduction_g_p>2} as a conjecture.

Our work is of a different nature. If we compare to the standard conjectures on special values of L-functions, we work on the ``analytic side'' of the problem, while Calegari--Emerton and Preston--Wake study the ``algebraic side''.

Let us say a few words about the proof of Merel's theorem. The essential point is the computation of the \textit{Eisenstein element} of $H_1(X_0(N), \cusps, \mathbf{Q})_+$ (the fixed part by complex conjugation of the singular homology relative to the cusps of the modular curve $X_0(N)$ of level $\Gamma_0(N)$). It is an element annihilated by $\tilde{I}$. 

The main idea of this paper is to determine, in well-chosen Hecke modules, the so called \textit{higher Eisenstein elements}, which have the property to be annihilated by a power of $\tilde{I}$. Results such as Theorem \ref{thm_Introduction_g_p>1} and Theorem \ref{thm_Introduction_g_p>2} are by-products of our study of higher Eisenstein elements. Unfortunately we could only determine a few of them so that we do not have a general formula for $g_p$.

We first describe the higher Eisenstein elements in the space of modular form. If $f$ is a modular form, let $a_0(f)$ be its constant coefficient at the cusp $\infty$. For simplicity, we assume that $p \geq 5$. There are modular forms $f_0$, $f_1$, ..., $f_{g_p}$ in $M_2(\Gamma_0(N), \mathbf{Z}/p\mathbf{Z})$, such that the following property hold. For any prime number $\ell$ not dividing $N$ and any integer $i$ such that $0 \leq i \leq g_p$, we have:
$$ (T_{\ell}-\ell-1)(f_{i}) = \frac{\ell-1}{2}\cdot \log(\ell) \cdot f_{i-1} \text{ (modulo }\mathbf{Z}\cdot f_0 + ... + \mathbf{Z}\cdot f_{i-2}\text{).}$$
By convention, we let $f_{-1}=0$. This determines $f_0$ up to an element of $(\mathbf{Z}/p\mathbf{Z})^{\times}$. We normalize $f_0$ so that its $q$-expansion at the cusp $\infty$ is
$$\frac{N-1}{24}+\sum_{n\geq 1} \left(\sum_{d \mid n \atop \gcd(d,N)=1} d\right)\cdot q^n \text{ (modulo }p\text{).}$$
This is, of course, the unique (normalized) Eisenstein series of weight $2$ and level $\Gamma_0(N)$. Note that the constant coefficient of $f_0$ is $0$ modulo $p$. The image of $f_i$ in $M_2(\Gamma_0(N), \mathbf{Z}/p\mathbf{Z})/\left( \mathbf{Z}\cdot f_0 + ... + \mathbf{Z}\cdot f_{i-2}\right)$ is uniquely determined.

We can show that the constant coefficients of $f_1$, ..., $f_{g_p-1}$ are $0$ modulo $p$, and that the constant coefficient of $f_{g_p}$ is non-zero modulo $p$. In particular, the following assertions are equivalent.
\begin{enumerate}
\item We have $a_0(f_1) = 0$.
\item We have $g_p \geq 2$.
\item We have $\sum_{k=1}^{\frac{N-1}{2}} k \cdot \log(k) \equiv 0 \text{ (modulo }p\text{).}$
\end{enumerate}

If $g_p \geq 2$, then $f_2$ exists and the following assertions are equivalent.
\begin{enumerate}
\item We have $a_0(f_2) = 0$.
\item We have $g_p \geq 3$. 
\item We have $\sum_{k=1}^{\frac{N-1}{2}} k \cdot \log(k) \equiv \sum_{k=1}^{\frac{N-1}{2}} k \cdot \log(k)^2 \equiv 0 \text{ (modulo }p\text{).}$
\end{enumerate}

We prove the following finer result.
\begin{thm}\label{thm_Introduction_constant_coeff}
\begin{enumerate}
\item We have $$a_0(f_1) \equiv \frac{1}{6}\cdot \sum_{k=1}^{\frac{N-1}{2}} k \cdot \log(k) \text{ (modulo }p\text{).}$$
\item Assume that $g_p \geq 2$. We have $$a_0(f_2) = \frac{1}{12}\cdot \sum_{k=1}^{\frac{N-1}{2}} k \cdot \log(k)^2 \text{ (modulo }p\text{).}$$
\end{enumerate}
\end{thm}
Contrary to point (i), point (ii) does not seem to follow from manipulation on the $q$-expansion of modular forms.

There is a notion of higher Eisenstein elements modulo $p^r$ for any integer $r$ such that $1 \leq r \leq t$. Theorem \ref{thm_Introduction_constant_coeff} generalizes modulo $p^r$, but we shall work modulo $p$ in this introduction because this is most relevant with respect to the study of $g_p$.

The quantity $\sum_{k=1}^{\frac{N-1}{2}} k \cdot \log(k)$ can thus be interpreted as the constant coefficient of an higher Eisenstein series, and also as a derivative of an L-function. Similarly, we can think of $\sum_{k=1}^{\frac{N-1}{2}} k \cdot \log(k)^2$ as the second derivative of an $L$-function. This point of view has been made precise in \cite{Lecouturier_class_group}, where these quantities are related to derived Stickelberger elements and to the class group of the cyclotomic field $\mathbf{Q}(\zeta_p, \zeta_N)$, proving a kind of class number number formula. Here, $\zeta_p$ and $\zeta_N$ are respectively $p$th and $N$th primitive roots of unity.  These results on class groups will in fact be essential for us to construct the second higher Eisenstein element in the space of odd modular symbols.

We construct higher Eisenstein elements in three other Hecke modules. Before we give our results, we briefly define what we mean by higher Eisenstein elements in general. Let $M$ be a $\tilde{\mathbb{T}}$-module such that $M \otimes_{\tilde{\mathbb{T}}} \tilde{\mathbf{T}}$ is free of rank one over $\tilde{\mathbf{T}}$. Let $M^0 \subset M$ be the submodule annihilated by $T_0 \in \tilde{\mathbb{T}}$. There is a sequence of elements $e_0$, $e_1$, ..., $e_{g_p}$ in $M/p\cdot M$, called the \textit{higher Eisenstein elements}, satisfying the following properties.
\begin{enumerate}
\item We have $e_0 \neq 0$.
\item For all prime number $\ell$ not dividing $N$ and all integer $i$ such that $0 \leq i \leq g_p$, we have:
$$(T_{\ell}-\ell-1)(e_i) = \frac{\ell-1}{2}\cdot \log(\ell)\cdot e_{i-1}  \text{ (modulo }\mathbf{Z}\cdot e_0 + ... + \mathbf{Z}\cdot e_{i-2}\text{)}$$
(with the convention $e_{-1}=0$).
In the case $p=\ell=2$ (which is excluded in this introduction but will be considered in the paper), the term $\frac{\ell-1}{2} \cdot \log(\ell)$ is replaced by $\log(x)$ where $2=x^2$ modulo $N$.
\end{enumerate}

The following additional properties hold.

a) The element $e_0$ is unique, up to a scalar in $(\mathbf{Z}/p\mathbf{Z})^{\times}$.  There is an element $\tilde{e}_0 \in M$ congruent to $e_0$ modulo $p$, and which is annihilated by $\tilde{I}$. Furthermore, $\tilde{e}_0$ is unique up to multiplication by an element of $1+p\mathbf{Z}_p$.

b) We have $e_0$, ..., $e_{g_p-1}$ $\in M^0/p\cdot M^0$ and $e_{g_p} \not\in M^0/p\cdot M^0$. 

c) If we fix $e_0$ in $(M^0/p\cdot M^0)[I]$, then the image of $e_i$ in $(M^0/p\cdot M^0)/\left( \mathbf{Z}\cdot e_0 + ... + \mathbf{Z}\cdot e_{i-2}\right)$ is uniquely determined.

Again, there is an analogous theory modulo $p^r$, for any integer $r$ such that $1 \leq r \leq t$, which we avoid in this introduction. Many of our definitions and theorems are valid modulo appropriate powers of $p$. The reader should consider all statements modulo $p$ as simplifications.

Most of our paper is devoted to the study of higher Eisenstein elements in three different modules, using completely different methods.

\subsection{The supersingular module}\label{introduction_subsection_supersingular} Consider the free $\mathbf{Z}_p$-module $M:=\mathbf{Z}_p[S]$ on the set $S$ of isomorphism classes of supersingular elliptic curves over $\overline{\mathbf{F}}_N$. If $E \in S$, let $j(E) \in \mathbf{F}_{N^2}$ be the $j$-invariant of $E$ and $[E] \in \mathbf{Z}_p[S]$ be the element corresponding to $E$. The element $\tilde{e}_0$ (unique up to $\mathbf{Z}_p^{\times}$) is well-known:
$$\tilde{e}_0 = \sum_{E\in S} \frac{1}{w_E} \cdot [E] \in M$$
where $w_E \in \{1,2,3\}$ is half the number of automorphism of $E$. 

We determine completely the element $e_1$ (which is unique modulo $(\mathbf{Z}/p\mathbf{Z})\cdot e_0$). Let 
$$H(X) = \sum_{i=0}^{\frac{N-1}{2}} {{\frac{N-1}{2}}\choose{i}}^2\cdot X^i \in \mathbf{F}_N[X]$$ 
be the classical Hasse polynomial and
$$P(X) = \Res_T(H'(T), 256\cdot (1-T+T^2)^3-T^2\cdot (1-T)^2\cdot X) \in \mathbf{F}_N[X]$$
where $\Res_X$ means the resultant relatively to the variable $X$. Since $p>2$, we can extend $\log$ to a morphism $\mathbf{F}_{N^2}^{\times} \rightarrow \mathbf{Z}/p\mathbf{Z}$, still denoted by $\log$.

\begin{thm}\label{thm_Introduction_w_1_S}
We have, modulo $\left(\mathbf{Z}/p\mathbf{Z}\right)\cdot e_0$:
$$e_1 = \frac{1}{12}\cdot \sum_{E \in S} \log(P(j(E))) \cdot [E] \text{ .}$$
\end{thm}

There is a bilinear pairing
$$\bullet : M \times M \rightarrow \mathbf{Z}_p$$
such that $[E] \bullet [E']=0$ if $E\neq E'$ and $[E] \bullet [E] = w_E$. This induces a perfect pairing:
$$\bullet : M/p\cdot M \times M/p\cdot M \rightarrow \mathbf{Z}/p\mathbf{Z} \text{ .}$$

It is not hard to show that $e_i \bullet e_j$ only depends on $i+j$ and that

$$e_i \bullet e_j \equiv 0 \text{ (modulo }p\text{)} \iff g_p\geq i+j+1 \text{ .}$$
Thus, the determination of $g_p$ is equivalent to the determination of the $e_i \bullet e_j$ modulo $p$. We now state two results about this pairing.
\begin{thm}\label{thm_Introduction_pairing_supersingular}
\begin{enumerate}
\item We have $$e_1 \bullet e_0 = \frac{1}{12}\cdot  \sum_{\lambda \in L} \log(H'(\lambda)) \text{ .}$$
\item We have $$e_2 \bullet e_0 = e_1 \bullet e_1 =\sum_{\lambda \in L} \frac{1}{24}\cdot \log(H'(\lambda))^2 - \frac{1}{18}\cdot \log(\lambda)^2 $$
where $L \subset \mathbf{F}_{N^2}^{\times}$ is the set of roots of $H$ (these are simple roots).
\end{enumerate}
\end{thm}
This gives us criteria to determine wether $g_p \geq 2$ and $g_p \geq 3$ respectively. 

One can in fact compute directly in an elementary way the discriminant of $H$ (modulo $N$), which gives us in particular the following formula:
$$e_1 \bullet e_0 = \frac{1}{3} \cdot \sum_{k=1}^{\frac{N-1}{2}} k \cdot \log(k) \text{ .}$$
This gives an other proof of Theorem \ref{thm_Introduction_g_p>1}.

The quantity $e_i \bullet e_0$ is the degree of $e_i$. For $i=0$, Eichler mass formula (in characteristic $0$) tells us that
$$e_0 \bullet e_0 \equiv \frac{N-1}{12} \equiv 0 \text{ (modulo }p\text{).}$$ 
We thus feel justified to call our formulas for $e_0 \bullet e_i$ ``higher Eichler's formulas''. We can only state them for $i \in \{1,2\}$ (\cf Theorem \ref{thm_introduction_identity}).

\subsection{Odd modular symbols}\label{section_intro_odd_modSymb}

Consider $M^- = H_1(Y_0(N), \mathbf{Z}_p)^-$, the largest torsion-free quotient of $H_1(Y_0(N), \mathbf{Z}_p)$ on which the complex conjugation acts by multiplication by $-1$, where $Y_0(N)$ is the open modular curve. Let $M_+ = H_1(X_0(N), \cusps, \mathbf{Z}_p)_+$, the fixed subspace by the complex conjugation of the homology (relative to the cusps) of the classical modular curve $X_0(N)$. There is a perfect $\tilde{\mathbb{T}}$-equivariant bilinear pairing, called the \textit{intersection pairing}, $M_+ \times M^- \rightarrow \mathbf{Z}_p$. We denote this pairing by $\bullet$. 

Let 
$$\xi_{\Gamma_0(N)} : \mathbf{Z}_p[\Gamma_0(N) \backslash \SL_2(\mathbf{Z})] \rightarrow H_1(X_0(N), \cusps, \mathbf{Z}_p)$$
be the usual Manin surjection, given by
$$\xi_{\Gamma_0(N)}\left( \Gamma_0(N) \cdot g \right) = \{g(0), g(\infty)\} $$
where, if $\alpha,\beta \in \mathbf{P}^1(\mathbf{Q})$, we denote by $\{\alpha, \beta\}$ the cohomology class of the image of the geodesic path between $\alpha$ and $\beta$ in the modular curve $X_0(N)$ (via the complex upper-half plane parametrization). There is a natural identification
$\Gamma_0(N) \backslash \SL_2(\mathbf{Z}) \xrightarrow{\sim} \mathbf{P}^1(\mathbf{Z}/N\mathbf{Z})$
given by
$$\Gamma_0(N)\cdot  \begin{pmatrix}
a & b\\
c& d
\end{pmatrix} \mapsto [c:d] \text{ .}$$

We denote by $m_i^-$, $i=1, ..., g_p$ the higher Eisenstein elements in $M^-/p\cdot M^-$. The element $\tilde{m}_0^-$ is easy to describe as a generator of the kernel of the map $H_1(Y_0(N), \mathbf{Z}_p)^- \rightarrow H_1(X_0(N), \mathbf{Z}_p)^-$. Since $X_0(N)$ has two cusps, this kernel is isomorphic to $\mathbf{Z}_p$. We normalize $\tilde{m}_0^-$ so that $$\{0, \infty\} \bullet \tilde{m}_0^- = -1 \text{ .}$$

The element $m_1^-$ was essentially determined by Mazur \cite[Proposition $18.8$]{Mazur_Eisenstein}. For any $x \in (\mathbf{Z}/N\mathbf{Z})^{\times}$, we have in $\mathbf{Z}/p\mathbf{Z}$:
$$ \left( (1+c)\cdot \xi_{\Gamma_0(N)}([x:1])\right) \bullet m_1^- = \log(x) \text{ , }$$
where $c$ is the complex conjugation. 

The analogous statement modulo $p^t$ shows that 
\begin{equation}
I\cdot H_1(X_0(N), \mathbf{Z}_p)_+ = \left\{(1+c)\cdot \sum_{x \in (\mathbf{Z}/N\mathbf{Z})^{\times}} \lambda_x \cdot \xi_{\Gamma_0(N)}([x:1]) \text{ such that } \sum_{x \in (\mathbf{Z}/N\mathbf{Z})^{\times}} \lambda_x \cdot \log(x) \equiv 0 \text{ (modulo }p^t\text{)}\right\} \text{ .}
\end{equation}
This shows that $I\cdot H_1(X_0(N), \mathbf{Z}_p)_+$ is generated by the elements $$(1+c)\cdot \xi_{\Gamma_0(N)}\left([x\cdot y:1]-[x:1]-[y:1]\right)$$ for $x,y \in (\mathbf{Z}/N\mathbf{Z})^{\times}$.
Our main result about the higher Eisenstein elements in $M^-$ is the determination of $m_2^-$, which exists if and only if $g_p\geq 2$. We describe $m_2^-$ with coefficients not in $\mathbf{Z}/p\mathbf{Z}$, but in a certain $K$-group, which is cyclic of order $p$ (see below). If $A$ is a ring, let $K_2(A)$ be the second $K$-group of $A$ defined by Quillen. Let $\Lambda = \mathbf{Z}_p[\mathbf{Z}/p\mathbf{Z}]$ and $J$ be the augmentation ideal of $\Lambda$. Let $K$ be the unique subfield of $\mathbf{Q}(\zeta_N)$ of degree $p$ over $\mathbf{Q}$ where $\zeta_N$ is a primitive $N$th root of unity. The group $\Gal(K/\mathbf{Q})$ is canonically isomorphic to $(\mathbf{Z}/N\mathbf{Z})^{\times} \otimes \mathbf{Z}/p\mathbf{Z}$ via the $N$th cyclotomic character. Since we have fixed a choice of $\log$, we can identify canonically $\Gal(K/\mathbf{Q})$ with $\mathbf{Z}/p\mathbf{Z}$. We let $\mathcal{O}_K$ be the ring of integer of $K$ and $\mathcal{K} = K_2(\mathcal{O}_K[\frac{1}{Np}])/p\cdot K_2(\mathcal{O}_K[\frac{1}{Np}])$. Note that $\mathcal{K}$ is equipped with a canonical action of $\Lambda$. We prove the following result about the structure of $\mathcal{K}$.
\begin{thm}\label{intro_odd_modSymb_relation_K_C}
\begin{enumerate}
\item The group $\mathcal{K}/J\cdot \mathcal{K}$ is cyclic of order $p$.
\item The group $J\cdot \mathcal{K}/J^2\cdot \mathcal{K}$ is cyclic of order dividing $p$, with equality if and only if $\sum_{k=1}^{\frac{N-1}{2}} k \cdot \log(k) \equiv 0 \text{ (modulo }p\text{).}$
\end{enumerate}
\end{thm}

If $x$ and $y$ are in $\mathbf{Z}[\zeta_N, \frac{1}{Np}]$, we denote by $\{x,y\} \in K_2(\mathbf{Z}[\zeta_N, \frac{1}{Np}])$ the Steinberg symbol associated to $x$ and $y$. We also let $(x,y)$ be the image of $\{x,y\}$ via the norm map $K_2(\mathbf{Z}[\zeta_N, \frac{1}{Np}]) \rightarrow \mathcal{K}$.
\begin{prop}
There is a unique group isomorphism $\iota : \mathcal{K}/J\cdot \mathcal{K} \simeq \mathbf{Z}/p\mathbf{Z}$ such that for all $u$, $v$ $\in (\mathbf{Z}/N\mathbf{Z})^{\times}$, we have
$$\iota(1-\zeta_N^u, 1-\zeta_N^v) \equiv \log\left(\frac{u}{v}\right) \text{ (modulo }p\text{).}$$
\end{prop}
Thus, we have a $K$-theoretic version of $m_1^-$, given by 
\begin{align*}
\iota^{-1} \left( \left((1+c)\cdot \sum_{x \in (\mathbf{Z}/N\mathbf{Z})^{\times}} \lambda_x \cdot \xi_{\Gamma_0(N)}([x:1]) \right)\bullet m_1^- \right)&= \sum_{x \in (\mathbf{Z}/N\mathbf{Z})^{\times}} \lambda_x \cdot \left(1-\zeta_N^x, 1-\zeta_N\right) \\&= \left(\prod_{x \in (\mathbf{Z}/N\mathbf{Z})^{\times}} (1-\zeta_N^x)^{\lambda_x}, 1-\zeta_N\right) \text{ .}
\end{align*}
Let $\Delta = [\overline{1}]-[\overline{0}] \in \Lambda$; it is a generator of $J$. The multiplication by $\Delta$ induces a surjective morphism $\mathcal{K}/J\cdot \mathcal{K} \rightarrow J\cdot \mathcal{K}/J^2\cdot \mathcal{K}$ since $\Delta$ is a generator of $J$. By Theorem \ref{intro_odd_modSymb_relation_K_C}, it is an isomorphism if and only if $g_p \geq 2$. In this case, let $\delta : J\cdot \mathcal{K}/J^2\cdot \mathcal{K} \xrightarrow{\sim} \mathcal{K}/J\cdot \mathcal{K}$ be the inverse isomorphism, and $\delta' = \delta \circ \iota^{-1} : \mathbf{Z}/p\mathbf{Z} \xrightarrow{\sim} J\cdot \mathcal{K}/J^2\cdot \mathcal{K}$.

We found our formula for $m_2^-$ under the influence of the work of Alexander Goncharov and Romyar Sharifi (\cf \cite{Goncharov} et \cite{Sharifi_survey}). Our formula is conditional on a conjecture inspired by conjectures of Sharifi. Let $$\xi_{\Gamma_1(N)} : \mathbf{Z}_p[\Gamma_1(N) \backslash \PSL_2(\mathbf{Z})] \rightarrow H_1(X_1(N), \cusps, \mathbf{Z}_p)$$ be the Manin surjective map, given by $$\xi_{\Gamma_1(N)}\left(\Gamma_1(N)\cdot g \right) = \{g(0), g(\infty)\} \text{ .}$$
Let $$M_{\Gamma_1(N)}^0 = \left\{\Gamma_1(N)\cdot \begin{pmatrix}a & b \\ c& d \end{pmatrix} \in \Gamma_1(N) \backslash \PSL_2(\mathbf{Z}), \text{ } \gcd(c\cdot d, N)=1\right\} $$
and let $C_{\Gamma_1(N)}^0$ be the set of cusps of $X_1(N)$ above the cusp $\Gamma_0(N) \cdot 0$ of $X_0(N)$. The Manin surjective map induces a surjection
$$\xi_{\Gamma_1(N)}^0 : M_{\Gamma_1(N)}^0 \rightarrow H_1(X_1(N), C_{\Gamma_1(N)}^0, \mathbf{Z}_p) \text{ .}$$

Let $\varpi' : M_{\Gamma_1(N)}^0 \rightarrow K_2\left(\mathbf{Z}[\zeta_N, \frac{1}{Np}]\right) \otimes_{\mathbf{Z}} \mathbf{Z}_p$ be the $\mathbf{Z}_p$-linear map defined by
$$\varpi' \left(\Gamma_1(N)\cdot \begin{pmatrix}a & b \\ c& d \end{pmatrix} \right) = \left\{1-\zeta_N^c, 1-\zeta_N^d\right\}\otimes 1 $$
(it does not depend on the choice of $\begin{pmatrix}a & b \\ c& d \end{pmatrix}$).

One easily shows that $\varpi'$ factors through $\xi_{\Gamma_1(N)}^0$, thus inducing a group homomorphism
$$\varpi : H_1(X_1(N), C_{\Gamma_1(N)}^0, \mathbf{Z}_p) \rightarrow  K_2\left(\mathbf{Z}[\zeta_N, \frac{1}{Np}]\right)  \otimes_{\mathbf{Z}} \mathbf{Z}_p \text{ .}$$
The natural analogue of Sharifi's conjecture (proved by Fukaya Takako and Kazuya Kato in \cite[Theorem 5.2.3]{Fukaya_Kato}) would be the following conjecture.
\begin{conj}\label{Intro_fr_conjecture_sharifi}
The map $\varpi$ is annihilated by the Hecke operators $T_n - \sum_{d} \frac{n}{d}\cdot \langle d \rangle$ for any integer $n \geq 1$ (where the sum is over the divisors of $n$ prime to $N$ and $\langle d \rangle$ is the $d$th diamond operator).
\end{conj}

Conjecture \ref{Intro_fr_conjecture_sharifi} seems to have been considered by Sharifi himself. We refer to sections \ref{odd_modSymb_Sharifi_conj} and \ref{odd_modSymb_Section_Sharifi} for a more detailed discussion about Conjecture \ref{Intro_fr_conjecture_sharifi}.

\begin{thm}\label{Intro_formula_m_2^-}
Assume that Conjecture \ref{Intro_fr_conjecture_sharifi} holds.  Assume $g_p \geq 2$, \ie
$$\sum_{k=1}^{\frac{N-1}{2}} k \cdot \log(k) \equiv 0 \text{ (modulo } p\text{).}$$

Let $x$, $y$ $\in (\mathbf{Z}/N\mathbf{Z})^{\times}$. Then, we have the following equality in $J\cdot \mathcal{K}/J^2\cdot \mathcal{K}$:
\begin{align*}
\delta'\left(\left( 2\cdot (1+c)\cdot \xi_{\Gamma_0(N)}([x\cdot y:1]-[x:1] - [y:1]) \right)\bullet m_2^- \right)&= (1-\zeta_N^x, 1-\zeta_N) \\& -(1-\zeta_N^x, 1-\zeta_N^y) - (1-\zeta_N^{y}, 1-\zeta_N)\text{ .}
\end{align*}
\end{thm}

\subsection{Even modular symbols}\label{section_intro_even_modSymb}
We denote by $m_i^+$, $i=1, ..., g_p$ the higher Eisenstein elements in $M_+/p\cdot M_+$. 

Merel determined the element $\tilde{m}_0^+$ (unique up to $\mathbf{Z}_p^{\times}$) in terms of Manin symbols. We recall his result below, using a slightly different formula. 

We will need to use the Bernoulli polynomial functions. Recall that $\B_1 : \mathbf{R} \rightarrow \mathbf{R}$
is the function defined by $\B_1(x) = x-\lfloor x \rfloor-\frac{1}{2}$
if $x\not\in \mathbf{Z}$ and $\B_1(x) = 0$ if $x \in \mathbf{Z}$.

Consider the boundary map $\partial : H_1(X_0(N), \cusps, \mathbf{Z}_p) \rightarrow \mathbf{Z}_p[\cusps]^0$ given by $\partial(\{\alpha, \beta\}) = (\Gamma_0(N) \cdot \beta) - (\Gamma_0(N) \cdot \alpha)$.

\begin{thm}[Merel]\label{thm_Introduction_w_0^+}
Let $F_{0,p} : \mathbf{P}^1(\mathbf{Z}/N\mathbf{Z}) \rightarrow \mathbf{Z}_p$ be such that if $x = [c:d] \in  \mathbf{P}^1(\mathbf{Z}/N\mathbf{Z})$, we have:
$$
6\cdot F_{0,p}(x) =  \sum_{(s_1,s_2) \in (\mathbf{Z}/2N\mathbf{Z})^2 \atop (d-c)s_1+(d+c)s_2 \equiv 0 \text{ (modulo } N\text{)}} (-1)^{s_1+s_2} \B_1\left(\frac{s_1}{2N}\right)\B_1\left(\frac{s_2}{2N}\right) \text{ .}
$$
This is independent of the choice of $c$ and $d$ such that $x = [c:d]$. 
We have 
$$\tilde{m}_0^+ =  \sum_{x \in \mathbf{P}^1(\mathbf{Z}/N\mathbf{Z})} F_{0,p}(x) \cdot\xi_{\Gamma_0(N)}(x) \in M_+ \text{.}$$
 Furthermore, one has $\partial \tilde{m}_0^+=\frac{N-1}{12}\cdot( (\Gamma_0(N)\cdot 0) - (\Gamma_0(N) \cdot \infty))$, \ie $\tilde{m}_0^+ \bullet \tilde{m}_0^- = \frac{N-1}{12}$.
\end{thm}
We warn the reader that $\tilde{m}_0^+$ is $\frac{1}{12} \cdot \mathcal{E}$ where $\mathcal{E}$ is computed in \cite[Corollaire $4$]{Merel_accouplement}.

Our main result about $M_+$ is a formula for $m_1^+$.

\begin{thm}\label{thm_Introduction_w_1^+>3}
Let $F_{1,p} : \mathbf{P}^1(\mathbf{Z}/N\mathbf{Z}) \rightarrow \mathbf{Z}/p\mathbf{Z}$ be such that
\begin{align*}
12 \cdot F_{1,p}([c:d]) &=  \sum_{(s_1,s_2) \in (\mathbf{Z}/2N\mathbf{Z})^2 \atop (d-c)s_1+(d+c)s_2 \equiv 0 \text{ (modulo } N\text{)}} (-1)^{s_1+s_2} \B_1\left(\frac{s_1}{2N}\right)\B_1\left(\frac{s_2}{2N}\right) \cdot \log\left(\frac{s_2}{d-c}\right)  \\&
- \sum_{(s_1,s_2) \in (\mathbf{Z}/2N\mathbf{Z})^2 \atop (d-c)s_1+(d+c)s_2 \not\equiv 0 \text{ (modulo } N\text{)}} (-1)^{s_1+s_2} \B_1\left(\frac{s_1}{2N}\right)\B_1\left(\frac{s_2}{2N}\right) \cdot \log((d-c)s_1+(d+c)s_2)) 
\end{align*}
if $[c:d]\neq [1:1]$ and $F_{1,p}([1:1]) = 0$. 

We have the following equality in $M_+/pM_+$:
$$m_1^+ = \sum_{x \in \mathbf{P}^1(\mathbf{Z}/N\mathbf{Z})} F_{1,p}(x) \cdot \xi_{\Gamma_0(N)}(x) \text{ .}$$
Furthermore, one has $\partial m_1^+ =  \left( \frac{1}{3}\sum_{k=1}^{\frac{N-1}{2}} k \cdot \log(k) \right) \cdot ( (\Gamma_0(N)\cdot 0) - (\Gamma_0(N) \cdot \infty))$, \ie $$m_1^+ \bullet m_0^- = \frac{1}{3}\sum_{k=1}^{\frac{N-1}{2}} k \cdot \log(k) \text{ .}$$
\end{thm}

The last assertion, combined with the first equality of point (i) of Theorem \ref{thm_Introduction_comparison_pairing_modSymb} below gives another proof of Theorem \ref{thm_Introduction_g_p>1}.

\subsection{Comparison}
We get additional results by comparing our results in the three above settings. First, combining our results in the spaces of even and odd modular symbols, a computation gives us the following formulas.

\begin{thm}\label{thm_Introduction_comparison_pairing_modSymb}
\begin{enumerate}
\item We have $m_0^+ \bullet m_1^- = m_1^+\bullet m_0^- = \frac{1}{3}\cdot \sum_{k=1}^{\frac{N-1}{2}} k \cdot \log(k)$.
\item We have $m_1^+ \bullet m_1^- = m_2^+ \bullet m_0^- = m_0^+ \bullet m_2^- = \frac{1}{6} \cdot \sum_{k=1}^{\frac{N-1}{2}} k \cdot \log(k)^2$.
\end{enumerate}
\end{thm}

This proves Theorem \ref{thm_Introduction_g_p>2}. Note that all equalities but the last of each line follows formally from algebraic properties of pairing of Hecke modules.

In particular, the boundary of $m_2^+$ is $\left(  \frac{1}{6} \cdot \sum_{k=1}^{\frac{N-1}{2}} k \cdot \log(k)^2 \right) \cdot ((\Gamma_0(N) \cdot 0) - (\Gamma_0(N)\cdot \infty))$. However, we do not know an expression for $m_2^+$ in terms of Manin symbols similar to Theorems \ref{thm_Introduction_w_0^+} and \ref{thm_Introduction_w_1^+>3}. It is unclear to us whether such an expression is to be expected as a quadratic expression of logarithms or as a formula in some algebraic number theoretic group considered in section \ref{section_intro_odd_modSymb}.

We also get the following result relating the three Hecke modules.

\begin{thm}\label{thm_Introduction_comparison_pairings}
Let $i$ and $j$ be integers such that $0 \leq i,j\leq g_p$ and $i+j \leq g_p$. We have:
$$m_i^+ \bullet m_j^- = e_i \bullet e_j \text{ .}$$
Furthermore, this quantity depends only on $i+j$. It is $0$ if $i+j < g_p$ and it is non-zero if $i+j = g_p$.
\end{thm}

The combination of Theorem \ref{thm_Introduction_pairing_supersingular}, Theorem \ref{thm_Introduction_comparison_pairing_modSymb} and Theorem \ref{thm_Introduction_comparison_pairings} allows us to deduce the following identity. We were not able to find an elementary proof of it.

\begin{thm}\label{thm_introduction_identity}
Assume that $\sum_{k=1}^{\frac{N-1}{2}} k \cdot \log(k) \equiv 0 \text{ (modulo }p\text{)}$. We then have
$$\sum_{\lambda \in L} \frac{1}{4}\cdot \log(H'(\lambda))^2 - \frac{1}{3}\cdot \log(\lambda)^2 \equiv   \sum_{k=1}^{\frac{N-1}{2}} k \cdot \log(k)^2 \text{ (modulo }p\text{)}\text{ .}$$
\end{thm}

This is most advanced instance (known to us) of what we call an \textit{higher Eichler mass formula}. We do not know how to connect the algebraic objects of section \ref{section_intro_odd_modSymb} with supersingular $j$ (or $\lambda$) invariants.

Obviously, there is a theory of higher Eisenstein elements for all levels and all weights. We finish this introduction by giving some reasons as to why we focus on the prime level and weight $2$ cases.
\begin{enumerate}
\item Mazur's original question about the rank $g_p$ was in this setting.
\item It is known that in this case, the Hecke algebra has nice properties. For example it is a Gorenstein ring at the Eisenstein maximal ideals, and the Eisenstein ideal $I$ is locally principal, \ie $I/I^2$ is cyclic.
\item We know, thanks to Mazur, some multiplicity one results for the homology of $X_0(N)$.
\item There is an explicit and simple description of $m_1^-$ coming from the Shimura covering $X_1(N) \rightarrow X_0(N)$. 
Mazur deduced from this a description of $I/I^2$ in terms of modular symbols.
\item In this case, the supersingular module presents itself. Our methods in turn has applications to supersingular elliptic curves.
\item The Galois cohomology group arising in the determination of $m_2^-$ is well-understood thanks to the results of \cite{Lecouturier_class_group}.
\end{enumerate}

It is not clear to us in which settings we can obtain similar results. Mazur's question makes sense in any weight and level (although there are possibly several maximal Eisenstein ideals). However, we do not know multiplicity one in general, and Eisenstein ideals are not locally principal in general neither. See the list of alternative settings already proposed by Mazur in \cite[p. 39]{Mazur_Eisenstein}.

\section{The formalism of higher Eisenstein elements}\label{Section_Formalism}

\subsection{Algebraic setting}\label{Formalism_algebraic_setting}
In this section, we develop the theory of higher Eisenstein elements in a tentative axiomatic setting. Let $\tilde{\mathbb{T}}$ be a $\mathbf{Z}_p$-algebra which is free of finite rank as a $\mathbf{Z}_p$-module. Let $\tilde{I}$ be an ideal of $\tilde{\mathbb{T}}$ and $\Ann_{\tilde{\mathbb{T}}}(\tilde{I})$ be the annihilator of $\tilde{I}$ in $\tilde{\mathbb{T}}$. Let $\mathbb{T} = \tilde{\mathbb{T}}/\Ann_{\tilde{\mathbb{T}}}(\tilde{I})$ and $I \subset \mathbb{T}$ be the image of $\tilde{I}$ in $\mathbb{T}$. Let $\tilde{\mathbf{T}}$ be the $\tilde{I}$-adic completion of $\tilde{\mathbb{T}}$, and $\mathbf{T}$ be the $I$-adic completion of $\mathbb{T}$. Let $t \geq 1$ be an integer. 

We assume the following hypotheses.
\begin{enumerate}
\item \label{hyp_T/I_tilde} We have $\tilde{\mathbb{T}}/\tilde{I} \simeq \mathbf{Z}_p$. In particular, the ideal $\tilde{\mathfrak{P}} := \tilde{I}+(p)$ of $\tilde{\mathbb{T}}$ is the unique maximal ideal containing $\tilde{I}$.
\item \label{hyp_Ann_free_one} The $\mathbf{Z}_p$-module $\Ann_{\tilde{\mathbb{T}}}(\tilde{I})$ is free of rank one.
\item \label{hyp_T_free} The $\mathbf{Z}_p$-module $\mathbb{T}$ is free.
\item \label{hyp_T/I} The group $\mathbb{T}/I$ is cyclic of order $p^t$.
\item \label{hyp_I/I^2}The $\mathbb{T}/I$-module $I/I^2$ is free of rank one. We fix a group isomorphism $e : I/I^2 \xrightarrow{\sim} \mathbf{Z}/p^t\mathbf{Z}$. If $\eta \in \tilde{I}$, we denote by $e(\eta)$ the image by $e$ of the class of $\eta$ in $I/I^2$.
\item \label{hyp_gorenstein} The ring $\mathbf{T}$ is Gorenstein, \ie the $\mathbf{T}$-module $\Hom_{\mathbf{Z}_p}(\mathbf{T}, \mathbf{Z}_p)$ is free of rank one.
\end{enumerate}

Since $\tilde{\mathbb{T}}$ is a Noetherian ring, if $M$ is a finitely generated $\tilde{\mathbb{T}}$-module then the $\tilde{I}$-adic completion of $M$ is canonically isomorphic to $M \otimes_{\tilde{\mathbb{T}}} \tilde{\mathbf{T}}$ \cite[Proposition 10.13]{Atiyah}. 

\begin{thm}[Construction of higher Eisenstein elements]\label{Formalism_existence_eisenstein}
Let $M$ be a $\tilde{\mathbb{T}}$-module which is free of finite rank over $\mathbf{Z}_p$. Assume that $M \otimes_{\tilde{\mathbb{T}}} \tilde{\mathbf{T}}$ is free of rank one over $\tilde{\mathbf{T}}$. Let $M^0 \subset M$ be the submodule annihilated by $\Ann_{\tilde{\mathbb{T}}}(\tilde{I})$. Let $r$ be an integer such that $1 \leq r \leq t$.

There exists a maximal positive integer $n(r,p)$ and a sequence of elements $e_0$, $e_1$, ..., $e_{n(r,p)}$ in $M/p^r\cdot M$, called the \textup{higher Eisenstein elements} of $M$, such that the following properties hold.
\begin{enumerate}
\item We have $e_0 \not\in p \cdot (M/p^r\cdot M)$
\item For all $\eta \in \tilde{I}$, we have 
$$\eta(e_i) \equiv e(\eta)\cdot e_{i-1} \text{ (modulo }\mathbf{Z}\cdot e_0+ ... + \mathbf{Z}\cdot e_{i-2}\text{)}$$
(with the convention $e_{-1}=0$).
\end{enumerate}

We have the following properties.

a) The element $e_0$ is unique up to $(\mathbf{Z}/p^r\mathbf{Z})^{\times}$. There exists $\tilde{e}_0 \in M[\tilde{I}]$ whose class in $M/p^r\cdot M$ is $e_0$. Furthermore, the element $\tilde{e}_0$ is unique up to $\mathbf{Z}_p^{\times}$, and $\tilde{e}_0 \not\in M^0 $.

b) We have $e_0$, ..., $e_{n(r,p)-1}$ $\in M^0/p^r \cdot M^0$ and $e_{n(r,p)} \not\in M^0/p^r\cdot M^0$ . 

c) If we fix $e_0$ in $(M^0/p^r \cdot M^0)[I]$, then for every integer $i$ such that $1 \leq i \leq n(r,p)$ the image of $e_i$ in $\left(M/p^r\cdot M\right)/\left(\mathbf{Z}\cdot e_0 + ... + \mathbf{Z}\cdot e_{i-2}\right)$ is uniquely determined.

d) The integer $n(r,p)$ is the largest integer $n \geq 1$ such that the group $\tilde{I}^n\cdot (\tilde{\mathbb{T}}/p^r\cdot \tilde{\mathbb{T}})/\tilde{I}^{n+1}\cdot (\tilde{\mathbb{T}}/p^r\cdot \tilde{\mathbb{T}})$ is cyclic of order $p^r$. Furthermore, we have $n(r,p) \leq \rk_{\mathbf{Z}_p}(\mathbf{T})$ with equality if $r=1$. In particular, the integer $n(r,p)$ only depends on $\tilde{\mathbb{T}}$ and $r$, and not on the choice of $M$.

\end{thm}
\begin{proof}
By hypothesis (\ref{hyp_T/I_tilde}), there is a unique maximal ideal $\tilde{\mathfrak{P}}$ containing $\tilde{I}$ in $\tilde{\mathbb{T}}$. We denote by $M_{\tilde{\mathfrak{P}}}$ the $\tilde{\mathfrak{P}}$-adic  (or equivalently $\tilde{I}$-adic) completion of $M$. 

\begin{lem}\label{Formalism_flatness_Hecke}
\begin{enumerate}
\item The algebra $\tilde{\mathbf{T}}$ is flat as a $\tilde{\mathbb{T}}$-module.
\item The canonical map $M \rightarrow M_{\tilde{\mathfrak{P}}}$
is surjective.
\item Let $M$ be a $\tilde{\mathbb{T}}$-module which is free as a $\mathbf{Z}_p$-module. Then the natural morphism of $\tilde{\mathbf{T}}$-modules
$$ \mathrm{Hom}_{\mathbf{Z}_p}(M_{\tilde{\mathfrak{P}}}, \mathbf{Z}_p) \rightarrow \mathrm{Hom}_{\mathbf{Z}_p}(M, \mathbf{Z}_p)_{\tilde{\mathfrak{P}}}$$
is an isomorphism.
\end{enumerate}
\end{lem}
\begin{proof}
The ring $\tilde{\mathbb{T}}$ has finitely many maximal ideals, \ie is a semi-local ring since it is a finitely generated $\mathbf{Z}_p$-module. The ring $\tilde{\mathbb{T}}$ is $p$-adically complete and semi-local, so is the direct sum of its completions at its maximal ideals \cite[Chap. III, \S 2, no. 12]{Bourbaki_Comm_Alg}. Thus, any module over $\tilde{\mathbb{T}}$ is the direct sum of its completions at the maximal ideals of $\tilde{\mathbb{T}}$. Lemma \ref{Formalism_flatness_Hecke} follows immediately.
\end{proof}

Let $M^*:=\Hom_{\mathbf{Z}_p}(M, \mathbf{Z}_p)$, equipped with its natural structure of $\tilde{\mathbb{T}}$-module. The $\tilde{\mathbf{T}}$-module $(M^*)_{\tilde{\mathfrak{P}}}$ is free of rank one by Lemma \ref{Formalism_flatness_Hecke} (iii) and hypothesis (\ref{hyp_gorenstein}). Since $M$ is free over $\mathbf{Z}_p$, we have:
$$M/p^r\cdot M = \Hom\left(M^*/p^r\cdot M^*, \mathbf{Z}/p^r\mathbf{Z}\right) \text{ .}$$
Thus, we have: 
$$(M/p^r\cdot M)[\tilde{I}^n] = \Hom\left(M^*/(p^r+ \tilde{I}^n)\cdot M^*, \mathbf{Z}/p^r\mathbf{Z}\right)\text{ .}$$
We get a canonical group isomorphism
\begin{equation}\label{Formalism_restriction_J_n-1_J_n}
(M/p^r\cdot M)[\tilde{I}^n]/(M/p^r\cdot M)[\tilde{I}^{n-1}] \xrightarrow{\sim} \Hom\left(\tilde{I}^{n-1}\cdot (M^*/p^r\cdot M^*)/\tilde{I}^{n}\cdot (M^*/p^r\cdot M^*), \mathbf{Z}/p^r\mathbf{Z}\right)\text{ .}
\end{equation}
Since $(M^*)_{\tilde{\mathfrak{P}}}$ is free of rank one over $\tilde{\mathbf{T}}$ and that $\tilde{\mathfrak{P}}$ is the unique maximal ideal of $\tilde{\mathbb{T}}$ containing $\tilde{I}$, we have group isomorphisms
\begin{equation}\label{Formalism_restriction_J_n-1_J_n_2}
(M^*/p^r\cdot M^*)/\tilde{I}^n\cdot (M^*/p^r\cdot M^*)  \simeq \tilde{\mathbf{T}}/(p^r+\tilde{I}^n)\cdot \tilde{\mathbf{T}} \simeq \tilde{\mathbb{T}}/(p^r+\tilde{I}^n)\cdot \tilde{\mathbb{T}}\text{ .}
\end{equation}
Let $J$ be the image of $\tilde{I}$ in $\tilde{\mathbb{T}}/p^r\cdot \tilde{\mathbb{T}}$. 
By (\ref{Formalism_restriction_J_n-1_J_n}) and (\ref{Formalism_restriction_J_n-1_J_n_2}), we get a (non-canonical) group isomorphism
\begin{equation}\label{construction_e_i}
(M/p^r\cdot M)[\tilde{I}^n]/(M/p^r\cdot M)[\tilde{I}^{n-1}] \xrightarrow{\sim}  \Hom(J^{n-1}/J^n, \mathbf{Z}/p^r\mathbf{Z}) \text{ .}
\end{equation}
In particular, for $n=1$, we have group isomorphisms:
$$(M/p^r\cdot M)[\tilde{I}] \simeq \Hom\left(\tilde{\mathbb{T}}/(p^r+\tilde{I})\cdot \tilde{\mathbb{T}}, \mathbf{Z}/p^r\mathbf{Z}\right) \simeq \mathbf{Z}/p^r\mathbf{Z} \text{ ,}$$
where the last isomorphism follows from Hypothesis (\ref{hyp_T/I_tilde}).

This shows the existence of $e_0 \in (M/p^r\cdot M)[\tilde{I}]$ such that $e_0 \not\in p \cdot (M/p^r\cdot M)$. Furthermore $e_0$ is unique up to $(\mathbf{Z}/p^r\mathbf{Z})^{\times}$.
Similarly, we have isomorphisms of $\mathbf{Z}_p$-modules:
\begin{equation}\label{e_0_tilde}
M[\tilde{I}] \simeq  \Hom_{\mathbf{Z}_p}(\tilde{\mathbb{T}}/\tilde{I}, \mathbf{Z}_p) \simeq \mathbf{Z}_p \text{ .}
\end{equation}
This shows the existence of $\tilde{e}_0 \in M[\tilde{I}]$ which reduces to $e_0$ modulo $p^r$, and its unicity up to a scalar. 

\begin{lem}\label{Formalism_T_0_Eisenstein}
We have $M[\tilde{I}] = \Ann_{\tilde{\mathbb{T}}}(\tilde{I}) \cdot M$.
\end{lem}
\begin{proof}
The inclusion $\Ann_{\tilde{\mathbb{T}}}(\tilde{I}) \cdot M \subset M[\tilde{I}]$ is obvious. We first claim that $M/\Ann_{\tilde{\mathbb{T}}}(\tilde{I}) \cdot M$ is torsion-free. For any ideal $\mathfrak{Q}$ of $\tilde{\mathbb{T}}$ different from $\tilde{\mathfrak{P}}$, we have
\begin{equation}\label{Formalism_T_0_Eisenstein_eq0}
\Ann_{\tilde{\mathbb{T}}}(\tilde{I}) \cdot M_{\mathfrak{Q}} = 0 \text{ .}
\end{equation}
By the proof of Lemma \ref{Formalism_flatness_Hecke}, $M$ is the direct sum of its completions at the maximal ideals of $\tilde{\mathbb{T}}$. Thus, we have a canonical isomorphism of $\tilde{\mathbb{T}}$-modules:
\begin{equation}\label{Formalism_T_0_Eisenstein_eq1}
\Ann_{\tilde{\mathbb{T}}}(\tilde{I}) \cdot M \xrightarrow{\sim} \Ann_{\tilde{\mathbb{T}}}(\tilde{I}) \cdot M_{\tilde{\mathfrak{P}}} \text{ .}
\end{equation}
By (\ref{Formalism_T_0_Eisenstein_eq0}) and (\ref{Formalism_T_0_Eisenstein_eq1}), we have an isomorphism of $\mathbf{Z}_p$-modules
\begin{equation}\label{Formalism_T_0_Eisenstein_eq2}
M/\Ann_{\tilde{\mathbb{T}}}(\tilde{I})\cdot M \simeq \left( M_{\tilde{\mathfrak{P}}}/\Ann_{\tilde{\mathbb{T}}}(\tilde{I})\cdot M_{\tilde{\mathfrak{P}}}\right) \bigoplus_{\mathfrak{Q} \neq \tilde{\mathfrak{P}}} M_{\mathfrak{Q}}  \text{ ,}
\end{equation}
where the sum is over the maximal ideals $\mathfrak{Q}$ of $\tilde{\mathbb{T}}$ different from $\tilde{\mathfrak{P}}$. Since $M_{\tilde{\mathfrak{P}}}$ is free of rank one over $\tilde{\mathbf{T}}$, we have an isomorphism of $\tilde{\mathbf{T}}$-modules 
\begin{equation}\label{Formalism_T_0_Eisenstein_eq3}
M_{\tilde{\mathfrak{P}}}/\Ann_{\tilde{\mathbb{T}}}(\tilde{I})\cdot M_{\tilde{\mathfrak{P}}} \simeq  \tilde{\mathbf{T}}/\Ann_{\tilde{\mathbb{T}}}(\tilde{I})\cdot \tilde{\mathbf{T}} \text{ .}
\end{equation}
As in (\ref{Formalism_T_0_Eisenstein_eq2}), we have an isomorphism of $\mathbf{Z}_p$-modules 
\begin{equation}\label{Formalism_T_0_Eisenstein_eq4}
\tilde{\mathbb{T}}/\Ann_{\tilde{\mathbb{T}}}(\tilde{I})\cdot \tilde{\mathbb{T}}  \simeq \tilde{\mathbf{T}}/\Ann_{\tilde{\mathbb{T}}}(\tilde{I})\cdot \tilde{\mathbf{T}} \bigoplus_{\mathfrak{Q}\neq \tilde{\mathfrak{P}}} \tilde{\mathbb{T}}_{\mathfrak{Q}} \text{ .}
\end{equation}
By hypothesis (\ref{hyp_T_free}), the former $\mathbf{Z}_p$-module is torsion-free. By (\ref{Formalism_T_0_Eisenstein_eq2}),  (\ref{Formalism_T_0_Eisenstein_eq3}) and (\ref{Formalism_T_0_Eisenstein_eq4}), the $\mathbf{Z}_p$-module $M/\Ann_{\tilde{\mathbb{T}}}(\tilde{I})\cdot M $ is torsion-free. 

By (\ref{e_0_tilde}) the $\mathbf{Z}_p$-module $M[\tilde{I}]$ is free of rank one.  Since $M[\tilde{I}]$ is a direct summand of the $\mathbf{Z}_p$-module $M$, the $\mathbf{Z}_p$-module $M[\tilde{I}]/\Ann_{\tilde{\mathbb{T}}}(\tilde{I})\cdot M$ is a free $\mathbf{Z}_p$-module of rank zero or one. If the rank of $M[\tilde{I}]/\Ann_{\tilde{\mathbb{T}}}(\tilde{I})\cdot M$ is one, then $\Ann_{\tilde{\mathbb{T}}}(\tilde{I})\cdot M = 0$. By (\ref{Formalism_T_0_Eisenstein_eq1}), we get $\Ann_{\tilde{\mathbb{T}}}(\tilde{I}) \cdot \tilde{\mathbf{T}}=0$. As above, this implies $\Ann_{\tilde{\mathbb{T}}}(\tilde{I})=0$, so $\tilde{\mathbb{T}}=\mathbb{T}$. This contradicts hypothesis (\ref{hyp_T/I}). This concludes the proof of Lemma \ref{Formalism_T_0_Eisenstein}.
\end{proof}

\begin{lem}\label{Formalism_T_0J}
We have $M^0 = \tilde{I} \cdot M$.
\end{lem}
\begin{proof}
The inclusion $\tilde{I}\cdot M \subset M^0$ is obvious. Since $M_{\tilde{\mathfrak{P}}}$ is free of rank one over $\tilde{\mathbf{T}}$, we have isomorphisms of $\mathbf{Z}_p$-modules:
$$M/\tilde{I}\cdot M \simeq \tilde{\mathbf{T}}/\tilde{I} \simeq \mathbf{Z}_p \text{ .}$$
We thus have a surjective group homomorphism
$$\mathbf{Z}_p \simeq M/\tilde{I}\cdot M \rightarrow M/M^0\text{ .}$$
The $\mathbf{Z}_p$-module $M/M^0$ is torsion-free by definition of $M^0$ and the fact that $M$ is a free $\mathbf{Z}_p$-module. Thus, we have either $M^0 =  \tilde{I}\cdot M$ or $M^0 =M$. If $M^0=M$, then we have $\Ann_{\tilde{\mathbb{T}}}(\tilde{I}) \cdot M = 0$, which is impossible by Lemma \ref{Formalism_T_0_Eisenstein}.
\end{proof}

Denote by $\varphi : M \rightarrow M_{\tilde{\mathfrak{P}}}$ the canonical map.

\begin{lem}\label{Formalism_e_0_completion}
The element $\varphi(\tilde{e}_0)$ is a generator of $M_{\tilde{\mathfrak{P}}}[\tilde{I}]$. In particular, the element $e_0$ is sent to a generator of $(M_{\tilde{\mathfrak{P}}}/p^r\cdot M_{\tilde{\mathfrak{P}}})[\tilde{I}]$.
\end{lem}
\begin{proof}
This follows from Lemma \ref{Formalism_T_0_Eisenstein} since $$M[\tilde{I}] = \Ann_{\tilde{\mathbb{T}}}(\tilde{I})  \cdot M \simeq \Ann_{\tilde{\mathbb{T}}}(\tilde{I})  \cdot M_{\tilde{\mathfrak{P}}} = M_{\tilde{\mathfrak{P}}}[\tilde{I}] \text{ ,}$$
where the middle isomorphism follows from (\ref{Formalism_T_0_Eisenstein_eq1}).
\end{proof}

\begin{lem}\label{Formalism_J/J^2}
The canonical map $\tilde{I} \rightarrow I$ induces a group isomorphism $\tilde{I}/\tilde{I}^2 \xrightarrow{\sim} I/I^2$. 
\end{lem}
\begin{proof}
We have, by definition and by Hypothesis (\ref{hyp_I/I^2}), a surjective group homomorphism $\tilde{I}/\tilde{I}^2 \rightarrow I/I^2 \simeq \mathbf{Z}/p^t\mathbf{Z}$. It remains to show that the kernel of the map $\tilde{I}/\tilde{I}^2 \rightarrow I/I^2$ is zero. This kernel is the image of $\tilde{I} \cap\Ann_{\tilde{\mathbb{T}}}(\tilde{I})$ in  $\tilde{I}/\tilde{I}^2$. Thus, it suffices to prove that $\tilde{I} \cap\Ann_{\tilde{\mathbb{T}}}(\tilde{I})=0$. By hypothesis (\ref{hyp_Ann_free_one}), there exists $T_0 \in \tilde{\mathbb{T}}$ such that $\Ann_{\tilde{\mathbb{T}}}(\tilde{I}) = \mathbf{Z}_p\cdot T_0$. Let $x \in \tilde{I} \cap\Ann_{\tilde{\mathbb{T}}}(\tilde{I})$. For the sake of a contradiction, assume that $x \neq 0$. Then there exists $n \in \mathbf{Z}_p \backslash \{0\}$ such that $x = n \cdot T_0$. Since the image of $x$ in $\tilde{\mathbb{T}}/\tilde{I} \simeq \mathbf{Z}_p$ is zero, we have $T_0 \in \tilde{I}$, so $\Ann_{\tilde{\mathbb{T}}}(\tilde{I}) \subset \tilde{I}$. This contradicts hypothesis (\ref{hyp_T/I}).
\end{proof}

By Lemma \ref{Formalism_J/J^2}, the ideal $\tilde{I}\cdot \tilde{\mathbf{T}}$ is principal. In particular, for every $n \geq 1$ the group $\tilde{I}^{n-1} \cdot (\tilde{\mathbf{T}}/p^r\cdot \tilde{\mathbf{T}})/\tilde{I}^{n} \cdot (\tilde{\mathbf{T}}/p^r\cdot \tilde{\mathbf{T}})$ is cyclic of order $ \leq p^r$. Thus, $J^{n-1}/J^n$ is cyclic of order $\leq p^r$. We let $n(r,p)$ be the largest integer $n\geq 1$ such that the group $J^{n}/J^{n+1}$ has order $p^r$. By (\ref{construction_e_i}), there exists $e_0$, $e_1$, ..., $e_{n(r,p)}$ satisfying properties (ii), and $n(r,p)$ is the largest such integer.

To prove property (a), it suffices to prove that $\tilde{e}_0 \not\in M^0$. Assume for the sake of a contradiction that $\tilde{e}_0 \in M^0$. By Lemma \ref{Formalism_T_0J}, we have $\tilde{e}_0 \in \tilde{I} \cdot M$. Using Lemma \ref{Formalism_T_0_Eisenstein}, we get $\Ann_{\tilde{\mathbb{T}}}(\tilde{I})  \cdot M \subset \tilde{I}\cdot M$. Since $M_{\tilde{\mathfrak{P}}}$ is free of rank one over $\tilde{\mathbf{T}}$, we have $\Ann_{\tilde{\mathbb{T}}}(\tilde{I})\cdot \tilde{\mathbf{T}} \subset \tilde{I}\cdot \tilde{\mathbf{T}}$. As in (\ref{Formalism_T_0_Eisenstein_eq0}), we get $\Ann_{\tilde{\mathbb{T}}}(\tilde{I}) \cdot \tilde{\mathbb{T}} \subset \tilde{I}\cdot \tilde{\mathbb{T}}$, which contradicts Hypothesis (\ref{hyp_T/I}). Property (c) follows from (\ref{construction_e_i}).

Lemma \ref{Formalism_T_0J} shows that $e_0$, ..., $e_{n(r,p)-1}$ $\in M^0/p^r\cdot M^0$. Assume $e_{n(r,p)} \not\in  M^0/p^r\cdot M^0$. By Lemma \ref{Formalism_T_0J}, we have $e_{n(r,p)} \in \tilde{I} \cdot (M/p^r\cdot M)$. Let $f_i$ be the image of $e_i$ in $M_{\tilde{\mathfrak{P}}}/p^r\cdot M_{\tilde{\mathfrak{P}}}$. There exists $f_{n(r,p)+1} \in M_{\tilde{\mathfrak{P}}}/p^r\cdot M_{\tilde{\mathfrak{P}}}$ satisfying Property (ii). By Lemma \ref{Formalism_e_0_completion}, the group $(\tilde{\mathbf{T}}/p^r\cdot \tilde{\mathbf{T}})[\tilde{I}^{n(r,p)+2}]/(\tilde{\mathbf{T}}/p^r\cdot \tilde{\mathbf{T}})[\tilde{I}^{n(r,p)+1}]$ is cyclic of order $p^r$. Thus, $J^{n(r,p)+1}/J^{n(r,p)+2}$ is cyclic of order $p^r$, which contradicts the definition of $n(r,p)$.

We finally show Property (d). The map $r \mapsto n(r,p)$ is obviously a decreasing function of $r$. Thus, it suffices to prove $n(1,p) = \rk_{\mathbf{Z}_p}(\mathbf{T})$. Let $\eta$ be a generator of $\tilde{I}\cdot \tilde{\mathbf{T}}$ and $R(X) \in \mathbf{Z}_p[X]$ be the characteristic polynomial of $\eta$ acting on the free $\mathbf{Z}_p$-module $\tilde{\mathbf{T}}$. We then have a ring isomorphism
\begin{equation}\label{Formalism_newton_R(X)_T}
\tilde{\mathbf{T}} \simeq \mathbf{Z}_p[X]/(R(X)) \text{ .}
\end{equation}
We have $\rk_{\mathbf{Z}_p}(\tilde{\mathbf{T}}) = \rk_{\mathbf{Z}_p}(\mathbf{T})+1$ by hypothesis (\ref{hyp_Ann_free_one}), so $\text{deg}(R) =  \rk_{\mathbf{Z}_p}(\mathbf{T})+1$. By Hypothesis \ref{hyp_T/I_tilde}, we have $R(0)=0$. Since $\tilde{\mathbf{T}}$ is local, we have $R(X) \equiv X^{ \rk_{\mathbf{Z}_p}(\mathbf{T})+1} \text{ (modulo }p \text{)}$. Thus, we have a ring isomorphism:
$$\tilde{\mathbf{T}}/p\cdot \tilde{\mathbf{T}} \simeq (\mathbf{Z}/p\mathbf{Z})[X]/(X^{ \rk_{\mathbf{Z}_p}(\mathbf{T})+1})$$
such that $\eta$ (modulo $p$) is sent to $X$. Thus, $J^{n}/J^{n+1}$ is isomorphic to $(X^{n})/(X^{n+1})$, which is non-zero if and only if $n \leq \rk_{\mathbf{Z}_p}(\mathbf{T})$, and also if and only if $n \leq n(1,p)$ by definition. Thus we have $n(1,p) = \rk_{\mathbf{Z}_p}(\mathbf{T})$.
This concludes the proof of Theorem \ref{Formalism_existence_eisenstein}.
\end{proof}
Another characterization of $n(r,p)$ that will be used later. 
\begin{prop}\label{Formalism_odd_modSymb_critere_higher_eisenstein}
The integer $n(r,p)$ is the largest integer $k \geq 1$ such that $p^t$ is in $I^k+p^r\cdot I$.
\end{prop}
\begin{proof}
We apply Theorem \ref{Formalism_existence_eisenstein} with $M=\tilde{\mathbb{T}}$, which obviously satisfies the hypotheses of the theorem. Let $T_0$ be a generator of the $\mathbf{Z}_p$-module $\Ann_{\tilde{\mathbb{T}}}(\tilde{I})$. By hypothesis (\ref{hyp_T/I}), we can assume that $T_0 - p^t \in \tilde{I}$. We can choose $\tilde{e}_0 = T_0$ by Lemma \ref{Formalism_T_0_Eisenstein}. Since $r \leq t$, the image of $T_0-p^t$ in $M/p^r\cdot M$ is $e_0$. By Lemma \ref{Formalism_T_0J}, we have $M^0 = \tilde{I}$. 

\begin{lem}\label{Formalism_odd_modSymb_critere_higher_eisenstein_lemma}
The integer $n(r,p)$ is the largest integer $k\geq 1$ such that $e_0 \in \tilde{I}^k \cdot (M/p^r\cdot M)$.
\end{lem}
\begin{proof}
By Theorem \ref{Formalism_existence_eisenstein}, if $\eta \in \tilde{I} \backslash \tilde{I}^2$, there exists $c \in (\mathbf{Z}/p^r\mathbf{Z})^{\times}$ such that $\eta^{n(r,p)}\cdot e_{n(r,p)} = c\cdot e_0$. Thus, we have $e_0 \in \tilde{I}^{n(r,p)}\cdot (M/p^r\cdot M)$. Assume for the sake of a contradiction that $e_0 \in \tilde{I}^{n(r,p)+1}\cdot (M/p^r\cdot M)$, then we have $\eta^{n(r,p)} \cdot e_{n(r,p)} \in \tilde{I}^{n(r,p)+1}\cdot (M/p^r\cdot M)$. The group $$\tilde{I}^{n(r,p)} \cdot \left(\tilde{\mathbb{T}}/p^r\cdot \tilde{\mathbb{T}}\right)/\tilde{I}^{n(r,p)+1} \cdot \left(\tilde{\mathbb{T}}/p^r\cdot \tilde{\mathbb{T}}\right)$$ is cyclic of order $p^r$, so we have $e_{n(r,p)} \in \tilde{I}\cdot (M/p^r\cdot M)$. By Lemma \ref{Formalism_T_0J}, we have $e_{n(r,p)} \in M^0/p^r\cdot M^0$. This contradicts Theorem \ref{Formalism_existence_eisenstein} b).
\end{proof}

By lemma \ref{Formalism_odd_modSymb_critere_higher_eisenstein_lemma}, the integer $n(r,p)$ is the largest integer $k\geq 1$ such that $T_0-p^t \in \tilde{I}^k + p^r\cdot \tilde{\mathbb{T}}$. Since $T_0-p^t \in \tilde{I}$ and $\tilde{I}\cap (p^r) = p^r\cdot \tilde{I}$, it is also the largest integer $k\geq 1$ such that $T_0-p^t \in \tilde{I}^k + p^r\cdot \tilde{I}$. This concludes the proof of Proposition \ref{Formalism_odd_modSymb_critere_higher_eisenstein} since $\Ker\left(\tilde{\mathbb{T}} \rightarrow \mathbb{T}\right) = \mathbf{Z}_p\cdot T_0$ and that the unique $\alpha \in \mathbf{Z}_p$ such that $\alpha\cdot T_0 - p^t \in \tilde{I}$ is $\alpha=1$. 
\end{proof}

\subsection{The Newton polygon of $\tilde{\mathbf{T}}$}\label{Formalism_Newton_polygon}
As in (\ref{Formalism_newton_R(X)_T}), there is a (non-canonical) ring isomorphism
$$\tilde{\mathbf{T}} \simeq \mathbf{Z}_p[X]/(R(X))$$
for some monic polynomial $R = \sum_{i=0}^{\rk_{\mathbf{Z}_p}(\tilde{\mathbf{T}})+1} a_i \cdot X^i \in \mathbf{Z}_p[X]$. Recall that the \textit{Newton polygon} of $R$ is the lower convex hull of the points $\{(i, v_p(a_i)), \text{ }i \in \{0, ..., \rk_{\mathbf{Z}_p}(\tilde{\mathbf{T}})+1\}\}$ (where $v_p$ is the usual $p$-adic valuation and we omit the points with $a_i=0$). The Newton polygon of $R$ only depends on $\tilde{\mathbf{T}}$. The Newton polygon of $\tilde{\mathbf{T}}$ is by definition the Newton polygon of $R$, and we denote it by $\NP(\tilde{\mathbf{T}})$. 

We now recall a finer invariant of $\tilde{\mathbf{T}}$ than $\NP(\tilde{\mathbf{T}})$, introduced in by Wake--Wang-Erickson in \cite[Section 8]{Wake}. By convention, we let $v_p(0) = \infty$. Define a sequence $z_0$, ..., $z_{\rk_{\mathbf{Z}_p}(\tilde{\mathbf{T}})+1}$ inductively by $z_0 = v_p(a_0)$ and $z_i = \min(z_{i-1}, v_p(a_i))$ (by convention, $\min(z, \infty)=z$ for any $z \in \mathbf{Z} \cup \{\infty\}$). One easily sees that $\NP(\tilde{\mathbf{T}})$ is the lower convex hull of the points $\{(i, z_i), \text{ } i\in \{0, ...,  \rk_{\mathbf{Z}_p}(\tilde{\mathbf{T}})+1\}\}$ (we omit the points with $z_i = \infty$). Let $ i \in  \{0, ..., \rk_{\mathbf{Z}_p}(\tilde{\mathbf{T}})+1\}$. Then $z_i$ is the supremum of the set of integers $r\geq 1$ such that there exists a surjective ring homomorphism
$$f: \mathbf{Z}_p[X]/(R(X)) \twoheadrightarrow (\mathbf{Z}/p^r\mathbf{Z})[x]/(x^{i+1})$$
such that $f(X) = x$ (this supremum can be $\infty$)  \cite[Lemma 8.1.2]{Wake}. We choose $R$ so that $X$ corresponds to a generator of $\tilde{I}\cdot \tilde{\mathbf{T}}$ via (\ref{Formalism_newton_R(X)_T}). By Hypothesis (\ref{hyp_T/I_tilde}) we have $a_0 = 0$. By Hypothesis (\ref{hyp_T/I}) and Lemma \ref{Formalism_J/J^2}, we have $v_p(a_1) = t$. Thus, we have $z_0 = \infty$ and $z_i \leq t$ for all $i \geq 1$.

By Theorem \ref{Formalism_existence_eisenstein} d), one easily sees that for all integer $r$ such that $1 \leq r \leq t$, the integer $n(r,p)$ is the largest integer $i \geq 1$ such that there exists a surjective ring homomorphism
$$f': \mathbf{Z}_p[X]/(R(X)) \twoheadrightarrow (\mathbf{Z}/p^r\mathbf{Z})[x]/(x^{i+1}) $$
such that $f'(X) = x$. Thus, for all $r \in \{1, ..., t\}$ and $i \in\{1, ..., \rk_{\mathbf{Z}_p}(\tilde{\mathbf{T}})+1\}$, we have:
$$n(r,p) = \max\{j \in \{1, ..., \rk_{\mathbf{Z}_p}(\tilde{\mathbf{T}})+1\}, z_j \geq r \} $$
and
$$z_i = \max\{r' \in \{1, ..., t\}, n(r',p) \geq i \} \text{ .}$$
Hence, knowing the sequence $(n(r,p))_{1 \leq r \leq t}$ amounts to knowing the sequence $(z_i)_{0 \leq i \leq \rk_{\mathbf{Z}_p}(\tilde{\mathbf{T}})+1}$. Thus, the sequence $(n(r,p))_{1 \leq r \leq t}$ is a finer invariant of $\tilde{\mathbf{T}}$ than $\NP(\tilde{\mathbf{T}})$ (\cf \cite[8.1.3]{Wake} for an example where the Newton polygon does not determine the $(z_i)_{0 \leq i \leq \rk_{\mathbf{Z}_p}(\tilde{\mathbf{T}})+1}$).

\subsection{Pairing between higher Eisenstein elements}
The proof of the following is easy and left to the reader.
\begin{prop}\label{Formalism_pairing}
Let $M$ and $M'$ be two $\tilde{\mathbb{T}}$-modules satisfying the assumptions of Theorem \ref{Formalism_existence_eisenstein}. Let $1 \leq r \leq t$ be an integer. Let $e_0$, ..., $e_{n(r,p)}$ (resp. $e_0'$, ..., $e_{n(r,p)}'$) be the elements of $M/p^r\cdot M$ (resp. $M'/p^r\cdot M'$) constructed in Theorem \ref{Formalism_existence_eisenstein}. Let 
$$ \bullet : M \times M' \rightarrow \mathbf{Z}/p^r\mathbf{Z}$$
be a perfect $\tilde{\mathbb{T}}$-equivariant bilinear pairing. Let $i \in \{0, ..., n(r,p)\}$. Then an element $m$ (resp. $m'$) of $\tilde{I}^i\cdot (M/p^r\cdot M)$ (resp. $\tilde{I}^i\cdot (M'/p^r\cdot M')$) is in $\tilde{I}^{i+1} \cdot (M/p^r\cdot M)$ (resp. $\tilde{I}^{i+1}\cdot (M'/p^r\cdot M')$) if and only if $m \bullet e_i' \equiv 0 \text{ (modulo }p^r\text{)}$ (resp. $e_i \bullet m' \equiv 0 \text{ (modulo }p^r\text{)}$). 
\end{prop}

\begin{corr}\label{Formalism_criterion_pairing}
Keep the notation of Proposition \ref{Formalism_pairing} and consider the elements $e_0$, ..., $e_{n(r,p)}$, $e_0'$, ..., $e_{n(r,p)}'$ modulo $p^r$. The product
$$e_i \bullet e_j'$$
only depends on $i+j$, is zero modulo $p^r$ if $i+j < n(r,p)$ and non-zero modulo $p^r$ if $i+j=n(r,p)$. (Note that if $i+j>n(r,p)$, the product is not well-defined since $e_i$ is only defined modulo $e_0$, ..., $e_{i-1}$ and the same for $e_j'$). 

In particular, we have $n(r,p) \geq 2$ if and only if $e_1 \bullet e_0' \equiv 0 \text{ (modulo }p \text{)}$, and $\rk_{\mathbf{Z}_p}(\mathbb{T}) \geq 3$ if and only if $e_1 \bullet e_0' \equiv e_1 \bullet e_1' \equiv 0 \text{ (modulo }p \text{)}$.
\end{corr}

\subsection{The special case of weight $2$ and prime level}\label{Formalism_special_case_modular_forms}

In this section, as well as in the rest of the article, let $\tilde{\mathbb{T}}$ be the the Hecke algebra with $\mathbf{Z}_p$ coefficients acting faithfully on the space of  modular forms of weight $2$ and level $\Gamma_0(N)$. The ideal $\tilde{I}$ is the \textit{Eisenstein ideal}, generated by the Hecke operators $T_{\ell}-\ell-1$ if $\ell \neq N$ and by $U_N-1$. We easily see that $\mathbb{T}$ is the Hecke algebra acting faithfully on the space of cuspidal modular forms of weight $2$ and level $\Gamma_0(N)$. 

We now check the hypotheses of section \ref{Section_Formalism}. 
Hytothesis (\ref{hyp_T/I_tilde}) is obvious. Hypothesis (\ref{hyp_Ann_free_one}) follows from the fact that there is a unique Eisenstein series of weight $2$ and level $\Gamma_0(N)$ since $N$ is prime. Hypothesis (\ref{hyp_T_free}) is obvious. Hypothesis (\ref{hyp_T/I}) follows from \cite[Proposition II.9.7]{Mazur_Eisenstein}. Hypothesis (\ref{hyp_I/I^2}) follows from \cite[Propositions II.18.7 and II.18.10]{Mazur_Eisenstein}. Hypothesis (\ref{hyp_gorenstein}) follows from \cite[Corollary II.16.3]{Mazur_Eisenstein}. Thus, we can apply Theorem \ref{Formalism_existence_eisenstein}.  As in \cite[Proposition $1.8$ (ii)]{Emerton_supersingular}, we let $T_0 \in \tilde{\mathbb{T}}$ be such that $\Ann_{\tilde{\mathbb{T}}}(\tilde{I}) = \mathbf{Z}_p\cdot T_0$ and $T_0-\frac{N-1}{\nu}\in \tilde{I}$ (where $\nu = \gcd(N-1, 12)$).

By \cite[Proposition 18.9]{Mazur_Eisenstein}, there is a group isomorphism $e : I/I^2  \xrightarrow{\sim} \mathbf{Z}/p^t\mathbf{Z}$ such that for all prime $\ell$ not dividing $N$, we have:
$$e(T_{\ell}-\ell-1) = \frac{\ell-1}{2}\cdot \log(\ell) $$
where, if $p=\ell=2$, this equality means
$$e(T_2-3) = \log(x)$$
where $x^2 \equiv 2$ (modulo $N$).

\begin{rem}
In order to study $g_p$, only higher Eisenstein elements \textit{modulo} $p$ (and not modulo $p^r$ for $r>1$) are important.
\end{rem}

In practice, when we construct the $e_i$'s, we need only to check condition (ii) of Theorem \ref{Formalism_existence_eisenstein} for Hecke operators $T_{\ell}-\ell-1$ for primes $\ell$ outside a finite set.

\begin{prop}\label{Formalism_property_Hecke_outside}
We keep the notation of Theorem \ref{Formalism_existence_eisenstein}. Let $S$ be a finite set of rational primes containing $N$. Let $i$ be an integer such that $1 \leq i \leq n(r,p)$ and $m \in M/p^r\cdot M$ be such that for all $\ell \not\in S$ prime, we have:
$$(T_{\ell}-\ell-1)(m) = \frac{\ell-1}{2}\cdot \log(\ell)\cdot e_{i-1} \text{ .}$$
Then we have 
$$m \equiv e_i \text{ (modulo the subgroup generated by }e_0\text{, ..., }e_{i-1}\text{).}$$
\end{prop}
\begin{proof}
The element $m-e_i\in M/p^r$ is annihilated by $T_{\ell}-\ell-1$ for all prime $\ell \not\in S$. It thus suffices to prove the following result.
\begin{lem}\label{Formalism_property_Hecke_outside_lem}
Let $W \subset M/p^r\cdot M$ be set of elements annihilated by $T_{\ell}-\ell-1$ for $\ell$ prime not in $S$. Then $W = (\mathbf{Z}/p^r\mathbf{Z}) \cdot e_0$.
\end{lem}
\begin{proof}
We prove it by induction on $r$. Let $T$ be the $\mathbf{Z}_p$-subalgebra of $\tilde{\mathbb{T}}$ generated by the operators $T_{\ell}$ for $\ell \not\in S$. 

Assume first that $r=1$.
Let $w \in W$. By the Deligne--Serre lifting lemma \cite[Lemme 6.11]{Deligne_Serre}, there is a finite extension $\mathcal{O} \subset \overline{\mathbf{Q}}_p$ of $\mathbf{Z}_p$ ($\mathcal{O}$ is a discrete valuation ring) and an element $\tilde{w} \in M \otimes_{\mathbf{Z}} \mathcal{O}$ which is proper for the action of $T$ and such that the eigenvalue of $\tilde{w}$ for $T_{\ell}$ modulo $\pi$ (an uniformizer of $\mathcal{O}$) is the eigenvalue of $w$ for $T_{\ell}$ if $\ell \not\in S$.

The strong multiplicity one theorem shows that $T \otimes_{\mathbf{Z}_p} \overline{\mathbf{Q}}_p = \mathbb{T}\otimes_{\mathbf{Z}_p} \overline{\mathbf{Q}}_p$.  Thus, there is a ring homomorphism
$$\varphi : \mathbb{T} \rightarrow \overline{\mathbf{Q}}_p$$
such that for all $t \in T$, $\varphi(t)$ is the eigenvalue of $t$ on $\tilde{w}$. There is an associated semi-simple Galois representation $\rho_{\varphi} : \Gal(\overline{\mathbf{Q}}/\mathbf{Q}) \rightarrow \text{GL}_2(\overline{\mathbf{Q}}_p)$. The associated semi-simple residual representation must be the direct sum of the trivial character and the Teichm\"uller character. This shows that, in fact, for all prime $\ell$ not dividing $N$, $T_{\ell}-\ell-1$ annihilates $w$. We conclude that $w$ and $e_0$ are proportional modulo $p$, which concludes the case $r=1$.

Now, let $r \geq 1$ be any integer. By the case $r=1$, there exists $\lambda \in \mathbf{Z}$ such that
$$w - \lambda \cdot e_0 \in p\cdot M/p^r \cdot M \text{ .}$$
The the element $$\frac{x-\lambda\cdot e_0}{p} \in M/p^{r-1}\cdot M$$
is annihilated by $T_{\ell}-\ell-1$ for all $\ell \not\in S$. By induction this concludes the proof of Lemma \ref{Formalism_property_Hecke_outside_lem}.

\end{proof}
\end{proof}

\section{The supersingular module}\label{Section_Supersingular}
We keep the notation of Chapters \ref{Section_introduction} and \ref{Section_Formalism}. In particular, $p \geq 2$ is a prime such that $p^t$ divides exactly the numerator of $\frac{N-1}{12}$. Let $r$ be an integer such that $1 \leq r \leq t$.

\subsection{Preliminary results and notation}\label{Section_Supersingular_intro_SS}

Let $S$ be the set of isomorphism classes of supersingular elliptic curves over $\overline{\mathbf{F}}_N$, and $M := \mathbf{Z}_p[S]$ be the free $\mathbf{Z}_p$-module with basis the elements of $S$. If $E \in S$, we denote by $[E]$ the corresponding element in $M$ and we let $w_E = \frac{\Card(\Aut(E))}{2} \in \{1,2,3\}$. The ring $\tilde{\mathbb{T}}$ acts on $M$ in the following way. If $n\geq 1$, we have $$T_n([E]) = \sum_{C} [E/C]$$
where $C$ goes through the cyclic subgroup schemes of order $n$ in $E$. If $n$ and $N$ are coprimes, these subgroup schemes correspond to the subgroups of $E(\overline{\mathbf{F}}_N)$ isomorphic to $\mathbf{Z}/n\mathbf{Z}$. We also have 
$$U_N([E]) = [E^{(N)}]$$
where $U_N$ is the $N$th Hecke operator and $E^{(N)}$ is the Frobenius transform of $E$.

\begin{thm}\cite[Theorem 0.5]{Emerton_supersingular}\label{Supersingular_Emerton_supersingular}
The $\tilde{\mathbf{T}}$-module $M \otimes_{\tilde{\mathbb{T}}} \tilde{\mathbf{T}}$ is free of rank one.
\end{thm}

Thus, we can apply Theorem \ref{Formalism_existence_eisenstein} to $M$. One easily sees that $M^0$ is the augmentation subgroup of $M$. There is a $\tilde{\mathbb{T}}$-equivariant bilinear pairing
$$\bullet : M \times M \rightarrow \mathbf{Z}_p$$
given by $[E] \bullet [E'] = 0$ if $[E] \neq [E']$ and $[E] \bullet [E] = w_E$.

We let
$$\tilde{e}_0 = \sum_{E \in S} \frac{1}{w_E}\cdot [E] \in \mathbf{Q}_p[S] \text{ .}$$
The following result is well-known.
\begin{prop}
\begin{enumerate}
\item We have $\tilde{e}_0 \in M[\tilde{I}]$.
\item We have:
$$\tilde{e}_0 \bullet \tilde{e}_0 = \frac{N-1}{12} \text{ .}$$
\end{enumerate}
\end{prop}
\begin{proof}
We first check that $\tilde{e}_0 \in M$. This is obvious if $p \geq 5$, so we need to check it when $p \in \{2,3\}$. If $p=2$ (resp. if $p=3$), there is some $E$ in $S$ with $w_E=2$  (resp. $w_E=3$) if and only if $4$ divides $N-3$ (resp. $3$ divides $N-2$). Since $p$ divides the numerator of $\frac{N-1}{12}$, we have $\tilde{e}_0 \in M$. 

The fact that $\tilde{e}_0$ is annihilated by $\tilde{I}$ is an easy and well-known computation. The last assertion is equivalent to \textit{Eichler mass formula}:
\begin{equation}\label{Supersingular_Eichler_mass_formula_eq}
\sum_{E \in S} \frac{1}{w_E} = \frac{N-1}{12} \text{ .}
\end{equation}
\end{proof}
We let $e_0$ be the image of $\tilde{e}_0$ in $M/p^r\cdot M$. We denote by $e_0$, $e_1$, ..., $e_{n(r,p)}$ the higher Eisenstein elements in $M/p^r\cdot M$. In this chapter, we give an explicit formula for $e_1$. 

One idea would be to consider the element whose coefficient in $[E]$ is $\log(S'(j(E)))$, where $S(X) \in \mathbf{F}_N[X]$ is the monic polynomial whose roots are the $j$-invariants of elements in $S$. It turns out that this element is not $e_1$. For reasons we do not fully understand, this idea leads us to $e_1$ after adding an auxiliary $\Gamma(2)$-structure. In other words, we replace the $j$-invariants by the Legendre $\lambda$-invariants. The analogue of the polynomial $S$ in this context is the well-known Hasse polynomial. In the next section, we mainly study the relation between the Hasse polynomial and the Hecke operators.

\subsection{The Hasse polynomial}\label{Supersingular_section_Hasse}
Recall the definition of the \textit{Hasse polynomial}:
$$H(X) = \sum_{i=0}^{\frac{N-1}{2}} {{\frac{N-1}{2}}\choose{i}}^2\cdot X^i \in \mathbf{F}_N[X] \text{ .}$$
Let
$$P(Y) = \Res_X(H'(X), 256\cdot(1-X+X^2)^3-X^2\cdot (1-X)^2\cdot Y) \in \mathbf{F}_N[Y] \text{ .}$$

\subsubsection{The discriminant of the Hasse polynomial}
\label{Supersingular_discriminant_hasse}

The computation of $e_1 \bullet e_0$ is related to the discriminant $\Disc(H)$ of the Hasse polynomial.

\begin{thm}\label{Supersingular_discriminant_thm}
Let $m = \frac{N-1}{2}$. We have, in $\mathbf{F}_N$:
$$\Disc(H) = \frac{(-1)^{\frac{m(m-1)}{2}}}{m!} \cdot \prod_{k=1}^{m} k^{4k} \text{ .}$$
In particular, we have $\Disc(H) \neq 0$, so the roots of $H$ are simples.
\end{thm}
\begin{proof}
The following \textit{Picard--Fuchs} differential equation is well-known (\cf the proof of \cite[Theorem V.4.1]{Silverman}).
\begin{lem}
Let $A=4X(1-X)$ and $B=4(1-2X)$. Then for all $n \geq 0$, we have the following differential equation:
\begin{equation}\label{diff_equation}
A\cdot H^{(n+2)}+(n+1)\cdot B\cdot H^{(n+1)}-(2n+1)^2\cdot H^{(n)}=0
\end{equation}
where $H^{(n)}$ is the $n$th derivative of $H$.
\end{lem}
For all $n \geq 0$, let $r_n = \Res(H^{(n)}, H^{(n+1)})$, where $\Res$ is the resultant. Note that $r_0 = (-1)^{\frac{m(m-1)}{2}}\text{Disc}(H)$.  We can rewrite the differential equation (\ref{diff_equation}) as
$$A\cdot H^{(n+2)}+(n+1)\cdot B\cdot H^{(n+1)}=(2n+1)^2\cdot H^{(n)}$$
and then take the resultant of both sides with $H^{(n+1)}$. 
Thus, for all $0 \leq n \leq m-2$, we have: 
$$r_n = (-4)^{m-n-1}\cdot (2n+1)^{2n+2} \cdot H^{(n+1)}(0)\cdot H^{(n+1)}(1) \cdot r_{n+1} \text{ .}$$
Note that $r_{m-1} = \Res(H^{(m-1)}, m!) = m!$. We have to compute $a_n:=H^{(n)}(0)$ and $b_n:=H^{(n)}(1)$. We immediately see that $a_n = n!\cdot {m \choose n}^2$. We can again use the differential equation (\ref{diff_equation}) to compute $b_n$. For all $0 \leq n \leq m-1$, we get: 
$$b_{n} = -\frac{4(n+1)}{(2n+1)^2}\cdot b_{n+1} \text{ .}$$
Furthermore, we have $b_{m} = m!$.
We thus have, for all $0 \leq n \leq m$:
$$b_n = m!\cdot \prod_{i=n}^{m-1}\frac{-4\cdot(i+1)}{(2i+1)^2} \text{ .}$$
Thus, we have:
$$\prod_{i=1}^{m-1}H^{(i)}(0)\cdot H^{(i)}(1) = m!^{m-1}\cdot \prod_{i=1}^{m-1}i!\cdot {m \choose i}^2\cdot \left(\frac{-4\cdot (i+1)}{(2i+1)^2}\right)^i \text{ .}$$
This gives us, after simplification:
$$\Res(H,H') = r_0 = m!^m \cdot \prod_{i=1}^{m-1} (2i+1)^2\cdot (i+1)^i \cdot i! \cdot {m \choose i}^2 \text{ .}$$
A direct computation finally shows that:
$$\Res(H,H') = m!^{-1}\cdot \prod_{k=1}^m k^{4k} \text{ .}$$
This concludes the proof of Theorem \ref{Supersingular_discriminant_thm}.
\end{proof}

\subsubsection{Relation between the Hasse polynomial and the modular polynomials}

We denote by $L$ the set of roots of $H$ in $\overline{\mathbf{F}}_N$. By Theorem \ref{Supersingular_discriminant_thm}, the roots of $H$ are simples, so $\Card(L) = \frac{N-1}{2}$. 

Let $L'$ be the set of isomorphism classes of triples $(E, P_1, P_2)$ where $E$ is a supersingular elliptic curve over $\overline{\mathbf{F}}_N$ and $(P_1, P_2)$ is a basis of the $2$-torsion $E[2](\overline{\mathbf{F}}_N)$. Each such triple is isomorphic to a unique triple of the form $(E_{\lambda} : y^2=x(x-1)(x-\lambda), (0,0),(1,0))$ where $\lambda \in \overline{\mathbf{F}}_N \backslash \{0,1\}$. We call $\lambda$ the \textit{lambda-invariant} of $(E, P_1, P_2)$. The relation between the $j$-invariant and the lambda-invariant is:
\begin{equation}\label{Supersingular_relation_j_lambda_eq}
j = \frac{256\cdot (1-\lambda+\lambda^2)^3}{\lambda^2\cdot(1-\lambda)^2} \text{ .}
\end{equation} 

We identify $L'$ with the set of supersingular lambda-invariants. It is well-known that $\lambda \in L'$ if and only if $H(\lambda)=0$ \cite[Theorem V.4.1(b)]{Silverman}. Thus, one can (and do) identify $L'$ and $L$. Since the supersingular Legendre elliptic curves in characteristic $N$ are defined over $\mathbf{F}_{N^2}$, we have $L \subset \mathbf{F}_{N^2}$. 

If $\lambda$ and $\lambda'$ are in $\overline{\mathbf{F}}_N \backslash \{0,1\}$ (not necessarily in $L$) and $\ell \geq 3$ is a prime number different from $N$, we say that $\lambda$ and $\lambda'$ are $\ell$\textit{-isogenous}, and we write $\lambda \sim_{\ell} \lambda'$, if there is a rational isogeny $\phi : E_{\lambda} \rightarrow E_{\lambda'}$ of degree $\ell$ preserving the $\Gamma(2)$-structure, \ie $\phi(0,0) = (0,0)$ and $\phi(1,0)=(1,0)$.

It is well-known (\cf for instance \cite[Proposition 2.2]{Lecouturier_Betina_2}) that there exists a unique polynomial $\varphi_{\ell} \in \mathbf{Z}[X,Y]$ which is monic in $X$ and $Y$, and such that we have $\lambda \sim_{\ell} \lambda'$ if and only if $\varphi_\ell(\lambda, \lambda')=0$.
The polynomial $\varphi_{\ell}(X,Y)$ is the ``Legendre version'' of the classical $\ell$th modular polynomial. 

The following result is the key to compute $e_1$.

\begin{thm}\label{Supersingular_main_thm}
Let $\ell$ be a prime number not dividing $2\cdot N$ and let $\lambda \in L$. We have, in $\mathbf{F}_{N^2}$:
\begin{equation}\label{Supersingular_main_thm_eq}
\prod_{\lambda' \sim_{\ell} \lambda} H'(\lambda') = \ell^{\ell-1}\cdot H'(\lambda)^{\ell+1} \text{ .} 
\end{equation}
\end{thm}
\begin{rem}
Equation (\ref{Supersingular_main_thm_eq}) makes sense even if $\lambda \in \overline{\mathbf{F}}_N \backslash \left(\{0,1\} \cup L\right)$, but it does not hold in general. The supersingularity of $\lambda$ is used crucially in Lemma \ref{Supersingular_Robert} below.
\end{rem}
\begin{proof}
The idea is to reinterpret $H(\lambda)$ and $H'(\lambda)$ as modular forms of level $1$, and more precisely as Eisenstein series. If $k \geq 2$ is an integer, let
$$E_k = 1+\frac{(-1)^{k-1}\cdot 4k}{B_k}\cdot \sum_{n\geq 1} \sigma_{k-1}(n)q^n$$ be the unique Eisenstein series of weight $k$ and level $\SL_2(\mathbf{Z})$ with constant coefficient $1$.

Recall that a modular form of level $\SL_2(\mathbf{Z})$ with coefficients in an  $\mathbf{F}_N$-algebra $A$ can be considered, following Deligne and Katz \cite[p. 77-78]{Katz_properties}, as a rule on pairs $(E, \omega)$ were $E$ is an elliptic curve over $A$ and $\omega$ is a differential form generating $H^0(E, \Omega^1)$ over $A$. 

Recall that if $\lambda \in \overline{\mathbf{F}}_N \backslash \{0,1\}$, we have denoted by $E_{\lambda}$ the Legendre elliptic curve $y^2=x(x-1)(x-\lambda)$. We let $\omega_{\lambda} = \frac{dx}{y}$. The following fact, essentially due du Katz, is crucial to interpret $H'(\lambda)$.
\begin{lem}[Katz]\label{Supersingular_Katz}
Let $m = \frac{N-1}{2}$.
For all $\lambda \in \overline{\mathbf{F}}_N \backslash \{0,1\}$, we have the following equality in $\overline{\mathbf{F}}_N$:
$$E_{N+1}(E_{\lambda}, \omega_{\lambda}) = \frac{3}{m+1} \cdot \left(m\cdot H(\lambda) - (\lambda-1) \cdot H'(\lambda)\right) + (\lambda+1)\cdot H(\lambda) \text{ .}$$
\end{lem}
\begin{proof}
Recall the following (unpublished) theorem of Katz \cite[Theorem $3.1$]{Katz}. Let $(E,\omega)$ be a pair of the form $\left(y^2=4x^3-g_2x-g_3, \frac{dx}{y}\right)$. Then, $\frac{1}{12} \cdot E_{N+1}(E, \omega)$ is the coefficient of $x^{N-2}$ in the polynomial $\left(4x^3-g_2x-g_3\right)^{\frac{N-1}{2}}$.
  
Let $\lambda \in \overline{\mathbf{F}}_N \backslash \{0,1\}$, and $f(x) = 4x(x-1)(x-\lambda)$. The pair $\left(y^2=f(x), \frac{dx}{y}\right)$ is isomorphic to the pair $\left(y^2=g(x), \frac{dx}{y}\right)$ with $g(x) = f(x+c)=4x^3-g_2x-g_2$ (for some $g_2$ and $g_3$), where $c = \frac{\lambda+1}{3}$. 

We have $$E_{N+1}(E_{\lambda}, \omega_{\lambda}) = E_{N+1}\left(y^2=f(x), 2\cdot \frac{dx}{y}\right) = 2^{-(N+1)}\cdot E_{N+1}\left(y^2=g(x), \frac{dx}{y}\right) \text{, }$$
so $\frac{2^{N+1}}{12}\cdot E_{N+1}\left(E_{\lambda}, \omega_{\lambda}\right)$ is the coefficient of $x^{N-2}$ in $g(x)^{\frac{N-1}{2}}=f(x+c)^{\frac{N-1}{2}}$. We have$$\left(f^{\frac{N-1}{2}}\right)^{(N-2)}(c) = \left(f^{\frac{N-1}{2}}\right)^{(N-2)}(0)+c\cdot \left(f^{\frac{N-1}{2}}\right)^{(N-1)}(0)$$
where $\left(f^{\frac{N-1}{2}}\right)^{(k)}$ is the $k$th derivative of $f^{\frac{N-1}{2}}$, since $(f^{\frac{N-1}{2}})^{(k)}=0$ if $k\geq N$. A direct computation finally shows that
$$\frac{2^{N+1}}{12}\cdot E_{N+1}(E_{\lambda}, \omega_{\lambda}) = \frac{2^{N-1}}{m+1}\cdot \left(m\cdot H(\lambda)-(\lambda-1)\cdot H'(\lambda)\right) + \frac{\lambda+1}{3} \cdot 2^{N-1} \cdot H(\lambda) \text{ .}$$ 
\end{proof}

Let $\ell$ be a prime number not dividing $2\cdot N$. If $\varphi: E_{\lambda} \rightarrow E_{\lambda'}$ is an isogeny of degree $\ell$ preserving the $\Gamma(2)$-structure then
$$\varphi^*(\omega_{\lambda'}) = c_{\varphi}\cdot \omega_{\lambda}$$
for  some
$c_{\varphi} \in \overline{\mathbf{F}}_N^{\times}$. 
Note that there are only two $\ell$-isogenies $E_{\lambda} \rightarrow E_{\lambda'}$. These are $\varphi$ and $-\varphi$. Consequently, $c_{\varphi}^2$ only depends on $\lambda$ and $\lambda'$ (satisfying $\varphi_{\ell}(\lambda, \lambda')=0$). In fact one can show that 
\begin{equation}\label{Supersingular_c_phi_lambda_lambda'_eq_0}
c_{\varphi}^2=:c(\lambda, \lambda')
\end{equation}
is a rational function of $\lambda$ and $\lambda'$. More precisely, we have \cite[Lemma 8]{Madsen}:
\begin{equation}\label{Supersingular_c_phi_lambda_lambda'_eq}
c_{\varphi}^2 = -\ell \cdot \frac{\lambda(1-\lambda)}{\lambda'(1-\lambda')}\cdot \frac{\partial_Y \varphi_{\ell}(\lambda', \lambda) }{\partial_X \varphi_{\ell}(\lambda', \lambda) } \text{ .}
\end{equation}
Note that this equality is true in $\mathbf{Q}(\lambda, \lambda')$ and is deduced in $\mathbf{F}_N(\lambda, \lambda')$ by reducing modulo $N$ our objects (elliptic curves, isogenies...).
\begin{lem}\label{Supersingular_c_phi}
We have, over any field of characteristic $\neq 2$, without assuming that $E_{\lambda}$ is supersingular:
$$\prod_{\varphi} c_{\varphi}^2 = \ell^2$$
where the product is over a choice of non isomorphic degree $\ell+1$ isogenies with origin $E_{\lambda}$ and preserving the $\Gamma(2)$-structure. Equivalently, we have the following equality in $\mathbf{Q}(\lambda)$:
$$\frac{\Res_X(\partial_Y \varphi_{\ell}(X, \lambda), \varphi_{\ell}(X,\lambda))}{\Res_X(\partial_X \varphi_{\ell}(X, \lambda), \varphi_{\ell}(X,\lambda))} = \ell^{-(\ell-1)}\text{ .}$$
\end{lem}
\begin{proof}
\begin{lem}\label{Supersingular_identity_product_lambda}
Let $F$ be a separably closed field of characteristic different from $2$. For all $\lambda \in F \backslash \{0,1\}$, we have in $F$:
\begin{equation}\label{identity_product_lambda_1}
\prod_{\lambda' \sim_{\ell} \lambda} \lambda' = \lambda^{\ell+1}
\end{equation}
and
\begin{equation}\label{identity_product_lambda_2}
\prod_{\lambda' \sim_{\ell} \lambda} (1-\lambda') = (1-\lambda)^{\ell+1} \text{ .}
\end{equation}
\end{lem}
\begin{proof}
This is well-known, \cf for instance \cite[p. 122]{Pi_AGM}.
\end{proof}

The equivalence between the two equalities comes from the relation $$\prod_{\varphi} c_{\varphi}^2 = \ell^{\ell+1}\cdot \frac{\Res_X(\partial_Y \varphi_{\ell}(X, \lambda), \varphi_{\ell}(X,\lambda))}{\Res_X(\partial_X \varphi_{\ell}(X, \lambda), \varphi_{\ell}(X,\lambda))} $$ which is a direct consequence of (\ref{Supersingular_c_phi_lambda_lambda'_eq}) and Lemma \ref{Supersingular_identity_product_lambda}. Let $F_{\ell}(\lambda) = \prod_{\varphi} c_{\varphi}^2$. This is a rational fraction in $\lambda$, which has no poles or zeros outside $0$ and $1$. Thus, we have $F_{\ell}(\lambda) = C_{\ell}\cdot \lambda^{n_{\ell}}\cdot(\lambda-1)^{m_{\ell}}$ where $n_{\ell}$, $m_{\ell}$ $\in \mathbf{Z}$ and $C_{\ell} \in \overline{\mathbf{F}}_N^{\times}$.

The relations $$\varphi_{\ell}(X,Y) = X^{\ell+1}Y^{\ell+1}\cdot \varphi_{\ell}(X^{-1},Y^{-1})$$ and $$\prod_{\lambda' \text{ s.t. } \varphi_{\ell}(\lambda', \lambda)=0} \lambda' = \lambda^{\ell+1}$$ show that $F_{\ell}(\lambda) = F_{\ell}(\lambda^{-1})$.
This last relation shows that $n_{\ell} = -(n_{\ell}+m_{\ell})$.
One can easily show that $\varphi_{\ell}(1-X, 1-Y) = \varphi_{\ell}(X,Y)$. This is deduced for example from the fact that $\lambda\left(\frac{-1}{z}\right) = 1-\lambda(z)$ if $z$ is in the complex upper-half plane. The relation $\varphi_{\ell}(1-X, 1-Y) = \varphi_{\ell}(X,Y)$ shows that $F_{\ell}(1-\lambda) = F_{\ell}(\lambda)$, so $n_{\ell}=m_{\ell}$. Since $n_{\ell}=-(n_{\ell}+m_{\ell})$, we have $n_{\ell} = m_{\ell} = 0$, so $F_{\ell}$ is a constant. 

In order to show that $C_{\ell} = \ell^2$, it suffices to prove it over $\mathbf{C}$ using the $q$-expansions, where $q = e^{i\pi z}$. More precisely, let $\lambda = \lambda(q)$ be fixed, where $q = e^{i \pi z}$. The $\ell+1$ roots of $\varphi_{\ell}(\lambda, X)$ are the $\lambda_i = \lambda\left(e^{\frac{i\pi\cdot(z+2i)}{\ell}}\right)$ where $i \in \{0,1,..., \ell-1\}$ and $\lambda_{\ell} = \lambda(q^{\ell})$. We thus have
$$C_{\ell} = c(\lambda, \lambda_{\ell})^2\cdot \prod_{i=0}^{\ell-1}c(\lambda, \lambda_i)^2 $$
where $c$ was defined in (\ref{Supersingular_c_phi_lambda_lambda'_eq_0}).

This last factor is a constant function of $z$. Note that $c(\lambda, \lambda_{\ell}) = \frac{\ell}{c(\lambda_{\ell}, \lambda)}$. By \cite[(4.6.1) p. 137]{Pi_AGM}, for all $i \in \{0,1,..., \ell-1\}$ we have:
$$c(\lambda, \lambda_i) = \frac{\theta_3\left(e^{\frac{i\pi\cdot(z+2i)}{\ell}} \right)^2}{\theta_3\left(e^{i\pi z} \right)^2} \text{ , }$$
where $\theta_3(q) = \sum_{n\in \mathbf{Z}} q^{n^2}$.
Thus, $c(\lambda, \lambda_{i})$ goes to $1$ when $z$ goes to $i\infty$. Similarly, $c(\lambda_{\ell}, \lambda)$ goes to $1$ when $z$ goes to $i\infty$. This shows $C_{\ell}=\ell^2$.
\end{proof}

In order to conclude the proof of Theorem \ref{Supersingular_main_thm}, we need one last result, which is only true in characteristic $N$ and which uses the supersingularity is used in an essential way.
\begin{lem}\cite[Th\'eor\`eme B]{Robert}\label{Supersingular_Robert}
Let $\lambda \in L$ and $\varphi: E_{\lambda} \rightarrow E_{\lambda'}$ be an isogeny of degree $\ell$. We have:
$$E_{N+1}(E_{\lambda'}, \omega_{\lambda'}) = \ell\cdot E_{N+1}(E_{\lambda}, \varphi^*(\omega_{\lambda'})) \text{ .}$$
\end{lem}

Using Lemma \ref{Supersingular_Katz} and (\ref{identity_product_lambda_2}), we get
$$\prod_{\lambda' \sim_{\ell} \lambda} H'(\lambda') = \ell^{\ell+1}\cdot \prod_{\varphi} c_{\varphi}^{-(N+1)} \cdot H(\lambda)^{\ell+1} \text{ .}$$
Lemma \ref{Supersingular_c_phi} shows that $\prod_{\varphi} c_{\varphi}^{N+1} = (\ell^2)^{\frac{N+1}{2}} = \ell^2$, which concludes the proof of Theorem \ref{Supersingular_main_thm}.
\end{proof}

We let
\begin{equation}\label{Supersingular_modular_polynomial_varphi_2}
\varphi_2(X,Y) := Y^2\cdot (1-X)^2 + 16\cdot X\cdot Y - 16\cdot X \in \mathbf{Z}[X,Y]\text{ .}
\end{equation}
The following result motivates the definition of $\varphi_2$.

\begin{prop}\label{Supersingular_motivation_varphi_2}
Let $z \in \mathbf{C}$ and $\lambda = \lambda(q)$ (where $q = e^{i \pi z}$). We have in $\mathbf{C}[Y]$:
$$\left(Y-\lambda(e^{\frac{i\pi z}{2}})\right) \cdot \left(Y-\lambda(e^{\frac{i\pi (z+2)}{2}}) \right) = \frac{1}{(1-\lambda)^2}\cdot \varphi_2(\lambda, Y) \text{ .}$$ 
\end{prop}
\begin{proof}
We have: $$(Y-\lambda(e^{\frac{i \pi z}{2}}))(Y-\lambda(e^{\frac{i \pi (z+2)}{2}})) = Y^2 - (\lambda(q^{\frac{1}{2}})+\lambda(-q^{\frac{1}{2}}))\cdot Y+\lambda(q^{\frac{1}{2}})\lambda(-q^{\frac{1}{2}})$$
where $q^{\frac{1}{2}} := e^{\frac{i \pi z}{2}}$. For clarity, let $\lambda_1 = \lambda(q^{\frac{1}{2}})$ et $\lambda_2=\lambda(-q^{\frac{1}{2}})$. We have the following $q$-expansion identity:
$$\lambda(q) = 16q\prod_{n=1}^{\infty} \left( \frac{1+q^{2n}}{1+q^{2n-1}} \right)^8 \text{ .}$$ 
We follow \cite[p.$63$]{Pi_AGM} by letting $Q_0=\prod_{n=1}^{\infty} (1-q^{2n})$, $Q_1=\prod_{n=1}^{\infty} (1+q^{2n})$, $Q_2=\prod_{n=1}^{\infty} (1+q^{2n-1})$ and $Q_3=\prod_{n=1}^{\infty} (1-q^{2n-1})$. 
We have:
\begin{equation}\label{Supersingular_U_2_Q_i_eq3}
\lambda(q) = 16q\cdot \left(\frac{Q_1}{Q_2}\right)^8
\end{equation}
and
\begin{equation}\label{Supersingular_U_2_Q_i_eq4}
\lambda(q^{\frac{1}{2}})\cdot \lambda(-q^{\frac{1}{2}}) = -16^2q\prod_{n=1}^{\infty}\frac{(1+q^n)^{16}}{(1-q^{2n-1})^8} = -16^2q\cdot \left(\frac{(Q_1\cdot Q_2)^2}{Q_3}\right)^8 \text{ .}
\end{equation}
We have  \cite[p.$64$,$65$]{Pi_AGM}:
\begin{equation}\label{Supersingular_U_2_Q_i_eq1}
Q_1\cdot Q_2\cdot Q_3=1
\end{equation}
and 
\begin{equation}\label{Supersingular_U_2_Q_i_eq2}
Q_2^8=Q_3^8+16q\cdot Q_1^8 \text{ .}
\end{equation}

By (\ref{Supersingular_U_2_Q_i_eq4}) and (\ref{Supersingular_U_2_Q_i_eq1}), we have:
\begin{equation}\label{Supersingular_U_2_Q_i_eq6}
\lambda(q^{\frac{1}{2}})\cdot \lambda(-q^{\frac{1}{2}}) = -16^2q(Q_1\cdot Q_2)^{3\cdot 8} = -16^2q\cdot \left(\frac{Q_1}{Q_2}\right)^{3\cdot 8}\cdot Q_2^{6\cdot 8} \text{ .}
\end{equation}
By (\ref{Supersingular_U_2_Q_i_eq1}) and (\ref{Supersingular_U_2_Q_i_eq2}), we have:
$$Q_2^{3\cdot 8}\cdot \left(\frac{Q_1}{Q_2}\right)^8\cdot\left(1-16q\cdot \left(\frac{Q_1}{Q_2}\right)^8\right)=1$$
which gives 
\begin{equation}\label{Supersingular_U_2_Q_i_eq5}
Q_2^{6\cdot 8} = \frac{16^2q^2}{\lambda^2\cdot(1-\lambda)^2} \text{ .}
\end{equation}
By (\ref{Supersingular_U_2_Q_i_eq6}) and (\ref{Supersingular_U_2_Q_i_eq5}), we have:
$$\lambda(q^{\frac{1}{2}})\cdot \lambda(-q^{\frac{1}{2}}) = \frac{-16\cdot \lambda(q)}{(1-\lambda(q))^2}$$ that is $\lambda_1\lambda_2 = \frac{-16\lambda}{(1-\lambda)^2}$.
By \cite[p.$115$, ($4.3.6$)]{Pi_AGM}, we have $\lambda(-q)=\frac{\lambda(q)}{1-\lambda(q)}$, which gives $$(1-\lambda_1)\cdot (1-\lambda_2)=1 \text{ , }$$ that is $$\lambda_1+\lambda_2=\lambda_1\cdot \lambda_2 = \frac{-16\cdot \lambda}{(1-\lambda)^2} \text{ .}$$ This concludes the proof of Proposition \ref{Supersingular_motivation_varphi_2}.
\end{proof}

The following result is the analogue of Theorem \ref{Supersingular_main_thm} for $\ell=2$.
\begin{thm}\label{Supersingular_U_2_multiplicative}
Let $\lambda \in L$ and let $\lambda_1$ and $\lambda_2$ be the roots of the polynomial $\varphi_2(X, \lambda)$. We have:
$$\lambda^{N-1}\cdot H'(\lambda_1)\cdot H'(\lambda_2) = \frac{\lambda^2\cdot (\lambda-1)}{4} \cdot H'(\lambda)^2 \text{ .}$$
\end{thm}
\begin{proof}
\begin{lem}
Keep the notation of Theorem \ref{Supersingular_U_2_multiplicative}, except that $\lambda \in \overline{\mathbf{F}}_N\backslash \{0,1\}$ is now arbitrary. We have: 
 \begin{equation}\label{U_2_H}
\lambda^{N-1}\cdot H(\lambda_1)\cdot H(\lambda_2) = H(\lambda)^2
\end{equation}
\end{lem}
\begin{proof}
It suffices to prove the following identity, in $\mathbf{F}_N[X]$:
\begin{equation}\label{Supersingular_H_F_H^2_eq}
\Res_X(H(X), \varphi_2(X,Y))=H(Y)^2 \text{ .}
\end{equation}
We have $$\Res_X(H(X), \varphi_2(X,Y))=\prod_{\lambda' \in L} \varphi_2(\lambda', Y) =\prod_{\lambda' \in L} \lambda'^2\cdot (\lambda_1'-Y)\cdot(\lambda_2'-Y) \text{ , }$$ where $\lambda_1'$ et $\lambda_2'$ are the roots of $\varphi_2(\lambda', Y) = 0$ (if $\varphi_2(\lambda', Y)$ has double roots then we let $\lambda_1' = \lambda_2'$). Since $\prod_{\lambda' \in L} \lambda' = 1$, we get 
\begin{equation}\label{Supersingular_H_F_H^2_eq_1}
\Res_X(H(X), \varphi_2(X,Y)) = \prod_{\lambda' \in L} (\lambda_1'-Y)\cdot(\lambda_2'-Y) \text{ .}
\end{equation}
If $\lambda' \in L$ then $\lambda_1'$ and $\lambda_2'$ are also in $L$, since by Proposition \ref{Supersingular_motivation_varphi_2} the elliptic curves $E_{\lambda_1'}$ and $E_{\lambda_2'}$ are $2$-isogenous to $E_{\lambda}$. If $\lambda' \in \overline{\mathbf{F}}_N \backslash \{0,1,-1\}$ (resp. $\lambda' = -1$) then the polynomial $\varphi_2(\lambda', Y)$ has two distinct roots (resp. a double root $Y=2$). Conversely, if $\lambda_1 \in \overline{\mathbf{F}}_N \backslash \{0, 1, 2\}$ (resp. $\lambda_1=2$) then the polynomial $\varphi_2(X, \lambda_1)$ has two distinct roots (resp. a double root $X=-1$). Thus, we have $$\prod_{\lambda' \in L} (\lambda_1'-Y)\cdot(\lambda_2'-Y) = \prod_{\lambda' \in L} (\lambda'-Y)^2 = H(Y)^2 \text{ .}$$ 
\end{proof}

We are going to differentiate (\ref{U_2_H}) two times. Let $K = \overline{\mathbf{F}}_N(\lambda)$ be the function field of $\mathbf{P}^1(\lambda)$. We have a derivation $\frac{d}{d\lambda}$ which sends $\lambda$ to $1$. Let $F = K(\lambda_1)$ where $\lambda_1$ has minimal polynomial $X^2+\frac{-2\lambda^2+16\lambda-16}{\lambda^2}X+1$ over $K$.
\begin{lem}
The $\overline{\mathbf{F}}_N$-derivation $\frac{d}{d\lambda}$ of $K$ extends in an unique way to a $\overline{\mathbf{F}}_N$-derivation $\frac{\partial}{\partial \lambda}$ of $F$. More precisely, we have in $F$:
$$\frac{\partial \lambda_1}{\partial \lambda} = -\frac{\frac{\partial }{\partial Y} \varphi_2(\lambda_1, \lambda)}{\frac{\partial }{\partial X} \varphi_2(\lambda_1, \lambda)} \text{ .}$$
\end{lem}
\begin{proof}
This follows from the fact that $\varphi_2$ is irreducible of degree $2$ as a polynomial in $X$ over $\mathbf{F}_N(Y)$, and that $N$ is prime to $2$.
\end{proof}
Let $\lambda_2$ be the other root of $X^2+\frac{-2\lambda^2+16\lambda-16}{\lambda^2}X+1$ in $F$. 
\begin{lem}
We have:
$$ \frac{\partial \lambda_1}{\partial \lambda} \cdot \frac{\partial \lambda_2}{\partial \lambda} = \frac{4}{\lambda\cdot (\lambda-1)}$$
\end{lem}
\begin{proof}
It is just a formal computation.
\end{proof}
By differentiating (\ref{U_2_H}) two times ($\lambda$ is considered as a formal variable), we get:
\begin{equation}\label{H_messy_lambda_1_lambda_2}
\frac{4\cdot \lambda^{N-1}}{\lambda\cdot (\lambda-1)}\cdot H'(\lambda_1)\cdot H'(\lambda_2) = H'(\lambda)^2+ G(\lambda, \lambda_1, \lambda_2)
\end{equation}
where $G(\lambda, \lambda_1, \lambda_2)$ is a sum of two terms of the form $H(\lambda)$, $H(\lambda_1)$ or $H(\lambda_2)$ times a polynomial in $\lambda$, $\lambda_1$, $\lambda_2$, $ \frac{\partial \lambda_1}{\partial \lambda}$ et $\frac{\partial \lambda_2}{\partial \lambda} $. If $\lambda \in L$ then we have $H(\lambda)=H(\lambda_1)=H(\lambda_2)=0$. By (\ref{H_messy_lambda_1_lambda_2}), we get
$$
\frac{4\cdot \lambda^{N-1}}{\lambda\cdot (\lambda-1)}\cdot H'(\lambda_1)\cdot H'(\lambda_2) = H'(\lambda)^2 \text{ .}$$
This concludes the proof of Theorem \ref{Supersingular_U_2_multiplicative}.
\end{proof}

\subsection{The supersingular module of Legendre elliptic curves}\label{Section_Supersingular_Gamma(2)_structure}
Keep the notation of sections \ref{Section_Supersingular_intro_SS} and \ref{Supersingular_section_Hasse}. We denote by $v$ the $p$-adic valuation of $N^2-1$ (thus, $v=t$ if $p\geq 5$, $v=t+1$ if $p=3$ and $v=t+3$ if $p=2$).

In view of the properties satisfied by the Hasse polynomial, the higher Eisenstein element $e_1$ is more easily determined after adding an auxiliary $\Gamma(2)$-structure. 

Let $\tilde{\mathbb{T}}'$ be the full $\mathbf{Z}_p$-Hecke algebra acting faithfully the space of modular forms of weight $2$ and level $\Gamma_0(N) \cap \Gamma(2)$. If $n \geq 1$ is an integer, we denote by $T_n'$ the $n$th Hecke operator in $\tilde{\mathbb{T}}'$. The ring $\tilde{\mathbb{T}}$ acts on the $\mathbf{Z}_p$-module $M' := \mathbf{Z}_p[L]$. If $\ell$ is a prime not dividing $2\cdot N$ and $\lambda \in L$, we have in $M'$:
$$T_{\ell}'([\lambda]) = \sum_{\lambda' \in L \atop \varphi_{\ell}(\lambda, \lambda')=0} [\lambda']\text{ .}$$
Let 
$$\tilde{e}_0' = \sum_{\lambda \in L} [\lambda] \in M' \text{ .}$$ 
\begin{thm}\label{Supersingular_w_1_L}
Let $\Lambda : \mathbf{F}_{N^2}^{\times} \rightarrow \mathbf{Z}/p^v\mathbf{Z}$ be a surjective group homomorphism. 
Let $e_0'$ be the image of $\tilde{e}_0'$ in $M/p^v\cdot M$ and let
$$e_1' = \sum_{\lambda \in L} \Lambda(H'(\lambda))\cdot [\lambda] \in M'/p^v\cdot M' \text{ .}$$
For all prime $\ell$ not dividing $2\cdot N$, we have in $M'/p^v\cdot M'$:
$$(T_{\ell}'-\ell-1)(e_0')=0$$ 
and
$$(T_{\ell}'-\ell-1)(e_1') = (\ell-1)\cdot \Lambda(\ell) \cdot e_0' \text{ .}$$
\end{thm}
\begin{proof}
The first equality is obvious. The second equality is a direct consequence of Theorem \ref{Supersingular_main_thm}.
\end{proof}

We let $\pi : M' \rightarrow M$ be the forgetful map. It is a $\mathbf{Z}_p$-equivariant group homomorphism defined by $\pi([\lambda]) = [E_{\lambda}] \in S$ (recall that $E_{\lambda}$ is the Legendre elliptic curve $y^2 = x(x-1)(x-\lambda)$). For all prime $\ell$ not dividing $2\cdot N$ and all $x \in M'$, we have in $M$:
\begin{equation}\label{Supersingular_equivariance_T_ell_prime_T_ell}
\pi\left(T_{\ell}'(x)\right) = T_{\ell}\left(\pi(x)\right) \text{ .}
\end{equation}
We have, in $M$:
\begin{equation}\label{Supersingular_e_0_tilde_prime_e_0_tilde_eq}
\pi(\tilde{e}_0') = 6 \cdot \tilde{e}_0 \text{ .}
\end{equation}
Theorem \ref{Supersingular_w_1_L}, (\ref{Supersingular_equivariance_T_ell_prime_T_ell}) and (\ref{Supersingular_e_0_tilde_prime_e_0_tilde_eq}) allow us to compute $e_1$. However, some complications arise when $p \in \{2,3\}$, so we treat the cases $p \geq 5$, $p=3$ and $p=2$ separately.

\subsection{The case $p\geq 5$}
In this section, we assume $p\geq 5$. Keep the notation of sections \ref{Section_Supersingular_intro_SS}, \ref{Supersingular_section_Hasse} and \ref{Section_Supersingular_Gamma(2)_structure}.

\begin{thm}\label{Supersingular_w_1_S_5}
Assume that $p \geq 5$. We extend $\log$ to a surjective group morphism $\log : \mathbf{F}_{N^2}^{\times} \rightarrow \mathbf{Z}/p^r\mathbf{Z}$.
\begin{enumerate}
\item We have, in $M/p^r\cdot M$ modulo the subgroup generated by $e_0$:
$$12\cdot e_1 = \sum_{E \in S} \log(P(j(E))) \cdot [E] $$
where $j(E)$ is the $j$-invariant of $E$ (the fact that $P(j(E)) \neq 0$ is included in the statement). 
\item We have, in $\mathbf{Z}/p^r\mathbf{Z}$:
\begin{align*}
e_1\bullet e_0 &= \frac{1}{12}\cdot \sum_{\lambda \in L} \log(H'(\lambda)) \\& = \frac{1}{3}\cdot \sum_{k=1}^{\frac{N-1}{2}} k \cdot \log(k) \text{ .}
\end{align*}
\item Assume $e_1 \bullet e_0 = 0$. We have, in $\mathbf{Z}/p^r\mathbf{Z}$:
$$72 \cdot e_1 \bullet e_1 = 3 \cdot \left( \sum_{\lambda \in L} \log(H'(\lambda))^2 \right) -4  \cdot \left( \sum_{\lambda \in L} \log(\lambda)^2 \right) \text{ .}$$
\end{enumerate}
\end{thm}
\begin{proof}
We prove (i). We choose $\Lambda : \mathbf{F}_{N^2}^{\times} \rightarrow \mathbf{Z}/p^v\mathbf{Z}$ so that for all $x \in \mathbf{F}_{N^2}^{\times}$, we have $\Lambda(x) \equiv \log(x) \text{ (modulo }p^r\text{)}$.
We abuse notation and still denote by $\pi : M'/p^r\cdot M' \rightarrow M/p^r\cdot M$ the forgetful map. Let $e_1''$ be the image of $e_1'$ in $M'/p^r\cdot M'$. Let $\ell$ be a prime not dividing $2\cdot N$.
By Theorem \ref{Supersingular_w_1_L}, (\ref{Supersingular_equivariance_T_ell_prime_T_ell}) and (\ref{Supersingular_e_0_tilde_prime_e_0_tilde_eq}) we have in $M/p^r\cdot M$:
$$(T_{\ell}-\ell-1)(\pi(e_1'')) = 6\cdot (\ell-1)\cdot \log(\ell)\cdot e_0 \text{ .}$$
By Proposition \ref{Formalism_property_Hecke_outside}, we have 
$$\pi(e_1'') = 12\cdot e_1 \text{ (modulo the subgroup generated by }e_0 \text{).}$$
Let $E \in S$. The coefficient of $\pi(e_1'')$ in $[E]$ is by definition $$\sum_{\lambda \in L \atop j(E_{\lambda}) = j(E)} \log(H'(\lambda)) \text{ .}$$
By (\ref{Supersingular_relation_j_lambda_eq}), we have
$$\sum_{\lambda \in L \atop j(E_{\lambda}) = j(E)} \log(H'(\lambda)) = \log(P(j(E))) \text{ .}$$
This concludes the proof of Theorem \ref{Supersingular_w_1_S_5}.

We prove (ii). We have:
\begin{align*}
12\cdot e_1 \bullet e_0 &= \sum_{E \in S} \log(P(j(E)) \\&
=\sum_{\lambda \in L} \log(H'(\lambda)) \\& = \log(\Disc(H)) \\& = 4\cdot \sum_{k=1}^{\frac{N-1}{2}} k \cdot \log(k) \text{ .}
\end{align*}
The first equality follows from (i). The last equality follows from Theorem \ref{Supersingular_discriminant_thm}, using the fact that $\log\left(\left(\frac{N-1}{2}\right)!\right) = 0$ (since $p>2$ and $\left(\frac{N-1}{2}\right)!^4 = 1$ in $\mathbf{F}_N$).

We prove (iii). We identify an element of $L$ with its $\lambda$-invariant. For $E\in S$, let $F_E \subset L$ be the fiber above $E$, \ie the set of $\lambda \in L$ such that $\pi([\lambda]) = [E]$. There exists $\lambda \in L$ such that
\begin{equation}\label{Supersingular_F_E_fiber}
F_E= \left\{\lambda, \frac{1}{\lambda}, 1-\lambda, \frac{\lambda-1}{\lambda}, \frac{\lambda}{\lambda-1}, \frac{1}{1-\lambda}\right\}
\end{equation}
(we do $\textbf{not}$ count multiplicity). Let $c_E \in \{1,2,3,6\}$ be the cardinality of $F_E$. We have $w_E = \frac{6}{c_E}$. We have, by definition of $P$, for any $E \in S$ and $\lambda \in F_E$:
\begin{equation}\label{Supersingular_computation_e_1_e_1_>5}
P(j(E))^{w_E} =  H'(\lambda) \cdot H'\left(\frac{1}{\lambda}\right) \cdot H'(1-\lambda) \cdot H'\left( \frac{\lambda-1}{\lambda}\right) \cdot H'\left(\frac{\lambda}{\lambda-1}\right) \cdot H'\left(\frac{1}{1-\lambda}\right) \text{ .}
\end{equation}
We have $H(X) = (-1)^m\cdot H(1-X)$ and that $H(\frac{1}{X}) = X^{-m} \cdot H(X)$ where $m = \frac{N-1}{2}$. Indeed, $H$ is monic and the roots of $H$ are permuted by the transformation $\lambda \mapsto 1-\lambda$ and $\lambda \mapsto \frac{1}{\lambda}$. By differentiating these two relations  with respect to $X$ and using (\ref{Supersingular_computation_e_1_e_1_>5}), we get:
\begin{equation}\label{Supersingular_P(j(E))_H'(Lambda)_F_E_eq}
P(j(E))^{w_E} =  H'(\lambda)^6 \cdot \lambda^4 \cdot (1-\lambda)^4  \text{ .}
\end{equation}
Thus, we have, in $\mathbf{Z}/p^r\mathbf{Z}$:
\begin{align*}
12^2\cdot e_1 \bullet e_1 &= \sum_{E \in S} w_E \cdot \log(P(j(E)))^2 \\&=\frac{1}{6} \cdot \left(\sum_{\lambda \in L}  6 \cdot \log(H'(\lambda)) + 4 \cdot \log(\lambda) + 4 \cdot \log(1-\lambda) \right)^2 \\&
= 6 \cdot \left( \sum_{\lambda \in L} \log(H'(\lambda))^2 \right) -8  \cdot \left( \sum_{\lambda \in L} \log(\lambda)^2 \right) 
\end{align*}
This concludes the proof of Theorem \ref{Supersingular_w_1_S_5}
\end{proof}
\begin{rems}
\begin{enumerate}
\item Theorem \ref{Supersingular_w_1_S_5} (ii) will be proved independently using modular symbols in Section \ref{comparison_section_computation_i+j=1}. 
\item Theorem \ref{Supersingular_w_1_S_5} (ii) is the first instance of what we call a \textit{higher Eichler mass formula}, by analogy with the classical Eichler mass formula (\ref{Supersingular_Eichler_mass_formula_eq}).
\item In Corollary \ref{comparison_higher_Eichler_quadratic_formula}, we will show (using modular symbols) that if $e_1 \bullet e_0 = 0$, then $e_1 \bullet e_1 = 12\cdot \sum_{k=1}^{\frac{N-1}{2}} k\cdot \log(k)^2$. This is the second instance of a higher Eichler mass formula. We have note been able to prove this identity directly. See Conjecture \ref{Supersingular_conjecture_supersingular_5} (iii) for a generalization when $e_1 \bullet e_0 \neq 0$. 
\end{enumerate}
\end{rems}

\subsection{The case $p=3$}
In this section, we assume $p=3$. Keep the notation of sections \ref{Section_Supersingular_intro_SS}, \ref{Supersingular_section_Hasse} and \ref{Section_Supersingular_Gamma(2)_structure}.

\begin{thm}\label{Supersingular_w_1_S_3}
Extend and lift $\log$ to a surjective group homomorphism $\log: \mathbf{F}_{N^2}^{\times} \rightarrow \mathbf{Z}/3^{r+1}\mathbf{Z}$. We have, in $\left(M/3^{r+1}\cdot M\right)/\left(\mathbf{Z}\cdot (3\cdot e_0) \right)$:
$$
12 \cdot e_1 \equiv 2\cdot \log(2) \cdot \tilde{e}_0 +  \sum_{E \in S} \log(P(j(E))) \cdot [E] \text{ .}
$$
\end{thm}
\begin{proof}
It obviously suffices to prove Theorem \ref{Supersingular_w_1_S_3} when $r=t$, which we assume until the end of the proof. Note that $v=t+1$, so we can let $\Lambda=\log$.
We abuse notation and still denote by $\pi : M'/3^{t+1}\cdot M' \rightarrow M/3^{t+1}\cdot M$ the forgetful map. Let $\ell$ be a prime not dividing $2\cdot N$.
By Theorem \ref{Supersingular_w_1_L}, (\ref{Supersingular_equivariance_T_ell_prime_T_ell}) and (\ref{Supersingular_e_0_tilde_prime_e_0_tilde_eq}) we have in $M/3^{t+1}\cdot M$:
$$(T_{\ell}-\ell-1)(\pi(e_1')) = 6\cdot (\ell-1)\cdot \log(\ell)\cdot e_0 \text{ .}$$
By Proposition \ref{Formalism_property_Hecke_outside}, we have 
$$\pi(e_1') = 12\cdot e_1 \text{ (modulo the subgroup generated by the image of }\tilde{e}_0 \text{ in }M/3^{t+1}\cdot M \text{).}$$
Let $E \in S$. The coefficient of $\pi(e_1')$ in $[E]$ is by definition $$\sum_{\lambda \in L \atop j(E_{\lambda}) = j(E)} \log(H'(\lambda)) \text{ .}$$
By (\ref{Supersingular_relation_j_lambda_eq}), we have
$$\sum_{\lambda \in L \atop j(E_{\lambda}) = j(E)} \log(H'(\lambda)) = \log(P(j(E))) \text{ .}$$
Thus, there exists $C_3 \in \mathbf{Z}/3^{t+1}\mathbf{Z}$, uniquely defined modulo $3$, such that we have in $M/3^{t+1}\cdot M$:
\begin{equation}\label{Supersingular_e_1_p=3_eq1}
12\cdot e_1 = C_3 \cdot \tilde{e}_0 +  \sum_{E \in S} \log(P(j(E))) \cdot [E] \text{ .}
\end{equation}

By pairing  (\ref{Supersingular_e_1_p=3_eq1}) with $\tilde{e}_0$, we get in $\mathbf{Z}/3^{t+1}\mathbf{Z}$:
$$12\cdot e_1 \bullet e_0 =  C_3 \cdot \tilde{e}_0 \bullet \tilde{e}_0 + \log(\text{Disc}(H)) \text{ .}$$

By Corollary \ref{Comparison_corr_comparison}, Theorem \ref{Comparison_Merel_sqrt_u}, Eichler mass formula (\ref{Supersingular_Eichler_mass_formula_eq}), Theorem \ref{Supersingular_discriminant_thm} and Lemma \ref{even_modSymb_computation_square} (which are independent of the results of this section), we have in $\mathbf{Z}/3^{t+1}\mathbf{Z}$:
$$- 3 \cdot \sum_{k=1}^{N-1} k^2\cdot\log(k) = C_3\cdot \frac{N-1}{12}-3 \cdot \sum_{k=1}^{N-1} k^2\cdot\log(k) -\frac{N-1}{6}\cdot \log(2) \text{ .}$$
This concludes the proof of Theorem \ref{Supersingular_w_1_S_3}.
\end{proof}

\begin{rem}
One could use Theorem \ref{Supersingular_w_1_S_3} to compute $e_1 \bullet e_1$ if $t\geq 2$ and $r \leq t-1$. 
\end{rem}

\subsection{The case $p=2$}
In this section, we assume $p=2$. Keep the notation of sections \ref{Section_Supersingular_intro_SS}, \ref{Supersingular_section_Hasse} and \ref{Section_Supersingular_Gamma(2)_structure}.

\begin{thm}\label{Supersingular_w_1_S_2}
Let $\tilde{\Lambda} : \mathbf{F}_{N^2}^{\times} \rightarrow \mathbf{Z}/2^{t+3}\mathbf{Z}$ be a surjective group homomorphism such that for all $x \in \mathbf{F}_N^{\times}$, we have $\tilde{\Lambda}(x) \equiv 2\cdot \log(x) \text{ (modulo }2^{r+1}\text{)}$. Let $\epsilon_2 \in \{1,-1\}$ be defined by $$\tilde{\Lambda}\left(\left(\frac{N-1}{2}\right)!\right) \equiv 2^{t+1}\cdot \epsilon_2 \text{ (modulo }2^{t+3}\text{).}$$ 
Let $\Lambda$ be the reduction of $\tilde{\Lambda}$ modulo $2^{r+3}$. We have, in $M/2^{r+3}\cdot M$ modulo the subgroup generated by $8\cdot e_0$:
$$24 \cdot e_1 =2 \cdot C_2 \cdot \tilde{e}_0 + \sum_{E \in S} \Lambda(P(j(E))) \cdot [E] \text{ .}$$
where 
$$C_2 \equiv \frac{2^{t+2}}{N-1} \cdot \epsilon_2 \text{ (modulo }4\text{).}$$
\end{thm}
\begin{proof}
It obviously suffices to prove Theorem \ref{Supersingular_w_1_S_2} when $r=t$ (so $v=r+3$), which we assume until the end of the proof. 
We abuse notation and still denote by $\pi : M'/2^{t+3}\cdot M' \rightarrow M/2^{t+3}\cdot M$ the forgetful map. Let $\ell$ be a prime not dividing $2\cdot N$.
By Theorem \ref{Supersingular_w_1_L}, (\ref{Supersingular_equivariance_T_ell_prime_T_ell}) and (\ref{Supersingular_e_0_tilde_prime_e_0_tilde_eq}) we have in $M/2^{t+3}\cdot M$:
$$(T_{\ell}-\ell-1)(\pi(e_1')) = 12\cdot (\ell-1)\cdot \log(\ell)\cdot e_0\text{ .}$$
By Proposition \ref{Formalism_property_Hecke_outside}, there exists $C_2' \in \mathbf{Z}/2^{t+3}\mathbf{Z}$ (uniquely defined modulo $8$) such that we have, in $M/2^{t+3}\cdot M$ modulo the subgroup generated by the image of $\tilde{e}_0$ in $M/2^{t+3}\cdot M$:
\begin{equation}\label{Supersingular_C_2_e_1}
24\cdot e_1 = C_2' \cdot \tilde{e}_0 + \sum_{E \in S}\tilde{\Lambda}(P(j(E))) \cdot [E] \text{ .}
\end{equation}
By pairing (\ref{Supersingular_C_2_e_1}) with $\tilde{e}_0$, we get in $\mathbf{Z}/2^{r+3}\mathbf{Z}$:
$$
24\cdot e_1 \bullet e_0 = C_2' \cdot \tilde{e}_0 \bullet \tilde{e}_0+\tilde{\Lambda}(\text{Disc}(H)) \text{ (modulo }2^{t+3}\text{).}
$$

By Corollary \ref{Comparison_corr_comparison}, Theorem \ref{Comparison_Merel_sqrt_u}, Eichler mass formula (\ref{Supersingular_Eichler_mass_formula_eq}) and Theorem \ref{Supersingular_discriminant_thm} (which are independent of the results of this section), we have in $\mathbf{Z}/2^{t+3}\mathbf{Z}$:

$$-2^{t+2} +8\cdot \left(\sum_{k=1}^{\frac{N-1}{2}} k \cdot \log(k) \right) \equiv C_2'\cdot \frac{N-1}{12}+8\cdot \left(\sum_{k=1}^{\frac{N-1}{2}} k \cdot \log(k) \right) - \tilde{\Lambda}\left(\left(\frac{N-1}{2}\right)!\right) \text{ .}$$
Thus, we have in $\mathbf{Z}/2^{t+3}\mathbf{Z}$:
$$2^{t}\cdot \left(4+ \frac{C_2'}{3}\cdot \frac{N-1}{2^{t+2}} - 2\cdot \epsilon_2\right) = 0 \text{ .}$$
Thus, we have 
$$C_2' \equiv 6\cdot \frac{2^{t+2}}{N-1}\cdot (\epsilon_2-2) \equiv 2 \cdot C_2 \text{ (modulo }8\text{).}$$
This concludes the proof of Theorem \ref{Supersingular_w_1_S_2}.
\end{proof}

The following result is an elementary consequence of Theorem \ref{Supersingular_w_1_S_2}, for which we have not found an elementary proof.
\begin{corr}
There exists $\lambda \in L$ such that $H'(\lambda)$ is not a square of $\mathbf{F}_{N^2}^{\times}$.
\end{corr}
\begin{proof}
If $E\in S$, let $F_E$ be defined as in (\ref{Supersingular_F_E_fiber}).
By (\ref{Supersingular_P(j(E))_H'(Lambda)_F_E_eq}), for all $\lambda \in F_E$ we have: 
$$P(j(E))^{w_E} =  H'(\lambda)^6 \cdot \lambda^4 \cdot (1-\lambda)^4  \text{ .}$$
Thus exists $E \in S$ such that $P(j(E))$ is not a fourth power in $\mathbf{F}_{N^2}^{\times}$ if and only if there is $\lambda \in L$ such that $H'(\lambda)$ is not a square of $\mathbf{F}_{N^2}^{\times}$. Such a $E$ exists by Theorem \ref{Supersingular_w_1_S_2} since we have $C_2 \in (\mathbf{Z}/4\mathbf{Z})^{\times}$.
\end{proof}

The following conjecture was checked numerically for $N<2000$ (without assuming $N \equiv 1 \text{ (modulo }8\text{)}$ anymore). We do not know the significance of this empirical fact.
\begin{conj}\label{Supersingular_carre}
Assume $N \equiv 1 \text{ (modulo }4\text{)}$. For all $\lambda \in L$, $H'(\lambda)$ is not a square of $\mathbf{F}_{N^2}^{\times}$.
\end{conj}
\begin{rem}
Conjecture \ref{Supersingular_carre} does not hold if $N \equiv 3 \text{ (modulo }4\text{)}$, since in this case there exists $\lambda \in \mathbf{F}_N \cap L$, so $H'(\lambda) \in \mathbf{F}_N^{\times}$ is a square of $\mathbf{F}_{N^2}^{\times}$ \cite[Proposition 4.3]{Lecouturier_Betina_2}.
\end{rem}

\subsection{Conjectural identities satisfied by the supersingular lambda invariants}\label{Supersingular_module_various_conjectures}
In this section, we collect various (often conjectural) arithmetic properties satisfied by the elements of $L$. These identities are motivated by the theory of the Eisenstein ideal.

\begin{conj}\label{Supersingular_conjecture_supersingular_5}
Assume $p \geq 5$. Extend $\log$ to a surjective group homomorphism $\log : \mathbf{F}_{N^2}^{\times} \rightarrow \mathbf{Z}/p^r\mathbf{Z}$.

\begin{enumerate}
\item We have $\sum_{\lambda \in L} \log(\lambda)^2 = -32 \cdot \log(2) \cdot \left( \sum_{k=1}^{\frac{N-1}{2}}k\cdot \log(k) \right)$.
\item There exists $\lambda \in L$ such that $\log(\lambda)\neq 0$.
\item We have $\sum_{\lambda \in L} \log(H'(\lambda))^2 = 4\cdot \left(\sum_{k=1}^{\frac{N-1}{2}} k \cdot \log(k)^2\right) - 3 \cdot 16\cdot \log(2) \cdot \left(\sum_{k=1}^{\frac{N-1}{2}} k \cdot \log(k)\right)$.

\item Assume $\sum_{k=1}^{\frac{N-1}{2}} k \cdot \log(k) = 0$. For all $\lambda \in L$, we have
$$\sum_{\lambda' \in L \backslash\{\lambda\}} \log(\lambda'-\lambda) \cdot \log(H'(\lambda')) = \log(H'(\lambda))^2 \text{ .}$$

\item Assume $\sum_{k=1}^{\frac{N-1}{2}} k \cdot \log(k) = 0$.  For all $\lambda \in L$, we have
$$\sum_{\lambda' \in L \backslash\{\lambda\}} \log(\lambda'-\lambda) \cdot \log(\lambda') = \log(\lambda)^2 \text{ .}$$

\end{enumerate}
\end{conj}

\begin{rems}\label{Supersingular_relation_remark}
\begin{enumerate}
\item These conjectures have been numerically checked (using SAGE) for $N < 1000$.
\item Using (derivatives of) the relations $H(1-X) = (-1)^m\cdot H(X)$ and $H(\frac{1}{X}) = X^{-m}\cdot H(X)$ (where $m = \frac{N-1}{2}$), we easily prove that
$$\sum_{\lambda \in L} \log(H'(\lambda))\cdot \log(\lambda) = \sum_{\lambda \in L} \log(H'(\lambda))\cdot \log(1-\lambda) = -\sum_{\lambda \in L} \log(\lambda)^2 = -2 \cdot \sum_{\lambda \in L} \log(\lambda)\cdot \log(1-\lambda)$$
which allows us to reformulate Conjecture \ref{Supersingular_conjecture_supersingular_5} (i).
\item The term $\log(2)$ in Conjecture \ref{Supersingular_conjecture_supersingular_5} (i) comes from a criterion of Ribet concerning the existence of congruences between cuspidal newforms of weight $2$ and level $\Gamma_0(2N)$, and Eisenstein series \cite[Theorem $2.4$]{Yoo_optimal}.
\item It seems that Akshay Venkatesh found a proof of Conjecture \ref{Supersingular_conjecture_supersingular_5} (i) using the Hecke operator $U_2$. 
\item We will show in Proposition \ref{Supersingular_preuve_non_nuls} that (ii) holds if $\log(2) \not\equiv 0 \text{ (modulo }p\text{)}$.
\item If $\sum_{k=1}^{\frac{N-1}{2}} k \cdot \log(k)=0$, Conjecture \ref{Supersingular_conjecture_supersingular_5} (iii) is proved in Corollary \ref{comparison_higher_Eichler_quadratic_formula}, using modular symbols. 
\item Conjecture \ref{Supersingular_conjecture_supersingular_5} (iv) and (v) are suggested by the theory of the refined $\mathscr{L}$-invariant of de Shalit \cite{de_shalit_L}, Oesterl\'e, Mazur--Tate \cite{MT} and Mazur--Tate--Teitelbaum \cite{MTT}, and its generalization to level $\Gamma(2) \cap \Gamma_0(N)$ studied in \cite{Lecouturier_Betina}. 
\end{enumerate}
\end{rems}

We end this section with (conjectural) relations analogous to the ones of Conjecture \ref{Supersingular_conjecture_supersingular_5} for $p=3$ and $p=2$.
\begin{conj}\label{Supersingular_conjecture_supersingular_3}
Assume $p=3$. Extend and lift $\log$ to a surjective group homomorphism $\log : \mathbf{F}_{N^2}^{\times} \rightarrow \mathbf{Z}/p^{r+1}\mathbf{Z}$. 
\begin{enumerate}
\item There exists $\lambda \in L$ such that $\log(\lambda) \not\equiv 0 \text{ (modulo }3\text{)}$ (\ie $\lambda$ is not a cube of $\mathbf{F}_{N^2}^{\times}$).
\item We have $$\sum_{\lambda \in L} \log(\lambda)^2 = 4 \cdot \log(2) \cdot \left( \sum_{k=1}^{\frac{N-1}{2}}k\cdot \log(k) \right) \text{ (modulo }9\text{)} \text{ .}$$
Furthermore, both sides of the equality are congruent to $0$ modulo $3$. 

\item For al $\lambda \in L$, we have:
$$\sum_{\lambda' \in L \backslash\{\lambda\}} \log(\lambda'-\lambda) \cdot \log(H'(\lambda')) = \log(H'(\lambda))^2 \text{ (modulo }3\text{)}$$

\item For al $\lambda \in L$, we have:
$$\sum_{\lambda' \in L \backslash\{\lambda\}} \log(\lambda'-\lambda) \cdot \log(\lambda') = \log(\lambda)^2 \text{ (modulo }3\text{)}$$
\end{enumerate}
\end{conj}

\begin{rems}
\begin{enumerate}
\item Conjecture \ref{Supersingular_conjecture_supersingular_3} (i) should be true even if $N \not\equiv 1 \text{ (modulo }9\text{)}$ (which is assumed since $p=3$).
\item Conjecture \ref{Supersingular_conjecture_supersingular_3} (ii), (iii) and (iv), although similar to \ref{Supersingular_conjecture_supersingular_5}, is not modulo $3^{r+1}$ in general.
\end{enumerate}
\end{rems}
The following result, in which $N$ is an arbitrary odd prime, is kind of opposite to Conjecture \ref{Supersingular_conjecture_supersingular_3} (i).
\begin{prop}\label{Supersingular_cube}
For all $\lambda \in L$, $\frac{\lambda(1-\lambda)}{2}$ is a cube of $\mathbf{F}_{N^2}^{\times}$.
\end{prop}
\begin{proof}
Since $j = \frac{256\cdot (\lambda^2-\lambda+1)^3}{\lambda^2\cdot (1-\lambda)^2}$, it suffices to show that supersingular $j$-invariants are cubes of $\mathbf{F}_{N^2}^{\times}$. This is well-known \cite[Theorem $1.2$ (b)]{Morton}.
\end{proof}

We conclude this section by stating a result which goes in the opposite direction of Conjecture \ref{Supersingular_conjecture_supersingular_5} (ii).
\begin{prop}\cite[Proposition $3.1$]{legendre_fourth}\label{Supersingular_puissance_4} 
Every $\lambda \in L$ is a fourth power modulo $N$ (we do not assume anything on $N$). If $N^2 \equiv 1  \text{ (modulo } 8 \text{)}$, then every $\lambda \in L$ is a eighth power modulo $N$.
\end{prop}

\subsection{Eisenstein ideals of level $\Gamma_0(N) \cap \Gamma(2)$}\label{Supersingular_module_section_eis_ideals_gamma(2)}
Assume $p\geq 3$. Keep the notation of sections \ref{Section_Supersingular_intro_SS}, \ref{Supersingular_section_Hasse} and \ref{Section_Supersingular_Gamma(2)_structure}. Extend $\log$ to a surjective group homomorphism $\log : \mathbf{F}_{N^2}^{\times} \rightarrow \mathbf{Z}/p^r\mathbf{Z}$.

In view of the role played by the Hasse polynomial, we reformulate the problem of Eisenstein elements in the context of the congruence subgroup $\Gamma_0(N) \cap \Gamma(2)$. The modular curve $X(\Gamma_0(N) \cap \Gamma(2))$ has $6$ cusps. Thus, the space of Eisenstein series of weight $2$ and level $\Gamma_0(N) \cap \Gamma(2)$ has dimension $5$. It admits a basis of eigenforms for the Hecke operators $T_{\ell}$ ($\ell$ prime $\neq 2,N$), $U_2$ and $U_N$. These Eisenstein series are characterized by the pair $(a_N, a_2)$ of their Fourier coefficients (at the cusp $\infty$) at $N$ and $2$, which belongs to $\{(1,1), (1, 2), (1, 0), (N, 1), (N, 0)\}$. Since $p^t$ divides $N-1$, these coefficients are in $\{(1, 1), (1, 2), (1, 0)\}$ modulo $p^t$. We define three Eisenstein ideals.

For $\alpha \in  \{0,1,2\}$, let $I_{\alpha}$ be the ideal of the Hecke algebra acting on $M_2(\Gamma(2) \cap \Gamma_0(N))$ generated by the $T_{\ell}-\ell-1$ ($\ell$ prime number different from $2$ and $N$), $U_2-\alpha$ and $U_N-1$. Recall that these Hecke operators generate a $\mathbf{Z}_p$-algebra $\tilde{\mathbb{T}}'$ acting on $M' = \mathbf{Z}_p[L]$. We have described the action of $T_{\ell}$ (given by the modular polynomial $\varphi_{\ell}$) and $U_N$ (given by $[\lambda] \mapsto [\lambda^N]$). We are going do describe $U_2$ (this might be well-known, be we could not find a reference).

\begin{prop}\label{Supersingular_formule_U_2}
For all $\lambda \in L$, we have in $M'$:
$$U_2([\lambda]) = [\lambda_1]+ [\lambda_2]$$
where $\lambda_1$ et $\lambda_2$ are the roots of the polynomial $\varphi_2(\lambda, Y)$.
\end{prop}
\begin{proof}
This follows from Proposition \ref{Supersingular_motivation_varphi_2}.
\end{proof}

Define the following elements in $M'/p^r\cdot M'$, if $p \geq 5$: 
\begin{itemize}
\item $e_0^0 = \sum_{\lambda \in L} \log(\lambda)\cdot [\lambda]$
\item $e_0^1 = \sum_{\lambda \in L} (\log(1-\lambda) - 2 \cdot \log(\lambda) - 4\cdot \log(2)) \cdot [\lambda]$
\item $e_0^2 = \sum_{\lambda \in L} [\lambda]$
\end{itemize}
If $p=3$, we define in the same way $e_0^0$ and $e_0^2$. To define $e_0^1$, note that by Proposition \ref{Supersingular_cube}, for all $\lambda \in L$ there exists $a_{\lambda} \in \mathbf{F}_{N^2}^{\times}$ such that $\frac{1-\lambda}{\lambda^2\cdot 2^4} = a_{\lambda}^3$. We then let
$$e_0^1 = \sum_{\lambda \in L} a_{\lambda} \cdot [\lambda] \in M'/3^r\cdot M' $$
(this does not depend on the choice of $a_{\lambda}$ since $r \leq t < v = v_3(N-1) = t+1$).

\begin{prop}\label{Supersingular_Eisenstein_alpha}
For all $\alpha \in \{0,1,2\}$, $e_0^{\alpha}$ is annihilated by the ideal $I_{\alpha}$.
\end{prop}
\begin{proof}
The property for $T_{\ell}-\ell-1$ for a prime $\ell \neq 2, N$ follows from Lemma \ref{Supersingular_identity_product_lambda}. The fact that these elements are killed by $U_N-1$ is obvious.
The property for $U_2$ is a formal computation using Proposition \ref{Supersingular_formule_U_2}.
\end{proof}

\begin{conj}\label{Supersingular_non_nuls}
The elements $e_0^1$ and $e_0^0$ are non-zero modulo $p$.
\end{conj}
\begin{rems}
\begin{enumerate}
\item The fact that $e_0^0 \neq 0$ is equivalent to Conjecture \ref{Supersingular_conjecture_supersingular_5} (ii) if $p\geq 5$ and to Conjecture \ref{Supersingular_conjecture_supersingular_3} (i) if $p=3$. 
\item By \cite[Theorems $1.2$ and $1.3$]{Yoo}, the $\mathbf{Z}/p^r\mathbf{Z}$-module $(M'/p^r\cdot M')[I_{\alpha}]$ is free of rank one. Thus,  if Conjecture \ref{Supersingular_non_nuls} is true, we have $(M'/p^r\cdot M')[I_{\alpha}] = \mathbf{Z}/p^r\mathbf{Z}\cdot e_0^{\alpha}$.
\end{enumerate}
\end{rems}

\begin{prop}\label{Supersingular_preuve_non_nuls}
Assume that $2$ is not a $p$th power modulo $N$, \ie $\log(2) \not\equiv 0 \text{ (modulo }p\text{)}$. Then Conjecture \ref{Supersingular_non_nuls} holds.
\end{prop}
\begin{proof}
Assume that for all $\lambda \in L$, $\log(\lambda) \equiv 0 \text{ (modulo }p\text{)}$. We use the same notation as in Proposition \ref{Supersingular_formule_U_2}. We have $$\log(\lambda_1\cdot \lambda_2) \equiv  0 \equiv \log\left(\frac{-16\cdot \lambda}{(1-\lambda)^2}\right) \equiv  \log(-16) \text{ (modulo }p\text{)} \text{ .}$$ This is a contradiction since $p \geq 3$ and $\log(2) \neq 0$. Thus, we have $e_0^0$ is non-zero modulo $p$.

Assume that for all $\lambda \in L$, $$\log(1-\lambda)-2\cdot \log(\lambda) - 4 \cdot \log(2) \equiv 0 \text{ (modulo }p \text{).}$$ We use the same notation as in Proposition \ref{Supersingular_formule_U_2}. We then have 
\begin{align*}
0 &\equiv \log\left((1-\lambda_1)(1-\lambda_2)\right) \\&\equiv 2\cdot \log(\lambda_1) + 2 \cdot \log(\lambda_2) + 8 \cdot \log(2) \\& \equiv 2\cdot \log\left(\frac{-16 \cdot \lambda}{(1-\lambda)^2}\right) + 8 \cdot \log(2) \text{ (modulo }p\text{).}
\end{align*}
Thus, we have $\log(\lambda) - 2 \cdot \log(1-\lambda) + 8 \cdot \log(2) \equiv 0 \text{ (modulo }p\text{)}$, so $-3 \cdot \log(\lambda) \equiv 0 \text{ (modulo }p\text{)}$.
If $p>3$, this is a contradiction by the previous case. If $p=3$, we just work modulo $9$ instead of working modulo $3$.
\end{proof}

We now want to determine $(M'/p^r\cdot M')[I_{\alpha}^2]$. The only result we got in this direction is the following.
\begin{thm}\label{Supersingular_U_2_log}
Let
$$e_{1}^2:=e_0^0+\frac{1}{2}\cdot e_0^1+\frac{1}{2}\cdot \sum_{\lambda \in L} \log(H'(\lambda))\cdot [\lambda] \text{ .}$$
\begin{enumerate}
\item We have, for all prime number $\ell$ not dividing $2\cdot N$:
$$(T_{\ell}-\ell-1)(e_1^2) = \frac{\ell-1}{2}\cdot \log(\ell)\cdot e_0^2 \text{ .}$$
\item We have:
$$
(U_2-2)(e_1^2) = \log(2)\cdot e_0^0
$$
\item We have
$$(U_N-1)(e_1^2)=0 \text{ .}$$
\end{enumerate}
In particular, we have $e_1^2 \in (M'/p^r\cdot M')[I_{\alpha}^2]$.
\end{thm}
\begin{proof}
Point (i) follows from Proposition \ref{Supersingular_Eisenstein_alpha} and Theorem \ref{Supersingular_main_thm}. Point (ii) follows from Proposition \ref{Supersingular_formule_U_2}. Point (iii) is obvious.
\end{proof}
\begin{rem}
Although we will not need it, results of Yoo show that $(M'/p^r\cdot M')[I_2^2] = \mathbf{Z}/p^r\mathbf{Z} \cdot e_0^2 \oplus \mathbf{Z}/p^r\mathbf{Z} \cdot e_1^2$.
\end{rem}

We have note been able to conjecture any explicit formula for the higher Eisenstein elements corresponding to $I_0$ and $I_1$. We have no conjecture for the analogous element $e_2^2$ annihilated by $I_2^3$ (when it exists, \ie when $n(r,p) \geq 2$). The degree of $e_2^2$ is $e_1 \bullet e_1$, which is $$\frac{1}{24}\cdot \left( \sum_{\lambda \in L} \log(H'(\lambda))^2 \right) - \frac{1}{18} \cdot \left( \sum_{\lambda \in L} \log(\lambda)^2 \right)$$ 
if $p\geq 5$ by Theorem \ref{Supersingular_w_1_S_5} (iii).

One idea would be to express $e_2^2$ in terms of $\log(H'(\lambda))^2$, $\log(\lambda)^2$, $\log(1-\lambda)^2$ or more general quadratic formula in the logarithms of differences of supersingular invariants. More precisely, define the following elements of $M'/p^r\cdot M'$:
\begin{enumerate}
\item $$\alpha_1 = \sum_{\lambda \in L} \left(\sum_{\lambda' \in L \atop \lambda' \neq \lambda} \log(\lambda'-\lambda)^2 \right)\cdot [\lambda]$$
\item $$\alpha_2 = \sum_{\lambda \in L} \left(\sum_{\lambda', \lambda'' \in L \atop \lambda' \neq \lambda, \lambda'\neq \lambda'', \lambda\neq \lambda''} \log(\lambda'-\lambda)\cdot \log(\lambda''-\lambda) \right)\cdot [\lambda]$$
\item $$\alpha_3 = \sum_{\lambda \in L} \left(\sum_{\lambda', \lambda'' \in L \atop \lambda' \neq \lambda, \lambda''\neq \lambda', \lambda\neq \lambda''} \log(\lambda'-\lambda)\cdot \log(\lambda''-\lambda') \right)\cdot [\lambda]$$
\item $$\alpha_4 = \sum_{\lambda \in L} \log(\lambda)^2\cdot [\lambda]$$
\item $$\alpha_5 = \sum_{\lambda \in L} \log(1-\lambda)^2\cdot [\lambda]$$
\item $$\alpha_6 = \sum_{\lambda \in L} \log(\lambda)\cdot \log(1-\lambda)\cdot [\lambda]$$
\end{enumerate}
Then we have checked numerically for $N=181$ and $p=5$ that $e_2^2$ modulo $p$ is not a linear combination of the elements $\alpha_i$.

\section{Odd modular symbols: extension of the theory of Sharifi}\label{Section_odd_modSymb}

We keep the notation of chapters \ref{Section_introduction} and \ref{Section_Formalism}. In all this chapter, unless explicitly stated, we assume $p \geq 5$. Let $r$ be an integer such that $1 \leq r \leq t$.

\subsection{Odd and even modular symbols}\label{odd_modSymb_preliminary_setting}
In this section, we do not assume $p\geq 5$. Thus, we only assume that $p \geq 2$ is a prime such that $p^r$ divides the numerator of $\frac{N-1}{12}$. 

Let $M^- = H_1(Y_0(N), \mathbf{Z}_p)^-$ (resp. $H^- = H_1(X_0(N), \mathbf{Z}_p)^-$) be the largest torsion-free quotient of $H_1(Y_0(N), \mathbf{Z}_p)$ (resp. $H_1(X_0(N), \mathbf{Z}_p)$) anti-invariant by the complex conjugation. Let $M_+=H_1(X_0(N), \cusps, \mathbf{Z}_p)_+$ (resp. $H_+ = H_1(X_0(N), \mathbf{Z}_p)_+$) be the subspace of $H_1(X_0(N), \cusps, \mathbf{Z}_p)$ (resp. $H_1(X_0(N), \mathbf{Z}_p)$) fixed by the complex conjugation. With the notation of Theorem \ref{Formalism_existence_eisenstein}, we easily see that $H_+ = (M_+)^0$.  

The intersection product induces perfect $\tilde{\mathbb{T}}$-equivariant pairings:
$$ \bullet :  M_+ \times M^- \rightarrow \mathbf{Z}_p $$
and
$$ \bullet :  H_+ \times H^- \rightarrow \mathbf{Z}_p \text{ .}$$

In order to apply the theory of higher Eisenstein elements as in chapter \ref{Section_Formalism}, we need to check the hypotheses of Theorem \ref{Formalism_existence_eisenstein} as a variant of Mazur's well-known result \cite[Proposition II.18.3]{Mazur_Eisenstein}. 
\begin{prop}\label{even_modSymb_multiplicity_one}
The $\tilde{\mathbf{T}}$-modules $M_+ \otimes_{\tilde{\mathbb{T}}} \tilde{\mathbf{T}}$ and $M^-\otimes_{\tilde{\mathbb{T}}} \tilde{\mathbf{T}}$ are free of rank $1$.
\end{prop}
\begin{proof}
Since $M_+$ and $M^-$ are dual $\tilde{\mathbb{T}}$-modules, it suffices to prove that $M_+ \otimes_{\tilde{\mathbb{T}}} \tilde{\mathbf{T}}$ is free of rank one over $\tilde{\mathbf{T}}$ by Lemma \ref{Formalism_flatness_Hecke} (iii).  By \cite[Proposition II.18.3]{Mazur_Eisenstein}, $H_+ \otimes_{\mathbb{T}} \mathbf{T}$ is free of rank one over $\mathbf{T}$. We now prove that this implies that $M_+ \otimes_{\tilde{\mathbb{T}}} \tilde{\mathbf{T}}$ is free of rank one over $\tilde{\mathbf{T}}$. Consider the exact sequence of $\mathbf{Z}_p$-modules
\begin{equation}\label{odd_even_modSymb_exact_seq1}
0 \rightarrow H_+ \rightarrow M_+ \xrightarrow{\pi} \mathbf{Z}_p \rightarrow 0 \text{ ,}
\end{equation}
where the map $\pi$ sends $\{0,\infty\}$ to $1$. It is a $\tilde{\mathbb{T}}$-equivariant exact sequence if we identify $\mathbf{Z}_p$ with $\tilde{\mathbb{T}}/\tilde{I}$ with its obvious $\tilde{\mathbb{T}}$-module structure.  By Lemma \ref{Formalism_flatness_Hecke} (i), $\tilde{\mathbf{T}}$ is flat over $\tilde{\mathbb{T}}$. Thus, (\ref{odd_even_modSymb_exact_seq1}) gives an exact sequence of $\tilde{\mathbf{T}}$-modules
\begin{equation}\label{odd_even_modSymb_exact_seq2}
0 \rightarrow H_+\otimes_{\mathbb{T}} \mathbf{T} \rightarrow M_+\otimes_{\tilde{\mathbb{T}}} \tilde{\mathbf{T}} \xrightarrow{\pi'} \tilde{\mathbf{T}}/\tilde{I}\cdot \tilde{\mathbf{T}}  \rightarrow 0 \text{ .}
\end{equation}

We claim that the map $\textbf{e} :\tilde{\mathbf{T}} \rightarrow M_+ \otimes_{\tilde{\mathbb{T}}} \tilde{\mathbf{T}}$ given by $T \mapsto \{0,\infty\} \otimes T$ is an isomorphism of $\tilde{\mathbf{T}}$-modules. This follows (\ref{odd_even_modSymb_exact_seq2}), the fact that $\pi'\left(\{0, \infty\} \otimes T \right)$ is the image of $T$ in $\tilde{\mathbf{T}}/\tilde{I}\cdot \tilde{\mathbf{T}}$ and the fact that the restriction of $\textbf{e}$ to $\tilde{I} \cdot \tilde{\mathbf{T}}$ gives an isomorphism of $\mathbf{T}$-modules $\tilde{I} \cdot \tilde{\mathbf{T}} \rightarrow H_+ \otimes_{\tilde{\mathbb{T}}} \tilde{\mathbf{T}}$. The latter fact comes from \cite[Theorem II.18.10]{Mazur_Eisenstein} and the fact that the map $\tilde{I} \rightarrow I$ is an isomorphism of $\tilde{\mathbb{T}}$-modules.
\end{proof}

Thus, Theorem  \ref{Formalism_existence_eisenstein} holds for $M^-$ and $M_+$. Recall (\cf Section \ref{section_intro_odd_modSymb}) that we have denoted by $m_0^-$, $m_1^-$, ..., $m_{n(r,p)}^-$ the higher Eisenstein elements in $M^-/p^r\cdot M^-$. Recall also that we have fixed an element $\tilde{m}_0^- \in M^-[\tilde{I}]$ reducing to $m_0^-$ modulo $p^r$, such that $\{0, \infty\} \bullet \tilde{m}_0^- = -1$. The kernel of the map $M^- \rightarrow H^-$ is $\mathbf{Z}_p\cdot \tilde{m}_0^-$. If $i \in \{1, ..., n(r,p)\}$, we denote by $\overline{m}_i^-$ the image of $m_i^-$ in $H^-$. Since $m_1^-$ is uniquely defined modulo the subgroup generated by $m_0^-$, the element $\overline{m}_1^-$ is uniquely defined and is annihilated by the Eisenstein ideal $I$. If $i>1$, the element $\overline{m}_i^-$ is uniquely defined modulo the subgroup generated by $\overline{m}_1^-$, ..., $\overline{m}_{i-1}^-$.

By intersection duality, the elements $\overline{m}_i^-$ can be considered as elements of $\Hom_{\mathbf{Z}}(H_+, \mathbf{Z}/p^r\mathbf{Z})$. If $k \in \{1, ..., n(r,p)\}$ then $\overline{m}_k^-$ modulo $p^r$ can also be considered as a group homomorphism
$$I^{k-1}\cdot (H_+/p^r \cdot H_+)/I^k \cdot (H_+/p^r \cdot H_+) \rightarrow \mathbf{Z}/p^r\mathbf{Z} \text{ .}$$ 
Indeed, we have an exact sequence of abelian groups:
\begin{align*}
0 \rightarrow \Hom((H_+/p^r \cdot H_+)/I^k \cdot (H_+/p^r \cdot H_+)/ , \mathbf{Z}/p^r\mathbf{Z}) &\rightarrow  \Hom((H_+/p^r \cdot H_+)/I^{k+1} \cdot (H_+/p^r \cdot H_+)/, \mathbf{Z}/p^r\mathbf{Z})  \\&\rightarrow \Hom(I^{k}\cdot (H_+/p^r \cdot H_+)/I^{k+1}\cdot (H_+/p^r \cdot H_+), \mathbf{Z}/p^r\mathbf{Z}) \\& \rightarrow 0 \text{ .}
\end{align*}

Equivalently, for all $k \in \{0,1, ..., n(r,p)\}$ the element $m_k^-$ is uniquely determined by its pairing with $\tilde{I}^k\cdot (M_+/p^r\cdot M^+)$. The element $\overline{m}_1^-$ was essentially determined by Mazur \cite[II.18.8]{Mazur_Eisenstein} (although it was not formulated in this way). Recall the Manin surjective map $$\xi_{\Gamma_0(N)} : \mathbf{Z}[\mathbf{P}^1(\mathbf{Z}/N\mathbf{Z})] \rightarrow H_1(X_0(N), \cusps, \mathbf{Z}_p)$$ whose definition is recalled in Section \ref{section_intro_odd_modSymb}. The $\mathbf{Z}_p$-module $H_1(X_0(N), \mathbf{Z}_p)$ is generated by the element $\xi_{\Gamma_0(N)}([x:1])$ for $x \in (\mathbf{Z}/N\mathbf{Z})^{\times}$ \cite[Proposition 3]{Merel_accouplement}. The group $H_1(X_0(N), \mathbf{Z}_p)$ is acyclic for the complex conjugation \cite[Proposition 5]{Merel_accouplement}.  In particular, we have $H_+ = (1+c)\cdot H_1(X_0(N), \mathbf{Z}_p)$, where $c$ is the complex conjugation. Thus, the element $\overline{m}_1^-$ is uniquely determined by its pairing with $(1+c)\cdot \xi_{\Gamma_0(N)}([x:1])$ for $x \in (\mathbf{Z}/N\mathbf{Z})^{\times}$. This was essentially computed by Mazur  \cite[II.18.8]{Mazur_Eisenstein} (although Mazur does not take into account the complex conjugation).

\begin{thm}\label{odd_even_modSymb_determination_m_1^-}
For all $x \in (\mathbf{Z}/N\mathbf{Z})^{\times}$, we have in $\mathbf{Z}/p^r\mathbf{Z}$:
$$\left((1+c)\cdot \xi_{\Gamma_0(N)}([x:1]) \right) \bullet m_1^-= \log(x) \text{ .}$$
\end{thm}
\begin{proof}
We easily check that there is a unique group homomorphism $f : H_+ \rightarrow \mathbf{Z}/p^r\mathbf{Z}$ such that for all $x \in( \mathbf{Z}/N\mathbf{Z})^{\times}$, we have in $\mathbf{Z}/p^r\mathbf{Z}$:
 $$f\left((1+c)\cdot \xi_{\Gamma_0(N)}([x:1]) \right) = \log(x) \text{ .}$$
We easily see (\cf  \cite[II.18.8]{Mazur_Eisenstein}) that the map $f$ annihilates $I \cdot H_+$. 
Let $\ell$ be a prime not dividing $2\cdot N$. A simple computation shows that we have, in $H_+$:
$$(T_{\ell}-\ell-1)(\{0, \infty\}) = -(1+c)\cdot \sum_{i=1}^{\frac{\ell-1}{2}} \{0, \frac{i}{\ell}\} \text{ .}$$
Thus, we have in $\mathbf{Z}/p^r\mathbf{Z}$:
$$f\left( (T_{\ell}-\ell-1)(\{0, \infty\}) \right) = \frac{\ell-1}{2}\cdot \log(\ell) \cdot \left(\{0, \infty\} \bullet m_0^-\right) = \left((T_{\ell}-\ell-1)(\{0, \infty\})\right) \bullet m_1^-\text{ .}$$
The elements $(T_{\ell}-\ell-1)(\{0, \infty\})$ generate $H_+/I\cdot H_+$ when $\ell$ goes through the primes not dividing $2\cdot N$ \cite[II.18.10]{Mazur_Eisenstein}. Thus, for all $x \in H_+$ we have in $\mathbf{Z}/p^r\mathbf{Z}$:
$$f(x) = x \bullet m_1^- \text{ .}$$
This concludes the proof of Theorem \ref{odd_even_modSymb_determination_m_1^-}.
\end{proof}
\begin{rem}
Let $f : H_+ \rightarrow A$ where $A$ is an abelian group. By intersection duality, $f$ corresponds to an element $\hat{f} \in H_1(X_0(N), A)^-$. Merel gave a formula for $\hat{f}$  in terms of Manin symbols (\cf Lemma \ref{even_modSymb_duality_lemma_RUI_2} for the case where $3$ is invertible in $A$). This allows us to compute $\overline{m}_1^-$ in terms of Manin symbols. If $p\neq 3$ the formula is given in Lemma \ref{even_modSymb_injectivity_trace} (i).
\end{rem}
The aim of this chapter is to determine the element $m_2^-$ modulo $p^r$ when $n(r,p) \geq 2$ and $p\geq 5$. We have seen that $m_2^-$ can be considered as a group isomorphism $$I\cdot (H_+/p^r\cdot H_+) / I^2\cdot (H_+/p^r\cdot H_+) \rightarrow \mathbf{Z}/p^r\mathbf{Z} \text{ .}$$ We will construct such a map, using the modular curve $X_1(N)$ and its Eisenstein ideals.

\subsection{Refined Hida theory}\label{odd_modSymb_section_hida}
In this section, we will use the following notation.
 \begin{itemize}
  \item $\sigma = \begin{pmatrix}
0 & -1 \\
1 & 0
\end{pmatrix} $ and $\tau = \begin{pmatrix}
0 & -1 \\
1 & -1
\end{pmatrix}$ $\in \SL_2(\mathbf{Z})$.
 \item If $\Gamma$ is a subgroup of $\Gamma_0(N)$ containing $\Gamma_1(N)$, let $X_{\Gamma}$ be the compact modular curve associated to $\Gamma$. 
 \item  $C_{\Gamma}^0$ (resp. $C_{\Gamma}^{\infty}$) is the set of cusps of $X_{\Gamma}$ above the cusp $\Gamma_0(N) \cdot 0$ (resp. $\Gamma_0(N) \cdot \infty$) of $X_0(N)$. 
 \item $C_{\Gamma} = C_{\Gamma}^0 \cup C_{\Gamma}^{\infty}$
 \item  $\tilde{H}_{\Gamma}' = H_1(X_{\Gamma}, C_{\Gamma}, \mathbf{Z}_p)$, $\tilde{H}_{\Gamma} = H_1(X_{\Gamma}, C_{\Gamma}^0, \mathbf{Z}_p)$ and $H_{\Gamma} = H_1(X_{\Gamma}, \mathbf{Z}_p)$. 
 \item $\partial : \tilde{H}_{\Gamma}'  \rightarrow \mathbf{Z}_p[C_{\Gamma}]^0$ is the boundary map, sending the geodesic path $\{\alpha, \beta\}$ to $[\beta]-[\alpha]$ where $\alpha$, $\beta$ $\in \mathbf{P}^1(\mathbf{Q})$.
 \item $\left(\tilde{H}_{\Gamma}\right)_+$ is the subgroup of elements of $\tilde{H}_{\Gamma}$ fixed by the complex conjugation. A similar notation applies to $H_{\Gamma}$.
 \item $D_{\Gamma} \subset (\mathbf{Z}/N\mathbf{Z})^{\times}$ is the subgroup generated by the classes of the lower right corners of the elements of $\Gamma$ and by the class of $-1$. 
 \item $\Lambda_{\Gamma} = \mathbf{Z}_p[(\mathbf{Z}/N\mathbf{Z})^{\times}/D_{\Gamma}]$.
 \item If $\Gamma_1$ and $\Gamma_2$ are subgroups of $\SL_2(\mathbf{Z})$ such that $\Gamma_1(N) \subset \Gamma_1 \subset \Gamma_2 \subset \Gamma_0(N)$, we let $J_{1\rightarrow 2} = \Ker(\Lambda_{\Gamma_1} \rightarrow \Lambda_{\Gamma_2})$. It is a principal ideal of $\Lambda_{\Gamma_1}$, generated by $[x]-1$ where $x$ is a generator of $\Ker((\mathbf{Z}/N\mathbf{Z})^{\times}/D_{\Gamma_1} \rightarrow (\mathbf{Z}/N\mathbf{Z})^{\times}/D_{\Gamma_2})$.  
 \item $\tilde{\mathbb{T}}_{\Gamma}'$ (resp. $\tilde{\mathbb{T}}_{\Gamma}$, resp. $\mathbb{T}_{\Gamma}$) is the $\mathbf{Z}_p$-Hecke algebra acting faithfully on $\tilde{H}_{\Gamma}'$ (resp. $\tilde{H}_{\Gamma}$, resp. $H_{\Gamma}$) generated by the Hecke operators $T_n$ for $n \geq 1$ and the diamond operators. The $d$th diamond operator is denoted by $\langle d \rangle$. By convention, it corresponds on modular form to the action of a matrix whose lower right corner is congruent to $d$ modulo $N$.
\end{itemize}

We will need some ``refined Hida control'' results, describing the kernel of the various maps in homology induced by the degeneracy maps between the various modular curves. 

Manin proved \cite[Theorem 1.9]{Manin} that we have a surjection
$$\mathcal{\xi}_{\Gamma} : \mathbf{Z}_p[\Gamma \backslash \PSL_2(\mathbf{Z})] \rightarrow H_1(X_{\Gamma}, C_{\Gamma}, \mathbf{Z}_p)$$
such that $\xi_{\Gamma}(\Gamma\cdot g)$ is the class of the geodesic path $\{g(0), g(\infty)\}$, where $X_{\Gamma}$ is the compact modular curve associated to $\Gamma$. Furthermore, he proved that the kernel of $\xi_{\Gamma}$ is spanned by the sum of the (right) $\sigma$-invariants and $\tau$-invariants.

Recall that $\Gamma_1(N) \subset \Gamma \subset \Gamma_0(N)$. Consider the bijection $$\kappa : \Gamma \backslash \PSL_2(\mathbf{Z}) \xrightarrow{\sim} \left((\mathbf{Z}/N\mathbf{Z})^2\backslash \{(0,0)\} \right) / D_{\Gamma}$$ given by $\kappa(\Gamma \cdot g)= [c,d]$
where $g = \begin{pmatrix}
a & b \\
c & d
\end{pmatrix}$ and $[c,d]$ is the class of $(c,d)$ modulo $D_{\Gamma}$. By abuse of notation, we identify $\Gamma \backslash \PSL_2(\mathbf{Z})$ and $\left((\mathbf{Z}/N\mathbf{Z})^2\backslash \{(0,0)\} \right) / D_{\Gamma}$.

The map $(\mathbf{Z}/N\mathbf{Z})^{\times}/ D_{\Gamma} \rightarrow C_{\Gamma}^0$ (resp. $(\mathbf{Z}/N\mathbf{Z})^{\times}/D_{\Gamma} \rightarrow C_{\Gamma}^{\infty}$) given by $u \mapsto \langle u \rangle \cdot (\Gamma \cdot 0)$ (resp. $u \mapsto \langle u \rangle \cdot (\Gamma\cdot \infty)$) (where $\langle \cdot\rangle$ denotes the diamond operator) is a bijection. If $u \in (\mathbf{Z}/N\mathbf{Z})^{\times}/D_{\Gamma} $, we denote by $[u]_{\Gamma}^0$ (resp. $[u]_{\Gamma}^{\infty}$) the image of $u$ in $C_{\Gamma}^0$ (resp. $C_{\Gamma}^{\infty}$). In other words, we have $[u]_{\Gamma}^0 = \Gamma \cdot \frac{c}{d}$ for some coprime integers $c$ and $d$ not divisible by $N$, and such that the image of $d$ in $(\mathbf{Z}/N\mathbf{Z})^{\times}/D_{\Gamma} $ is $u$. Similarly, $[u]_{\Gamma}^{\infty} = \Gamma \cdot \frac{a}{N\cdot b}$ for some coprime integers $a$ and $b$ not divisible by $N$, and such that the image of $a$ in $(\mathbf{Z}/N\mathbf{Z})^{\times}/D_{\Gamma} $ is $u^{-1}$. 

Let $\begin{pmatrix} a&b \\ c&d \end{pmatrix} \in \SL_2(\mathbf{Z})$. We describe $\partial(\xi_{\Gamma}([c,d]))$ in the various cases that can happen.

\begin{itemize}
\item If $c \equiv 0 \text{ (modulo }N\text{)}$ then $a \equiv d^{-1} \text{ (modulo }N\text{)}$.
Thus, we have $\partial(\xi_{\Gamma}([c,d])) = [d]_{\Gamma}^{\infty} - [d]_{\Gamma}^{0}$.

\item If $d \equiv 0 \text{ (modulo }N\text{)}$ then we have $b \equiv -c^{-1} \text{ (modulo }N\text{)}$.
Thus, we have $\partial(\xi_{\Gamma}([c,d])) = [c]_{\Gamma}^{0} - [c]_{\Gamma}^{\infty}$.

\item If  $c\cdot d \not\equiv 0 \text{ (modulo }N\text{)}$ then we have $\partial(\xi_{\Gamma}([c,d])) = [c]_{\Gamma}^{0} - [d]_{\Gamma}^{0}$.
\end{itemize}

In particular, the set of $[c,d]$ such that $\partial(\xi_{\Gamma}([c,d])) \in \mathbf{Z}[C_{\Gamma}^0]$ coincides with the set of $[c,d]$ such that $c\cdot d \not\equiv 0 \text{ (modulo }N \text{)}$. Let $M_{\Gamma}^0$ be the sub-$\mathbf{Z}_p$-module of $\mathbf{Z}_p[\left((\mathbf{Z}/N\mathbf{Z})^2\backslash \{(0,0)\} \right) / D_{\Gamma}]$ generated by the symbols $[c,d]$ with $c\cdot d \not\equiv 0 \text{ (modulo }N\text{)}$. 

The following statement is well-known, but we could not find a reference.
\begin{prop}\label{generation_Manin_C_0^{(p^r)}} 
The map $\xi_{\Gamma}$ induces a surjective homomorphism
$$\xi_{\Gamma}^0 : M_{\Gamma}^0 \rightarrow \tilde{H}_{\Gamma}$$
whose kernel is $R_{\Gamma}^0 = (M_{\Gamma}^0)^{\tau} + (M_{\Gamma}^0)^{\sigma} + \sum_{d \in (\mathbf{Z}/N\mathbf{Z})^{\times}} \mathbf{Z}_p \cdot [-d,d]$ where $(M_{\Gamma}^0)^{\tau}$ (resp. $(M_{\Gamma}^0)^{\sigma}$) is the subgroup of elements of $M_{\Gamma}^0$ fixed by the right action of $\tau$ (resp. $\sigma$).
\end{prop}
\begin{proof}
Let $\xi_{\Gamma}' = \xi_{\Gamma} \circ \kappa^{-1} :  \mathbf{Z}_p[\left((\mathbf{Z}/N\mathbf{Z})^2\backslash \{(0,0)\} \right) / D_{\Gamma}] \rightarrow \tilde{H}_{\Gamma}'$ and $\xi_{\Gamma}^0$ be the restriction of $\xi_{\Gamma}'$ to $M_{\Gamma}^0$. The computation of $\partial$ shows that $\xi_{\Gamma}^0$ takes values in $\tilde{H}_{\Gamma}$. 
Let $y \in \tilde{H}_{\Gamma}$. Since $\xi_{\Gamma}'$ is surjective, there is some element $x = \sum_{[c,d] \in  \left((\mathbf{Z}/N\mathbf{Z})^2\backslash \{(0,0)\} \right) / D_{\Gamma}} \lambda_{[c,d]} \cdot [c,d] \in \mathbf{Z}_p[\left((\mathbf{Z}/N\mathbf{Z})^2\backslash \{(0,0)\} \right) / D_{\Gamma}]$ such that $\xi_{\Gamma}'(x)=y$. Since $\partial \xi_{\Gamma}'(x) \in \mathbf{Z}_p[C_{\Gamma}^0]$, we have $\lambda_{[d,0]} = \lambda_{[0,d]}$ for all $d \in  \left((\mathbf{Z}/N\mathbf{Z})^2\backslash \{(0,0)\} \right) / D_{\Gamma}$. Since $\xi_{\Gamma}'([0,d]+[d,0]) = 0$, the element $y$ is in the image of $\xi_{\Gamma}^0$. Thus, we have proved that $\xi_{\Gamma}^0$ is surjective.

Let $x = \sum_{[c,d] \in  \left((\mathbf{Z}/N\mathbf{Z})^2\backslash \{(0,0)\} \right) / D_{\Gamma}} \lambda_{[c,d]} \cdot [c,d] - \mu_{[c,d]} \cdot [c,d] \in \Ker(\xi_{\Gamma}^0) = \Ker(\xi_{\Gamma}') \cap M_{\Gamma}^0$ with $\lambda_{[c,d]}=\lambda_{[c,d]\cdot \tau}$ and $\mu_{[c,d]}=\mu_{[c,d]\cdot \sigma}$ for all $[c,d] \in  \left((\mathbf{Z}/N\mathbf{Z})^2\backslash \{(0,0)\} \right) / D_{\Gamma}$.  We also have $\lambda_{[d,0]} = \mu_{[d,0]}$ and $\lambda_{[0,d]} = \mu_{[0,d]}$ for all $d \in (\mathbf{Z}/N\mathbf{Z})^{\times}/D_{\Gamma}$.
Note that for all $d \in (\mathbf{Z}/N\mathbf{Z})^{\times}/D_{\Gamma}$, we have:
\begin{equation}\label{[-u,u]_relation}
[d,-d]=([d,0]+[0,d]+[d,-d])-([d,0]+[0,d]) \in \Ker(\xi_{\Gamma}^0) \text{ .}
\end{equation}
Hence, $x - \sum_{d \in (\mathbf{Z}/N\mathbf{Z})^{\times}/D_{\Gamma}} \lambda_{[d,0]} \cdot [d,-d] \in (M_{\Gamma}^0)^{\sigma} + (M_{\Gamma}^0)^{\tau}$ and $x$ has the form sought after.

\end{proof}

\begin{corr}\label{surjection_homology}
The map $\tilde{H}_{\Gamma_1} \rightarrow \tilde{H}_{\Gamma_2}$ is surjective.
\end{corr}
\begin{proof}
The map $M_{\Gamma_1}^0 \rightarrow M_{\Gamma_2}^0$ is surjective. We conclude using Proposition \ref{generation_Manin_C_0^{(p^r)}}.
\end{proof}

The ring $\Lambda_{\Gamma_i}$ acts naturally on $R_{\Gamma_i}^0$, $M_{\Gamma_i}$ and $\tilde{H}_{\Gamma_i}$ (for $i=1,2$).

\begin{prop}\label{odd_modSymb_refined_Hida}
\begin{enumerate}
\item The kernel of the homomorphism $\tilde{H}_{\Gamma_1} \rightarrow \tilde{H}_{\Gamma_2}$ is $J_{1 \rightarrow 2}  \cdot \tilde{H}_{\Gamma_1}$. 
\item The kernel of the homomorphism $H_{\Gamma_1} \rightarrow H_{\Gamma_2}$ is $J_{1 \rightarrow 2}  \cdot H_{\Gamma_1}$.
\end{enumerate}
\end{prop}

\begin{proof}
We prove point (i).
Consider the following commutative diagram, where the rows are exact:
$$\xymatrix{
   0 \ar[r] &  R_{\Gamma_1}^0 \ar[r] \ar[d]  & M_{\Gamma_1}^0 \ar[r] \ar[d]  & \tilde{H}_{\Gamma_1}\ar[r] \ar[d]& 0  \\
    0 \ar[r] & R_{\Gamma_2}^0 \ar[r]  & M_{\Gamma_2}^0  \ar[r] & \tilde{H}_{\Gamma_2} \ar[r] & 0   
  } $$

It is clear that the kernel of the middle vertical arrow is $J_{1 \rightarrow 2}\cdot M_{\Gamma_1}^0$. The cokernel of the left vertical map is zero by Proposition \ref{generation_Manin_C_0^{(p^r)}} (using $p>3$). The snake lemma concludes the proof of point (i).

We now prove point (ii). Using point (i), it suffices to show that $H_{\Gamma_1} \cap \left(J_{1\rightarrow 2} \cdot \tilde{H}_{\Gamma_1}\right) = J_{1\rightarrow 2}\cdot H_{\Gamma_1}$. Consider the following commutative diagram, where the rows are exact:
$$\xymatrix{
   0 \ar[r] &   H_{\Gamma_1} \ar[r] \ar[d]  & \tilde{H}_{\Gamma_1} \ar[r] \ar[d]  &  \mathbf{Z}_p[C_{\Gamma_1}^0] ^0\ar[r] \ar[d]& 0  \\
    0 \ar[r] &  H_{\Gamma_1} \ar[r]  & \tilde{H}_{\Gamma_1}  \ar[r] &   \mathbf{Z}_p[C_{\Gamma_1}^0]^0 \ar[r] & 0   
  }$$
Here, the vertical maps are induced by the action of $[d]-1$ where $d$ is a fixed generator of $\Ker((\mathbf{Z}/N\mathbf{Z})^{\times}/D_{\Gamma_1} \rightarrow (\mathbf{Z}/N\mathbf{Z})^{\times}/D_{\Gamma_2})$, and $ \mathbf{Z}_p[C_{\Gamma_1}^0] ^0$ is the augmentation subgroup of $ \mathbf{Z}_p[C_{\Gamma_1}^0] $. Note that $J_{1 \rightarrow 2}$ is principal, generated by $[d]-1$. Thus, to prove (ii) it suffices to show (using the snake Lemma) that the map $\tilde{H}_{\Gamma_1}[J_{1 \rightarrow 2}] \rightarrow \mathbf{Z}_p[C_{\Gamma_1}^0] ^0[J_{1 \rightarrow 2}]$ is surjective. 

It suffices to show that the boundary map $M_{\Gamma_1}^0[J_{1 \rightarrow 2}] \rightarrow  \mathbf{Z}_p[C_{\Gamma_1}^0] ^0[J_{1 \rightarrow 2}]$ is surjective. Since we can identify $C_{\Gamma_1}^0$ with $(\mathbf{Z}/N\mathbf{Z})^{\times}/D_{\Gamma_1}$, the action of $(\mathbf{Z}/N\mathbf{Z})^{\times}/D_{\Gamma_1}$ on $C_{\Gamma_1}^0$ is free. Thus, any element of $\mathbf{Z}_p[C_{\Gamma_1}^0] ^0[J_{1 \rightarrow 2}]$ is of the form $\sum_{x \in C_{\Gamma_1}^0} \lambda_x \cdot (\sum_{k=0}^{m-1} [d^{k-1}])\cdot [x]$ where $m$ is the order of $d$. We have $\sum_{x \in C_{\Gamma_1}^0} \lambda_x = 0$. Thus, $\mathbf{Z}_p[C_{\Gamma_1}^0] ^0[J_{1 \rightarrow 2}]$ is spanned over $\mathbf{Z}_p$ by the elements $(\sum_{k=0}^{m-1} [d^{k-1}])\cdot ([u]-[v])$ for $u$, $v$ $\in C_{\Gamma_1}^0$. If we identify $u$ and $v$ with elements of $(\mathbf{Z}/N\mathbf{Z})^{\times}/D_{\Gamma_1}$ and lift them to elements of $(\mathbf{Z}/N\mathbf{Z})^{\times}$, $(\sum_{k=0}^{m-1} [d^{k-1}])\cdot ([u]-[v])$ is the boundary of the Manin symbol $(\sum_{k=0}^{m-1} [d^{k-1}])\cdot [u,v]$, which is annihilated by $J_{1 \rightarrow 2}$. This concludes the proof of point (ii).

\end{proof}

We have the analogous (certainly well-known) statement for the Hecke algebras.

\begin{prop}\label{odd_modSymb_refined_Hida_Hecke}
The kernel of the restriction map $\tilde{\mathbb{T}}_{\Gamma_1}\rightarrow \tilde{\mathbb{T}}_{\Gamma_2}$ (resp. $\mathbb{T}_{\Gamma_1} \rightarrow \mathbb{T}_{\Gamma_2}$) is $J_{1 \rightarrow 2} \cdot \tilde{\mathbb{T}}_{\Gamma_1}$ (resp. $J_{1 \rightarrow 2} \cdot \mathbb{T}_{\Gamma_1}$).
\end{prop}
\begin{proof}
We prove the first assertion. There is an isomorphism of $\tilde{\mathbb{T}}_{\Gamma_i}$-modules
\begin{equation}\label{duality_T_M}
\Hom_{\mathbf{Z}_p}(\tilde{\mathbb{T}}_{\Gamma_i}, \mathbf{Z}_p) \simeq M_2(\Gamma_i, \mathbf{Z}_p) \text{ ,}
\end{equation}
where $M_2(\Gamma_i, \mathbf{Z}_p) = M_2(\Gamma_i, \mathbf{Z}) \otimes_{\mathbf{Z}} \mathbf{Z}_p$ and $M_2(\Gamma_i, \mathbf{Z})$ is the set of modular form for $\Gamma_i$ whose $q$-expansion at the cusp $\infty$ has coefficients in $\mathbf{Z}$. We have
$$M_2(\Gamma_2, \mathbf{Z}_p)= M_2(\Gamma_1, \mathbf{Z}_p)[J_{1 \rightarrow 2}]\text{, }$$
where the action of $\Lambda_{\Gamma_1}$ on $M_2(\Gamma_1, \mathbf{Z}_p)$ is via the diamond operators.  This proves the first assertion using (\ref{duality_T_M}). The proof of the second assertion is identical, replacing $ M_2(\Gamma_i, \mathbf{Z}_p)$ by the space of cusp forms  $S_2(\Gamma_i, \mathbf{Z}_p)$.
\end{proof}

\subsection{Eisenstein ideals of $X_1^{(p^r)}(N)$}\label{odd_modSymb_section_eisenstein_ideal}
We keep the notation of section \ref{odd_modSymb_section_hida} and add the following ones.
\begin{itemize}
\item $P_r=(\mathbf{Z}/N\mathbf{Z})^{\times}/\left( (\mathbf{Z}/N\mathbf{Z})^{\times} \right)^{p^r}$.
 \item $P_r' =\left( (\mathbf{Z}/N\mathbf{Z})^{\times} \right)^{p^r}$.
 \item $\Lambda_r = \mathbf{Z}_p[P_r]$.
 \item $J_r \subset \Lambda_r$ is the augmentation ideal.
 \item $J_r' = \Ker\left(\mathbf{Z}_p[(\mathbf{Z}/N\mathbf{Z})^{\times}] \rightarrow \Lambda_r \right)$.
 \item $\Gamma_1^{(p^r)}(N) \subset \Gamma_0(N)$ is the subgroup of $\Gamma_0(N)$ corresponding to the matrices whose diagonal entries are in $P_r'$ modulo $N$.
  \item If $\Gamma = \Gamma_1^{(p^r)}(N)$, we let $X_1^{(p^r)}(N)=X_{\Gamma}$, $\tilde{H}^{(p^r)} = \tilde{H}_{\Gamma}$, $\left(\tilde{H}^{(p^r)}\right)_+ = \left(\tilde{H}_{\Gamma}\right)_+$, $H^{(p^r)} = H_{\Gamma}$, $\left(H^{(p^r)} \right)_+ = \left(  H_{\Gamma} \right)_+$, $\widetilde{\mathbb{T}'}^{(p^r)} = \tilde{\mathbb{T}}_{\Gamma}'$, $\tilde{\mathbb{T}}^{(p^r)} = \tilde{\mathbb{T}}_{\Gamma}$, $\mathbb{T}^{(p^r)} = \mathbb{T}_{\Gamma}$, $C_0^{(p^r)} = C_{\Gamma}^0$ and $C_{\infty}^{(p^r)} = C_{\Gamma}^{\infty}$.
 \item If $\Gamma = \Gamma_0(N)$, we recall that $\tilde{H} = \tilde{H}_{\Gamma}$, $H = H_{\Gamma}$, $\tilde{\mathbb{T}} = \tilde{\mathbb{T}}_{\Gamma}$ and $\mathbb{T}= \mathbb{T}_{\Gamma}$.
 \item We define two Eisenstein ideals in $\widetilde{\mathbb{T}'}^{(p^r)}$, $\tilde{\mathbb{T}}^{(p^r)}$ and $\mathbb{T}^{(p^r)}$. The set $C_{\infty}^{(p^r)}$ is annihilated by the Eisenstein ideal $\tilde{I}_{\infty}'$ of $\widetilde{\mathbb{T}'}^{(p^r)}$ generated by the operators $T_{n} - \sum_{d \mid n, \gcd(d,N)=1} \langle d \rangle \cdot \frac{n}{d}$. Similarly, $C_0^{(p^r)}$ is annihilated by the Eisenstein ideal $\tilde{I}_{0}'$ of $\widetilde{\mathbb{T}'}^{(p^r)}$ generated by the operators $T_{n} - \sum_{d \mid n, \gcd (d,N)=1} \langle d \rangle \cdot d$.
 \item We denote by $\tilde{I}_{\infty}$ and $\tilde{I}_0$ (resp. $I_{\infty}$ and $I_0$) the respective images of $\tilde{I}_{\infty}'$ and $\tilde{I}_0'$ in $\tilde{\mathbb{T}}^{(p^r)}$ (resp. $\mathbb{T}^{(p^r)}$).
\end{itemize}

The main goal of this section is to give an explicit description of $\tilde{H}^{(p^r)}/\tilde{I}_{\infty} \cdot \tilde{H}^{(p^r)}$. The Hecke algebra $\widetilde{\mathbb{T}'}^{(p^r)}$ (resp. $\tilde{\mathbb{T}}^{(p^r)}$, $\mathbb{T}^{(p^r)}$) acts faithfully on the space of modular forms of weight $2$ and level $\Gamma_1^{(p^r)}(N)$ (resp. which vanish at the cusps in $C_0^{(p^r)}$, resp. which are cuspidal). 

Let
$$ \zeta^{(p^r)} = \sum_{x \in (\mathbf{Z}/N\mathbf{Z})^{\times}}  \B_2\left(\frac{x}{N}\right)\cdot [x^{-1}] \in \Lambda_r $$
and
$$\delta^{(p^r)} =  \sum_{x \in P_r}  [x] \in \Lambda_r \text{ .}$$

The following lemma will be useful in our proofs. It is an immediate consequence of Nakayama's lemma, since $\Lambda_r$ is a local ring.

\begin{lem}\label{odd_modSymb_Nakayama}
Let $f : M_1 \rightarrow M_2$ be a morphism of finitely generated $\Lambda_r$-modules. Let $\overline{f} :  M_1 \rightarrow M_2/J_r\cdot M_2$ be the map obtained from $f$. Then $f$ is surjective if and only if $\overline{f}$ is surjective.
\end{lem}

The following result is analogous to Mazur's computation of $\mathbb{T}/I$ \cite[Proposition II.9.7]{Mazur_Eisenstein}. In fact, our proof uses Mazur's results and techniques. 
\begin{thm}\label{odd_modSymb_Eisenstein_ideal_X1}
Assume that $p \geq 5$.
\begin{enumerate}
\item The map $\Lambda_r\rightarrow \tilde{\mathbb{T}}^{(p^r)}$ given by $[d] \mapsto \langle d \rangle$ gives an isomorphism of $\Lambda_r$-modules $$\Lambda_r/(\zeta^{(p^r)}) \xrightarrow{\sim} \tilde{\mathbb{T}}^{(p^r)}/\tilde{I}_{\infty} \text{ .}$$

\item The map $\Lambda_r\rightarrow \mathbb{T}^{(p^r)}$ given by $[d] \mapsto \langle d \rangle$ gives an isomorphism of $\Lambda_r$-modules $$\Lambda_r/\left(\zeta^{(p^r)}, p^{t-r}\cdot \delta^{(p^r)}\right) \xrightarrow{\sim} \mathbb{T}^{(p^r)}/I_{\infty} \text{ .}$$

\item The groups $\tilde{\mathbb{T}}^{(p^r)}/\tilde{I}_{\infty}$ and $\mathbb{T}^{(p^r)}/I_{\infty}$ are finite.
\end{enumerate}
\end{thm}
\begin{proof}
The assertion (iii) follows from (i), (ii) and the well-known property of Stickelberger elements (non vanishing of $L(\chi,2)$ for any even Dirichlet character $\chi$).

We first prove the following result.
\begin{lem}\label{odd_modSymb_delta_modular_form}
Let $I$ be an ideal of $\Lambda_r$ and 
$$ E_{\infty} :=\sum_{n \geq 1} \left(\sum_{d \mid n \atop \gcd(d,N)=1}  [d] \cdot \frac{n}{d}\right)\cdot q^n \in \Lambda_r[[q]] \text{ .}$$
\begin{enumerate}
\item The series $\frac{N-1}{24\cdot p^r}\cdot \delta^{(p^r)} + E_{\infty}$ is the $q$-expansion at the cusp $\infty$ of a modular form of weight $2$ and level $\Gamma_1^{(p^r)}(N)$ over $\Lambda_r$.
\item The following assertions are equivalent.

\begin{enumerate}
\item There exists $\alpha \in \Lambda_r/I$ such that the image $F_{\alpha}$ of $\alpha + E_{\infty}$ in $(\Lambda_r/I)[[q]]$ satisfies:
\begin{itemize}
\item $F_{\alpha}$ is the $q$-expansion at the cusp $\infty$ of a modular form of weight $2$, level $\Gamma_1^{(p^r)}(N)$ over $\Lambda_r/I$ which vanishes at the cusps in $C_0^{(p^r)}$.
\item For all $d \in P$, we have $\langle d \rangle F_{\alpha} = [d] \cdot F_{\alpha}$.
\end{itemize}
\item We have $\zeta^{(p^r)} \in I$.
\end{enumerate}
\item The following assertions are equivalent.

\begin{enumerate}
\item There exists $\alpha \in \Lambda_r/I$ such that the image $F_{\alpha}$ of $\alpha + E_{\infty}$ in $(\Lambda_r/I)[[q]]$ satisfies:
\begin{itemize}
\item $F_{\alpha}$ is the $q$-expansion at the cusp $\infty$ of a cuspidal modular form of weight $2$, level $\Gamma_1^{(p^r)}(N)$ over $\Lambda_r/I$.
\item For all $d \in P$, we have $\langle d \rangle F_{\alpha} = [d] \cdot F_{\alpha}$.
\end{itemize}
\item We have $\zeta^{(p^r)} \in I$ and $p^{t-r}\cdot \delta^{(p^r)} \in I$.
\end{enumerate}
\end{enumerate}
\end{lem}
\begin{proof}

Fix an embedding of $\overline{\mathbf{Q}}_p \hookrightarrow \mathbf{C}$. For any non-trivial character $\epsilon : P_r \rightarrow \mathbf{C}^{\times}$, the $q$-expansion $$\sum_{n \geq 1} \left(\sum_{d \mid n, \atop \gcd(d,N)=1}  \epsilon(d) \cdot \frac{n}{d}\right)\cdot q^n \in \mathbf{C}[[q]]$$ is the $q$-expansion at the cusp $\infty$ of an Eisenstein series of weight $2$ and level $\Gamma_1^{(p^r)}(N)$, which we denote by $E_{\epsilon,1}$ (\cf for instance \cite[Theorem $4.6.2$]{Diamond_Shurman}). Furthermore, we have already seen that $\frac{N-1}{24} + \sum_{n \geq 1} \left(\sum_{d \mid n, \gcd(d,N)=1} \frac{n}{d}\right)\cdot q^n$ is the $q$-expansion at the cusp $\infty$ of an Eisenstein series of level $\Gamma_0(N)$, denoted by $E_2$. 

Using our fixed embedding $\mathbf{Q}_p \hookrightarrow \mathbf{C}$, we get a natural injective ring homomorphism $\iota : \Lambda_r\rightarrow \prod_{\epsilon \in \hat{P}_r} \mathbf{C}$ where $\hat{P}_r$ is the set of characters of $P_r$.

Thus, we have shown that $F_{\infty}:=\frac{N-1}{24 \cdot p^r} \cdot \delta^{(p^r)} + E_{\infty} \in \Lambda_r[[q]]$ is the $q$-expansion at the cusp $\infty$ of a modular form of weight $2$ and level $\Gamma_1^{(p^r)}(N)$ (still denoted by $F_{\infty}$) over $\prod_{\epsilon \in \hat{P}_r} \mathbf{C}$. By the $q$-expansion principle \cite[Corollary 1.6.2]{Katz_properties}, such a modular form is over $\Lambda_r$. This proves point (i) of Lemma \ref{odd_modSymb_delta_modular_form}. 

We now prove point (ii). Since for all $d \in P_r$ we have $\langle d \rangle E_{\epsilon,1} = \epsilon(d) \cdot E_{\epsilon, 1}$, the $q$-expansion principle shows that
\begin{equation}\label{odd_modSymb_action_Atkin_Lehner}
\langle d \rangle  F_{\infty} = [d]\cdot F_{\infty} \text{ .}
\end{equation}
Let $\alpha \in \Lambda_r$ such that point (a) of (ii) holds and let $\overline{F}_{\infty}$ be the image of $F_{\infty}$ modulo $I$. The element $F:= F_{\alpha} - \overline{F}_{\infty} = \alpha - \frac{N-1}{24 \cdot p^r} \cdot \delta^{(p^r)} \in \Lambda_r/I \subset (\Lambda_r/I)[[q]]$ is the $q$-expansion at the cusp $\infty$ of a modular form of weight $2$ and level $\Gamma_1^{(p^r)}(N)$ over $\Lambda_r/I$. Furthermore, for all $d\in P_r$, we have by the assumptions on $F_{\alpha}$ and (\ref{odd_modSymb_action_Atkin_Lehner}): 
$$\langle d \rangle F = \langle d \rangle F_{\alpha} - \langle d \rangle \overline{F}_{\infty} = [d]\cdot F_{\alpha} - [d] \cdot \overline{F}_{\infty} = [d]\cdot F \text{ .}$$

\begin{lem}\label{odd_modSymb_constant_modular_form}
Let $I$ be an ideal of $\Lambda_r$ and $F \in \Lambda_r/I$. Assume that $F$ is the $q$-expansion of a modular form of weight $2$ and level $\Gamma_1^{(p^r)}(N)$ over $\Lambda_r/I$ such that for all $d \in P_r$, we have $\langle d \rangle F = [d] \cdot F$. Then we have $F=0$.
\end{lem}
\begin{proof}
We prove by induction on $n\geq 0$ that $F \in J_r^n \cdot (\Lambda_r/I)$. This is true if $n=0$. Assume that this is true for some $n \geq 0$. By the $q$-expansion principle, $F$ is the $q$-expansion at the cusp $\infty$ of a modular form over the $\mathbf{Z}[\frac{1}{N}]$-module $J_r^n \cdot (\Lambda_r/I)$ (\cf \cite[Section 1.6]{Katz_properties} for the notion of a modular form over an abelian group). Let $\overline{F}$ be the image $F$ by the map $J_r^n \cdot (\Lambda_r/I) \twoheadrightarrow J_r^n \cdot (\Lambda_r/I)/J_r^{n+1} \cdot (\Lambda_r/I) $. The diamond operators act trivially on $\overline{F}$. Thus, $\overline{F}$ is the $q$-expansion at the cusp $\infty$ of a modular form of weight $2$ and level $\Gamma_0(N)$ with coefficients in the module $J_r^n \cdot (\Lambda_r/I)/J_r^{n+1} \cdot (\Lambda_r/I) $. Note that $J_r^n \cdot (\Lambda_r/I)/J_r^{n+1} \cdot (\Lambda_r/I) $ is a quotient of $J_r^n/J_r^{n+1} \simeq \mathbf{Z}/p^r\mathbf{Z}$. Since $p \geq 5$ and $\gcd(N,p)=1$, \cite[Lemma 5.9, Corollary 5.11]{Mazur_Eisenstein} shows that $\overline{F}=0$, \ie we have $F \in J_r^{n+1} \cdot (\Lambda_r/I)$. This concludes the induction step. Since $\bigcap_{n \geq 0} J_r^n \cdot (\Lambda_r/I) = 0$, we have $F=0$. This concludes the proof of Lemma \ref{odd_modSymb_constant_modular_form}.
\end{proof}

By Lemma \ref{odd_modSymb_constant_modular_form}, we have $F=0$, \ie $F_{\alpha}= \overline{F}_{\infty}$.
Let $w_N$ be the Atkin--Lehner involution. 
By \cite[Proposition 1]{Weisinger_thesis}, the $q$-expansion at the cusp $\infty$ of $w_N(E_{\epsilon, 1})$ is
$$\frac{1}{N}\cdot\left(\sum_{x \in (\mathbf{Z}/N\mathbf{Z})^{\times}} \epsilon(x) \cdot e^{\frac{2 i \pi x }{N}} \right) \cdot \left(-\frac{N}{4} \cdot \left( \sum_{x \in (\mathbf{Z}/N\mathbf{Z})^{\times}} \epsilon(x)^{-1}\cdot \B_2(\frac{x}{N}) \right)+  \sum_{n \geq 1} \sum_{d \mid n \atop \gcd(d,N)=1}  \epsilon\left(\frac{n}{d}\right)^{-1} \cdot \frac{n}{d} \cdot q^n \right) \in \Lambda_r[[q]] \text{ .}$$
Furthermore, we have $w_N(E_2) = -E_2$.

Let $\mu \in \overline{\mathbf{Z}}_p$ be the primitive $N$th root of unity corresponding to $e^{\frac{2 i \pi}{N}}$ under our fixed embedding $\overline{\mathbf{Q}}_p \hookrightarrow \mathbf{C}$. Let $\Lambda_r' = (\mathbf{Z}_p[\mu])[P_r]$ and $\mathcal{G'} = \sum_{x \in (\mathbf{Z}/N\mathbf{Z})^{\times}} \mu^x \cdot [x] \in \Lambda_r'$. The element $\mathcal{G}'$ is invertible in $\Lambda_r'$ since its degree is $-1$, which is prime to $p$, and $\Lambda_r'$ is a local ring whose maximal ideal is $J'+(\varpi)$ where $J'$ is the augmentation ideal of $\Lambda_r'$ and $\varpi$ is a uniformizer of $\mathbf{Z}_p[\mu]$.

By reformulating Weisinger's formula, the $q$-expansion principle shows that the $q$-expansion at the cusp $\infty$ of $F_0:=w_N(F_{\infty})$ is 
\begin{equation}
-\frac{1}{4}\cdot \mathcal{G}' \cdot \zeta^{(p^r)} + \mathcal{G}' \cdot \sum_{n\geq 1} \left(\sum_{d \mid n, \text{ gcd}(d,N)=1}  \left[\frac{n}{d}\right]^{-1} \cdot \frac{n}{d}\right)\cdot q^n \text{ . }
\end{equation}

Since the constant coefficients of the modular form associated to $F_{\alpha} = \overline{F}_{\infty}$ at the cusps of $C_0^{(p^r)}$ are zero, this proves that $\mathcal{G'} \cdot \zeta^{(p^r)} \in I$. Since $\mathcal{G'}$ is a unit of $\Lambda_r$, we have $\zeta^{(p^r)} \in I$.

Conversely, if $\zeta^{(p^r)} \in I$ then the image $\overline{F}_{\infty}$ of $F_{\infty}$ in $(\Lambda_r/I)[[q]]$ is the $q$-expansion at the cusp $\infty$ of a modular form of weight $2$ and level $\Gamma_1^{(p^r)}(N)$ over $\Lambda_r/I$ whose constant coefficient at the cusp $0$ is zero. Since the diamond operators act transitively on $C_0^{(p^r)}$, equation (\ref{odd_modSymb_action_Atkin_Lehner}) shows that the constant terms of the modular form associated to $\overline{F}_{\infty}$ at any cusp of $C_0^{(p^r)}$ is zero. This proves point (ii).

The proof of point (iii) is similar to the proof of point (ii). This concludes the proof of Lemma \ref{odd_modSymb_delta_modular_form}.
\end{proof}

We now prove that Theorem \ref{odd_modSymb_Eisenstein_ideal_X1} follows from Lemma \ref{odd_modSymb_delta_modular_form}. The ring homomorphism $\Lambda_r\rightarrow \tilde{\mathbb{T}}^{(p^r)}/\tilde{I}_{\infty}$ is surjective since we have $T_n - \sum_{d \mid n, \text{ gcd}(d,n\text{)}=1}\frac{n}{d}\cdot \langle d \rangle \in \tilde{I}_{\infty}$ for all $n \geq 1$. Let $K$ be its kernel. There is some $\alpha \in \Lambda_r$ such that the image of $\alpha + E_{\infty}$ in $\Lambda_r/K$ is the $q$-expansion at the cusp $\infty$ of a modular form of weight $2$ and level $\Gamma_1^{(p^r)}(N)$ over $\Lambda_r/K$ satisfying the two properties of Lemma \ref{odd_modSymb_delta_modular_form} (ii) (a), and $K$ is the smallest such ideal of $\Lambda_r$. By Lemma \ref{odd_modSymb_delta_modular_form} (i), we have $K = (\zeta^{(p^r)})$. Point (ii) of Theorem \ref{odd_modSymb_Eisenstein_ideal_X1} follows in a similar way.
\end{proof}

In the following result, we extend the winding homomorphism of Mazur \cite[p. 137]{Mazur_Eisenstein}.

\begin{thm}\label{odd_modSymb_structure_H_1}
The $\Lambda_r$-module $\left(\tilde{H}^{(p^r)}\right)_+/\tilde{I}_{\infty} \cdot  \left(\tilde{H}^{(p^r)}\right)_+$ is isomorphic to $\Lambda_r/(\zeta^{(p^r)})$. 
\end{thm}
\begin{proof} 
We first define a homomorphism of $\Lambda_r$-modules $\tilde{e}: \tilde{I}_{\infty} \rightarrow \left(\tilde{H}^{(p^r)}\right)_+$ as follows.

We have a map $\tilde{I}_{\infty}' \rightarrow \left(\tilde{H}^{(p^r)}\right)_+$
given by $\eta \mapsto \eta\cdot \{0,\infty\}$. 
This induces the desired map $\tilde{I}_{\infty} \rightarrow \left(\tilde{H}^{(p^r)}\right)_+$. 
Indeed, if $\eta \in \tilde{I}_{\infty}'$ maps to $0$ in $\tilde{I}_{\infty}$ then $\eta \in \tilde{I}_{\infty}' \cap \tilde{I}_0'$, so in particular $\eta$ annihilates all the Eisenstein series of $M_2(\Gamma_1^{(p^r)}(N), \mathbf{C})$. There is an Eisenstein series $E \in M_2(\Gamma_1^{(p^r)}(N), \mathbf{C})$ such that the divisor of the meromorphic differential form $E(z)dz$ is $(0)-(\infty)$. This Eisenstein series induces (via integration) a morphism $H_1(Y_1^{(p^r)}(N), \mathbf{Z}) \rightarrow \mathbf{C}$. By intersection duality, we get an element $\mathcal{E} \in H_1(X_1^{(p^r)}(N), C_0^{(p^r)}\cup C_{\infty}^{(p^r)}, \mathbf{C})$. Since $E$ is annihilated by $ \tilde{I}_{\infty}' \cap \tilde{I}_0'$, so is $\mathcal{E}$. Since $\mathcal{E} - \{0,\infty\} \in H_1(X_1^{(p^r)}(N), \mathbf{C})$ and $\eta$ acts trivially on $H_1(X_1^{(p^r)}(N), \mathbf{C})$, we see that $\eta \cdot \{0,\infty\}=0$. 

Let $e: I \rightarrow H_+$ be the winding homomorphism of Mazur, denoted by $e_+$ in \cite[Definition, p. 137]{Mazur_Eisenstein}. We denote by a bold letter the various $\tilde{\mathbb{T}}^{(p^r)}$ or $\mathbb{T}$-modules involved completed at $\tilde{I}_{\infty}$ or at $I$. Thus, for example, $\left(\tilde{\textbf{H}}^{(p^r)}\right)_+$ (resp. $\textbf{H}_+$) is the $\tilde{I}_{\infty}$ (resp. $I$)-adic completion of $\left(\tilde{H}^{(p^r)}\right)_+$ (resp. $H_+$). Let $\tilde{\textbf{e}}: \tilde{\textbf{I}}_{\infty} \rightarrow \left(\tilde{\textbf{H}}_1\right)_+$ (resp. $\textbf{e} : \textbf{I} \rightarrow \textbf{H}_+ $) be the map obtained after completion at the ideal $\tilde{I}_{\infty}$ (resp. $I$). 

\begin{lem}\label{odd_modSymb_completion_projection}
The map $\left(\tilde{H}^{(p^r)}\right)_+ \rightarrow H_+$ defines by passing to completion a group isomorphism $\left(\tilde{\textbf{H}}^{(p^r)}\right)_+/J_r\cdot \left(\tilde{\textbf{H}}^{(p^r)}\right)_+ \xrightarrow{\sim} \textbf{H}_+$.
\end{lem}
\begin{proof}
By point (ii) of Lemma \ref{Formalism_flatness_Hecke}, we have $\left(\tilde{\textbf{H}}^{(p^r)}\right)_+ = \left(\tilde{H}^{(p^r)}\right)_+/\left(\bigcap_{n \geq 0} \tilde{I}_{\infty}^n\right) \cdot \left(\tilde{H}^{(p^r)}\right)_+$ and $\textbf{H}_+ = H_+/\left(\bigcap_{n\geq 0} I^n\right) \cdot H_+$. By Proposition \ref{odd_modSymb_refined_Hida}, it suffices to show that the image of $\bigcap_{n\geq 0} \tilde{I}_{\infty}^n$ in $\mathbb{T}$ is $\bigcap_{n \geq 0} I^n$. Since $\Lambda_r$ is a local ring, by Theorem \ref{odd_modSymb_Eisenstein_ideal_X1} (i) there is a unique maximal ideal of $\tilde{\mathbb{T}}^{(p^r)}$ containing $\tilde{I}_{\infty}$, which we denote by $\tilde{\mathfrak{m}}$.  By Theorem \ref{odd_modSymb_Eisenstein_ideal_X1} (iii), the group $\tilde{\mathbb{T}}^{(p^r)}/\tilde{I}_{\infty}$ is finite, so there exists an integer $n_0 \geq 1$ such that $\tilde{\mathfrak{m}}^{n_0} \subset \tilde{I}_{\infty}$. Thus, we have $\bigcap_{n \geq 0} \tilde{I}_{\infty}^n = \bigcap_{n \geq 0} \tilde{\mathfrak{m}}^n$. Similarly, there is a unique maximal ideal $\mathfrak{m}$ of $\mathbb{T}$ containing $I$, and we have $\bigcap_{n \geq 0} I^n = \bigcap_{n \geq 0} \mathfrak{m}^n$. 

Any maximal ideal $\tilde{\mathfrak{m}}'$ of $\tilde{\mathbb{T}}^{(p^r)}$ contains the image of $J_r$ by the map $\Lambda_r \rightarrow \tilde{\mathbb{T}}^{(p^r)}$. Indeed, we have an injective ring homomorphism $\Lambda_r/ \mathfrak{m}' \hookrightarrow \tilde{\mathbb{T}}^{(p^r)}/\tilde{\mathfrak{m}}'$ where $\mathfrak{m}'$ is the pre image of $\tilde{\mathfrak{m}}'$ in $\Lambda_r$. Since $\tilde{\mathbb{T}}^{(p^r)}/\tilde{\mathfrak{m}}'$ is a finite field, $\Lambda_r \cap \tilde{\mathfrak{m}}'$ is a maximal ideal of $\Lambda_r$ and therefore contains $J_r$. If $\mathfrak{m}'$ is a maximal ideal of $\mathbb{T}$, there exists a unique maximal ideal $\tilde{\mathfrak{m}}'$ of $\tilde{\mathbb{T}}^{(p^r)}$ projecting to $\mathfrak{m}'$. The existence is obvious and the unicity follows from the fact that if $\tilde{\mathfrak{m}}_1$ and $\tilde{\mathfrak{m}}_2$ are two such ideals, then $\tilde{\mathfrak{m}}_1 = \tilde{\mathfrak{m}}_1 + J_r = \tilde{\mathfrak{m}}_2 + J_r = \tilde{\mathfrak{m}}_2$ by the previous remark and Proposition \ref{odd_modSymb_refined_Hida_Hecke}. The rings $\tilde{\mathbb{T}}^{(p^r)}$ and $\mathbb{T}$ are semi-local and $p$-adically complete, so we have $\tilde{\mathbb{T}}^{(p^r)} = \bigoplus_{\tilde{\mathfrak{m}}' \in \text{SpecMax}(\tilde{\mathbb{T}}^{(p^r)})} (\tilde{\mathbb{T}}^{(p^r)})_{\tilde{\mathfrak{m}}'}$ and $\mathbb{T} = \bigoplus_{\mathfrak{m}' \in \text{SpecMax}(\mathbb{T}) } \mathbb{T}_{\mathfrak{m}'}$, where the subscript means the completion. By the previous discussion, the map $(\tilde{\mathbb{T}}^{(p^r)})_{\tilde{\mathfrak{m}}'} \rightarrow \mathbb{T}_{\mathfrak{m}'}$ is surjective if $\tilde{\mathfrak{m}}'$ projects to $\mathfrak{m}'$. We have $\bigcap_{n \geq 0} \tilde{\mathfrak{m}}^n =  \bigoplus_{\tilde{\mathfrak{m}}' \in \text{SpecMax}(\tilde{\mathbb{T}}^{(p^r)}) \atop \tilde{\mathfrak{m}}' \neq \tilde{\mathfrak{m}}} (\tilde{\mathbb{T}}^{(p^r)})_{\tilde{\mathfrak{m}}'}$ and  $\bigcap_{n \geq 0} \mathfrak{m}^n =  \bigoplus_{\mathfrak{m}' \in \text{SpecMax}(\mathbb{T}) \atop \mathfrak{m}' \neq \mathfrak{m}} \mathbb{T}_{\tilde{\mathfrak{m}}'}$. Thus, the map $\bigcap_{n \geq 0} \tilde{\mathfrak{m}}^n \rightarrow \bigcap_{n \geq 0} \mathfrak{m}^n$ is surjective.
\end{proof}

Thus, we get a commutative diagram:
$$\xymatrix{
 \tilde{\textbf{I}}_{\infty} \ar[r]^{\tilde{\textbf{e}}} \ar[d]  & \left(\tilde{\textbf{H}}^{(p^r)}\right)_+  \ar[d]  \\
     \textbf{I} \ar[r]^{\textbf{e}} & \textbf{H}_+ }$$

Since $\textbf{e}$ is surjective (it is even an isomorphism by \cite[Theorem $18.10$]{Mazur_Eisenstein}), Lemmas \ref{odd_modSymb_Nakayama} and \ref{odd_modSymb_completion_projection} show that $\tilde{\textbf{e}}$ is surjective. By the Eichler--Shimura isomorphism (over $\mathbf{C}$), the $\mathbf{Z}_p$-rank of these two modules must be equal to the $\mathbf{Z}_p$-rank of $\tilde{\mathbf{T}}_1$. Thus, $\tilde{\textbf{e}}$ is an isomorphism. 

By passing to the quotient map, $\tilde{\textbf{e}}$ gives rise to an isomorphism of $\Lambda_r$-modules $\tilde{I}_{\infty} / \tilde{I}_{\infty}^2\simeq \left(\tilde{H}^{(p^r)}\right)_+/\tilde{I}_{\infty}$. The $\Lambda_r$-module $\left(\tilde{H}^{(p^r)}\right)_+/\tilde{I}_{\infty}$, and so $\tilde{I}_{\infty} / \tilde{I}_{\infty}^2$, is cyclic since it is cyclic modulo $J_r$ by Proposition \ref{odd_modSymb_refined_Hida}. By Nakayama's Lemma the $\tilde{\textbf{I}}_{\infty}$ is principal. Since a generator of $\tilde{\textbf{I}}_{\infty}$ is not a zero-divisor, we get:
$$\tilde{\mathbb{T}}^{(p^r)}/\tilde{I}_{\infty} \simeq \tilde{I}_{\infty}/\tilde{I}_{\infty}^2 \text{ .}$$
This concludes the proof of Theorem \ref{odd_modSymb_structure_H_1} by Theorem \ref{odd_modSymb_Eisenstein_ideal_X1}.
\end{proof}

The following result will be useful later.
\begin{prop}\label{odd_modSymb_image_I_0}
We have $\tilde{I}_0 \cdot \left(\tilde{H}^{(p^t)}\right)_+ = \left(H^{(p^t)}\right)_+$.
\end{prop}
\begin{proof}
The inclusion $\tilde{I}_0 \cdot \left(\tilde{H}^{(p^t)}\right)_+ \subset \left(H^{(p^t)}\right)_+$ is obvious. We have $\left(\tilde{H}^{(p^t)}\right)_+/\left(H^{(p^t)}\right)_+ = \mathbf{Z}_p[C_0^{(p^t)}]^0$. The $\Lambda_t$-module $\mathbf{Z}_p[C_0^{(p^t)}]^0$ is isomorphic to $J_t$.
Let $M =  \left(\tilde{H}^{(p^t)}\right)_+/\tilde{I}_0 \cdot \left(\tilde{H}^{(p^t)}\right)_+$; this is a $\Lambda_t$-module. By Proposition \ref{odd_modSymb_refined_Hida}, we have $M/J_t\cdot M = H_+/I\cdot H_+$.
In particular, $M$ is a cyclic $\Lambda_t$-module, which we write $\Lambda_t/K$ where $K \subset \Lambda_t$ is an ideal.  

We have $K \subset \delta^{(p^t)} \cdot \Lambda_t$. Indeed, we have a surjection of $\Lambda_t$-modules $\Lambda_r/K \twoheadrightarrow J_t   \simeq \Lambda_t/\delta^{(p^t)} \cdot \Lambda_t$.
This induces a surjection
$\Lambda_t/K \twoheadrightarrow \Lambda_t/((\delta^{(p^t)})+J) \cdot \Lambda_t\simeq \mathbf{Z}/p^t\mathbf{Z}$.
Since $\Lambda_t$ is a local ring with maximal ideal $J_t + (p)$, the image of $1 \in \Lambda_t$ in $\Lambda_t/K$ is mapped to the image of a unit of $\Lambda_t$ in $\Lambda_t  /\delta^{(p^t)} \cdot \Lambda_t$. Thus we have $K \subset \delta^{(p^t)} \cdot \Lambda_t$. We now prove the reverse inclusion.

Since $M/J_t\cdot M \simeq \mathbf{Z}/p^t\mathbf{Z}$, there is an element $\beta = \delta^{(p^t)} \cdot \alpha$ in $K$ such that $\alpha  \in 1 + J_t + (p)$. In particular, $\alpha$ is a unit and $K = \delta^{(p^t)} \cdot \Lambda_r$. 
Thus, we have $M \simeq J_t $ and the surjective map $\left(\tilde{H}^{(p^t)}\right)_+/\tilde{I_0} \cdot \left(\tilde{H}^{(p^t)}\right)_+ \rightarrow \left(\tilde{H}^{(p^t)}\right)_+/\left(H^{(p^t)}\right)_+$ is an isomorphism of free $\mathbf{Z}_p$-modules.
\end{proof}

\begin{corr} \label{odd_modSymb_construction}
The natural map $\tilde{H} \rightarrow H$ gives an isomorphism
$$ \left(H^{(p^t)}\right)_+/(I_0 + J_t)\cdot \left(H^{(p^t)}\right)_+ \xrightarrow{\sim} I\cdot H_+/I^2\cdot H_+\text{ .}$$
\end{corr}
\begin{proof}
This follows from Propositions \ref{odd_modSymb_refined_Hida} and \ref{odd_modSymb_image_I_0}.
\end{proof}
Note that Lemma \ref{odd_modSymb_construction} gives another proof of the fact that the image of $\left(H^{(p^t)}\right)_+$ in $H_+$ is $I\cdot H_+$, which was first proved by Mazur \cite[Section II, Lemma 18.7]{Mazur_Eisenstein}.

\subsection{Algebraic number theoretic criterion for $n(r,p) \geq 2$}\label{odd_modSymb_section_class_group}

We keep the notation of section \ref{odd_modSymb_section_eisenstein_ideal}, and add the following ones. 
\begin{itemize}
\item Fix an algebraic closure $\overline{\mathbf{Q}}$ of $\mathbf{Q}$.
\item If $n$ is a positive integer, let $\mu_n$ be the group of $n$th roots of unity in $\overline{\mathbf{Q}}$.
\item Let $\zeta_{p^r}$, $\zeta_N$ $\in \overline{\mathbf{Q}}$ be primitive $p^r$th and $N$th roots of unity respectively.
\item Let $\chi_p : \Gal(\overline{\mathbf{Q}}/\mathbf{Q}) \rightarrow \mathbf{Z}_p^{\times}$ be the $p$th cyclotomic character and $\omega_p : \Gal(\mathbf{Q}(\zeta_p)/\mathbf{Q}) \xrightarrow{\sim} (\mathbf{Z}/p\mathbf{Z})^{\times}$ be the Teichm\"uller character.
\item If $M$ is a $\mathbf{Z}_p[\Gal(\mathbf{Q}(\zeta_{p^{\infty}})/\mathbf{Q})]$-module, and $i \in \mathbf{Z}$,  we denote by $M_{(i)}$ the maximal quotient of $M$ where $\Gal(\mathbf{Q}(\zeta_{p^{\infty}})/\mathbf{Q})$ acts by $\chi_p^i$. Let $$M^{(i)} = \{m \in M, \forall g \in \Gal(\mathbf{Q}(\zeta_{p^{\infty}})/\mathbf{Q}), \text{ }g(m) = \chi_p^i(g)\cdot m\}\text{ .}$$ Note that if the action of $\Gal(\mathbf{Q}(\zeta_{p^{\infty}})/\mathbf{Q})$ on $M$ factors through $\Gal(\mathbf{Q}(\zeta_{p^r})/\mathbf{Q})$ and if $i \neq 0$, then $M^{(i)}$ is necessarily a $\mathbf{Z}/p^r\mathbf{Z}$-module.
\item We normalize class field theory by sending geometric Frobenius substitutions to uniformizers.
\item For simplicity, we denote the augmentation ideal $J_r$ of $\Lambda_r$ by $J$.
\end{itemize}

Let $K_r$ be the unique degree $p^r$-extension of $\mathbf{Q}$ contained in $\mathbf{Q}(\zeta_{N})$, and $L_r = K_r(\zeta_{p^r})$. Let $\mathcal{O}_r = \mathcal{O}_{K_r}[\frac{1}{Np}]$ where $\mathcal{O}_{K_r}$ is the ring of integers of $K_r$ and let $\mathcal{A}_r=\mathcal{O}_r[\zeta_{p^r}]$. Let $S_r$ (resp. $T_r$) be the set of infinite places of $K_r$ (resp. $L_r$) and finite places of $K_r$ (resp. $L_r$) dividing $Np$. Let $K_{S}$ be the maximal extension of $K_r$ unramified outside $S_r$. Let $\Gamma_r = \Gal(K_{S}/K_r)$ and $\Gamma_r' = \Gal(K_{S}/L_r)$. The group $\Gamma_r'$ is a normal subgroup of $\Gamma_r$, and we have canonical group isomorphisms $$\Gamma_r/\Gamma_r' \simeq \Gal(L_r/K_r) \simeq \Gal(\mathbf{Q}(\zeta_{p^r})/\mathbf{Q}) \simeq (\mathbf{Z}/p^r\mathbf{Z})^{\times}$$ (the last one coming from the $p^r$th cyclotomic character). The $N$th cyclotomic character gives a group isomorphism $\Gal(K_r/\mathbf{Q}) \simeq (\mathbf{Z}/N\mathbf{Z})^{\times}/\left((\mathbf{Z}/N\mathbf{Z})^{\times} \right)^{p^r}$. Thus, we have a canonical ring isomorphism $$\Lambda_r \xrightarrow{\sim} \mathbf{Z}_p[\Gal(K_r/\mathbf{Q})] \text{ .}$$ We let $\mathcal{K}_r = K_2(\mathcal{O}_r)/p^r \cdot K_2(\mathcal{O}_r)$; it is equipped with a structure of $\Lambda_r$-module since $\Gal(K_r/\mathbf{Q})$ acts on $K_2(\mathcal{O}_r)$. The aim of this section is to understand (partially) the filtration of $\mathcal{K}_r$ given by the powers of $J$.

Recall that $\frac{1}{p} \in \mathcal{O}_r$. The Chern character induces an isomorphism of $\Lambda_r$-modules $$K_2(\mathcal{O}_r)/p^r \cdot K_2(\mathcal{O}_r)  \simeq H^2_{\text{\'et}}(\mathcal{O}_r, \mu_{p^r}^{\otimes 2}) $$\cite{Tate_cohomology}. We have a canonical isomorphism $$H^2_{\text{\'et}}(\mathcal{O}_r, \mu_{p^r}^{\otimes 2}) \simeq H^2(\Gamma_r, \mu_{p^r}^{\otimes 2})\text{ .}$$ Since the real places and the places dividing $p$ of $K_r$ are in $S_r$, the $p$-cohomological dimension of $\Gamma_r$ is $\leq 2$ \cite[Proposition 8.3.18]{Neukirch_CNF}. Thus, by \cite[Proposition 3.3.11]{Neukirch_CNF}, the corestriction gives an isomorphism
\begin{equation}\label{odd_modSymb_cores_eq}
H^2(\Gamma_r', \mu_{p^r}^{\otimes 2})_{\Gamma_r/\Gamma_r'} \xrightarrow{\sim} H^2(\Gamma_r, \mu_{p^r}^{\otimes 2})
\end{equation}
where $H^2(\Gamma_r', \mu_{p^r}^{\otimes 2})_{\Gamma_r/\Gamma_r'}$ is the group of coinvariants of $H^2(\Gamma_r', \mu_{p^r}^{\otimes 2})$ for the action of $\Gamma_r/\Gamma_r'$.

\begin{thm}\label{comparison_K_C}
\begin{enumerate}
\item The group $\mathcal{K}_r/J\cdot \mathcal{K}_r$ is cyclic of order $p^r$.
\item The group following assertions are equivalent.
\begin{enumerate}
\item The group $J\cdot \mathcal{K}_r/J^2\cdot \mathcal{K}_r$ is cyclic of order $p^r$.
\item We have $n(r,p) \geq 2$, \ie $\sum_{k=1}^{\frac{N-1}{2}} k \cdot \log(k) \equiv 0 \text{ (modulo }p^r\text{)}$.
\end{enumerate}

\end{enumerate}
\end{thm}
\begin{rems}
\begin{enumerate}
\item The group $J\cdot \mathcal{K}_r/J^2\cdot \mathcal{K}_r$ is cyclic in any case, since $\mathcal{K}_r/J\cdot \mathcal{K}_r$ is cyclic. 
\item We give another criterion equivalent to (a) and (b) in Lemma \ref{big_thm_log_unit} in terms of certain norm residue symbols.
\item We give in Proposition \ref{odd_modSymb_local_cup_product} an explicit isomorphism between $\mathcal{K}_r/J\cdot \mathcal{K}_r$ and $(\mathbf{Z}/N\mathbf{Z})^{\times}[p^r]$.
\item If $r=1$, one can show that Theorem \ref{comparison_K_C} is a consequence of \cite[Theorem 1.7]{Lecouturier_class_group}.
\end{enumerate}
\end{rems}
\begin{proof}
Kummer theory gives us an exact sequence of $\Gal(L_r/\mathbf{Q})$-modules:
\begin{equation}  
0 \rightarrow \Pic(\mathcal{A}_r) \otimes_{\mathbf{Z}} \mathbf{Z}/p^r\mathbf{Z} \rightarrow H^2(\Gamma_r', \mu_{p^r}) \rightarrow H^2(\Gamma_r', \mathcal{A}_r^{\times})[p^r]\rightarrow 0 \text{ .}
\end{equation}
By \cite[Proposition ($8.3.11$) (iii)]{Neukirch_CNF}, we have an isomorphism of $\Gal(L_r/\mathbf{Q})$-modules:
$$H^2(\Gamma_r', \mathcal{A}_r^{\times})[p^r] \simeq \Ker\left(\bigoplus_{\mathfrak{q} \in T_r}\mathbf{Z}/p^r\mathbf{Z} \rightarrow \mathbf{Z}/p^r\mathbf{Z}\right) \text{ .}$$
where $\Gal(L_r/\mathbf{Q})$ acts on the right-hand side by permuting the elements of $T_r$ via the natural left action of  $\Gal(L_r/\mathbf{Q})$ on $T_r$.
Thus, we get the following $\Gal(L_r/\mathbf{Q})$-equivariant exact sequence  (by tensoring by $\mu_{p^r}$):
\begin{equation}\label{odd_modSymb_exact_sequence_H_2_bis}
0 \rightarrow \Pic(\mathcal{A}_r) \otimes \mu_{p^r}  \rightarrow H^2(\Gamma_r', \mu_{p^r}) \otimes \mu_{p^r} \rightarrow \Ker\left(\bigoplus_{\mathfrak{q} \in T_r} \mu_{p^r} \rightarrow \mu_{p^r}\right) \rightarrow 0 \text{ .}
\end{equation}
Over $L_r$, we have $\mu_{p^r} \simeq \mathbf{Z}/p^r\mathbf{Z}$. Thus $H^2(\Gamma_r', \mu_{p^r}) \otimes \mu_{p^r}= H^2(\Gamma_r', \mu_{p^r}^{\otimes 2})$.
Taking $\Gal(L_r/K_r)$-coinvariants in (\ref{odd_modSymb_exact_sequence_H_2_bis}), we get (using (\ref{odd_modSymb_cores_eq})) an exact sequence of $\Gal(L_r/\mathbf{Q}(\zeta_{p^r}))$-modules:
\begin{equation}\label{odd_modSymb_short_exact_sqn_1}
H_1(\Gal(L_r/K_r), M) \rightarrow \left(\Pic(\mathcal{A}_r) \otimes_{\mathbf{Z}} \mathbf{Z}/p^r\mathbf{Z}\right)_{(-1)} \rightarrow  \mathcal{K}_r \rightarrow M_{(0)} \rightarrow 0
\end{equation}
where  $M:=\Ker\left(\bigoplus_{\mathfrak{q} \in T_r} \mu_{p^r} \rightarrow \mu_{p^r}\right)$.

\begin{lem}\label{odd_modSymb_coinv_computation}
We have $H_1(\Gal(L_r/K_r), M) =0$ and $ M_{(0)} \simeq \mathbf{Z}/p^r\mathbf{Z}$.
\end{lem}
\begin{proof}
Recall that $\Gal(L_r/K_r) \simeq \Gal(\mathbf{Q}(\zeta_{p^r})/\mathbf{Q}) \simeq  (\mathbf{Z}/p^r\mathbf{Z})^{\times}$ is a cyclic group. Thus we have $H_1(\Gal(L_r/K_r), M) \simeq M^{\Gal(L_r/K_r)}/\Norm(M)$ where $\Norm(M) = \{\sum_{g \in \Gal(L_r/K_r)} g \cdot m, m\in M\}$. Recall that the action of $\Gal(L_r/K_r)$ on $M$ permutes the primes in $T_r$ and acts on $\mu_{p^r}$ via the $p$th cyclotomic character. Fix a prime $\mathfrak{n} \in T_r$ dividing $N$. Any other such prime equals $g(\mathfrak{n})$ for some $g \in \Gal(L_r/K_r)$. We have
$$M^{\Gal(L_r/K_r)} =\left\{  \bigoplus_{g(\mathfrak{n}) \in T_r \atop g \in \Gal(L_r/K_r)}  \zeta^{\chi_p(g)}, \zeta \in \mu_{p^r}  \right\}  $$
where $\bigoplus_{g(\mathfrak{n}) \in T_r \atop g \in \Gal(L_r/K_r)}  \zeta^{\chi_p(g)}$ is the element of $M$ whose component at $g(\mathfrak{n})$ is $ \zeta^{\chi_p(g)}$. Furthermore, this is by definition an element of $\Norm(M)$. This shows that $M^{\Gal(L_r/K_r)}=\Norm(M)$, and therefore  $H_1(\Gal(L_r/K_r), M) =0$. The computation of  $M_{(0)} =H_0(\Gal(L_r/K_r), M)$ is similar.
\end{proof}

By Lemma \ref{odd_modSymb_coinv_computation} and  (\ref{odd_modSymb_short_exact_sqn_1}), we have a canonical short exact sequence of $\Lambda_r$-modules:
\begin{equation}\label{odd_modSymb_short_exact_sqn_2}
0 \rightarrow  \left(\Pic(\mathcal{A}_r) \otimes_{\mathbf{Z}} \mathbf{Z}/p^r\mathbf{Z}\right)_{(-1)} \rightarrow  \mathcal{K}_r \rightarrow \mathbf{Z}/p^r\mathbf{Z} \rightarrow 0 \text{ .}
\end{equation}

\begin{lem}\label{odd_modSymb_K_cyclic}
The group $\mathcal{K}_r/J\cdot \mathcal{K}_r$ is cyclic of order $p^r$. Furthermore, the image of the injective map $ \left(\Pic(\mathcal{A}_r) \otimes_{\mathbf{Z}} \mathbf{Z}/p^r\mathbf{Z}\right)_{(-1)} \rightarrow  \mathcal{K}_r $ is $J \cdot \mathcal{K}_r$.
\end{lem}
\begin{proof}
As before, by \cite[Proposition 3.3.11]{Neukirch_CNF}, the corestriction gives a group isomorphism
$$\mathcal{K}_r/J\cdot \mathcal{K}_r \xrightarrow{\sim} H^2_{\text{\'et}}(\mathbf{Z}[\frac{1}{Np}], \mu_{p^r}^{\otimes 2}) \text{ .}$$
The latter group is (as before by Tate) isomorphic to $K_2(\mathbf{Z}[\frac{1}{Np}])/p^r\cdot K_2(\mathbf{Z}[\frac{1}{Np}])$. Since $\mathbf{Z}[\frac{1}{Np}]$ is euclidean, by \cite[Th\'eor\`eme 1]{Kahn} and the well-known description of $K_2(\mathbf{Q})$, we get that $K_2(\mathbf{Z}[\frac{1}{Np}])/p^r\cdot K_2(\mathbf{Z}[\frac{1}{Np}])$ is cyclic of order $p^r$.

The second assertion follows from the short exact sequence (\ref{odd_modSymb_short_exact_sqn_2}) since $J$ acts trivially on the right-hand side term $\mathbf{Z}/p^r\mathbf{Z}$.
\end{proof}

Lemma \ref{odd_modSymb_K_cyclic} proves point (i) of Theorem \ref{comparison_K_C}. Fix a prime $\mathfrak{n}$ above $N$ in $\mathbf{Q}(\zeta_{p^r})$. There is a canonical $\Gal(L_r/\mathbf{Q})$-equivariant group isomorphism $\Pic(\mathcal{A}_r) \xrightarrow{\sim} \mathcal{C}_r/C$ where $\mathcal{C}_r$ is the class group of $L_r$ and $C$ is the subgroup of $\mathcal{C}_r$ generated by the classes of the primes in $T_r$ (\cf for instance  \cite[Proposition ($8.3.11$) (ii)]{Neukirch_CNF}). Thus, Lemma \ref{odd_modSymb_K_cyclic} gives a canonical group isomorphism
\begin{equation}\label{odd_modSymb_J/J^2_K_C}
J\cdot \mathcal{K}_r/J^2 \cdot \mathcal{K}_r \xrightarrow{\sim} \left( \mathcal{C}_r/(C+J\cdot \mathcal{C}_r + p^r \cdot \mathcal{C}_r)\right)_{(-1)} \text{ .}
\end{equation}
We have denoted (abusively) by $J$ the augmentation ideal of $\mathbf{Z}[\Gal(L_r/\mathbf{Q}(\zeta_{p^r}))]$, which acts on $\mathcal{C}_r$, although $J$ is properly speaking the augmentation ideal of $\Lambda_r \simeq \mathbf{Z}_p[\Gal(L_r/\mathbf{Q}(\zeta_{p^r}))]$ (the point of this abuse is that $\mathcal{C}_r$ is not a $p$-group a priori, so we have to work with $\mathbf{Z}$-coefficients).

Let ${}_{N}\mathcal{C}_r$ be the kernel of the norm from $\mathcal{C}_r$ to the class group of $\mathbf{Q}(\zeta_{p^r})$. The group ${}_{N}\mathcal{C}_r/J \cdot \mathcal{C}_r$ is well understood in this case, thanks to genus theory. We now explain the result, following \cite{Gold_genus_theory}.

Let $\Gal(L_r/\mathbf{Q}(\zeta_{p^r}))^{(\Gal(\mathbf{Q}(\zeta_{p^r})/\mathbf{Q}))}$ be the product of copies of $\Gal(L_r/\mathbf{Q}(\zeta_{p^r}))$ indexed by the elements $\Gal(\mathbf{Q}(\zeta_{p^r})/\mathbf{Q})$ (which we identify with $\Gal(L_r/K_r)$ as usual). It is equipped with an action of $\Gal(\mathbf{Q}(\zeta_{p^r})/\mathbf{Q})$ given as follows. If $\tau_0 \in \Gal(\mathbf{Q}(\zeta_{p^r})/\mathbf{Q})$ and $(g_{\tau})_{\tau \in \Gal(\mathbf{Q}(\zeta_{p^r})/\mathbf{Q})} \in \Gal(L_r/\mathbf{Q}(\zeta_{p^r}))^{(\Gal(\mathbf{Q}(\zeta_{p^r})/\mathbf{Q}))}$, then we let
$$\tau_0 \cdot (g_{\tau})_{\tau \in \Gal(\mathbf{Q}(\zeta_{p^r})/\mathbf{Q})} = (g_{\tau \cdot \tau_0^{-1}})_{\tau \in \Gal(\mathbf{Q}(\zeta_{p^r})/\mathbf{Q})} \text{ .}$$

Let $f : \mathbf{Q}(\zeta_{p^r})^{\times} \rightarrow \Gal(L_r/K)^{(\Gal(\mathbf{Q}(\zeta_{p^r})/\mathbf{Q}))}$ be given by
$$f(x) = \left( \legendre{x, L_r/\mathbf{Q}(\zeta_{p^r})}{\tau(\mathfrak{n})}  \right)_{\tau \in \Gal(\mathbf{Q}(\zeta_{p^r})/\mathbf{Q})}$$
where $\legendre{\cdot , L_r/\mathbf{Q}(\zeta_{p^r})}{\tau(\mathfrak{n})}$ is the norm residue symbol of the extension $L_r/\mathbf{Q}(\zeta_{p^r})$ at the prime $\tau(\mathfrak{n})$. The group homomorphism $f$ is $\Gal(\mathbf{Q}(\zeta_{p^r})/\mathbf{Q})$-equivariant. This follows from the equality, for all $x \in \mathbf{Q}(\zeta_{p^r})^{\times}$ and $\tau \in \Gal(\mathbf{Q}(\zeta_{p^r})/\mathbf{Q})$, in $\Gal(L_r/\mathbf{Q}(\zeta_{p^r}))$:
\begin{equation}\label{odd_modSymb_galois_action_residue_symbol}
\legendre{x, L_r/\mathbf{Q}(\zeta_{p^r})}{\tau(\mathfrak{n})} = \legendre{\tau^{-1}(x), L_r/\mathbf{Q}(\zeta_{p^r})}{\mathfrak{n}} \text{ .}
\end{equation}

 We let $U = f(\mathbf{Z}[\zeta_{p^r}]^{\times})$. By local class field theory, the kernel of the restriction of $f$ to $U$ is the set of elements of $U$ which are everywhere locally norms of element of $L_r$. Since $L_r/\mathbf{Q}(\zeta_{p^r})$ is cyclic, the Hasse norm theorem shows that this kernel is $\mathbf{Z}[\zeta_{p^r}]^{\times} \cap \Norm_{L_r/\mathbf{Q}(\zeta_{p^r})}(\mathbf{Z}[\zeta_{p^r}]^{\times}) $. Thus, we have a canonical isomorphism 
\begin{equation}
U \xrightarrow{\sim} \mathbf{Z}[\zeta_{p^r}]^{\times}/\left( \mathbf{Z}[\zeta_{p^r}]^{\times} \cap \Norm_{L_r/\mathbf{Q}(\zeta_{p^r})}(\mathbf{Z}[\zeta_{p^r}]^{\times}) \right) \text{ .}
\end{equation}

Let $\mathfrak{a}$ be an ideal class in ${}_{N}\mathcal{C}_r$ and $I$ be a fractional ideal of $L_r$ whose class is in $\mathfrak{a}$. There is some $\alpha \in \mathbf{Q}(\zeta_{p^r})^{\times}$ such that $\Norm_{L_r/\mathbf{Q}(\zeta_{p^r})}(I) = (\alpha)$.
One easily checks that the map $\mathfrak{a} \mapsto f(\alpha)$ modulo $U$ is well-defined, \ie is independent of the choice of $I$ and $\alpha$. We have thus constructed a canonical $\Gal(\mathbf{Q}(\zeta_{p^r})/\mathbf{Q})$-equivariant group homomorphism 
$$\hat{f} : {}_{N}\mathcal{C}_r \rightarrow \Gal(L_r/\mathbf{Q}(\zeta_{p^r}))^{(\Gal(\mathbf{Q}(\zeta_{p^r})/\mathbf{Q}))}/U \text{ .}$$
The following result follows from \cite[Proposition 3.1]{Lecouturier_class_group} applied with $E = \mathbf{Q}(\zeta_{p^r})$ and $F=L_r$.
\begin{lem}
\begin{enumerate}
\item The kernel of $\hat{f}$ is $J \cdot \mathcal{C}_r$.
\item The image of $\hat{f}$ is $\Gal(L_r/\mathbf{Q}(\zeta_{p^r}))^{(\Gal(\mathbf{Q}(\zeta_{p^r})/\mathbf{Q})),0}/U$, where $$\Gal(L_r/\mathbf{Q}(\zeta_{p^r}))^{(\Gal(\mathbf{Q}(\zeta_{p^r})/\mathbf{Q})),0} = \{(g_{\tau})_{\tau \in \Gal(\mathbf{Q}(\zeta_{p^r})/\mathbf{Q})} \in\Gal(L_r/\mathbf{Q}(\zeta_{p^r}))^{(\Gal(\mathbf{Q}(\zeta_{p^r})/\mathbf{Q}))}, \text{ } \prod_{\tau \in \Gal(\mathbf{Q}(\zeta_{p^r})/\mathbf{Q})} g_{\tau} = 1 \} \text{ .}$$
\end{enumerate}
\end{lem}
We get a group isomorphism 
\begin{equation}\label{odd_modSymb_hat_f_iso}
\hat{f}' :  \left( ({}_{N}\mathcal{C}_r/J\cdot \mathcal{C}_r) \otimes_{\mathbf{Z}} \mathbf{Z}_p \right)_{(-1)}  \xrightarrow{\sim} \left(\Gal(L_r/\mathbf{Q}(\zeta_{p^r}))^{(\Gal(\mathbf{Q}(\zeta_{p^r})/\mathbf{Q})),0}/U \right)_{(-1)} \text{ .}
\end{equation}
There is a group isomorphism 
$$\left( \Gal(L_r/\mathbf{Q}(\zeta_{p^r}))^{(\Gal(\mathbf{Q}(\zeta_{p^r})/\mathbf{Q})),0} \right)_{(-1)} \xrightarrow{\sim} \Gal(L_r/\mathbf{Q}(\zeta_{p^r}))$$
given by 
$$(g_{\tau})_{\tau \in \Gal(\mathbf{Q}(\zeta_{p^r})/\mathbf{Q})} \mapsto \prod_{\tau \in \Gal(\mathbf{Q}(\zeta_{p^r})/\mathbf{Q})} g_{\tau}^{\chi_p^{-1}(\tau)} \text{ .}$$
Note that $\Gal(L_r/\mathbf{Q}(\zeta_{p^r})) \simeq \mathbf{Z}/p^r\mathbf{Z}$ is a $p$-group, so $g_{\tau}^{\chi_p^{-1}(\tau)}$ makes sense. By \cite[Lemma 2.7]{Lecouturier_class_group}, we have 
\begin{equation}\label{odd_modSymb_units_killed}
\left(\mathbf{Z}[\zeta_{p^r}]^{\times} \otimes_{\mathbf{Z}} \mathbf{Z}_p\right)_{(-1)}=0
\end{equation}
 (we are using the fact that $\chi_p^{-1} \not\equiv \chi_p \text{ (modulo }p^r\text{)}$ since $p>3$). Thus, the image of $U$ in $\left( \Gal(L_r/\mathbf{Q}(\zeta_{p^r}))^{(\Gal(\mathbf{Q}(\zeta_{p^r})/\mathbf{Q})),0} \right)_{(-1)}$ is trivial. Consequently, we have
\begin{equation}\label{odd_modSymb_right_hand_side_Hilbert}
\left( \Gal(L_r/\mathbf{Q}(\zeta_{p^r}))^{(\Gal(\mathbf{Q}(\zeta_{p^r})/\mathbf{Q})),0}/U \right)_{(-1)}\simeq \mathbf{Z}/p^r\mathbf{Z} \text{ .}
\end{equation}

Since $p$ does not divide the numerator of $B_2 = \frac{1}{6}$, the Herbrand--Ribet theorem shows that $(\Pic(\mathbf{Z}[\zeta_p]) \otimes_{\mathbf{Z}} \mathbf{Z}_p)_{(-1)}$ is trivial. By Nakayama's lemma, it follows that 
\begin{equation}\label{odd_modSymb_Picard_trivial}
(\Pic(\mathbf{Z}[\zeta_{p^r}])  \otimes_{\mathbf{Z}} \mathbf{Z}_p)_{(-1)}=0
\end{equation}
Thus, we have $(\mathcal{C}_r \otimes_{\mathbf{Z}} \mathbf{Z}_p)_{(-1)} =({}_{N}\mathcal{C}_r \otimes_{\mathbf{Z}} \mathbf{Z}_p)_{(-1)}$ and 
\begin{equation}\label{odd_modSymb_iso_norm_non_norm}
\left( ({}_{N}\mathcal{C}_r/J\cdot \mathcal{C}_r) \otimes_{\mathbf{Z}} \mathbf{Z}_p \right)_{(-1)} = \left( (\mathcal{C}_r/J\cdot \mathcal{C}_r) \otimes_{\mathbf{Z}} \mathbf{Z}_p\right)_{(-1)} \text{ .}
\end{equation}
Since the primes above $p$ in $L_r$ are fixed by $\Gal(L_r/K_r)$, we have 
\begin{equation}\label{odd_modSymb_norm_class_group}
\left(\mathcal{C}_r/(C+J\cdot \mathcal{C}_r + p^r\cdot \mathcal{C}_r )\right)_{(-1)} = \left(\mathcal{C}_r/(C'+J\cdot \mathcal{C}_r + p^r\cdot \mathcal{C}_r )\right)_{(-1)}
\end{equation}
 where $C' \subset \mathcal{C}_r$ is generated by the classes of primes above $N$ in $\mathcal{C}_r$. 

The multiplication by a generator of $J$ gives a surjective group homomorphism $\mathcal{K}_r/J\cdot \mathcal{K}_r \twoheadrightarrow J\cdot \mathcal{K}_r/J^2\cdot \mathcal{K}_r$. Thus, $J\cdot \mathcal{K}_r/J^2\cdot \mathcal{K}_r$ is a cyclic group of order dividing $p^r$. By (\ref{odd_modSymb_J/J^2_K_C}), (\ref{odd_modSymb_hat_f_iso}), (\ref{odd_modSymb_right_hand_side_Hilbert}), (\ref{odd_modSymb_iso_norm_non_norm}) and (\ref{odd_modSymb_norm_class_group}), the group $J\cdot \mathcal{K}_r/J^2\cdot \mathcal{K}_r$ is cyclic of order $p^r$ if and only if $\hat{f}'((C' \otimes_{\mathbf{Z}} \mathbf{Z}_p)_{(-1)})=0$. Let $\mathcal{I}$ be the group of fractional ideals of $\mathbf{Q}(\zeta_{p^r})$. By (\ref{odd_modSymb_units_killed}) and (\ref{odd_modSymb_Picard_trivial}), we have 
\begin{equation}\label{odd_modSymb_ideal_principal}
(\mathcal{I} \otimes_{\mathbf{Z}} \mathbf{Z}_p)_{(-1)} = (\mathbf{Q}(\zeta_{p^r})^{\times} \otimes_{\mathbf{Z}} \mathbf{Z}_p)_{(-1)} \text{ .}
\end{equation}
Recall that we have fixed a prime ideal $\mathfrak{n}$ above $N$ in $\mathcal{I}$. By (\ref{odd_modSymb_ideal_principal}), the image of $\mathfrak{n}$ in $(\mathcal{I} \otimes_{\mathbf{Z}} \mathbf{Z}_p)_{(-1)}$ is some $x \in (\mathbf{Q}(\zeta_{p^r})^{\times} \otimes_{\mathbf{Z}} \mathbf{Z}_p)_{(-1)}$. We let $\tilde{u}$ be any lift of $x$ in $\mathbf{Q}(\zeta_{p^r})^{\times} \otimes_{\mathbf{Z}} \mathbf{Z}_p$ and $u$ be the image of $\tilde{u}$ in $\mathbf{Q}(\zeta_{p^r})^{\times} \otimes_{\mathbf{Z}} \mathbf{Z}/p^r\mathbf{Z}$. By the discussion above, $J\cdot \mathcal{K}_r/J^2\cdot \mathcal{K}_r$ is a cyclic group of order $p^r$ if and only if we have, in $\Gal(L_r/\mathbf{Q}(\zeta_{p^r}))$:
\begin{equation*}
\prod_{\tau \in \Gal(\mathbf{Q}(\zeta_{p^r})/\mathbf{Q})} \legendre{u, L_r/\mathbf{Q}(\zeta_{p^r})}{\tau(\mathfrak{n})}^{\chi_p^{-1}(\tau)} = 1 \text{ .}
\end{equation*}
(this does not depend on the choice of $\tilde{u}$).
By (\ref{odd_modSymb_galois_action_residue_symbol}), this is equivalent to
\begin{equation*}
\legendre{\prod_{\tau \in \Gal(\mathbf{Q}(\zeta_{p^r})/\mathbf{Q})}  \tau^{-1}(u)^{\chi_p^{-1}(\tau)}, L_r/\mathbf{Q}(\zeta_{p^r})}{\mathfrak{n}} = 1 \text{ .}
\end{equation*}
Note that we have written the $\mathbf{Z}_p$-module structure of $\mathbf{Q}(\zeta_{p^r})^{\times} \otimes_{\mathbf{Z}} \mathbf{Z}/p^r\mathbf{Z}$ multiplicatively.
Let $$u_{\chi_p^{-1}} = \prod_{\tau \in \Gal(\mathbf{Q}(\zeta_{p^r})/\mathbf{Q})}  \tau(u)^{\chi_p(\tau)} \in (\mathbf{Q}(\zeta_{p^r})^{\times} \otimes_{\mathbf{Z}} \mathbf{Z}/p^r\mathbf{Z})^{(-1)} \text{ .}$$ (again, this does not depend on the choice of $\tilde{u}$). 
To conclude the proof of Theorem \ref{comparison_K_C}, it suffices to prove the following result.

\begin{lem}\label{big_thm_log_unit}
The following assertions are equivalent.
\begin{enumerate}
\item We have $\legendre{u_{\chi_p^{-1}}, L_r/\mathbf{Q}(\zeta_{p^r})}{\mathfrak{n}} = 1$.
\item We have $\sum_{k=1}^{\frac{N-1}{2}} k \cdot \log(k) \equiv 0 \text{ (modulo }p^r\text{)}$.
\end{enumerate}
\end{lem}
\begin{proof}
We first relate $u_{\chi_p^{-1}}$ to a Gauss sum. Let $\omega_N : (\mathbf{Z}/N\mathbf{Z})^{\times} \rightarrow \mathbf{Z}_N^{\times}$ be the Teichm\"uller character, characterized by $\omega_N(a) \equiv a \text{ (modulo }N\text{)}$ for all $a \in (\mathbf{Z}/N\mathbf{Z})^{\times}$. Let $$\mathcal{G}_r= \sum_{a \in (\mathbf{Z}/N\mathbf{Z})^{\times}} \omega_N(a)^{-\frac{N-1}{p^r}} \cdot \zeta_N^a \in \mathbf{Z}[\zeta_{p^r}, \zeta_N]\text{, }$$ where we view $\omega_N^{\frac{N-1}{p^r}}$ as taking values in the $p^r$th root of unity of $\mathbf{Z}[\zeta_{p^r}]$ using the choice of $\mathfrak{n}$. Galois theory shows that $\mathcal{G}_r^{p^r} \in \mathbf{Q}(\zeta_{p^r})$. Thus, we have $\mathcal{G}_r \in L_r$. Let $$\mathcal{G}_{r,\chi_p^{-1}} = \prod_{\tau \in  \Gal(\mathbf{Q}(\zeta_{p^r})/\mathbf{Q})} \tau(\mathcal{G}_r)^{\chi_p(\tau)} \in (L_r^{\times} \otimes_{\mathbf{Z}} \mathbf{Z}/p^r\mathbf{Z})^{(-1)} \text{ .}$$

\begin{lem}\label{proof_K_C_lemma_invariant_L_r}
The inclusion $\mathbf{Q}(\zeta_{p^r}) \hookrightarrow L_r$ gives a group isomorphism
$$(\mathbf{Q}(\zeta_{p^r})^{\times} \otimes_{\mathbf{Z}} \mathbf{Z}/p^r\mathbf{Z})^{(-1)} \xrightarrow{\sim} \left((L_r^{\times} \otimes_{\mathbf{Z}} \mathbf{Z}/p^r\mathbf{Z})^{(-1)}\right)^{\Gal(L_r/\mathbf{Q}(\zeta_{p^r}))}$$
where the $\Gal(L_r/\mathbf{Q}(\zeta_{p^r}))$ in the exponent means the invariants by $\Gal(L_r/\mathbf{Q}(\zeta_{p^r}))$.
\end{lem}
\begin{proof}
The long exact sequence of cohomology attached to the short exact sequence of $\Gal(L_r/\mathbf{Q}(\zeta_{p^r}))$-modules $1 \rightarrow \mu_{p^r} \rightarrow L_r^{\times} \xrightarrow{p^r} (L_r^{\times})^{p^r} \rightarrow 1$ gives us, using Hilbert 90:
\begin{equation}\label{proof_K_C_exact_sequence_galois_1}
1 \rightarrow \mu_{p^r} \rightarrow \mathbf{Q}(\zeta_{p^r})^{\times} \xrightarrow{p^r} \mathbf{Q}(\zeta_{p^r})^{\times} \cap (L_r^{\times})^{p^r} \rightarrow H^1(\Gal(L_r/\mathbf{Q}(\zeta_{p^r})), \mu_{p^r}) \rightarrow 1
\end{equation}
and
\begin{equation}\label{proof_K_C_exact_sequence_galois_2}
1 \rightarrow H^1(\Gal(L_r/\mathbf{Q}(\zeta_{p^r})), (L_r^{\times})^{p^r}) \rightarrow H^2(\Gal(L_r/\mathbf{Q}(\zeta_{p^r})), \mu_{p^r}) \text{ .}
\end{equation}

The long exact sequence of cohomology attached to the short exact sequence of $\Gal(L_r/\mathbf{Q}(\zeta_{p^r}))$-modules $1 \rightarrow(L_r^{\times})^{p^r}   \rightarrow L_r^{\times} \rightarrow L_r^{\times} \otimes_{\mathbf{Z}} \mathbf{Z}/p^r\mathbf{Z} \rightarrow 1$ gives us, using Hilbert 90:
\begin{equation}\label{proof_K_C_exact_sequence_galois_3}
1 \rightarrow \mathbf{Q}(\zeta_{p^r})^{\times} \cap (L_r^{\times})^{p^r}  \rightarrow \mathbf{Q}(\zeta_{p^r})^{\times} \rightarrow (L_r^{\times} \otimes_{\mathbf{Z}} \mathbf{Z}/p^r\mathbf{Z})^{\Gal(L_r/\mathbf{Q}(\zeta_{p^r}))} \rightarrow H^1(\Gal(L_r/\mathbf{Q}(\zeta_{p^r})), (L_r^{\times})^{p^r}) \rightarrow 1 \text{ .}
\end{equation}
Since $p>3$, we have $\chi_p^{-1} \not\equiv \chi_p \text{ (modulo }p\text{)}$, so we get:
\begin{equation}\label{proof_K_C_vanishing_cohomology}
H^1(\Gal(L_r/\mathbf{Q}(\zeta_{p^r})), \mu_{p^r})^{(-1)}=H^2(\Gal(L_r/\mathbf{Q}(\zeta_{p^r})), \mu_{p^r})^{(-1)}=1 \text{ .}
\end{equation}
Combining (\ref{proof_K_C_exact_sequence_galois_1}) and (\ref{proof_K_C_vanishing_cohomology}), we get
\begin{equation}\label{proof_K_C_vanishing_cohomology_2}
\left(\mathbf{Q}(\zeta_{p^r})^{\times} \cap (L_r^{\times})^{p^r}/(\mathbf{Q}(\zeta_{p^r})^{\times})^{p^r} \right)^{(-1)}=1 \text{ .}
\end{equation}
Combining (\ref{proof_K_C_exact_sequence_galois_2}) and (\ref{proof_K_C_vanishing_cohomology}), we get
\begin{equation}\label{proof_K_C_vanishing_cohomology_3}
H^1(\Gal(L_r/\mathbf{Q}(\zeta_{p^r})), (L_r^{\times})^{p^r})^{(-1)}=1 \text{ .}
\end{equation}
Combining (\ref{proof_K_C_exact_sequence_galois_3}) and (\ref{proof_K_C_vanishing_cohomology_3}), we get
\begin{equation}\label{proof_K_C_exact_sequence_galois_4}
\left((L_r^{\times} \otimes_{\mathbf{Z}} \mathbf{Z}/p^r\mathbf{Z})^{(-1)}\right)^{\Gal(L_r/\mathbf{Q}(\zeta_{p^r}))} =  \left(\mathbf{Q}(\zeta_{p^r})^{\times}/\mathbf{Q}(\zeta_{p^r})^{\times} \cap (L_r^{\times})^{p^r} \right)^{(-1)}\text{ .}
\end{equation}
We have an exact sequence
\begin{equation}\label{proof_K_C_exact_sequence_galois_5}
1 \rightarrow \mathbf{Q}(\zeta_{p^r})^{\times} \cap (L_r^{\times})^{p^r}/(\mathbf{Q}(\zeta_{p^r})^{\times})^{p^r} \rightarrow \mathbf{Q}(\zeta_{p^r})^{\times}/(\mathbf{Q}(\zeta_{p^r})^{\times})^{p^r} \rightarrow \mathbf{Q}(\zeta_{p^r})^{\times}/\mathbf{Q}(\zeta_{p^r})^{\times} \cap  (L_r^{\times})^{p^r} \rightarrow 1 \text{ .}
\end{equation}
\begin{lem}\label{proof_K_C_vanishing_cohomology_4}
We have $$H^1\left(\Gal(\mathbf{Q}(\zeta_{p^r})/\mathbf{Q}),  \left(\mathbf{Q}(\zeta_{p^r})^{\times} \cap (L_r^{\times})^{p^r}/(\mathbf{Q}(\zeta_{p^r})^{\times})^{p^r} \right) \otimes_{\mathbf{Z}/p^r\mathbf{Z}} \mu_{p^r} \right) = 1 \text{ .}$$
\end{lem}
\begin{proof}
Since $\Gal(\mathbf{Q}(\zeta_{p^r})/\mathbf{Q})$ is cyclic, the Herbrand quotient of $ \left(\mathbf{Q}(\zeta_{p^r})^{\times} \cap (L_r^{\times})^{p^r}/(\mathbf{Q}(\zeta_{p^r})^{\times})^{p^r} \right) \otimes_{\mathbf{Z}/p^r\mathbf{Z}} \mu_{p^r} $ is well-defined, and is trivial since this last group is finite by Kummer theory. By (\ref{proof_K_C_vanishing_cohomology_2}), we have 
$$H^0\left(\Gal(\mathbf{Q}(\zeta_{p^r})/\mathbf{Q}),  \left(\mathbf{Q}(\zeta_{p^r})^{\times} \cap (L_r^{\times})^{p^r}/(\mathbf{Q}(\zeta_{p^r})^{\times})^{p^r} \right) \otimes_{\mathbf{Z}/p^r\mathbf{Z}} \mu_{p^r} \right) = 1 \text{ .}$$ 
Using the previous remark on the Herbrand quotient, this proves that $$H^1\left(\Gal(\mathbf{Q}(\zeta_{p^r})/\mathbf{Q}),  \left(\mathbf{Q}(\zeta_{p^r})^{\times} \cap (L_r^{\times})^{p^r}/(\mathbf{Q}(\zeta_{p^r})^{\times})^{p^r} \right) \otimes_{\mathbf{Z}/p^r\mathbf{Z}} \mu_{p^r} \right) = 1 \text{ .}$$
\end{proof}

Combining Lemma \ref{proof_K_C_vanishing_cohomology_4}, (\ref{proof_K_C_vanishing_cohomology_2}) and (\ref{proof_K_C_exact_sequence_galois_5}), we get a group isomorphism:
\begin{equation}\label{proof_K_C_exact_sequence_galois_6}
\left(\mathbf{Q}(\zeta_{p^r})^{\times}/(\mathbf{Q}(\zeta_{p^r})^{\times})^{p^r} \right)^{(-1)} \xrightarrow{\sim} \left(  \mathbf{Q}(\zeta_{p^r})^{\times}/\mathbf{Q}(\zeta_{p^r})^{\times} \cap  (L_r^{\times})^{p^r} \right)^{(-1)}  \text{ .}
\end{equation}
To conclude the proof of Lemma \ref{proof_K_C_lemma_invariant_L_r}, it suffices to combine (\ref{proof_K_C_exact_sequence_galois_3}), (\ref{proof_K_C_vanishing_cohomology_3}) and (\ref{proof_K_C_exact_sequence_galois_6}).

\end{proof}

We have $\mathcal{G}_{r,\chi_p^{-1}} \in \left((L_r^{\times} \otimes_{\mathbf{Z}} \mathbf{Z}/p^r\mathbf{Z})^{(-1)}\right)^{\Gal(L_r/\mathbf{Q}(\zeta_{p^r}))}$. By Lemma \ref{proof_K_C_lemma_invariant_L_r}, there is a unique $g_{r,\chi_p^{-1}} \in (\mathbf{Q}(\zeta_{p^r})^{\times} \otimes_{\mathbf{Z}} \mathbf{Z}/p^r\mathbf{Z})^{(-1)}$ such that $\mathcal{G}_{r,\chi_p^{-1}}$ is the image of $g_{r,\chi_p^{-1}}$ in $(L_r^{\times} \otimes_{\mathbf{Z}} \mathbf{Z}/p^r\mathbf{Z})^{(-1)}$. 

\begin{lem}\label{proof_K_C_equality_gauss_unit}
There exists $\alpha \in (\mathbf{Z}/p^r\mathbf{Z})^{\times}$ such that $g_{r,\chi_p^{-1}}=u_{\chi_p^{-1}}^{\alpha}$.
\end{lem}
\begin{proof}
Let $\mathcal{I}_{\mathbf{Q}(\zeta_{p^r})}$ (resp. $\mathcal{C}_{\mathbf{Q}(\zeta_{p^r})}$) be the group of fractional ideals (resp. the class group) of $\mathbf{Q}(\zeta_{p^r})$. Similarly, let $\mathcal{I}_{L_r}$ be the group of fractional ideals of $L_r$. We have an exact sequence 
$$1 \rightarrow \mathbf{Z}[\zeta_{p^r}]^{\times} \rightarrow \mathbf{Q}(\zeta_{p^r})^{\times} \rightarrow \mathcal{I}_{\mathbf{Q}(\zeta_{p^r})} \rightarrow \mathcal{C}_{\mathbf{Q}(\zeta_{p^r})} \rightarrow 1\text{ .}$$
The snake lemma gives an exact sequence 
\begin{equation}\label{proof_K_C_exact_sequence_fractional_ideals}
1 \rightarrow  \mathcal{C}_{\mathbf{Q}(\zeta_{p^r})}[p^r] \rightarrow (\mathbf{Q}(\zeta_{p^r})^{\times}/\mathbf{Z}[\zeta_{p^r}]^{\times}) \otimes_{\mathbf{Z}} \mathbf{Z}/p^r\mathbf{Z} \rightarrow \mathcal{I}_{\mathbf{Q}(\zeta_{p^r})} \otimes_{\mathbf{Z}} \mathbf{Z}/p^r\mathbf{Z} \text{ .}
\end{equation}
\begin{lem}\label{proof_K_C_triviality_class_group}
We have $(\mathcal{C}_{\mathbf{Q}(\zeta_{p^r})}[p^r])^{(-1)}=1$.
\end{lem}
\begin{proof}
It suffices to show that $(\mathcal{C}_{\mathbf{Q}(\zeta_{p^r})}[p^r])^{(-1)}[p]=1$, \ie that $(\mathcal{C}_{\mathbf{Q}(\zeta_{p^r})}[p])^{(-1)}=1$. If $M$ is a $\Gal(\mathbf{Q}(\zeta_p)/\mathbf{Q})$-module and $i \in \mathbf{Z}$, let $M^{[i]}$ (resp. $M_{[i]}$) be the subgroup (resp. the maximal quotient) of $M$ on which $\Gal(\mathbf{Q}(\zeta_p)/\mathbf{Q})$ acts by $\omega_p^i$. It suffices to show that $(\mathcal{C}_{\mathbf{Q}(\zeta_{p^r})}[p])^{[-1]}=1$. It suffices to show that $(\mathcal{C}_{\mathbf{Q}(\zeta_{p^r})})^{[-1]}=1$, or equivalently that $(\mathcal{C}_{\mathbf{Q}(\zeta_{p^r})})_{[-1]}=1$. By Nakayama's lemma, it suffices to show that $(\mathcal{C}_{\mathbf{Q}(\zeta_{p})})_{[-1]}=1$, which follows from the Herbrand--Ribet theorem since $p$ does not divide $B_2 = \frac{1}{6}$.
\end{proof}
By (\ref{proof_K_C_exact_sequence_fractional_ideals}), we get an embedding
\begin{equation}\label{proof_K_C_embedding_fractional_ideals}
\left((\mathbf{Q}(\zeta_{p^r})^{\times}/\mathbf{Z}[\zeta_{p^r}]^{\times})\otimes_{\mathbf{Z}} \mathbf{Z}/p^r\mathbf{Z} \right)^{(-1)} \hookrightarrow \left(\mathcal{I}_{\mathbf{Q}(\zeta_{p^r})} \otimes_{\mathbf{Z}} \mathbf{Z}/p^r\mathbf{Z}\right)^{(-1)} \text{ .}
\end{equation}

\begin{lem}\label{proof_K_C_triviality_units}
The map $$\left(\mathbf{Q}(\zeta_{p^r})^{\times} \otimes_{\mathbf{Z}} \mathbf{Z}/p^r\mathbf{Z} \right)^{(-1)} \rightarrow \left((\mathbf{Q}(\zeta_{p^r})^{\times}/\mathbf{Z}[\zeta_{p^r}]^{\times})\otimes_{\mathbf{Z}} \mathbf{Z}/p^r\mathbf{Z} \right)^{(-1)}$$
is an isomorphism.
\end{lem}
\begin{proof}
We have an exact sequence
\begin{equation}
1 \rightarrow \mathbf{Z}[\zeta_{p^r}]^{\times}/\mathbf{Z}[\zeta_{p^r}]^{\times} \cap (\mathbf{Q}(\zeta_{p^r})^{\times})^{p^r} \rightarrow  \mathbf{Q}(\zeta_{p^r})^{\times} \otimes_{\mathbf{Z}} \mathbf{Z}/p^r\mathbf{Z} \rightarrow (\mathbf{Q}(\zeta_{p^r})^{\times}/\mathbf{Z}[\zeta_{p^r}]^{\times})\otimes_{\mathbf{Z}} \mathbf{Z}/p^r\mathbf{Z}  \rightarrow 1 \text{ .}
\end{equation}
To prove Lemma \ref{proof_K_C_triviality_units}, it suffices to prove that
\begin{equation}\label{proof_K_C_triviality_cohomology_units}
H^0\left(\Gal(\mathbf{Q}(\zeta_{p^r})/\mathbf{Q}), \left(\mathbf{Z}[\zeta_{p^r}]^{\times}/\mathbf{Z}[\zeta_{p^r}]^{\times} \cap (\mathbf{Q}(\zeta_{p^r})^{\times})^{p^r} \right) \otimes_{\mathbf{Z}/p^r\mathbf{Z}} \mu_{p^r}\right)=1
\end{equation}
and
$$H^1\left(\Gal(\mathbf{Q}(\zeta_{p^r})/\mathbf{Q}), \left(\mathbf{Z}[\zeta_{p^r}]^{\times}/\mathbf{Z}[\zeta_{p^r}]^{\times} \cap (\mathbf{Q}(\zeta_{p^r})^{\times})^{p^r}\right) \otimes_{\mathbf{Z}/p^r\mathbf{Z}} \mu_{p^r}\right)=1 \text{ .}$$
Since $\Gal(\mathbf{Q}(\zeta_{p^r})/\mathbf{Q})$ is cyclic and $ \left(\mathbf{Z}[\zeta_{p^r}]^{\times}/\mathbf{Z}[\zeta_{p^r}]^{\times} \cap (\mathbf{Q}(\zeta_{p^r})^{\times})^{p^r}\right) \otimes_{\mathbf{Z}/p^r\mathbf{Z}} \mu_{p^r}$ is a finite group, a Herbrand quotient argument like in Lemma \ref{proof_K_C_vanishing_cohomology_4} shows that it suffices to prove (\ref{proof_K_C_triviality_cohomology_units}). We have an exact sequence
$$1 \rightarrow \mathbf{Z}[\zeta_{p^r}]^{\times} \cap (\mathbf{Q}(\zeta_{p^r})^{\times})^{p^r}/(\mathbf{Z}[\zeta_{p^r}]^{\times})^{p^r} \rightarrow  \mathbf{Z}[\zeta_{p^r}]^{\times}/( \mathbf{Z}[\zeta_{p^r}]^{\times})^{p^r} \rightarrow  \mathbf{Z}[\zeta_{p^r}]^{\times}/ \mathbf{Z}[\zeta_{p^r}]^{\times} \cap (\mathbf{Q}(\zeta_{p^r})^{\times})^{p^r} \rightarrow 1 \text{ .}$$
To prove (\ref{proof_K_C_triviality_cohomology_units}), it thus suffices to prove that
\begin{equation}\label{usual_vanishing_units_cyclotomic}
(\mathbf{Z}[\zeta_{p^r}]^{\times} \otimes_{\mathbf{Z}} \mathbf{Z}/p^r\mathbf{Z})^{(-1)}=1
\end{equation}
and 
\begin{equation}\label{proof_K_C_triviality_cohomology_units_2}
H^1\left(\Gal(\mathbf{Q}(\zeta_{p^r})/\mathbf{Q}), \left(\mathbf{Z}[\zeta_{p^r}]^{\times} \cap (\mathbf{Q}(\zeta_{p^r})^{\times})^{p^r}/(\mathbf{Z}[\zeta_{p^r}]^{\times})^{p^r} \right) \otimes_{\mathbf{Z}/p^r\mathbf{Z}} \mu_{p^r} \right)=1 \text{ .}
\end{equation}
The equality (\ref{usual_vanishing_units_cyclotomic}) follows from \cite[Lemma 2.7]{Lecouturier_class_group}. As before, a Herbrand quotient argument shows that to prove (\ref{proof_K_C_triviality_cohomology_units_2}), it suffices to prove
$$ H^0\left(\Gal(\mathbf{Q}(\zeta_{p^r})/\mathbf{Q}), \left(\mathbf{Z}[\zeta_{p^r}]^{\times} \cap (\mathbf{Q}(\zeta_{p^r})^{\times})^{p^r}/(\mathbf{Z}[\zeta_{p^r}]^{\times})^{p^r} \right) \otimes_{\mathbf{Z}/p^r\mathbf{Z}} \mu_{p^r}\right)= 1 \text{ ,}$$
which follows from (\ref{usual_vanishing_units_cyclotomic}).
\end{proof}

By (\ref{proof_K_C_embedding_fractional_ideals}) and Lemma \ref{proof_K_C_triviality_units}, we get an embedding
\begin{equation}\label{injection_field_fractional_ideals}
\left(\mathbf{Q}(\zeta_{p^r})^{\times} \otimes_{\mathbf{Z}} \mathbf{Z}/p^r\mathbf{Z}\right)^{(-1)} \hookrightarrow \left(\mathcal{I}_{\mathbf{Q}(\zeta_{p^r})} \otimes_{\mathbf{Z}} \mathbf{Z}/p^r\mathbf{Z}\right)^{(-1)} \text{ .}
\end{equation}

By (\ref{injection_field_fractional_ideals}), in order to prove Lemma \ref{proof_K_C_equality_gauss_unit} it suffices to prove that there exists  $\alpha \in (\mathbf{Z}/p^r\mathbf{Z})^{\times}$  such that $u_{\chi_p^{-1}}$ and $g_{r,\chi_p^{-1}}^{\alpha}$ have the same image in $(\mathcal{I}_{\mathbf{Q}(\zeta_{p^r})} \otimes_{\mathbf{Z}} \mathbf{Z}/p^r\mathbf{Z})^{(-1)}$. Let $\mathcal{I}_N \subset \mathcal{I}_{\mathbf{Q}(\zeta_{p^r})}$ be the subgroup generated the ideals above $N$ in $\mathbf{Q}(\zeta_{p^r})$. This is a direct summand of the free $\mathbf{Z}$-module $\mathcal{I}_{\mathbf{Q}(\zeta_{p^r})}$. Thus, we have an embedding $(\mathcal{I}_N \otimes_{\mathbf{Z}} \mathbf{Z}/p^r\mathbf{Z})^{(-1)} \hookrightarrow  (\mathcal{I}_{\mathbf{Q}(\zeta_{p^r})} \otimes_{\mathbf{Z}} \mathbf{Z}/p^r\mathbf{Z})^{(-1)}$. Furthermore, the group $(\mathcal{I}_N \otimes_{\mathbf{Z}} \mathbf{Z}/p^r\mathbf{Z})^{(-1)}$ is cyclic of order $p^r$. 

By definition, the image of $u_{\chi_p^{-1}}$ in  $(\mathcal{I}_{\mathbf{Q}(\zeta_{p^r})} \otimes_{\mathbf{Z}} \mathbf{Z}/p^r\mathbf{Z})^{(-1)}$ is a generator of $(\mathcal{I}_N \otimes_{\mathbf{Z}} \mathbf{Z}/p^r\mathbf{Z})^{(-1)}$. To conclude the proof of Proposition \ref{proof_K_C_equality_gauss_unit}, it suffices to prove that the image of $g_{r,\chi_p^{-1}}$ in $(\mathcal{I}_{\mathbf{Q}(\zeta_{p^r})} \otimes_{\mathbf{Z}} \mathbf{Z}/p^r\mathbf{Z})^{(-1)}$ is a generator of $(\mathcal{I}_N \otimes_{\mathbf{Z}} \mathbf{Z}/p^r\mathbf{Z})^{(-1)}$. We first prove that this image lies in $(\mathcal{I}_N \otimes_{\mathbf{Z}} \mathbf{Z}/p^r\mathbf{Z})^{(-1)}$. Let $\mathfrak{N}$ be the unique prime ideal of $L_r$ above the fixed prime ideal $\mathfrak{n}$ of $\mathbf{Q}(\zeta_{p^r})$ dividing $N$. If $\tau \in \Gal(\mathbf{Q}(\zeta_{p^r})/\mathbf{Q})$, let $\Rep_r(\tau)$ be the unique integer in $\{1, ..., p^r-1\}$ such that $\tau(\zeta_{p^r}) = \zeta_{p^r}^{\Rep_r(\tau)}$. We have the following prime ideal decomposition in $L_r$ (\cf for instance \cite[Section 4]{Garett}):
\begin{equation}\label{prime_factorization_gauss_sum}
(\mathcal{G}_r) = \prod_{\tau \in \Gal(\mathbf{Q}(\zeta_{p^r})/\mathbf{Q})} \tau(\mathfrak{N})^{\Rep_r(\tau^{-1})} \text{ .}
\end{equation}
This allows us to compute the image of $\mathcal{G}_{r,\chi_p^{-1}}$ in $\mathcal{I}_{L_r}\otimes_{\mathbf{Z}} \mathbf{Z}/p^r\mathbf{Z}$, which is
\begin{equation}\label{proof_K_C_zero_image_gauss}
\left(\prod_{\tau \in  \Gal(\mathbf{Q}(\zeta_{p^r})/\mathbf{Q})} \tau(\mathfrak{N})^{\Rep_r(\tau^{-1})}\right)^{\sum_{k=1}^{p^r-1} k^2} = 1
\end{equation}
since $\sum_{k=1}^{p^r-1} k^2\equiv 0 \text{ (modulo }p^r\text{)}$. The kernel of the map $\mathcal{I}_{\mathbf{Q}(\zeta_{p^r})} \otimes_{\mathbf{Z}}\mathbf{Z}/p^r\mathbf{Z} \rightarrow \mathcal{I}_{L_r}  \otimes_{\mathbf{Z}}\mathbf{Z}/p^r\mathbf{Z}$ is $\mathcal{I}_N  \otimes_{\mathbf{Z}}\mathbf{Z}/p^r\mathbf{Z}$. By (\ref{proof_K_C_zero_image_gauss}), the image of $g_{r,\chi_p^{-1}}$ in  $(\mathcal{I}_{\mathbf{Q}(\zeta_{p^r})} \otimes_{\mathbf{Z}} \mathbf{Z}/p^r\mathbf{Z})^{(-1)}$ is in $(\mathcal{I}_N \otimes_{\mathbf{Z}} \mathbf{Z}/p^r\mathbf{Z})^{(-1)}$. Thus, there exists $\alpha \in \mathbf{Z}/p^r\mathbf{Z}$ such that $g_{r,\chi_p^{-1}}=u_{\chi_p^{-1}}^{\alpha}$. In order to prove that $\alpha \in (\mathbf{Z}/p^r\mathbf{Z})^{\times}$, it suffies to prove the following result.
\begin{lem}
The element $g_{r,\chi_p^{-1}}$ is not a $p$th power in $(\mathbf{Q}(\zeta_{p^r})^{\times}\otimes_{\mathbf{Z}} \mathbf{Z}/p^r\mathbf{Z})^{(-1)}$.
\end{lem}
\begin{proof}
For the sake of a contradiction, assume that $g_{r,\chi_p^{-1}}$ is a $p$th power in $(\mathbf{Q}(\zeta_{p^r})^{\times}\otimes_{\mathbf{Z}} \mathbf{Z}/p^r\mathbf{Z})^{(-1)}$. In particular, the image of $\mathcal{G}_{r, \chi_p^{-1}}$ in $L_r^{\times} \otimes_{\mathbf{Z}} \mathbf{Z}/p\mathbf{Z}$ is trivial. Thus, we have  
\begin{equation}\label{proof_K_C_norm_trivial_Gauss_1}
\prod_{\tau \in \Gal(\mathbf{Q}(\zeta_p)/\mathbf{Q})} \tau\left(\prod_{\tau' \in \Gal(\mathbf{Q}(\zeta_{p^r})/\mathbf{Q}(\zeta_p))} \tau'(\mathcal{G}_r)\right)^{\omega_p(\tau)} = 1 \text{ in } (L_r^{\times} \otimes_{\mathbf{Z}} \mathbf{Z}/p\mathbf{Z})^{(-1)} \text{ .}
\end{equation}
Note that $\prod_{\tau' \in \Gal(\mathbf{Q}(\zeta_{p^r})/\mathbf{Q}(\zeta_p))} \tau'(\mathcal{G}_r) = \Norm_{L_r/K_r(\zeta_p)}(\mathcal{G}_r)$.
\begin{lem}\label{proof_K_C_puissance_p_intersection}
\begin{enumerate}
\item The map $(K_r(\zeta_p)^{\times} \otimes_{\mathbf{Z}} \mathbf{Z}/p\mathbf{Z})^{(-1)} \rightarrow (L_r^{\times} \otimes_{\mathbf{Z}} \mathbf{Z}/p\mathbf{Z})^{(-1)}$ is injective.
\item The map $(L_1^{\times} \otimes_{\mathbf{Z}} \mathbf{Z}/p\mathbf{Z})^{(-1)} \rightarrow (K_r(\zeta_p)^{\times} \otimes_{\mathbf{Z}} \mathbf{Z}/p\mathbf{Z})^{(-1)}$ is injective.
\end{enumerate}
\end{lem}
\begin{proof}
This follows by considering the cohomology of the exact sequences $1 \rightarrow \mu_p \rightarrow L_r^{\times} \rightarrow  (L_r^{\times})^p \rightarrow 1$ and $1 \rightarrow \mu_p \rightarrow K_r(\zeta_p)^{\times} \rightarrow (K_r(\zeta_p)^{\times})^p \rightarrow 1$ as in the proof of Lemma \ref{proof_K_C_lemma_invariant_L_r}.
\end{proof}
By (\ref{proof_K_C_norm_trivial_Gauss_1}), we get that
\begin{equation}\label{proof_K_C_norm_trivial_Gauss_2}
\prod_{\tau \in \Gal(\mathbf{Q}(\zeta_p)/\mathbf{Q})} \tau\left( \Norm_{L_r/K_r(\zeta_p)}(\mathcal{G}_r) \right)^{\omega_p(\tau)} \text{ is a }p \text{th power in }  (K_r(\zeta_p)^{\times} \otimes_{\mathbf{Z}} \mathbf{Z}_p)^{(\omega_p^{-1})} \text{ .}
\end{equation}
 Here, if $M$ is a $\mathbf{Z}_p[\Gal(\mathbf{Q}(\zeta_p)/\mathbf{Q})]$, module, we let $M^{(\omega_p^{-1})}$ be the submodule of $M$ on which $\Gal(\mathbf{Q}(\zeta_p)/\mathbf{Q})$ acts by $\omega_p^{-1}$.
\begin{lem}\label{proof_K_C_equality_norm_gauss_sum}
We have, in $(K_r(\zeta_p)^{\times} \otimes_{\mathbf{Z}} \mathbf{Z}_p)^{(\omega_p^{-1})}$:
$$\prod_{\tau \in \Gal(\mathbf{Q}(\zeta_p)/\mathbf{Q})} \tau(\Norm_{L_r/K_r(\zeta_p)}(\mathcal{G}_r))^{\omega_p(\tau)} = \prod_{\tau \in \Gal(\mathbf{Q}(\zeta_p)/\mathbf{Q})} \tau(\mathcal{G}_1)^{\omega_p(\tau)} \text{ .}$$
\end{lem}
\begin{proof}
Using (\ref{prime_factorization_gauss_sum}), one checks that both sides have the same image in $(\mathcal{I}_{K_r(\zeta_p)} \otimes_{\mathbf{Z}} \mathbf{Z}_p)^{(\omega_p^{-1})}$. One then uses the fact that the map $(K_r(\zeta_p)^{\times} \otimes_{\mathbf{Z}} \mathbf{Z}_p)^{(\omega_p^{-1})} \rightarrow (\mathcal{I}_{K_r(\zeta_p)} \otimes_{\mathbf{Z}} \mathbf{Z}_p)^{(\omega_p^{-1})}$ is injective by \cite[Lemma 2.7]{Lecouturier_class_group} since $\omega_p^{-1} \neq \omega_p$ and $\omega_p^{-1}$ is odd.
\end{proof}
By Lemma \ref{proof_K_C_puissance_p_intersection} (ii), (\ref{proof_K_C_norm_trivial_Gauss_2}) and Lemma \ref{proof_K_C_equality_norm_gauss_sum}, we get that $\prod_{\tau \in \Gal(\mathbf{Q}(\zeta_p)/\mathbf{Q})} \tau(\mathcal{G}_1)^{\omega_p(\tau)}$ is a $p$th power in $L_1^{\times} \otimes_{\mathbf{Z}}\mathbf{Z}_p$. This contradicts \cite[Proposition 3.4]{Lecouturier_class_group}.
\end{proof}
This concludes the proof of Lemma \ref{proof_K_C_equality_gauss_unit}.
\end{proof}
By Lemma \ref{proof_K_C_equality_gauss_unit}, in order to conclude the proof of Lemma \ref{big_thm_log_unit}, it suffices to prove the following result.
\begin{lem}\label{big_thm_log_unit_2}
The following assertions are equivalent.
\begin{enumerate}
\item We have $\legendre{g_{r,\chi_p^{-1}}, L_r/\mathbf{Q}(\zeta_{p^r})}{\mathfrak{n}} = 1$.
\item We have $\sum_{k=1}^{\frac{N-1}{2}} k \cdot \log(k) \equiv 0 \text{ (modulo }p^r\text{)}$.
\end{enumerate}
\end{lem}
\begin{proof}
Let $\mathcal{L} : \mathbf{Q}_N(\zeta_N)^{\times} \otimes_{\mathbf{Z}} \mathbf{Z}/p^r\mathbf{Z} \rightarrow \mathbf{Z}/p^r\mathbf{Z}$ be the group homomorphism defined by
$$\mathcal{L}(a \otimes b) = \log\left(\overline{\frac{a}{(1-\zeta_N)^{v_N(a)}}}\right)\cdot b$$
where $a \in \mathbf{Q}_N(\zeta_N)^{\times}$, $v_N(a)$ is the $N$-adic valuation of $a$ (normalized by $v_N(1-\zeta_N)=1$), $\overline{\frac{a}{(1-\zeta_N)^{v_N(a)}}}$ is the reduction of $\frac{a}{(1-\zeta_N)^{v_N(a)}}$ modulo $(1-\zeta_N)$ and $b \in \mathbf{Z}/p^r\mathbf{Z}$.
We have $$N = \prod_{i=1}^{N-1}(\zeta_N^i-1) \text{ ,}$$ so 
\begin{equation}\label{vanishing_N_residue_symbol}
\mathcal{L}(N \otimes \overline{1}) = \mathcal{L}\left(\frac{N}{(\zeta_N-1)^{N-1}} \otimes \overline{1}\right) \equiv \log((N-1)!) \equiv 0 \text{ (modulo }p^r \text{).}
\end{equation}
Let $\mathfrak{N}'$ be the prime above $\mathfrak{n}$ in $\mathbf{Q}(\zeta_{p^r}, \zeta_N)$. We let $\mathcal{L}' :  \mathbf{Q}(\zeta_N, \zeta_{p^r})^{\times} \otimes_{\mathbf{Z}} \mathbf{Z}/p^r\mathbf{Z} \rightarrow \mathbf{Z}/p^r\mathbf{Z}$ be the composition of $\mathcal{L}$ with the group homomorphism $\mathbf{Q}(\zeta_N, \zeta_{p^r}) \otimes_{\mathbf{Z}} \mathbf{Z}/p^r\mathbf{Z} \rightarrow \mathbf{Q}_N(\zeta_N)^{\times} \otimes_{\mathbf{Z}} \mathbf{Z}/p^r\mathbf{Z}$ induced by the $\mathfrak{N}'$-adic completion. Let $\mathcal{L}'' :  \mathbf{Q}(\zeta_{p^r})^{\times} \otimes_{\mathbf{Z}} \mathbf{Z}/p^r\mathbf{Z} \rightarrow \mathbf{Z}/p^r\mathbf{Z}$ be the composition of $\mathcal{L}'$ with the group homomorphism $\mathbf{Q}(\zeta_{p^r})^{\times} \otimes_{\mathbf{Z}} \mathbf{Z}/p^r\mathbf{Z} \rightarrow  \mathbf{Q}(\zeta_N, \zeta_{p^r})^{\times} \otimes_{\mathbf{Z}} \mathbf{Z}/p^r\mathbf{Z}$.

One easily sees that the $\mathfrak{N}$-adic completion of $L_r$ is the extension of $\mathbf{Q}_N$ obtained by adjoining a $p^r$th root of $N$. The classical properties of the norm residue symbol (recalled for instance in \cite[Section 3.1]{Lecouturier_class_group}) show that $\legendre{g_{r,\chi_p^{-1}}, L_r/\mathbf{Q}(\zeta_{p^r})}{\mathfrak{n}} = 1$ if and only if $\mathcal{L}''\left(g_{r,\chi_p^{-1}} - N \otimes v \right)=0$ where $v \in \mathbf{Z}/p^r\mathbf{Z}$ is the $\mathfrak{n}$-adic valuation of $g_{r,\chi_p^{-1}}$. By (\ref{vanishing_N_residue_symbol}), this is equivalent to $\mathcal{L}''(g_{r, \chi_p^{-1}})=0$. By definition of $g_{r, \chi_p^{-1}}$, this is equivalent to $\mathcal{L}'(\mathcal{G}_{r, \chi_p^{-1}})=0$ where $\mathcal{G}_{r, \chi_p^{-1}}$ is viewed abusively as an element of $\mathbf{Q}(\zeta_N, \zeta_{p^r})^{\times} \otimes_{\mathbf{Z}} \mathbf{Z}/p^r\mathbf{Z}$.

It is well-known (\cf for instance  \cite[Section 2]{Garett}) that we have, for all $\tau \in \Gal(\mathbf{Q}(\zeta_{p^r})/\mathbf{Q})$:
\begin{equation*}
\frac{\tau(\mathcal{G}_r)}{(\zeta_N-1)^{\Rep_r(\tau)\cdot \frac{N-1}{p^r}}} \equiv -\frac{1}{(\Rep_r(\tau)\cdot \frac{N-1}{p^r})!} \text{ (modulo }\mathfrak{N}'\text{).}
\end{equation*}
where we recall that $\Rep_r(\tau)$ is the unique integer in $\{0, 1, ..., p^r-1\}$ such that $\tau(\zeta_{p^r})=\zeta_{p^r}^{\Rep_r(\tau)}$.
Thus, we have in $\mathbf{Z}/p^r\mathbf{Z}$:
\begin{equation}\label{kummer_congruence_gauss_sum}
\mathcal{L}'(\mathcal{G}_{r, \chi_p^{-1}}) = \sum_{a=1 \atop \gcd(a, p)=1}^{p^r-1} a\cdot \log\left(-\frac{1}{(a\cdot \frac{N-1}{p^r})!}\right) \text{ .}
\end{equation}
In order to conclude the proof of Lemma \ref{big_thm_log_unit_2}, it suffices to prove the following identity.

\begin{lem}\label{proof_K_C_identity_Gamma}
We have:
$$\sum_{a=1 \atop \gcd(a, p)=1}^{p^r-1} a\cdot \log((a\cdot \frac{N-1}{p^r})!)  \equiv -\frac{2\cdot (p-1)}{3}\cdot \sum_{k=1}^{\frac{N-1}{2}} k \cdot \log(k) \text{ (modulo }p^r\text{)} \text{ .}$$
\end{lem}
\begin{proof}
We prove this by induction on $r$. Let $s$ be an integer such that $1 \leq s \leq r$. We follow the computations of the forthcoming work of Karl Schaefer and Eric Stubley, but we do them modulo $p^s$ and not modulo $p$ as they do. We have, in $\mathbf{Z}/p^s\mathbf{Z}$:
\begin{align*}
\sum_{k=1}^{N-1} k^2\cdot \log(k) &= \sum_{a=1}^{p^s-1}\sum_{b=0}^{\frac{N-1}{p^s}-1} (a+b\cdot p^s)^2\cdot \log(a+b\cdot p^s) \\& = \sum_{a=1}^{p^s-1}\sum_{b=0}^{\frac{N-1}{p^s}-1} a^2\cdot \log(a+b\cdot p^s) \\&=  \sum_{a=1}^{p^s-1}a^2\cdot \sum_{b=0}^{\frac{N-1}{p^s}-1} \log(\frac{a}{p^s}+b)
\end{align*}
where in the last equality we used the fact that $\sum_{a=1}^{p^s-1} a^2 \equiv 0 \text{ (modulo }p^s\text{)}$.

Let $\Gamma_N: \mathbf{Z}_N \rightarrow \mathbf{Z}_N^{\times}$ be the Morita $N$-adic Gamma function. This is the unique continuous function $\mathbf{Z}_N \rightarrow \mathbf{Z}_N^{\times}$ satisfying $\Gamma_N(n) = (-1)^n\cdot \prod_{1 \leq i \leq n-1, \text{ pgcd(}n,N \text{)}=1} i$ if $n>1$ is an integer. 
We will use the following properties of $\Gamma_N$: 
\begin{enumerate}
\item $\Gamma_N(0)=1$.
\item If $x$, $y$ $\in \mathbf{Z}_N$ are such that $x \equiv y \text{ (modulo }N\text{)}$, then $\Gamma_N(x) \equiv \Gamma_N(y) \text{ (modulo }N\text{).}$
\item If $x \in \mathbf{Z}_N$ is such that for all integer $k$ in $\{0, ..., \frac{N-1}{p^s}-1\}$ we have $x + k \not\equiv 0 \text{ (modulo }N\text{)}$, then we have
$$\prod_{k=0}^{\frac{N-1}{p^s}-1} \Gamma_N(x+k) = (-1)^{\frac{N-1}{p^s}} \cdot \frac{\Gamma_N(x+\frac{N-1}{p^s})}{\Gamma_N(x)} \text{ .}$$
\end{enumerate}
Thus, we have:
\begin{align*}
\sum_{k=1}^{N-1} k^2\cdot \log(k) &= \sum_{a=1}^{p^s-1}a^2\cdot \log\left(\frac{\Gamma_N(\frac{a}{p^s}+\frac{N-1}{p^s})}{\Gamma_N(\frac{a}{p^s})}\right) \\& =  \sum_{a=1}^{p^s-1}a^2\cdot \log\left(\frac{\Gamma_N(\frac{a-1}{p^s})}{\Gamma_N(\frac{a}{p^s})}\right) \\&= \sum_{a=1}^{p^s-1} \left((a+1)^2-a^2\right)\cdot \log(\Gamma_N(\frac{a}{p^s})) \\&=  \sum_{a=1}^{p^s-1} \left((a+1)^2-a^2\right)\cdot \log(\Gamma_N(-\frac{N-1}{p^s}\cdot a)) \\&= \sum_{a=1}^{p^s-1} \left((p^s-a+1)^2-(p^s-a)^2\right)\cdot \log(\Gamma_N(1+a\cdot \frac{N-1}{p^s})) \\&= \sum_{a=1}^{p^s-1} \left((p^s-a+1)^2-(p^s-a)^2\right)\cdot \log((a\cdot \frac{N-1}{p^s})!)  \text{ .}
\end{align*}
Similarly, we have in $\mathbf{Z}/p^s\mathbf{Z}$:
\begin{align*}
\sum_{k=1}^{N-1}k\cdot \log(k) &=  \sum_{a=1}^{p^s-1}\sum_{b=0}^{\frac{N-1}{p^s}-1} (a+b\cdot p^s)\cdot \log(a+b\cdot p^s) \\& = \sum_{a=1}^{p^s-1} a\cdot \sum_{b=0}^{\frac{N-1}{p^s}-1} \log(\frac{a}{p^s}+b) \\& =\sum_{a=1}^{p^s-1} a\cdot \log\left(\frac{\Gamma_N(\frac{a}{p^s}+\frac{N-1}{p^s})}{\Gamma_N(\frac{a}{p^s})}\right) \\& =  \sum_{a=1}^{p^s-1} a\cdot \log\left(\frac{\Gamma_N(\frac{a-1}{p^s})}{\Gamma_N(\frac{a}{p^s})}\right) \\& = \sum_{a=1}^{p^s-1} (a+1-a)\cdot \log(\Gamma_N(\frac{a}{p^s}))\end{align*}
\end{proof}
Since $\sum_{k=1}^{N-1} k \cdot \log(k) = 0$, we get $\sum_{a=1}^{p^s-1} \log(\Gamma_N(\frac{a}{p^s})) \equiv  0 \text{ (modulo }p^s\text{)}$. Thus, we have:
$$\sum_{k=1}^{N-1} k^2  \cdot \log(k) \equiv -2 \cdot \sum_{a=1}^{p^s-1} a \cdot \log((a\cdot \frac{N-1}{p^s})!) \text{ .}$$
By Lemma \ref{even_modSymb_computation_square} (which is independent of everything else in this paper), we have:
\begin{equation}\label{proof_K_C_identiy_Gamma_lemma}
 \sum_{a=1}^{p^s-1} a \cdot \log((a\cdot \frac{N-1}{p^s})!) \equiv \frac{2}{3} \cdot \sum_{k=1}^{\frac{N-1}{2}} k \cdot \log(k)\text{ (modulo }p^s\text{).}
\end{equation}
This proves Lemma \ref{proof_K_C_identity_Gamma} for $r=1$. Assume that Lemma \ref{proof_K_C_identity_Gamma} is true for all $1 \leq s < r$. We have:
$$ \sum_{a=1}^{p^r-1} a \cdot \log((a\cdot \frac{N-1}{p^r})!) \equiv \sum_{s=0}^{r-1} \sum_{a=1 \atop \gcd(a,p)=1}^{p^{r-s}-1} p^s\cdot a \cdot \log(\frac{N-1}{p^r}\cdot p^s\cdot a)  \text{ (modulo }p^r\text{).}$$
By the induction hypothesis and (\ref{proof_K_C_identiy_Gamma_lemma}), we have:
$$\sum_{a=1 \atop \gcd(a,p)=1}^{p^r-1} a \cdot \log((a\cdot \frac{N-1}{p^r})!) \equiv \frac{2}{3}\cdot (1-\sum_{s=1}^{r-1} p^s\cdot (1-p)) $$
Since $\sum_{s=1}^{r-1} p^s \equiv -\frac{p}{p-1} \text{ (modulo }p^r\text{)}$, this concludes the induction, and thus the proof of Lemma \ref{proof_K_C_identity_Gamma}.
\end{proof}
This concludes the proof of Lemma \ref{big_thm_log_unit}.
\end{proof}
This concludes the proof of Theorem \ref{comparison_K_C}.
 \end{proof}

By point Theorem \ref{comparison_K_C} (i), we have a group isomorphism $\mathcal{K}_r/J\cdot \mathcal{K}_r \simeq \mathbf{Z}/p^r\mathbf{Z}$. Such an isomorphism follows canonically from the choice of $\log$ we have made throughout the article. If $x,y \in \mathbf{Z}[\zeta_N, \frac{1}{Np}]^{\times}$, we let $(x, y)_r$ be the image of the Steinberg symbol $\{x,y\} \in K_2(\mathbf{Z}[\zeta_N, \frac{1}{Np}])$ in $\mathcal{K}_r$ via the norm map $K_2(\mathbf{Z}[\zeta_N, \frac{1}{Np}]) \rightarrow K_2(\mathcal{O}_r)$. 

\begin{prop}\label{odd_modSymb_local_cup_product}
There is a unique group isomorphism
$$\iota_r :\mathcal{K}_r/J\cdot \mathcal{K}_r \xrightarrow{\sim} \mathbf{Z}/p^r\mathbf{Z}$$
such that for all $u$ and $v$ in $(\mathbf{Z}/N\mathbf{Z})^{\times}$, we have
$$\iota_r \left( (1-\zeta_N^u, 1-\zeta_N^v)_r\right) \equiv \log\left(\frac{u}{v}\right) \in \mathbf{Z}/p^r\mathbf{Z} \text{ .}$$
\end{prop}
\begin{proof}
By Matsumoto's Theorem \cite[Theorem 4.3.15]{Rosenberg}, we have a map $K_2(\mathbf{Q}(\zeta_N)) \rightarrow \mathbf{Z}/p^r\mathbf{Z}$ given by the Hilbert symbol
$$\{x,y\} \mapsto  \log\left( \overline{\frac{y^{v(x)}}{x^{v(y)}}} \right) \text{ .}$$
Here, $v$ is the $N$-adic valuation and $\overline{\frac{y^{v(x)}}{x^{v(y)}}}$ is the reduction modulo the prime above $N$ in $\mathbf{Z}[\zeta_N]$ of $\frac{y^{v(x)}}{x^{v(y)}}$. Composing with the (injective) map $K_2(\mathbf{Z}[\zeta_N, \frac{1}{Np}]) \rightarrow K_2(\mathbf{Q}(\zeta_N))$, we get a map $\varphi : K_2(\mathbf{Z}[\zeta_N, \frac{1}{Np}]) \rightarrow \mathbf{Z}/p^r\mathbf{Z}$ such that $\varphi(\{1-\zeta_N^u, 1-\zeta_N^v\}) \equiv \log(\frac{u}{v})$ modulo $p^r$. The map $\varphi$ is $\Gal(\mathbf{Q}(\zeta_N)/\mathbf{Q})$-equivariant, where the action of $\Gal(\mathbf{Q}(\zeta_N)/\mathbf{Q})$ on $\mathbf{Z}/p^r\mathbf{Z}$ is trivial. By \cite[Proposition 3.3.11]{Neukirch_CNF}, this induces the map $\iota_r$ of the statement (which is unique because the elements $\log(\frac{u}{v})$ generate $\mathbf{Z}/p^r\mathbf{Z}$).
\end{proof}

Let $\Delta_r = [x]-[1] \in \Lambda_r$ where $x \in P_r$ is such that $\log(x) \equiv 1\text{ (modulo }p^r\text{)}$. The element $\Delta_r$ is a generator of $J$. The multiplication by $\Delta_r$ gives a natural surjective homomorphism $\delta_r' : \mathcal{K}_r/J\cdot \mathcal{K}_r \rightarrow J\cdot \mathcal{K}_r/J^2\cdot \mathcal{K}_r$, which is an isomorphism if and only if $n(r,p) \geq 2$. In this case, we define $\delta_r : \mathbf{Z}/p^r\mathbf{Z} \xrightarrow{\sim} J\cdot \mathcal{K}_r/J^2\cdot \mathcal{K}_r$ by $\delta_r = \delta_r' \circ \iota_r^{-1}$.

\subsection{Refined Sharifi theory}\label{odd_modSymb_Section_Sharifi}
We follow the notation of sections \ref{odd_modSymb_section_hida}, \ref{odd_modSymb_section_eisenstein_ideal} and \ref{odd_modSymb_section_class_group}. 
In this section, we explain, inspired by Sharifi \cite{Sharifi_survey}, the link between $H_1(X_1^{(p^r)}(N), C_0^{(p^r)}, \mathbf{Z}_p)$ and the $K$-group $\mathcal{K}_r$ studied in section \ref{odd_modSymb_section_class_group}. 

As in section \ref{odd_modSymb_section_hida}, if $u, v \in (\mathbf{Z}/N\mathbf{Z})^{\times}$, we have the Manin symbol $\xi_{\Gamma_1(N)}([u,v]) \in \tilde{H}_{\Gamma_1(N)}$. Recall that $M_{\Gamma_1(N)}^0 = \mathbf{Z}_p[((\mathbf{Z}/N\mathbf{Z})^{\times})^2/\pm]$. Following Goncharov \cite{Goncharov}, consider the map
$$ \varpi' : M_{\Gamma_1(N)}^0 \rightarrow K_2\left(\mathbf{Z}[\zeta_N, \frac{1}{Np}]\right) \otimes_{\mathbf{Z}} \mathbf{Z}_p$$
given by 
$$\varpi'\left([u,v] \right) =\left\{1-\zeta_N^u, 1-\zeta_N^v\right\} \otimes 1 \text{ .}$$

\begin{lem}\label{odd_modSymb_annihilates_Manin}
The map $\varpi'$ is even with respect to the complex conjugation $c$. It factors through $\xi_{\Gamma_1(N)}$. Hence we get a map
$$\varpi : \left( \tilde{H}_{\Gamma_1(N)}\right)_+ \rightarrow K_2\left(\mathbf{Z}[\zeta_N, \frac{1}{Np}]\right) \otimes_{\mathbf{Z}} \mathbf{Z}_p\text{ }$$
such that
$$\varpi\left((1+c)\cdot \xi_{\Gamma_1(N)}([u,v])\right) = \left\{1-\zeta_N^u, 1-\zeta_N^v\right\} \text{ .} $$
\end{lem}
\begin{proof}
The proof is the same as the proof of \cite[Theorem 5.1]{Busuioc} (also independently found by Sharifi). We give the proof here for the convenience of the reader. By Proposition \ref{generation_Manin_C_0^{(p^r)}}, it suffices to prove the following equalities for all $u, v \in (\mathbf{Z}/N\mathbf{Z})^{\times}$:
\begin{enumerate}
\item $\varpi'([-u,v]) = \varpi'([u,v])$.
\item $\varpi'([u,v]) + \varpi'([-v,u]) = 0$.
\item If  $u+v \neq 0$, we have $\varpi'([u,v]) + \varpi'([-(u+v),u]) + \varpi'([v,-(u+v)]) = 0$.
\item $\varpi'([u,-u])=0$. 
\end{enumerate}
These equalities follow from those general identities.

\begin{enumerate}[label=(\alph*)]
\item The Steinberg relations. If $x \in \mathbf{Z}[\zeta_N, \frac{1}{Np}]^{\times}$ and $1-x \in \mathbf{Z}[\zeta_N, \frac{1}{Np}]^{\times}$, then we have $\{x, 1-x\}=0$.
\item Antisymmetry. For all $x$, $y$ $\in \mathbf{Z}[\zeta_N, \frac{1}{Np}]^{\times}$, we have $\{x,y\}=-\{y,x\}$.

\item For all root of unity $\zeta$ of order prime to $p$ and all $x \in \mathbf{Z}[\zeta_N, \frac{1}{Np}]^{\times}$, we have $\{\zeta,x\}=0$.
\end{enumerate}

We prove (i). We have $\varpi'([-u,v]) = \{1-\zeta_N^{-u}, 1-\zeta_N^v\} = \{\zeta_N^u-1, 1-\zeta_N^v\} - \{\zeta_N^u, 1-\zeta_N ^v\} = \{\zeta_N^u-1, 1-\zeta_N^v\} = \{1-\zeta_N^u, 1-\zeta_N^v\} = \varpi'([u,v])$. 

We prove (ii). We have $\varpi([-v,u]) = \{1-\zeta_N^{-v}, 1-\zeta_N^u\} = \{1-\zeta_N^{v}, 1-\zeta_N^u\}= -\{1-\zeta_N^{u}, 1-\zeta_N^v\} = -\varpi([u,v])$. 

We prove (iii). Note that we have $$\frac{\zeta_N^v\cdot (1-\zeta_N^u)}{1-\zeta_N^{u+v}}+ \frac{1-\zeta_N^v}{1-\zeta_N^{u+v}}=1 \text{ .}$$
This shows that 
$$\left\{\frac{\zeta_N^v\cdot (1-\zeta_N^u)}{1-\zeta_N^{u+v}}, \frac{1-\zeta_N^v}{1-\zeta_N^{u+v}} \right\} = 0 \text{ .}$$ 
Using the facts above and the bilinearity of $\{\cdot, \cdot\}$, this proves (iii). We have: $\{1-\zeta_N^{-u}, 1-\zeta_N^u\} = \{1-\zeta_N^{u}, 1-\zeta_N^u\}= 0$, which proves (iv).
\end{proof}

The following conjecture is inspired by the work of Sharifi. We refer to \cite[Section 2]{Sharifi_survey} for details about Sharifi's conjectures. However, our situation is not considered by Sharifi and Fukaya--Kato, who consider modular curves of level divisible by $p$ \cite[Section 5.2]{Fukaya_Kato}.

\begin{conj}\label{odd_modSymb_Sharifi_conj}
The map $\varpi$ is annihilated by the Hecke operator $T_{n}-\sum_{d \mid n, \gcd(d,N)=1} \frac{n}{d}\cdot \langle d \rangle$ for all $n \geq 1$ such that $\gcd(n,p)=1$.
\end{conj}

We end this section by stating a conjecture about our $K$-groups generalizing Theorem \ref{comparison_K_C}. The norm map gives a group homomorphism $K_2\left(\mathbf{Z}[\zeta_N, \frac{1}{Np}]\right) \otimes_{\mathbf{Z}} \mathbf{Z}_p\rightarrow \mathcal{K}_r$ sending $\{x,y\} \otimes 1$ to $(x,y)_r$. By  Proposition \ref{odd_modSymb_refined_Hida}, the map $\varpi$ induces a $\Lambda_r$-module homomorphism $$\varpi^{(p^r)} : \left(\tilde{H}^{(p^r)}\right)_+ \rightarrow \mathcal{K}_r$$ such that $$\varpi^{(p^r)}\left((1+c)\cdot \xi_{\Gamma_1^{(p^r)}(N)}([u,v])\right) = (1-\zeta_N^u, 1-\zeta_N^v)_r$$ for all $u,v \in (\mathbf{Z}/N\mathbf{Z})^{\times}$. By Proposition \ref{odd_modSymb_local_cup_product}, the map $\varpi^{(p^r)}$ is surjective modulo $J$, so it is surjective.

\begin{conj}\label{odd_modSymb_conjecture_class_group}
The map $\varpi^{(p^r)}$ induces an isomorphism of $\Lambda_r$-modules
$$\left(\tilde{H}^{(p^r)}\right)_+/\left(\tilde{I}_{\infty}+(p^r)\right) \cdot \left(\tilde{H}^{(p^r)}\right)_+  \xrightarrow{\sim} \mathcal{K}_r \text{ .}$$
\end{conj}

The $\Lambda_r$-module $\mathcal{K}_r$ is cyclic generated by $(1-\zeta_N^x, 1-\zeta_N^y)_r$ for all $x$, $y$ $\in (\mathbf{Z}/N\mathbf{Z})^{\times}$ such that $xy^{-1}$ is not a $p$th power. Using Stickelberger theory, one can show that the annihilator of $\mathcal{K}_r$ in $\Lambda_r$ contains $(p^r)+(\zeta^{(p^r)})$. Theorem \ref{odd_modSymb_structure_H_1} gives us the following structure theorem for the $\Lambda_r$-module $\mathcal{K}_r$.

\begin{prop}\label{odd_modSymb_consequence_class_group}
Assume Conjecture \ref{odd_modSymb_Sharifi_conj}. Then Conjecture \ref{odd_modSymb_conjecture_class_group} is true if and only if the annihilator of $\mathcal{K}_r$ in $\Lambda_r$ is $(\zeta^{(p^r)})+(p^r)$.

If $r=1$, this is true if and only if the $\mathbf{F}_p$-rank of $\mathcal{K}_1$ is the largest integer $i \in \{1, 2, ..., p-1\}$ such that for all $1 \leq j < i$, we have
$$\sum_{k\in (\mathbf{Z}/N\mathbf{Z})^{\times}} \B_2\left(\frac{k}{N}\right) \cdot \log(k)^j \equiv 0 \text{ (modulo }p\text{)}$$
(this is a particular case for $\chi = \omega_p^{-1}$ of \cite[Conjecture 1.10]{Lecouturier_class_group}).
\end{prop}

Using the computer software PARI/GP, Nicolas Mascot checked the truth of the last condition (when $r=1$) of Proposition \ref {odd_modSymb_consequence_class_group} for $p=5$ and $N \leq 12791$, thus proving Conjecture \ref{odd_modSymb_conjecture_class_group} (assuming Conjecture \ref{odd_modSymb_Sharifi_conj}) for theses values of $N$ and $p$. He did the computation with $\mathcal{K}_1$ replaced by the $\omega_p^{-1}$-part $\mathcal{C}_{(\omega_p^{-1})}$ of the class group of $L_r$ modulo $p$, since one can show that $\mathcal{C}_{(\omega_p^{-1})}$ and $\mathcal{K}_1$ are isomorphic $\Lambda_1$-modules.

\subsection{Evidence in favor of Conjecture \ref{odd_modSymb_Sharifi_conj}}
We follow the notation of section \ref{odd_modSymb_Section_Sharifi}.  We give some evidence in favor of Conjecture \ref{odd_modSymb_Sharifi_conj}. 

The following result was proved in level $\Gamma_1(p)$ and weight $2$ by Busioc \cite[Theorem 1.1]{Busuioc} (and independently by Sharifi) and her proof can be adapted directly to our case. We reproduce it here for the convenience of the reader.
\begin{thm}
The map $\varpi$ is annihilated by the Hecke operators $T_2-2-\langle 2 \rangle$ and $T_3-3-\langle 3 \rangle$.
\end{thm}
\begin{proof}
The proof relies on explicit formulas for Hecke operators acting on Manin symbols \cite{Merel_Universal}. For each integer $n \geq 1$ prime to $N$, let 
$$\mathcal{X}_n = \left\{\begin{pmatrix} a & b \\ c & d \end{pmatrix} \in M_2(\mathbf{Z})\text{, } a>b\geq 0 \text{, } d>c \geq 0\text{, }ad-bc=n\right\} \text{ .}$$
This is a finite set \cite[Lemma 8]{Merel_Universal}. By \cite[Theorem 2 and Proposition 20]{Merel_Universal}, for all $\Gamma_0(N)\cdot \gamma \in \Gamma_0(N) \backslash \SL_2(\mathbf{Z})$ we have in $H_1(X_1(N), \cusps, \mathbf{Z})$:
\begin{equation}\label{odd_modSymb_Merel_formula_Hecke}
T_n \left( \xi_{\Gamma_1(N)}(\Gamma_0(N)\cdot \gamma) \right) = \sum_{\alpha \in \mathcal{X}_n} \xi_{\Gamma_1(N)}\left(\Gamma_0(N)\cdot \gamma\cdot \alpha \right) \text{ .}
\end{equation}

We first prove that $\varpi$ is annihilated by the Hecke operator $T_2-2-\langle 2 \rangle$. We have
$$\mathcal{X}_2 = \left\{\begin{pmatrix} 1 & 0 \\ 0 & 2 \end{pmatrix}\text{, } \begin{pmatrix} 2 & 0 \\ 0 & 1 \end{pmatrix}\text{, } \begin{pmatrix} 1 & 0 \\ 1 & 2 \end{pmatrix}\text{, } \begin{pmatrix} 2 & 1 \\ 0 & 1 \end{pmatrix}\right\} \text{ .}$$
Let $u,v \in (\mathbf{Z}/N\mathbf{Z})^{\times}$. By (\ref{odd_modSymb_Merel_formula_Hecke}), we have:
\begin{equation}\label{odd_modSymb_Merel_formula_Hecke_2}
T_2\left(\xi_{\Gamma_1(N)}([u,v]) \right) = \xi_{\Gamma_1(N)}([u,2v])+\xi_{\Gamma_1(N)}([2u,v])+\xi_{\Gamma_1(N)}([u+v,2v])+\xi_{\Gamma_1(N)}([2u,u+v]) \text{ .}
\end{equation}
Assume first that $u + v \neq 0$. By (\ref{odd_modSymb_Merel_formula_Hecke_2}), we have in $\mathcal{K}_r$:
$$\varpi\left( (1+c)\cdot T_2\left(\xi_{\Gamma_1(N)}([u,v])\right) \right)=\{1-\zeta_N^u, 1-\zeta_N^{2v}\} + \{1-\zeta_N^{2u}, 1-\zeta_N^v\} + \{1-\zeta_N^{u+v}, 1-\zeta_N^{2v}\} + \{1-\zeta_N^{2u}, 1-\zeta_N^{u+v}\} \text{ .}$$
The following identity was discovered by McCallum and Sharifi \cite{McCallum_Sharifi}:
\begin{equation}\label{identity_p=2_steinberg}
\frac{(1-\zeta_N^{u+v})\cdot (1-\zeta_N^u)}{1-\zeta_N^{2u}}+\frac{\zeta_N^u\cdot (1-\zeta_N^{2v})\cdot(1-\zeta_N^{u})}{(1-\zeta_N^{2u})\cdot (1-\zeta_N^{v})} = 1 \text{ .}
\end{equation}
Using (\ref{identity_p=2_steinberg}) and the properties of $\{\cdot, \cdot\}$ stated in the proof of Lemma \ref{odd_modSymb_annihilates_Manin}, we  have in $K_2\left(\mathbf{Z}[\zeta_N, \frac{1}{Np}]\right)$:
\begin{align*}
0&=\left\{\frac{(1-\zeta_N^{u+v})\cdot (1-\zeta_N^u)}{1-\zeta_N^{2u}},\frac{\zeta_N^u\cdot (1-\zeta_N^{2v})\cdot (1-\zeta_N^{u})}{(1-\zeta_N^{2u})\cdot (1-\zeta_N^{v})}\right\} \\&=\{1-\zeta_N^{u+v}, 1-\zeta_N^{2v}\}+\{1-\zeta_N^{u+v}, 1-\zeta_N^u\}-(1-\zeta_N^{u+v}, 1-\zeta_N^{2u}\}-\{1-\zeta_N^{u+v}, 1-\zeta_N^{v}\}\\&+\{1-\zeta_N^{u}, 1-\zeta_N^{2v}\}-\{1-\zeta_N^{u}, 1-\zeta_N^{2u}\}-\{1-\zeta_N^{u}, 1-\zeta_N^{v}\}-\{1-\zeta_N^{2u}, 1-\zeta_N^{2v}\}\\&-\{1-\zeta_N^{2u}, 1-\zeta_N^{u}\}+\{1-\zeta_N^{2u}, 1-\zeta_N^{v}\} \\&= \{1-\zeta_N^u, 1-\zeta_N^{2v}\} + \{1-\zeta_N^{2u}, 1-\zeta_N^v\} + \{1-\zeta_N^{u+v}, 1-\zeta_N^{2v}\} + \{1-\zeta_N^{2u}, 1-\zeta_N^{u+v}\}  \\& +  \{1-\zeta_N^{u+v}, 1-\zeta_N^{u}\} - \{1-\zeta_N^{u+v}, 1-\zeta_N^{v}\} - \{1-\zeta_N^{u}, 1-\zeta_N^{v}\}-\{1-\zeta_N^{2u}, 1-\zeta_N^{2v}\} \text{ .}
\end{align*}
Furthermore, we have seen that the Manin relations hold, \ie:
\begin{align*}
 \{1-\zeta_N^{u}, 1-\zeta_N^{v}\}+ \{1-\zeta_N^{u+v}, 1-\zeta_N^{u}\} - \{1-\zeta_N^{u+v}, 1-\zeta_N^{v}\} = 0 \text{ .} 
\end{align*}
Using \ref{odd_modSymb_Merel_formula_Hecke_2}, we get:
\begin{align*}
\varpi\left( (1+c)\cdot T_2\left(\xi_{\Gamma_1(N)}([u,v])\right) \right) = 2\cdot \{1-\zeta_N^{u}, 1-\zeta_N^{v}\} + \{1-\zeta_N^{2u}, 1-\zeta_N^{2v}\} \text{ .}
\end{align*}
Since $\langle 2 \rangle \cdot \xi_{\Gamma_1(N)}([u,v]) = \xi_{\Gamma_1(N)}([2u,2v])$, we have:
$$\varpi\left( (1+c)\cdot (T_2-2-\langle 2 \rangle)\left(\xi_{\Gamma_1(N)}([u,v])\right) \right) = 0 \text{ .}$$

If $u+v=0$, then we have $\xi_{\Gamma_1(N)}([u,v]) = 0$ by the Manin relations. Thus, we also have in $\mathcal{K}_r$:
$$\varpi_r\left( (1+c)\cdot (T_2-2-\langle 2 \rangle)\left(\xi_{\Gamma_1(N)}([u,v])\right) \right) = 0 \text{ .}$$

The proof that $\varpi'$ is annihilated by the Hecke operator $T_3-3-\langle 3 \rangle$ is the same as for $T_2-2-\langle 2 \rangle$, using the identity (discovered by Sharifi in an unpublished work)
$$\frac{\zeta_N^{v-u}\cdot (1-\zeta_N^{3u})\cdot (1-\zeta_N^{v})}{(1-\zeta_N^u)\cdot (1-\zeta_N^{3v})}+\frac{(1-\zeta_N^{v-u})\cdot (1-\zeta_N^v)\cdot (1-\zeta_N^{u+v})}{1-\zeta_N^{3v}}=1 \text{ .}$$
\end{proof}
\begin{rem}
We were not able to prove that $T_5-5-\langle 5 \rangle$ annihilates $\varpi_r$. 
\end{rem}

Another evidence in favor of Conjecture \ref{odd_modSymb_Sharifi_conj} is that the analogous conjecture when $p$ divides the level has been proved by Fukaya and Kato \cite[Theorem 5.2.3]{Fukaya_Kato}. Their methods are $p$-adic, so it is not clear how to generalize them in our setting. 

We end this section by recalling briefly a construction due to Goncharov, used by Fukaya and Kato. If $u \in (\mathbf{Z}/N\mathbf{Z})^{\times}$, one can defined a Siegel unit $g_{0, \frac{u}{N}} \in \mathcal{O}(Y_1(N)_{\mathbf{Z}[\zeta_N, \frac{1}{N}]})^{\times} \otimes_{\mathbf{Z}} \mathbf{Z}[\frac{1}{N}]$, where $\mathcal{O}(Y_1(N)_{\mathbf{Z}[\zeta_N, \frac{1}{N}]})^{\times}$ is the ring of global sections of the open modular curve $Y_1(N)$ over $\mathbf{Z}[\zeta_N, \frac{1}{N}]$. We refer to \cite[Section 2.3.1]{Sharifi_survey} for its definition. The specialization of $g_{0, \frac{u}{N}}$ at the cusp $\Gamma_1(N)\cdot \infty$ is $1-\zeta_N^u$. We define a map 
$$g : \mathbf{Z}[\left((\mathbf{Z}/N\mathbf{Z})^{\times}\right)^2/\pm 1] \rightarrow K_2\left(\mathcal{O}(Y_1(N)_{\mathbf{Z}[\zeta_N, \frac{1}{N}]})^{\times}\right) \otimes_{\mathbf{Z}}  \mathbf{Z}[\frac{1}{N}] $$
by $g([u,v]) = \{g_{0, \frac{u}{N}}, g_{0, \frac{v}{N}}\}$ where $\{\cdot, \cdot\}$ is the Steinberg symbol. Goncharov proved that an analogous map at level $\Gamma(N)$ factors through the Manin relations \cite[Corollary 2.17]{Goncharov}. His proof in fact shows that $g$ factors through $\xi_{\Gamma_1(N)}$. Thus, we get a map
$$\Gamma : H_1(X_1(N), C_{\Gamma_1(N)}^0, \mathbf{Z}) \rightarrow  K_2\left(\mathcal{O}(Y_1(N)_{\mathbf{Z}[\zeta_N, \frac{1}{N}]})^{\times}\right) \otimes_{\mathbf{Z}} \mathbf{Z}[\frac{1}{N}]  \text{ .}$$
We expect that $\Gamma$ commutes with the action of the Hecke operators $T_n$ and $\langle n \rangle$ for $n$ prime to $N$. This was proved by Fukaya and Kato when $p$ divides the level \cite[Remark 3.3.16]{Fukaya_Kato}. If this is true, then Conjecture \ref{odd_modSymb_Sharifi_conj} is true by specializing at the cusp $\Gamma_1(N)\cdot \infty$. This is the idea behind the proof of \cite[Theorem 5.2.3]{Fukaya_Kato}. However, we were not able to prove that $\Gamma$ is Hecke-equivariant.

\subsection{Construction of the element $m_2^-$ under Conjecture \ref{odd_modSymb_Sharifi_conj}}\label{odd_modSymb_section_construction}

In this section, we assume Conjecture \ref{odd_modSymb_Sharifi_conj} but not yet that $n(r,p) \geq 2$. We follow the notation of sections \ref{odd_modSymb_section_hida}, \ref{odd_modSymb_section_eisenstein_ideal}, \ref{odd_modSymb_section_class_group} and \ref{odd_modSymb_Section_Sharifi}.  

We first construct an explicit group homomorphism:
$$\psi : I\cdot H_+/I^{2}\cdot H_+ \rightarrow J_t\cdot \mathcal{K}_t/J_t^2 \cdot \mathcal{K}_t \text{.}$$

Recall the Eisenstein ideals $\tilde{I}_0$ and $\tilde{I}_{\infty}$ defined in section \ref{odd_modSymb_section_eisenstein_ideal}, annihilating the cusps $\Gamma_1^{(p^t)}(N) \cdot 0$ and $\Gamma_1^{(p^t)}(N)\cdot \infty$ respectively. Since we assume that Conjecture \ref{odd_modSymb_Sharifi_conj} holds, $\varpi^{(p^t)}$ induces a surjective morphism of $\Lambda_t$-modules $\varphi: \left(\tilde{H}^{(p^t)}\right)_+ \rightarrow \mathcal{K}_t $
annihilating $\tilde{I}_{\infty} \cdot  \left(\tilde{H}^{(p^t)}\right)_+$. Let $\varphi'$ be the restriction of $\varphi$ to $\tilde{I}_0 \cdot \left(\tilde{H}^{(p^t)}\right)_+$. We claim that the image of $\varphi'$ is $J_t \cdot \mathcal{K}_t$.
Since $\varphi$ is surjective, it suffices to note that we have $T_{\ell}- \ell\cdot \langle \ell \rangle-1 \in \tilde{I}_{0}$ and $T_{\ell} - \langle \ell \rangle - \ell \in \tilde{I}_{\infty}$, so
$$(\langle \ell \rangle-1)\cdot (\ell-1) \in \tilde{I}_0 + \tilde{I}_{\infty} $$
for all prime number $\ell$ different from $N$. Similarly, we see that $\varphi'$ induces a surjective group homomorphism
$$\varphi'' : \tilde{I}_0 \cdot \left(\tilde{H}^{(p^t)}\right)_+/(\tilde{I}_0^2 +J\cdot \tilde{I}_0)\cdot \left(\tilde{H}^{(p^t)}\right)_+ \rightarrow J_t \cdot \mathcal{K}_t/J_t^2\cdot \mathcal{K}_t \text{ .}$$

By Proposition \ref{odd_modSymb_image_I_0} and Corollary \ref{odd_modSymb_construction}, we have 
$$\tilde{I}_0 \cdot \left(\tilde{H}^{(p^t)}\right)_+/(\tilde{I}_0^2 +J\cdot \tilde{I}_0)\cdot \left(\tilde{H}^{(p^t)}\right)_+ =  \left(H^{(p^t)}\right)_+/(I_0 + J)\cdot \left(H^{(p^t)}\right)_+ \simeq I\cdot H_+/I^2\cdot H_+ \text{ .}$$
Thus, we have a canonical group isomorphism
$$\psi' : I\cdot H_+/I^2\cdot H_+ \xrightarrow{\sim}\tilde{I}_0 \cdot \left(\tilde{H}^{(p^t)}\right)_+/(\tilde{I}_0^2 +J\cdot \tilde{I}_0)\cdot \left(\tilde{H}^{(p^t)}\right)_+ \text{ .}$$
We then let 
$$\psi = \varphi'' \circ \psi' : I\cdot H_+/I^{2}\cdot H_+ \rightarrow J_t\cdot \mathcal{K}_t/J_t^2 \cdot \mathcal{K}_t \text{ .}$$

As in section \ref{odd_modSymb_construction}, we abuse notation and denote $J_r$ by $J$.

\begin{prop}
If $n(r,p) \geq 2$, the map $\psi$ induces a surjective group homomorphism:
$$\overline{\psi}_r: I\cdot (H_+/p^r\cdot H_+)/I^{2}\cdot (H_+/p^r\cdot H_+) \rightarrow J\cdot \mathcal{K}_r/J^{2}\cdot \mathcal{K}_r \text{ .}$$
\end{prop}
\begin{proof}
The norm map yields a canonical surjective group homomorphism $J_t \cdot \mathcal{K}_t/J_t^2 \cdot \mathcal{K}_t \rightarrow J \cdot \mathcal{K}_r/J^2 \cdot \mathcal{K}_r$. Thus, $\psi$ gives a surjective group homomorphism
$$\psi_r : I\cdot H_+ / I^2\cdot H_+ \rightarrow J \cdot \mathcal{K}_r/J^2 \cdot \mathcal{K}_r \text{ .}$$
To construct $\overline{\psi}_r$, it suffices to show that $\psi_r$ vanishes on the image of $(p^r\cdot H_+) \cap (I\cdot H_+)$ in $I\cdot H_+/I^2\cdot H_+$.
It suffices to prove that
\begin{equation}\label{odd_modSymb_localization_I}
\left(I^2\cdot H_+ + (p^r\cdot H_+) \cap (I\cdot H_+) \right) /\left( I^2\cdot H_+ + p^r\cdot I\cdot  H_+  \right) = 0 
\end{equation}

Since $I \cdot \left(I^2\cdot H_+ + (p^r\cdot H_+) \cap (I\cdot H_+) \right)\subset \left( I^2\cdot H_+ + p^r\cdot I\cdot  H_+  \right)$, it suffices to show (\ref{odd_modSymb_localization_I}) after completion at $I$. Using the winding isomorphism \cite[p. 137]{Mazur_Eisenstein}, it suffices to show that
\begin{equation}\label{odd_modSymb_localization_I_2}
\left((p^r\cdot I)\cap I^2 + I^3 \right)/\left( p^r\cdot I^2 + I^3 \right) \otimes_{\mathbb{T}} \mathbf{T}= 0 \text{ .}
\end{equation}

\begin{lem}
We have $\left( (p^r\cdot I) \cap I^2 \right) \otimes_{\mathbb{T}} \mathbf{T} = \left( p^t\cdot I + p^r\cdot I^2 \right)  \otimes_{\mathbb{T}} \mathbf{T}$.
\end{lem}
\begin{proof}
We have $p^t\cdot I + p^r\cdot I^2 \subset  (p^r\cdot I) \cap I^2$ since $p^t \in I$ and $r \leq t$. Let $\eta$ be a generator of $I\cdot \mathbf{T}$, and $x \in (p^r\cdot I)\cap I^2$. We can write $x = p^r \cdot \eta \cdot u = \eta^2\cdot v$ for some $u,v \in \mathbf{T}$.
Since $\eta$ is not a zero divisor in $\mathbf{T}$, we have $p^r\cdot u = \eta \cdot v \in I$. Let $m\in  \mathbf{Z}$ such that $u-m \in I$. We have $p^r\cdot m \in I \cap \mathbf{Z}_p = p^t\cdot \mathbf{Z}_p$. Thus, we have $x \in p^t\cdot I + p^r\cdot I^2$.
\end{proof}
 
Since $n(r,p) \geq 2$, we have by Proposition \ref{Formalism_odd_modSymb_critere_higher_eisenstein}, $p^t\cdot I \subset I^3+p^r\cdot I^2$. This ends the proof of the proposition.
\end{proof}

If $n(r,p) \geq 2$, the group $J\cdot \mathcal{K}_r/J^{2}\cdot \mathcal{K}_r$ is cyclic of order $p^r$ by Theorem \ref{comparison_K_C}. As explained in Section \ref{odd_modSymb_preliminary_setting}, the map $\psi_r$ gives the construction of the higher Eisenstein element $m_2^-$.

\subsection{Explicit computation of $m_2^-$}\label{odd_modSymb_explicit_computation_section}

In this section, we assume that Conjecture \ref{odd_modSymb_Sharifi_conj} holds and that $n(r,p) \geq 2$. We keep the notation the previous sections of Chapter \ref{Section_odd_modSymb}.

Recall that by intersection duality, we can consider $m_2^-$ as a group homomorphism
$$I \cdot (H_+/p^r\cdot H_+) / I^2\cdot (H_+/p^r\cdot H_+) \rightarrow  J\cdot \mathcal{K}_r/J^{2}\cdot \mathcal{K}_r \text{ .}$$
By Proposition \ref{odd_modSymb_image_I_0}, an element of $I \cdot H_+$ is the image of an element of $\left(H^{(p^r)}\right)_+$, which can be written as
$$\sum_{[u,v] \in \left( (\mathbf{Z}/N\mathbf{Z})^{\times} \right)^2/P_r'} \lambda_{[u,v]} \cdot (1+c)\cdot \xi_{\Gamma_1^{(p^r)}(N)}([u,v])$$
for some $\lambda_{[u,v]} \in \mathbf{Z}_p$, with the boundary condition $\sum_{[u,v]} \lambda_{[u,v]}\cdot([v]-[u])=0$ in $\Lambda_r$.

Equivalently, we have
$$I\cdot H_+ = \left\{\sum_{x \in (\mathbf{Z}/N\mathbf{Z})^{\times}} \lambda_x\cdot (1+c)\cdot \xi_{\Gamma_0(N)}(x), \sum_{x} \lambda_x \cdot \log(x) \equiv 0 \text{ (modulo }p^t\text{)} \right\} \text{ .}$$

An element $\sum_{x \in  (\mathbf{Z}/N\mathbf{Z})^{\times}} \lambda_x \cdot [x] \in \mathbf{Z}_p[(\mathbf{Z}/N\mathbf{Z})^{\times}]$ satisfies  $\sum_{x \in  (\mathbf{Z}/N\mathbf{Z})^{\times}} \lambda_x \cdot \log(x) \equiv 0 \text{ (modulo }p^t\text{)}$ if and only if it is in the subgroup generated by the square of the augmentation ideal of $\mathbf{Z}_p[(\mathbf{Z}/N\mathbf{Z})^{\times}]$ and by $[\overline{1}]$. Thus, any such element is a linear combination of elements of the form $[x\cdot y]-[x]-[y]$. Therefore, $m_2^-$ is determined by the values $$\left((1+c)\cdot \xi_{\Gamma_0(N)}([x\cdot y:1]-[x:1]-[y:1])\right)\bullet m_2^-$$ for  all $x,y \in (\mathbf{Z}/N\mathbf{Z})^{\times}$. 

\begin{thm}\label{odd_modSymb_main_theorem_>5}
Assume that Conjecture \ref{odd_modSymb_Sharifi_conj} holds. Assume that $n(r,p) \geq 2$, \ie that
$$\sum_{k=1}^{\frac{N-1}{2}} k \cdot \log(k) \equiv 0 \text{ (modulo } p^r\text{).}$$

Let $x$, $y$ $\in (\mathbf{Z}/N\mathbf{Z})^{\times}$. Then we have the following equality in $J\cdot \mathcal{K}_r/J^2\cdot \mathcal{K}_r$:
\begin{align*}
\delta_r\left(\left( 2\cdot (1+c)\cdot \xi_{\Gamma_0(N)}([x\cdot y:1]-[x:1] - [y:1]) \right) \bullet m_2^-\right)&= (1-\zeta_N^x, 1-\zeta_N)_r \\&-(1-\zeta_N^x, 1-\zeta_N^{y^{-1}})_r -(1-\zeta_N^{y^{-1}}, 1-\zeta_N)_r
\end{align*}
where $(\cdot, \cdot)_r : \mathbf{Z}[\zeta_N, \frac{1}{Np}]^{\times} \times  \mathbf{Z}[\zeta_N, \frac{1}{Np}]^{\times} \rightarrow \mathcal{K}_r$ and $\delta_r : \mathbf{Z}/p^r\mathbf{Z}  \xrightarrow{\sim}J\cdot \mathcal{K}_r/J^2\cdot \mathcal{K}_r$ were defined in Section \ref{odd_modSymb_section_class_group}.
\end{thm}
\begin{proof}
Let $\pi : \left(\tilde{H}^{(p^r)}\right)_+ \rightarrow H_+$ be the canonical map. The group homomorphism $$\varphi :   \left(\tilde{H}^{(p^r)}\right)_+ \rightarrow \mathcal{K}_t$$ gives a surjective group homomorphism $$\varphi_r : \left(\tilde{H}^{(p^r)}\right)_+ \rightarrow  \mathcal{K}_r/J\cdot \mathcal{K}_r \text{ .}$$
Similarly,  the group homomorphism $$\varphi' :   \tilde{I}_0 \cdot \left(\tilde{H}^{(p^r)}\right)_+ \rightarrow \mathcal{K}_t$$ gives a surjective group homomorphism $$\varphi_r' : \tilde{I}_0 \cdot \left(\tilde{H}^{(p^r)}\right)_+ \rightarrow  J\cdot \mathcal{K}_r/J^2\cdot \mathcal{K}_r \text{ .}$$

For all $x, y \in (\mathbf{Z}/N\mathbf{Z})^{\times}$, we have in  $H_+$:
\begin{equation}\label{odd_modSymb_ultimate_proof_eq0}
\xi_{\Gamma_0(N)}([x\cdot y:1]-[x:1]-[y:1]) = \pi(\xi_{\Gamma_1^{(p^r)}(N)}([x,y^{-1}]-[x,1]+[y^{-1},1])) \text{ .}
\end{equation}
Furthermore $\xi_{\Gamma_1^{(p^r)}(N)}([x,y^{-1}]-[x,1]+[y^{-1},1]) \in \left(H^{(p^r)}\right)_+$ since
\begin{align*}
\partial\left( \xi_{\Gamma_1^{(p^r)}(N)}([x,y^{-1}]-[x,1]+[y^{-1},1])\right) &= [x]_{\Gamma_1^{(p^r)}(N)}^0-[y^{-1}]_{\Gamma_1^{(p^r)}(N)}^0  - \left([x]_{\Gamma_1^{(p^r)}(N)}^0-[1]_{\Gamma_1^{(p^r)}(N)}^0 \right) \\& + [y^{-1}]_{\Gamma_1^{(p^r)}(N)}^0-[1]_{\Gamma_1^{(p^r)}(N)}^0\\&
= 0 
\end{align*}
(\cf the computation of $\partial$ in Section \ref{odd_modSymb_section_hida}). Thus, to prove Theorem \ref{odd_modSymb_main_theorem_>5} it suffices to prove that for all $h \in \tilde{I}_0 \cdot  \left(\tilde{H}^{(p^r)}\right)_+$, we have in $J\cdot \mathcal{K}_r/J^2\cdot \mathcal{K}_r$:
\begin{equation}\label{odd_modSymb_delta_r_m_2}
\delta_r  \left( 2 \cdot \pi(h) \bullet m_2^-\right) = -\varphi_r'(h) \text{ .}
\end{equation}

Let $\ell$ be a prime not dividing $N$. Note that
$$T_{\ell} - \ell \cdot \langle \ell \rangle -1 = T_{\ell}-\ell-\langle \ell \rangle - (\ell-1)\cdot (\langle \ell \rangle-1) \in I_{\infty} -  (\ell-1)\cdot (\langle \ell \rangle-1)\text{ .}$$
Since Conjecture \ref{odd_modSymb_Sharifi_conj} is assumed to be true, for all $u \in \left(\tilde{H}^{(p^r)}\right)_+$ we have in $J\cdot \mathcal{K}_r/J^2\cdot \mathcal{K}_r$:
\begin{equation}\label{odd_modSymb_ultimate_proof_eq1}
\varphi_r'((T_{\ell}-\ell\langle \ell \rangle-1)(u)) = -(\ell-1)\cdot ([\ell]-1) \cdot \varphi_r(u) \text{ .}
\end{equation}
By construction, for all $[x,y] \in M_{\Gamma_1^{(p^r)}(N)}^0$, we have in $\mathbf{Z}/p^r\mathbf{Z}$: 
\begin{equation}\label{odd_modSymb_ultimate_proof_eq2}
\iota_r\left( \varphi_r((1+c)\cdot \xi_{\Gamma_1^{(p^r)}(N)}([x,y])) \right)= \log\left(\frac{x}{y}\right) = \left( (1+c)\cdot \xi_{\Gamma_0(N)}\left([x:y]\right)\right)\bullet m_1^-  \text{ .}
\end{equation}
Combining (\ref{odd_modSymb_ultimate_proof_eq1}) and (\ref{odd_modSymb_ultimate_proof_eq2}), we have in $J\cdot \mathcal{K}_r/J^2\cdot \mathcal{K}_r$:
\begin{align*}
\varphi_r'((T_{\ell}-\ell\langle \ell \rangle-1)(u)) &= -\delta_r\left(2\cdot \frac{\ell-1}{2}\cdot \log(\ell)\cdot (\pi(u) \bullet m_1^-)\right) \\&=  -\delta_r\left(2\cdot (T_{\ell}-\ell-1)(\pi(u)) \bullet m_2^- \right) \\&= -\delta_r\left(2\cdot \pi((T_{\ell}-\ell\langle \ell \rangle-1)(u)) \bullet m_2^- \right) \text{ .}
\end{align*}

By Proposition \ref{odd_modSymb_image_I_0}, the elements $(T_{\ell}-\ell\langle \ell \rangle-1)(u)$ span $\left(H^{(p^r)}\right)_+$ when $\ell$ and $u$ varies. This concludes the proof of (\ref{odd_modSymb_delta_r_m_2}).
\end{proof}

\section{Even modular symbols}\label{Section_even_modSymb}
In this chapter, unless explicitly stated, we allow $p=2$ and $p=3$. We keep the notation of chapters \ref{Section_introduction}, \ref{Section_Formalism} and  \ref{Section_odd_modSymb}. We assume as usual that $p$ divides the numerator of $\frac{N-1}{12}$. We let $v$ be the $p$-adic valuation of $N-1$. We extend $\log$ to a group homomorphism $(\mathbf{Z}/N\mathbf{Z})^{\times} \rightarrow \mathbf{Z}/p^{v}\mathbf{Z}$ (still abusively denoted by $\log$).

Recall that in Theorem \ref{thm_Introduction_w_0^+}, we defined an element $\tilde{m}_0^+ \in H_1(X_0(N), \cusps, \mathbf{Q})_+$ (independent of the choice of $p$). We have, in fact, $\tilde{m}_0^+ \in H_1(X_0(N), \cusps, \mathbf{Z}_p)_+$ even if when $p \in \{2,3\}$ the formula defining $F_{0,p}$ is not $p$-integral. We fix this choice of $\tilde{m}_0^+$ for the rest of the paper. 

Let $r$ be an integer such that $1 \leq r \leq t$. We denote by $m_0^+$, ..., $m_{n(r,p)}^+$ the higher Eisenstein elements in $M_+/p^r\cdot M_+$, where $m_0^+$ is the image of $\tilde{m}_0^+$ in $M_+/p^r\cdot M_+$.

In this chapter, we give an explicit formula for $m_1^+$ modulo $p^t$ if $p \geq 3$, and a formula for $m_1^+$ modulo $p^{t-1}$ if $t\geq 2$ and $p=2$. In particular, we prove Theorem \ref{thm_Introduction_w_1^+>3}. The formula when $p=3$ (resp. $p=2$) is given in Theorem \ref{thm_Introduction_w_1^+_3} (resp. Theorem \ref{thm_Introduction_w_1^+_2}).

\subsection{Some results about the homology of $X_0(N)$ and $X_1(N)$}

In this section, we gather some useful results about the homology of $X_0(N)$ and $X_1(N)$.

\begin{prop}\label{even_modSymb_reduction_mod_2_complex_conjug}
Let $n \geq 1$ be an integer.
\begin{enumerate}
\item  The inclusion $M_+ \hookrightarrow H_1(X_0(N), \cusps, \mathbf{Z}_p)$ gives a $\tilde{\mathbb{T}}$-equivariant group isomorphism
$$M_+/p^n\cdot M_+ \xrightarrow{\sim} H_1(X_0(N), \cusps, \mathbf{Z}/p^n\mathbf{Z})_+ \text{ .}$$
\item The surjection $H_1(Y_0(N), \mathbf{Z}_p) \twoheadrightarrow M^-$ gives a $\tilde{\mathbb{T}}$-equivariant group isomorphism
$$H_1(Y_0(N), \mathbf{Z}/p^n\mathbf{Z})^- \xrightarrow{\sim} M^-/p^n\cdot M^- \text{ .}$$
\end{enumerate}
\end{prop}
\begin{proof}
Point (ii) follows from point (i) by intersection duality. Thus, we only need to prove (i).
The multiplication by $p^n$ gives an exact sequence of $\tilde{\mathbb{T}}$-modules
\begin{equation}\label{even_modSymb_exact_sequence_cplx_conjug}
0 \rightarrow H_1(X_0(N), \cusps, \mathbf{Z}_p) \xrightarrow{p^n} H_1(X_0(N), \cusps, \mathbf{Z}_p) \rightarrow H_1(X_0(N), \cusps, \mathbf{Z}/p^n\mathbf{Z}) \rightarrow 0\text{ .}
\end{equation}
There is an action of $\mathbf{Z}/2\mathbf{Z}$ on each of the groups involved in (\ref{even_modSymb_exact_sequence_cplx_conjug}), given by the complex conjugation. 
Furthermore, (\ref{even_modSymb_exact_sequence_cplx_conjug}) is $\mathbf{Z}/2\mathbf{Z}$-equivariant. The long exact sequence of cohomology associated to $\mathbf{Z}/2\mathbf{Z}$ yields an exact sequence:
\begin{equation}\label{even_modSymb_exact_sequence_cplx_conjug_2}
0 \rightarrow M_+ \xrightarrow{p^n} M_+ \rightarrow H_1(X_0(N), \cusps, \mathbf{Z}/p^r\mathbf{Z})_+ \rightarrow H^1\left(\mathbf{Z}/2\mathbf{Z},   H_1(X_0(N), \cusps, \mathbf{Z}_p)\right)[p^n] \text{ .}
\end{equation}
We have a $\mathbf{Z}/2\mathbf{Z}$-equivariant exact sequence
$$0 \rightarrow H_1(X_0(N), \mathbf{Z}_p) \rightarrow H_1(X_0(N), \cusps, \mathbf{Z}_p) \rightarrow \mathbf{Z}_p \rightarrow 0 \text{ ,}$$
where the action of $\mathbf{Z}/2\mathbf{Z}$ on $\mathbf{Z}_p$ is trivial. The long exact sequence in cohomology yields an exact sequence
$$ H^1\left(\mathbf{Z}/2\mathbf{Z}, H_1(X_0(N), \mathbf{Z}_p)\right) \rightarrow H^1\left(\mathbf{Z}/2\mathbf{Z}, H_1(X_0(N), \cusps, \mathbf{Z}_p)\right) \rightarrow H^1\left(\mathbf{Z}/2\mathbf{Z},  \mathbf{Z}_p\right) \text{ .}$$
We have $H^1(\mathbf{Z}/2\mathbf{Z}, \mathbf{Z}_p)=0$ and $H^1\left(\mathbf{Z}/2\mathbf{Z}, H_1(X_0(N), \mathbf{Z}_p)\right)=0$ \cite[Proposition $5$]{Merel_accouplement}. Thus, we have $H^1\left(\mathbf{Z}/2\mathbf{Z}, H_1(X_0(N), \cusps, \mathbf{Z}_p)\right) = 0$, which concludes the proof of Proposition \ref{even_modSymb_reduction_mod_2_complex_conjug} by (\ref{even_modSymb_exact_sequence_cplx_conjug_2}).
\end{proof}
Thus, we will abuse notation and write $H_1(X_0(N), \cusps, \mathbf{Z}/p^n\mathbf{Z})_+$ (resp. $H_1(Y_0(N),\mathbf{Z}/p^n\mathbf{Z})^-$) for $M_+/p^n\cdot M_+$ (resp. $M^-/p^n\cdot M^-$). 

Let $\pi : X_1(N) \rightarrow X_0(N)$ be the standard degeneracy map. This produces by pull-back and push-forward two maps 
$$\pi^* : H_1(X_0(N), \cusps, \mathbf{Z}) \rightarrow H_1(X_1(N), \cusps, \mathbf{Z})$$
and 
$$\pi_* : H_1(X_1(N), \cusps, \mathbf{Z}) \rightarrow H_1(X_0(N), \cusps, \mathbf{Z}) \text{ .}$$
We will freely abuse notation and still denote the pull-back maps for different coefficient rings than $\mathbf{Z}$ by $\pi^*$ and $\pi_*$.

\begin{prop}\label{even_modSymb_injectivity_trace}
\begin{enumerate}
\item For any integer $n \geq 1$, the kernel of $\pi^* : H_1(X_0(N), \cusps, \mathbf{Z}/p^n\mathbf{Z}) \rightarrow H_1(X_1(N), \cusps, \mathbf{Z}/p^n\mathbf{Z})$ is cyclic of order $p^{\min(n,t)}$, annihilated by the Eisenstein ideal $\tilde{I}$ and by $1+c$ where $c$ is the complex conjugation. If $p\neq 3$, a generator of this kernel is $p^{\max(n-t,0)} \cdot \mathcal{E}_{p}^-$,
where $$\mathcal{E}_{p}^- := \frac{1}{3}\cdot \sum_{x \in \mathcal{R} \atop x \not\sim [1:1]} \log\left(\frac{x-1}{x+1}\right) \cdot \xi_{\Gamma_0(N)}(x) \in H_1(X_0(N), \mathbf{Z}/p^t\mathbf{Z}) \text{ .}$$
Here, $\mathcal{R}$ is the set of equivalences classes in $\mathbf{P}^1(\mathbf{Z}/N\mathbf{Z})$ for the equivalence relation $[c:d] \sim [-d:c]$.
\item If $p \geq 3$, then for any integer $n \geq 1$, the pull-back map $$M_+/p^n\cdot M_+ \rightarrow H_1(X_1(N), \cusps, \mathbf{Z}/p^n\mathbf{Z})$$ is injective.
\item For any integer $n \geq 1$, the kernel of the pull-back map $$M_+/2^n\cdot M_+ \rightarrow H_1(X_1(N), \cusps, \mathbf{Z}/2^n\mathbf{Z})$$ is spanned by the reduction of $2^{n-1}\cdot \tilde{m}_0^+$ modulo $2^r$.
\end{enumerate}
\end{prop}
\begin{proof}
We prove (i). Let $U = (\mathbf{Z}/N\mathbf{Z})^{\times}/\mu_{12}$, where $\mu_{12}$ is the $12$-torsion subgroup of $(\mathbf{Z}/N\mathbf{Z})^{\times}$. We have a commutative diagram whose rows are exact and whose vertical maps are surjective:
\begin{equation}\label{even_modSymb_diagram_shimura_complex_conjugation}
\xymatrix{
\Gamma_1(N) \ar[r]\ar[d]^{\gamma \mapsto \{z_1, \gamma(z_1)\}}  & \Gamma_0(N) \ar[r]^{\begin{pmatrix} a & b \\ c & d \end{pmatrix}  \mapsto d}\ar[d]^{\gamma \mapsto \{z_0, \gamma(z_0)\}}  & (\mathbf{Z}/N\mathbf{Z})^{\times} \ar[d] \ar[r] & 0 \\
H_1(Y_1(N), \mathbf{Z}) \ar[r] & H_1(Y_0(N), \mathbf{Z}) \ar[r]^{\varphi} & U  \ar[r] & 0 }
\end{equation}
where $z_1 \in Y_1(N)$ (resp. $z_0 \in Y_0(N)$) is any fixed point. The complex conjugation acts on $\Gamma_0(N)$ via $\begin{pmatrix} a & b \\ c & d \end{pmatrix} \mapsto \begin{pmatrix} -1 & 0 \\ 0 & 1 \end{pmatrix}^{-1}  \begin{pmatrix} a & b \\ c & d \end{pmatrix}  \begin{pmatrix} -1 & 0 \\ 0 & 1 \end{pmatrix} = \begin{pmatrix} a & -b \\ -c & d \end{pmatrix}$. Thus, the map $\varphi$ is annihilated by $1+c$. The map $\varphi$ is also annihilated by $\tilde{I}$ \cite[II.18]{Mazur_Eisenstein}. By intersection duality, we get an exact sequence
$$0 \rightarrow \Hom\left(U, \mathbf{Z}/p^n\mathbf{Z}\right) \rightarrow H_1(X_0(N), \cusps, \mathbf{Z}/p^n\mathbf{Z}) \xrightarrow{\pi^*} H_1(X_1(N), \cusps, \mathbf{Z}/p^n\mathbf{Z}) \text{ .} $$
The intersection duality is Hecke equivariant and changes the sign for the complex conjugation. This proves the first assertion of (i). The map $\varphi : H_1(Y_0(N), \mathbf{Z}) \rightarrow U$ corresponds by intersection duality to an element $\mathcal{E}^-$ of $H_1(X_0(N), \cusps, U)$, which is a generator of the kernel of $\pi^* : H_1(X_0(N), \cusps, U) \rightarrow H_1(X_1(N), \cusps, U)$. The following general result, essentially due to Merel \cite{Merel_FLT} and Rebolledo \cite{Rebolledo}, allows us to compute $\mathcal{E}^-$ in terms of Manin symbols.
\begin{lem}\label{even_modSymb_duality_lemma_RUI}
Keep the notation of Section \ref{odd_modSymb_section_hida}. Let $A$ be an abelian group in which $3$ is invertible.  Let $f : H_1(X_{\Gamma}, C_{\Gamma}, \mathbf{Z}) \rightarrow A$ be a group homomorphism and $\hat{f} \in H_1(Y_{\Gamma}, A)$ be the element corresponding to $f$ by intersection duality. Let $\mathcal{R}_{\Gamma}$ be the set of equivalence classes in $\Gamma \backslash \PSL_2(\mathbf{Z})$ for the equivalence relation $\Gamma \cdot g \sim \Gamma \cdot g \cdot \sigma$. The image of $\hat{f}$ in $H_1(X_{\Gamma}, A)$ is
$$\sum_{\Gamma \cdot g \in \mathcal{R}_{\Gamma}} f\left(\xi_{\Gamma}(\Gamma\cdot g)\right)\cdot  \xi_{\Gamma}(\Gamma\cdot g) +  \frac{1}{3}\cdot \sum_{g\cdot\Gamma \in \Gamma \backslash \PSL_2(\mathbf{Z})} \left(2\cdot f\left(\xi_{\Gamma}(\Gamma\cdot g\cdot \tau)\right)+  f\left(\xi_{\Gamma}(\Gamma\cdot g\cdot \tau^2)\right)\right) \cdot \xi_{\Gamma}(\Gamma\cdot g)\text{ .}$$
\end{lem}
\begin{proof}
We follow the notation of \cite[Section 1]{Merel_FLT}. Note that Merel assumes that $\begin{pmatrix} -1 & 0 \\ 0 & -1 \end{pmatrix} \in \Gamma$, but he uses the coset $\Gamma \backslash \SL_2(\mathbf{Z})$. Since we have $\Gamma \backslash \SL_2(\mathbf{Z}) = \Gamma \backslash \PSL_2(\mathbf{Z})$, this assumption of $\Gamma$ is not important. Let $\mathfrak{H}$ be the upper-half plane and $\pi : \mathfrak{H} \cup \mathbf{P}^1(\mathbf{Q}) \rightarrow X_{\Gamma}$ be the canonical surjection. Let $\rho = e^{\frac{2\pi i}{3}}$ and $\delta$ be the geodesic path between $i$ and $\rho$. Let $R = \pi\left(\SL_2(\mathbf{Z})\cdot \rho\right)$ and $I = \pi\left(\SL_2(\mathbf{Z})\cdot i\right)$. These sets are disjoints. If $\Gamma \cdot g \in \Gamma \backslash \PSL_2(\mathbf{Z})$, let $\xi_{\Gamma}'(\Gamma\cdot g)$ be the class of $\pi(g\cdot \delta)$ in $H_1(Y_{\Gamma}, R\cup I, \mathbf{Z})$. Let $f' : H_1(X_{\Gamma} - (R \cup I), \cusps, \mathbf{Z}) \rightarrow A$ be the composition $f$ with the canonical map $H_1(X_{\Gamma} - (R \cup I), \cusps, \mathbf{Z}) \rightarrow H_1(X_{\Gamma}, \cusps, \mathbf{Z})$. By intersection duality, $f'$ corresponds to an element $\hat{f}' \in H_1(Y_{\Gamma}, R\cup I, A)$. By \cite[Proposition 1]{Merel_FLT}, we have
\begin{equation}\label{even_modSymb_formula_duality_RUI}
\hat{f}' = \sum_{g\cdot\Gamma \in \Gamma \backslash \PSL_2(\mathbf{Z})} f\left(\xi_{\Gamma}(\Gamma\cdot g)\right) \cdot \xi_{\Gamma}'(\Gamma\cdot g) \text{ .}
\end{equation}
We have $\hat{f}' \in H_1(Y_{\Gamma}, A)$ since $f'$ factors through $H_1(X_{\Gamma}, \cusps, \mathbf{Z})$. Consider the image $\hat{f}''$ of $\hat{f}'$ in $H_1(X_{\Gamma}, R\cup I, A)$. We have $\hat{f}'' \in H_1(X_{\Gamma}, A)$. Lemma \ref{even_modSymb_duality_lemma_RUI} follows from the following result. This is a slight generalization of \cite[Th\'eor\`eme 2.2]{Rebolledo} (Rebolledo's result assumes that $6$ is invertible in $A$).
\begin{lem}\label{even_modSymb_duality_lemma_RUI_2}
Let $$x =\sum_{\Gamma\cdot g \in \Gamma\backslash \PSL_2(\mathbf{Z})} \lambda_{\Gamma\cdot g} \cdot \xi_{\Gamma}'(\Gamma\cdot g) \in H_1(X_{\Gamma}, R\cup I, A) \text{ .}$$
Assume that $x \in H_1(X_{\Gamma}, A)$. We have, in $H_1(X_{\Gamma}, \cusps, A)$:
$$x = \sum_{\Gamma\cdot g \in \mathcal{R}_{\Gamma}} \lambda_{\Gamma\cdot g} \cdot  \xi_{\Gamma}(\Gamma\cdot g)  + \frac{1}{3}\cdot \sum_{\Gamma\cdot g \in \Gamma \backslash \PSL_2(\mathbf{Z})} (2\cdot \lambda_{\Gamma\cdot g\cdot \tau} + \lambda_{\Gamma\cdot g\cdot \tau^2})\cdot \xi_{\Gamma}(\Gamma\cdot g) \text{ .}$$
\end{lem}
\begin{proof} For simplicity, we write $\lambda_g$ for $\lambda_{\Gamma\cdot g}$.
We have in $H_1(X_{\Gamma}, \cusps \cup R \cup I, A)$:
\begin{align*}
x&=\sum_{\Gamma\cdot g \in \Gamma\backslash \PSL_2(\mathbf{Z})} \lambda_{g} \cdot \{g(i), g(\rho)\} \\&
=\sum_{\Gamma\cdot g \in \Gamma\backslash \PSL_2(\mathbf{Z})} \lambda_{g} \cdot \{g(i), g(\infty)\}  - \sum_{\Gamma\cdot g \in \Gamma\backslash \PSL_2(\mathbf{Z})} \lambda_{g} \cdot \{g(\rho), g(\infty)\} \\&
= \sum_{\Gamma\cdot g \in \mathcal{R}_{\Gamma}} \left( \lambda_{g} \cdot \{g(i), g(\infty)\} + \lambda_{g\cdot \sigma} \cdot \{g\cdot \sigma(i), g\cdot\sigma(\infty)\} \right) \\&-\frac{1}{3}\cdot \sum_{\Gamma\cdot g \in \Gamma\backslash \PSL_2(\mathbf{Z})} \left( \lambda_g \cdot \{g(\rho), g(\infty)\} +\lambda_{g\cdot \tau} \cdot \{g\cdot \tau(\rho), g\cdot \tau(\infty) \}+\lambda_{g\cdot \tau^2} \cdot \{g\cdot \tau^2(\rho), g\cdot \tau^2(\infty) \} \right)\\&
= \sum_{\Gamma\cdot g \in \mathcal{R}_{\Gamma}} \left( \lambda_{g} \cdot \{g(i), g(\infty)\} + \lambda_{g\cdot \sigma} \cdot \{g\cdot \sigma(i), g\cdot\sigma(\infty)\} \right) \\&-\frac{1}{3}\cdot \sum_{\Gamma\cdot g \in \Gamma\backslash \PSL_2(\mathbf{Z})} \left( \lambda_g \cdot \{g(\rho), g(\infty)\} +\lambda_{g\cdot \tau} \cdot \{g\cdot \tau(\rho), g\cdot \tau(\infty) \}+\lambda_{g\cdot \tau^2} \cdot \{g\cdot \tau^2(\rho), g\cdot \tau^2(\infty) \} \right) \text{ .}
\end{align*}
Note that $\sigma(i) = i$, $\sigma(\infty) = 0$, $\tau(\rho)=\rho$ and $\tau(\infty)=0$. Since the boundary of $x$ is zero, we have $\lambda_{g \cdot \sigma} = -\lambda_{g}$ and $\lambda_{g} = -\lambda_{g}=-\lambda_{g\cdot \tau} -  \lambda_{g\cdot \tau^2}$ \cite[Th\'eor\`eme 3]{Merel_FLT}.
Thus, we have in $H_1(X_{\Gamma}, \cusps \cup R \cup I, A)$:
\begin{align*}
x&=\sum_{\Gamma\cdot g \in \mathcal{R}_{\Gamma}} \lambda_{g} \cdot\left( \{g(i), g(\infty)\} -  \{g(i), g(0)\} \right) \\&-\frac{1}{3}\cdot \sum_{\Gamma\cdot g \in \Gamma\backslash \PSL_2(\mathbf{Z})} \left( -(\lambda_{g\cdot \tau} +\lambda_{g\cdot \tau^2}) \cdot \{g(\rho), g(\infty)\} +\lambda_{g\cdot \tau} \cdot \{g(\rho), g(0) \}+\lambda_{g\cdot \tau^2} \cdot \{g(\rho), g\cdot \tau(0) \} \right) \\&
=\sum_{\Gamma\cdot g \in \mathcal{R}_{\Gamma}} \lambda_{g} \cdot \{g(0), g(\infty)\} \\&-\frac{1}{3}\cdot \sum_{\Gamma\cdot g \in \Gamma\backslash \PSL_2(\mathbf{Z})} \lambda_{g\cdot \tau} \cdot \{g(\infty), g(0) \}+\lambda_{g\cdot \tau^2} \cdot \{g(\infty), g\cdot \tau(0) \}  \text{ .}
\end{align*}
We have: $$\{g(\infty), g\cdot \tau(0)\} = \{g(\infty), g\cdot \tau(\infty)\}+ \{g\cdot \tau(\infty), g\cdot \tau(0)\} = -\{g(0), g(\infty)\}- \{g\cdot \tau(0),g\cdot \tau(\infty) \} \text{ .}$$ Thus, we have:
\begin{align*}
x&=\sum_{\Gamma\cdot g \in \mathcal{R}_{\Gamma}} \lambda_{g} \cdot \{g(0), g(\infty)\} \\&+\frac{1}{3}\cdot \sum_{\Gamma\cdot g \in \Gamma\backslash \PSL_2(\mathbf{Z})} \lambda_{g\cdot \tau} \cdot \{g(0), g(\infty)\} \\& +\frac{1}{3}\cdot \sum_{\Gamma\cdot g \in \Gamma\backslash \PSL_2(\mathbf{Z})}   \lambda_{g\cdot \tau^2} \cdot \{g(0), g(\infty)\}  \\& +\frac{1}{3}\cdot \sum_{\Gamma\cdot g \in \Gamma\backslash \PSL_2(\mathbf{Z})}   \lambda_{g\cdot \tau} \cdot \{g(0), g(\infty)\} \text{ .}
\end{align*}
This concludes the proof of Lemma \ref{even_modSymb_duality_lemma_RUI_2}.
\end{proof}
This concludes the proof of Lemma \ref{even_modSymb_duality_lemma_RUI}.
\end{proof}
We identify $H_1(X_0(N), \cusps, U)$ with $H_1(X_0(N), \cusps, \mathbf{Z}) \otimes_{\mathbf{Z}} U$. The homomorphism $\varphi : H_1(Y_0(N), \mathbf{Z}) \rightarrow U$ induces a map $\varphi' : H_1(X_0(N), \mathbf{Z}) \rightarrow U$. It is given by $$\varphi'\left(\xi_{\Gamma_0(N)}\left(\Gamma_0(N) \cdot \begin{pmatrix} a & b \\ c & d \end{pmatrix} \right)\right) = \overline{c}\cdot \overline{d}^{-1} \text{ ,}$$
where $\gcd(c,N)=\gcd(d,N)=1$ and if $x \in \mathbf{Z}$ is prime to $N$, then $\overline{x}$ is the image of $x$ in $U$. By Lemma \ref{even_modSymb_duality_lemma_RUI}, we have in $H_1(X_0(N), \cusps, \mathbf{Z}) \otimes_{\mathbf{Z}} U$:
\begin{align*}
\mathcal{E}_- &= \sum_{x \in \mathcal{R} \atop x \not\sim \infty} \xi_{\Gamma_0(N)}(x) \otimes  \overline{x}+ \frac{1}{3}\cdot \sum_{x \in \mathbf{P}^1(\mathbf{Z}/N\mathbf{Z}) \atop x \neq 0, \infty} \xi_{\Gamma_0(N)}(x) \otimes \overline{\frac{-1}{x\cdot (x+1)}} \\&
= \frac{1}{3}\cdot \sum_{x \in \mathcal{R} \atop x \not\sim \infty} \xi_{\Gamma_0(N)}(x) \otimes  \overline{\frac{x-1}{x+1}}\text{ .}
\end{align*}
This concludes the proof of point (i).

Point (ii) is an immediate consequence of point (i). We now prove point (iii). We have $\varphi\left((1+c)\cdot H_1(Y_0(N), \mathbf{Z})\right) = U^2$, where $U^2$ is the subgroup of squares in $U$. We have an exact sequence:
$$H_1(Y_1(N), \mathbf{Z}) \rightarrow H_1(Y_0(N), \mathbf{Z})/(1+c)\cdot H_1(Y_0(N), \mathbf{Z}) \rightarrow U/U^2 \rightarrow 0 \text{ .}$$
Let $H_1(Y_0(N), \mathbf{Z})^-$ (resp. $H_1(Y_1(N), \mathbf{Z})^-$) be the largest torsion-free quotient of $H_1(Y_0(N), \mathbf{Z})$ (resp. $H_1(Y_1(N), \mathbf{Z})$) annihilated by $1+c$.
\begin{lem}\label{even_modSymb_acyclicity_Y0}
The kernel of the map $H_1(Y_0(N), \mathbf{Z}) \rightarrow H_1(Y_0(N), \mathbf{Z})^-$ is $(1+c)\cdot H_1(Y_0(N), \mathbf{Z})$ where $c$ is the complex conjugation.
\end{lem}
\begin{proof}
The kernel of the map $H_1(X_0(N), \mathbf{Z}) \rightarrow H_1(X_0(N), \mathbf{Z})^-$ is $(1+c)\cdot H_1(Y_0(N), \mathbf{Z})$ by \cite[Proposition 5]{Merel_accouplement}. We have an exact sequence 
$$0 \rightarrow H_1(X_0(N), \mathbf{Z}) \rightarrow H_1(X_0(N), \cusps, \mathbf{Z}) \rightarrow \mathbf{Z} \rightarrow 0 \text{ .}$$ 
Since the cusps are fixed  by the complex conjugation, this gives on the plus subspaces an exact sequence:
$$0 \rightarrow H_1(X_0(N), \mathbf{Z})_+ \rightarrow H_1(X_0(N), \cusps, \mathbf{Z})_+\rightarrow \mathbf{Z} \rightarrow 0 \text{ .}$$ 
By intersection duality, we get a commutative diagram whose rows are exact:
$$\xymatrix{
0 \ar[r] & \mathbf{Z}\ar[r]\ar[d] &H^1(Y_0(N), \mathbf{Z})\ar[r]\ar[d] &H^1(X_0(N),  \mathbf{Z}) \ar[r]\ar[d]&0 \\
0 \ar[r] &\mathbf{Z} \ar[r] & H^1(Y_0(N), \mathbf{Z})^- \ar[r] & H^1(X_0(N),  \mathbf{Z})^- \ar[r]&  0 
} \text{ .}$$
By the snake lemma, the kernel of the map $H_1(Y_0(N), \mathbf{Z}) \rightarrow H_1(Y_0(N), \mathbf{Z})^-$ is $(1+c) \cdot H_1(Y_0(N), \mathbf{Z})$.
\end{proof}

By Lemma \ref{even_modSymb_acyclicity_Y0}, we have an exact sequence:
$$H_1(Y_1(N), \mathbf{Z})^- \rightarrow H_1(Y_0(N), \mathbf{Z})^- \rightarrow U/U^2 \rightarrow 0 \text{ .}$$
By intersection duality, we have an exact sequence
$$0 \rightarrow \Hom(U/U^2, \mathbf{Z}/2^n\mathbf{Z}) \rightarrow M_+/2^n\cdot M_+ \rightarrow H_1(X_1(N), \cusps, \mathbf{Z}_2)_+  \otimes_{\mathbf{Z}_2}\mathbf{Z}/2^n\mathbf{Z}   \text{ .}$$
Note that $H_1(X_1(N), \cusps, \mathbf{Z}_2)_+  \otimes_{\mathbf{Z}_2}\mathbf{Z}/2^n\mathbf{Z}$ is a subgroup of $H_1(X_1(N), \cusps, \mathbf{Z}_2)  \otimes_{\mathbf{Z}_2}\mathbf{Z}/2^n\mathbf{Z}$, which we identify with $H_1(X_1(N), \cusps, \mathbf{Z}/2^n\mathbf{Z})$. The map $\varphi : H_1(Y_0(N), \mathbf{Z}) \rightarrow U$ is annihilated by the Eisenstein ideal. The map $H_1(Y_0(N), \mathbf{Z})^- \rightarrow U/U^2$ is also annihilated by the Eisenstein ideal. Thus, the image of $\Hom(U/U^2, \mathbf{Z}/2^n\mathbf{Z})$ in $M_+/2^n\cdot M_+$ has order $2$ and is Eisenstein, so is generated by $2^{n-1}\cdot \tilde{m}_0^+$ modulo $2^n$. 
\end{proof}

\begin{rem}\label{even_modSymb_intersection_shimura_cuspidal_rk}
Proposition \ref{even_modSymb_injectivity_trace} (iii) is another way to express that the Shimura subgroup and the cuspidal subgroup of $J_0(N)$ intersect at a point of order $2$ \cite[Proposition II.11.11]{Mazur_Eisenstein}.
\end{rem}

\subsection{The method used to compute $m_1^+$ }\label{even_modSymbidea}
In this section, we give the idea that lead us to the proof of Theorems \ref{thm_Introduction_w_1^+>3} and \ref{thm_Introduction_w_1^+_3}. This section is not necessary for the proof, and is mainly for the convenience of the reader. 

We begin by describing the first higher Eisenstein element in the space of modular forms. Following \cite{Emerton_supersingular}, let $\mathcal{N}$ be the $\mathbf{Z}$-module of weight $2$ modular forms $f$ of level $\Gamma_0(N)$ and weight $2$ for which $a_n(f)\in \mathbf{Z}$ if $n\geq 1$ and $a_0(f) \in \mathbf{Q}$, where $\sum_{n \geq 0} a_n(f) \cdot q^n$ is the $q$-expansion of $f$ at the cusp $\infty$. We let $\mathcal{M} = \mathcal{N} \otimes_{\mathbf{Z}} \mathbf{Z}_p$. By \cite[Corollary II.$16.3$]{Mazur_Eisenstein} and \cite[Theorem $0.5$ and Proposition $1.9$]{Emerton_supersingular}, the $\tilde{\mathbf{T}}$-module $\mathcal{M} \otimes_{\tilde{\mathbb{T}}} \tilde{\mathbf{T}}$ is free of rank one. By \cite[Proposition $1.3$ (i)]{Emerton_supersingular}, the pairing $\mathcal{M} \times \tilde{\mathbb{T}} \rightarrow \mathbf{Q}_p$
given by $(f,T) \mapsto a_1(T(f))$
takes values in $\mathbf{Z}_p$ and induces a canonical $\tilde{\mathbb{T}}$-equivariant isomorphism $\mathcal{M} \xrightarrow{\sim} \Hom_{\mathbf{Z}_p}(\tilde{\mathbb{T}}, \mathbf{Z}_p)$. 
Let $T_0 \in \tilde{\mathbb{T}}$ be such that 
\begin{equation}\label{comparison_relation_a_1_a_0_T_0}
a_1(T_0(f)) = \frac{24}{\nu}\cdot a_0(f)
\end{equation}
for all $f \in \mathcal{N}$. We have $\Ker(\tilde{\mathbb{T}} \rightarrow \mathbb{T}) = \mathbf{Z}_p \cdot T_0$. By \cite[Proposition $1.8$ (ii)]{Emerton_supersingular}, we have 
\begin{equation}\label{comparison_T_0-n in I}
T_0 - \frac{N-1}{\nu}\in \tilde{I} \text{ .}
\end{equation}
This is coherent with the choice of section \ref{Formalism_special_case_modular_forms}.

Let $\tilde{f}_0 \in \mathcal{M}$ be the modular form whose $q$-expansion at the cusp $\infty$ is
$$\frac{N-1}{24}+\sum_{n\geq 1} \left(\sum_{d\mid n \atop \gcd(d,N)=1} d\right)\cdot q^n \text{ .}$$ Let $f_0$, ..., $f_{n(t,p)}$ be the higher Eisenstein elements of $\mathcal{M}/p^t\cdot \mathcal{M}$. 
We normalize $f_0$ to be the image of $\tilde{f}_0$ in $\mathcal{M}/p^t\cdot \mathcal{M}$. The image of $f_1$ in $(\mathcal{M}/p^t\cdot \mathcal{M})/\mathbf{Z}\cdot f_0$ is uniquely determined.

\begin{prop}\label{even_modSymbhigher_modular_form}
There is an element $f_1'$ in $\mathcal{M}/p^t\cdot \mathcal{M}$ such that for all prime number $\ell$ not dividing $N$ we have $$a_{\ell}(f_1) = \frac{\ell-1}{2}\cdot \log(\ell)$$ (with the usual interpretation if $\ell=p=2$, \cf chapter \ref{Section_introduction}). The images of $f_1'$ and $f_1$ in $\left( \mathcal{M}/p^t\cdot \mathcal{M}\right)/\mathbf{Z}\cdot f_0$ coincide.
\end{prop}
\begin{proof}
Let $\phi : \tilde{\mathbb{T}} \rightarrow \mathbf{Z}_p$ be the $\mathbf{Z}_p$-linear homomorphism corresponding to $\tilde{f}_0$, given by
$$\phi(T_n) = \sum_{d \mid n \atop \gcd(d,N)=1}^n d$$
and 
$$\phi(T_0)=\frac{N-1}{\nu} \text{.}$$

If $T \in \tilde{\mathbb{T}}$, we have by definition $T-\phi(T) \in \tilde{I}$. The winding homomorphism of Mazur yields a map $\alpha : \tilde{I}/\tilde{I}^2 \rightarrow \mathbf{Z}/p^t\mathbf{Z}$
sending $T_{\ell}-\phi(T_{\ell})$ to $\frac{\ell-1}{2}\cdot \log(\ell)$ for all prime $\ell$ not dividing $N$. 
The map $\psi : \tilde{\mathbb{T}} \rightarrow \mathbf{Z}/p^t\mathbf{Z}$
given by $\psi(T) = \alpha(T-\phi(T))$ is a group morphism, so defines a modular form $f_1'$ with $a_n(f_1') = \psi(T_n)$ for all integers $n \geq 0$. In particular, we have $a_{\ell}(f_1') = \frac{\ell-1}{2} \cdot \log(\ell)$. It is straightforward to check that $(T_{\ell}-\ell-1)(f_1') = \frac{\ell-1}{2} \cdot \log(\ell) \cdot f_0$ for all prime $\ell$ not dividing $N$. This concludes the proof of Proposition \ref{even_modSymbhigher_modular_form}.
\end{proof}
\begin{rem}
Theorem \ref{thm_Introduction_constant_coeff} (which will be proved in chapter \ref{Section_comparison}) shows that we have
$$a_0(f_1') = a_0(f_1) =  \frac{1}{6}\cdot \sum_{k=1}^{\frac{N-1}{2}} k \cdot \log(k) \text{.}$$
\end{rem}

The key fact is that $f_1'$ is related to an Eisenstein series of level $\Gamma_1(N)$ and weight $2$. If $\epsilon : (\mathbf{Z}/N\mathbf{Z})^{\times} \rightarrow \mathbf{C}^{\times}$ is an even character, there is an Eisenstein series $E_{\epsilon, 1} \in M_2(\Gamma_1(N), \epsilon)$ whose $q$-expansion at the cusp $\infty$ is:
$$2\cdot \sum_{n\geq 1 } \left( \sum_{d\mid n} \epsilon(d)\cdot \frac{n}{d} \right)q^n $$
(\cf for instance \cite[Theorem $4.6.2$]{Diamond_Shurman}).
Similarly, there is an Eisenstein series $E_{1, \epsilon} \in M_2(\Gamma_1(N), \epsilon)$ whose $q$-expansion at the cusp $\infty$ is:
$$L(-1, \epsilon) + 2\cdot \sum_{n\geq 1 } \left( \sum_{d\mid n} \epsilon(d)\cdot d \right)q^n \text{.}$$
Let $\epsilon : (\mathbf{Z}/N\mathbf{Z})^{\times} \rightarrow \mathbf{C}^{\times}$
be the even character defined by $$\epsilon(x) = e^{\frac{2i\pi \cdot \log(x)}{p^{t}}} \text{.}$$ We now define
\begin{equation}\label{even_modSymb_Eisenstein_series_moyenne}
E = \frac{E_{1, \epsilon}-\tau(\epsilon)\cdot E_{\epsilon^{-1},1}}{4} 
\end{equation}
where $\tau(\epsilon)  = \sum_{a=1}^{N-1} \epsilon(a) \cdot e^{\frac{2\cdot i \cdot \pi \cdot a}{N}}$.
The key observation is that the $q$-expansion of $E$ is $\tilde{f}_0 + \pi \cdot f_1' + O(\pi^2)$
where $\pi = 1-\zeta_{p^t}$ is a uniformizer of $\mathbf{Z}_p^{\text{unr}}[\zeta_{p^t}]$. Thus, there should be a way to compute $m_1^+$ using \textit{Eisenstein elements} in $H_1(X_1(N), \cusps, \mathbf{C})_+$. Formulas for all the Eisenstein elements of level $\Gamma(N)$ have recently been given in \cite{Debargha_Merel}, where $\Gamma(N) \subset \SL_2(\mathbf{Z})$ is the subgroup of matrices congruent to the identity modulo $N$.

\subsection{Eisenstein elements of level in $H_1(X_1(N), \cusps, \mathbf{C})$}
Recall that we have defined $\B_1: \mathbf{R} \rightarrow \mathbf{R}$ by 
$$\B_1(x) = x-\lfloor x \rfloor-\frac{1}{2}$$
if $x \not\in\mathbf{Z}$ and $\B_1(x)=0$ else, where $\lfloor x \rfloor$ is the integer part of $x$. Following \cite[Section 2.2]{Merel_accouplement}, let $D_2: \mathbf{R} \rightarrow \mathbf{R}$ be given by $D_2(x) = 2\cdot( \B_1(x) - \B_1(x + \frac{1}{2}))$. It is a periodic function with period $1$ such that $D_2(0) = D_2\left(\frac{1}{2}\right)=0$, $D_2(x) = -1$ if $x \in ]0, \frac{1}{2}[$ and $D_2(x) = 1$ if $x \in ]\frac{1}{2}, 1[$. 

For the convenience of the reader, we state the main result of \cite{Debargha_Merel}. As in section \ref{odd_modSymb_section_hida}, let $$\xi_{\Gamma(N)} : \mathbf{Z}[\Gamma(N) \backslash \PSL_2(\mathbf{Z})] \rightarrow H_1(X(N), \cusps, \mathbf{Z})$$ be the Manin surjective map. Since $N$ is odd, we have canonical identifications $$\SL_2(\mathbf{Z}/N\mathbf{Z})/\pm1 \simeq \Gamma(N) \backslash \PSL_2(\mathbf{Z}) \simeq \pm \Gamma(2N) \backslash \Gamma(2) \text{ .}$$

Let $F : (\mathbf{Z}/N\mathbf{Z})^2 \rightarrow \mathbf{C}$ be defined by
$$F(x,y) = \sum_{(s_1,s_2) \in (\mathbf{Z}/2N\mathbf{Z})^2} e^{\frac{2i \pi \cdot ( s_1(x'+y') +s_2(x'-y'))}{2N}} \cdot \B_1\left(\frac{s_1}{2N}\right) \B_1\left(\frac{s_2}{2N}\right) $$
where $(x', y') \in (\mathbf{Z}/2N\mathbf{Z})^{2}$ is a lift of $(x,y)$ such that $x'+y' \equiv 1 \text{ (modulo }2\text{)}$.
We have $F(x,y)=F(-x,-y)$.

If $P \in (\mathbf{Z}/N\mathbf{Z})^2$, let 
$$\tilde{\mathcal{E}}_P = \sum_{\gamma \in \SL_2(\mathbf{Z}/N\mathbf{Z})/\pm1} F(\gamma^{-1}P) \cdot \xi_{\Gamma(N)}(\gamma) \in H_1(X(N), \cusps, \mathbf{C}) \text{ .}$$
We have $\tilde{\mathcal{E}}_P = \tilde{\mathcal{E}}_{-P}$. Let $\mathcal{E}_P$ be the image of $\tilde{\mathcal{E}}_P$ to $H_1(X_1(N), \cusps, \mathbf{C})$.The following result is an easy consequence of the work of Debargha and Merel (\cite{Debargha_Merel}).

\begin{thm}\label{even_modSymb_DM_thm}
Let $\ell$ be a prime not dividing $N$. For all $x \in (\mathbf{Z}/N\mathbf{Z})^{\times}$, we have:
$$T_{\ell}(\mathcal{E}_{(x,0)}) = \ell \cdot \mathcal{E}_{(x,0)} + \mathcal{E}_{(\ell^{-1}x, 0)} $$
and
$$T_{\ell}(\mathcal{E}_{(0,x)}) = \ell\cdot \mathcal{E}_{(0,\ell x)} + \mathcal{E}_{(0,x)} $$
where $T_{\ell}$ is the $\ell$th Hecke operator.
\end{thm}
\begin{proof}
Recall the notation of section \ref{odd_modSymb_section_hida} for the cusps of $X_1(N)$. For all prime $\ell$ not dividing $N$, the Hecke operator $T_{\ell} - \ell-\langle \ell \rangle$ (resp. $T_{\ell} - \ell\langle \ell \rangle - 1$) annihilates $C_{\Gamma_1(N)}^{\infty}$ (resp. $C_{\Gamma_1(N)}^{0}$), where $\langle \ell \rangle$ is the $\ell$th diamond operator. By definition we have $\langle \ell \rangle \cdot [x]_{\Gamma_1(N)}^{\infty} = [\ell\cdot x]_{\Gamma_1(N)}^{\infty}$ and $\langle \ell \rangle \cdot [x]_{\Gamma_1(N)}^{0} = [\ell \cdot x]_{\Gamma_1(N)}^{0}$.

Let $E_N$ be the $\mathbf{C}$-vector space generated by the elements $\mathcal{E}_{(x,0)}$ and $\mathcal{E}_{(0,x)}$ for all $x \in (\mathbf{Z}/N\mathbf{Z})^{\times}/\pm 1$. Let $\partial_N : H_1(X_1(N), \cusps, \mathbf{C}) \rightarrow \mathbf{C}[C_{\Gamma_1(N)}]^0$ be the boundary map. By \cite[Theorem 1]{Debargha_Merel}, the restriction of $\partial_N$ to $E_N$ gives a Hecke-equivariant isomorphism $E_N \xrightarrow{\sim}  \mathbf{C}[C_{\Gamma_1(N)}]^0$. By \cite[Theorem 13]{Debargha_Merel}, for all $x \in (\mathbf{Z}/N\mathbf{Z})^{\times}/\pm 1$ we have
$$ \partial_N(\mathcal{E}_{(x,0)}) = 2N\cdot \left(\sum_{\mu \in (\mathbf{Z}/N\mathbf{Z})^{\times}} (F(\mu \cdot (1,0))+\frac{1}{4})\cdot [(\mu\cdot x)^{-1}]_{\Gamma_1(N)}^{\infty} \right) - \frac{N}{2} \cdot \left(\sum_{c \in C_{\Gamma_1(N)}} [c] \right) $$
and
$$ \partial_N(\mathcal{E}_{(0,x)}) = 2N\cdot \left(\sum_{\mu \in (\mathbf{Z}/N\mathbf{Z})^{\times}} (F(\mu \cdot (1,0))+\frac{1}{4})\cdot [\mu \cdot x]_{\Gamma_1(N)}^{0} \right) - \frac{N}{2} \cdot \left(\sum_{c \in C_{\Gamma_1(N)}} [c] \right) \text{ .}$$
Thus, for all prime $\ell$ not dividing $N$ we have
$$\partial_N( \langle \ell \rangle \mathcal{E}_{(x,0)}) =  \langle \ell \rangle \partial_N(\mathcal{E}_{(x,0)}) = \partial_N(\mathcal{E}_{(\ell^{-1}x,0)}) \text{ .}$$
By injectivity of $\partial_N$, we have $\langle \ell \rangle \mathcal{E}_{(x,0)} = \mathcal{E}_{(\ell^{-1}x,0)}$. Similarly, we have $(T_{\ell}- \ell - \langle \ell \rangle)(\mathcal{E}_{(x,0)}) = 0$ since it is true after applying $\partial_N$. This proves the first relation of Theorem \ref{even_modSymb_DM_thm}. The second relation is proved in a similar way.
\end{proof}

The counterparts of the modular forms $E_{\epsilon,1}$ and $E_{1, \epsilon}$ of section \ref{even_modSymbidea} are the following modular symbols
$$\mathcal{E}_{\infty} = \sum_{x \in (\mathbf{Z}/N\mathbf{Z})^{\times}/\pm 1 } [x] \cdot \mathcal{E}_{(x,0)} \in H_1(X_1(N), \cusps, \mathbf{C}[(\mathbf{Z}/N\mathbf{Z})^{\times}/\pm 1])$$
and
$$\mathcal{E}_0 = \sum_{x \in (\mathbf{Z}/N\mathbf{Z})^{\times}/\pm 1 } [x]^{-1} \cdot \mathcal{E}_{(0,x)} \in H_1(X_1(N), \cusps, \mathbf{C}[(\mathbf{Z}/N\mathbf{Z})^{\times}/\pm 1]) \text{ .}$$

As an immediate application of Theorem \ref{even_modSymb_DM_thm}, for all prime $\ell$ not dividing $N$ we have
\begin{equation}\label{even_modSymb_Eisenstein_diamond_1}
(T_{\ell}-\ell-[\ell])(\mathcal{E}_{\infty})=0
\end{equation}
and
\begin{equation}\label{even_modSymb_Eisenstein_diamond_2}
(T_{\ell}-\ell\cdot [\ell]-1)(\mathcal{E}_0)=0 \text{ .}
\end{equation}

Let $$\mathcal{E}' =  \pi^*(\tilde{m}_0^+) \in H_1(X_1(N), \cusps, \mathbf{Z}_p) \text{ .}$$ Recall that there is a canonical bijection $\Gamma_1(N) \backslash \PSL_2(\mathbf{Z}) \xrightarrow{\sim} (\mathbf{Z}/N\mathbf{Z})^{2} \backslash \{(0,0)\}/\pm 1$ (\cf Section \ref{odd_modSymb_section_hida}), and that we denote by $[c,d]$ the class of $(c,d)$ in $(\mathbf{Z}/N\mathbf{Z})^{2} \backslash \{(0,0)\}/\pm 1$ for any $(c,d) \in (\mathbf{Z}/N\mathbf{Z})^{2}  \backslash \{(0,0)\}$. We sometimes abusively view $[c,d]$ as an element of $\Gamma_1(N) \backslash \PSL_2(\mathbf{Z})$.

\begin{lem}\label{even_modSymb_reecriture_m_0^+}
We have 
\begin{align*} 6\cdot \mathcal{E}' &= \sum_{[c,d] \in (\mathbf{Z}/N\mathbf{Z})^{2} \backslash \{(0,0)\}/\pm 1} \left( \sum_{(s_1, s_2) \in (\mathbf{Z}/2N\mathbf{Z})^2 \atop s_1(d-c)+s_2(d+c) \equiv 0 \text{ (modulo }N\text{)}} (-1)^{s_1+s_2} \B_1\left(\frac{s_1}{2N}\right)\B_1\left(\frac{s_2}{2N}\right) \right) \cdot \xi_{\Gamma_1(N)}([c,d]) \\&= - \sum_{[c,d] \in (\mathbf{Z}/N\mathbf{Z})^{2} \backslash \{(0,0)\}/\pm 1} \left( \sum_{(s_1, s_2) \in (\mathbf{Z}/2N\mathbf{Z})^2 \atop s_1(d-c)+s_2(d+c) \not\equiv 0 \text{ (modulo }N\text{)}} (-1)^{s_1+s_2} \B_1\left(\frac{s_1}{2N}\right)\B_1\left(\frac{s_2}{2N}\right)\right)\cdot \xi_{\Gamma_1(N)}([c,d]) \text{ .}
\end{align*}
\end{lem}
\begin{proof}
The second equality follows from the first equality and from the identity
$$\sum_{(s_1, s_2) \in (\mathbf{Z}/2N\mathbf{Z})^2 } (-1)^{s_1+s_2} \B_1\left(\frac{s_1}{2N}\right)\B_1\left(\frac{s_2}{2N}\right) = \left(\sum_{s \in (\mathbf{Z}/2N\mathbf{Z})^2} (-1)^s \cdot  \B_1\left(\frac{s}{2N}\right)\right)^2 = 0 \text{ .}$$
Let $(t_1,t_2) \in (\mathbf{Z}/N\mathbf{Z})^2$. Then 
$$\sum_{(s_1,s_2) \in(\mathbf{Z}/N\mathbf{Z})^2 \atop (s_1,s_2) \equiv (t_1, t_2) \text{ (modulo }N\text{)} } (-1)^{s_1+s_2} \B_1\left(\frac{s_1}{2N}\right) \B_1\left(\frac{s_2}{2N}\right) = \frac{1}{4} \cdot D_2\left(\frac{t_1}{N}\right) D_2\left(\frac{t_2}{N}\right) \text{ .} $$
Thus, we have
\begin{align*}
 &\sum_{[c,d] \in (\mathbf{Z}/N\mathbf{Z})^{2} \backslash \{(0,0)\}/\pm 1} \left( \sum_{(s_1, s_2) \in (\mathbf{Z}/2N\mathbf{Z})^2 \atop s_1(d-c)+s_2(d+c) \equiv 0 \text{ (modulo }N\text{)}} (-1)^{s_1+s_2} \B_1\left(\frac{s_1}{2N}\right)\B_1\left(\frac{s_2}{2N}\right) \right) \cdot \xi_{\Gamma_1(N)}([c,d]) \\&= \sum_{[c,d] \in (\mathbf{Z}/N\mathbf{Z})^{2} \backslash \{(0,0)\}/\pm 1} \left(  \sum_{(t_1,t_2)\in (\mathbf{Z}/N\mathbf{Z})^2 \atop (t_1:t_2)=(c+d:c-d) \in \mathbf{P}^1(\mathbf{Z}/N\mathbf{Z})} \frac{1}{4}\cdot D_2\left(\frac{t_1}{N}\right)D_2\left(\frac{t_2}{N}\right) \right) \cdot \xi_{\Gamma_1(N)}([c,d]) \text{ .}
\end{align*}
Note that for all $[u:v] \in \mathbf{P}^1(\mathbf{Z}/N\mathbf{Z}) = \Gamma_0(N) \backslash \SL_2(\mathbf{Z})$, we have
$$\pi^*(\xi_{\Gamma_0(N)}([u:v])) = \sum_{[c,d] \in (\mathbf{Z}/N\mathbf{Z})^{2} \backslash \{(0,0)\}/\pm 1 \atop [c:d] = [u:v]} \xi_{\Gamma_1(N)}([c,d]) \text{ .}$$
The first equality of the lemma then follows from the above computation Theorem \ref{thm_Introduction_w_0^+}.
\end{proof}

Define two maps $G_{\infty}$, $G_0$ : $\Gamma_1(N) \backslash \SL_2(\mathbf{Z}) \rightarrow \mathbf{Z}[\frac{1}{2N}][(\mathbf{Z}/N\mathbf{Z})^{\times}/\pm1]$ as follows. \\ Let $\gamma=\begin{pmatrix}
 a & b \\
 c & d
  \end{pmatrix} \in \PSL_2(\mathbf{Z})$. We define
  $$G_{\infty}(\Gamma_1(N) \cdot \gamma) = \sum_{(s_1, s_2) \in (\mathbf{Z}/2N\mathbf{Z})^2 \atop (d-c)s_1+(d+c)s_2 \not\equiv 0 \text{ (modulo }N\text{)}} (-1)^{s_1+s_2} \B_1\left(\frac{s_1}{2N}\right)\B_1\left(\frac{s_2}{2N}\right) \cdot [(d-c)s_1+(d+c)s_2]^{-1}$$
and
$$G_0(\Gamma_1(N) \cdot \gamma) = \sum_{(s_1, s_2) \in (\mathbf{Z}/2N\mathbf{Z})^2 \atop (d-c)s_1+(d+c)s_2 \equiv 0 \text{ (modulo }N\text{)}} (-1)^{s_1+s_2} \B_1\left(\frac{s_1}{2N}\right)\B_1\left(\frac{s_2}{2N}\right) \cdot [(b-a)s_1+(b+a)s_2] \text{ .}$$
Let
$$\mathcal{E}_{\infty}' = \sum_{\Gamma_1(N) \cdot \gamma \in \Gamma_1(N) \backslash \PSL_2(\mathbf{Z})} G_{\infty}(\Gamma_1(N) \cdot \gamma) \cdot \xi_{\Gamma_1(N)}(\Gamma_1(N) \cdot \gamma) \in H_1(X_1(N), \cusps, \mathbf{Z}[\frac{1}{2N}][(\mathbf{Z}/N\mathbf{Z})^{\times}/\pm 1]) $$
and
$$\mathcal{E}_{0}' = \sum_{\Gamma_1(N) \cdot \gamma \in \Gamma_1(N) \backslash \PSL_2(\mathbf{Z})} G_0(\Gamma_1(N) \cdot \gamma) \cdot \xi_{\Gamma_1(N)}(\Gamma_1(N) \cdot \gamma)  \in H_1(X_1(N), \cusps, \mathbf{Z}[\frac{1}{2N}][(\mathbf{Z}/N\mathbf{Z})^{\times}/\pm 1]) \text{ .}$$

\begin{lem}\label{even_modSymb_Hecke_relation_bis}
For all prime $\ell$ not dividing $N$, we have
$$
(T_{\ell}-\ell-[\ell])(\mathcal{E}_{\infty}')=0
$$
and
$$
(T_{\ell}-\ell\cdot [\ell]-1)(\mathcal{E}_0')=0 \text{ .}
$$
\end{lem}
\begin{proof}
Consider the following two elements of $\mathbf{C}[(\mathbf{Z}/N\mathbf{Z})^{\times}/\pm 1]$: $$\mathcal{G} = \sum_{x \in (\mathbf{Z}/N\mathbf{Z})^{\times}}e^{\frac{2i\pi x}{N}} \cdot [x]$$ and $$\delta = \sum_{x \in (\mathbf{Z}/N\mathbf{Z})^{\times}} [x] \text{ .}$$ 
An easy computation shows that we have, in $H_1(X_1(N), \cusps, \mathbf{C}[(\mathbf{Z}/N\mathbf{Z})^{\times}/\pm 1])$:
\begin{align*}
\frac{2}{N} \cdot \mathcal{E}_{\infty} &=  \sum_{[c,d] \in (\mathbf{Z}/N\mathbf{Z})^2 \backslash \{(0,0)\} / \pm1} \xi_{\Gamma_1(N)}([c,d]) \cdot \sum_{x \in (\mathbf{Z}/N\mathbf{Z})^{\times}\atop (s_1,s_2) \in (\mathbf{Z}/2N\mathbf{Z})^2} (-1)^{s_1+s_2}\cdot \B_1\left(\frac{s_1}{2N}\right)\B_1\left(\frac{s_2}{2N}\right)[2x]\cdot e^{\frac{2i\pi x((d-c)s_1+(d+c)s_2)}{N}}
\\&=
6 \cdot \delta \cdot \mathcal{E}'  +  [2]\cdot \mathcal{G} \cdot \mathcal{E}_{\infty}' \text{ .}
  \end{align*}
  
The last equality follows from Lemma \ref{even_modSymb_reecriture_m_0^+}. By (\ref{even_modSymb_Eisenstein_diamond_1}), the operator $T_{\ell}-\ell-[\ell]$ annihilates $\mathcal{E}_{\infty}$. Furthermore, $T_{\ell}-\ell-1$ annihilates $\mathcal{E}'$ and $[\ell]-1$ annihilates $\delta$, so $T_{\ell}-\ell-[\ell]$ annihilates $\delta \cdot \mathcal{E}'$. Thus, $T_{\ell}-\ell - [\ell]$ annihilates $\mathcal{G} \cdot \mathcal{E}_{\infty}'$. 

\begin{lem}\label{odd_modSymb_Gauss_sum_non_zero}
The element $\mathcal{G}$ is not a zero divisor of $\mathbf{C}[(\mathbf{Z}/N\mathbf{Z})^{\times}/\pm 1]$.
\end{lem}
\begin{proof}
It suffices to prove that if $\alpha : (\mathbf{Z}/N\mathbf{Z})^{\times} \rightarrow \mathbf{C}^{\times}$ is any character such that $\alpha(-1)=1$, then we have $\sum_{x \in (\mathbf{Z}/N\mathbf{Z})^{\times}} e^{\frac{2i \pi x}{N}} \cdot \alpha(x) \neq 0$. This is a well-known fact on Gauss sums.
\end{proof}

By Lemma \ref{odd_modSymb_Gauss_sum_non_zero}, the operator $T_{\ell}-\ell - [\ell]$ annihilates $\mathcal{E}_{\infty}'$, as wanted. The proof that $T_{\ell}-\ell\cdot [\ell]-1$ annihilates $\mathcal{E}_0'$ is similar.
\end{proof}

Let $\mathcal{U} = \frac{\mathcal{E}_0'+\mathcal{E}_{\infty}'}{2} \in H_1(X_1(N), \cusps, \mathbf{Z}[\frac{1}{2N}][(\mathbf{Z}/N\mathbf{Z})^{\times}/\pm 1])$. This is the modular symbol counterpart of (\ref{even_modSymb_Eisenstein_series_moyenne}). Let $J \subset \mathbf{Z}[(\mathbf{Z}/N\mathbf{Z})^{\times}/\pm 1]$ be the augmentation ideal. By Lemma \ref{even_modSymb_reecriture_m_0^+}, we have $\mathcal{U} \in J\cdot H_1(X_1(N), \cusps, \mathbf{Z}[\frac{1}{2N}][(\mathbf{Z}/N\mathbf{Z})^{\times}/\pm 1])$. Lemma \ref{even_modSymb_Hecke_relation_bis} shows that we have
\begin{equation}\label{even_modSymb_crucial_Hecke}
(T_{\ell}-\ell-1)(\mathcal{U})=([\ell]-1)\cdot \frac{\mathcal{E}_{\infty}' + \ell \cdot \mathcal{E}_0'}{2}
\end{equation}
This is the fundamental equality which allow us to compute $m_1^+$. We will study the cases $p\geq 5$, $p=3$ and $p=2$ separately.

\subsection{The case $p \geq 5$}\label{even_modSymb_the_case_p>3}
The following theorem is a generalization of Theorem \ref{thm_Introduction_w_1^+>3} modulo $p^t$.

\begin{thm}\label{thm_Introduction_w_1^+>3_p^t}
Assume that $p\geq 5$.
Let $F_{1,p} : \mathbf{P}^1(\mathbf{Z}/N\mathbf{Z}) \rightarrow \mathbf{Z}/p^t\mathbf{Z}$ be defined as follows. Let $[c:d] \in \mathbf{P}^1(\mathbf{Z}/N\mathbf{Z})$. If $[c:d] \neq [1:1]$, let
\begin{align*}
12 \cdot F_{1,p}([c:d]) &=  \sum_{(s_1,s_2) \in (\mathbf{Z}/2N\mathbf{Z})^2 \atop (d-c)s_1+(d+c)s_2 \equiv 0 \text{ (modulo } N\text{)}} (-1)^{s_1+s_2} \B_1\left(\frac{s_1}{2N}\right)\B_1\left(\frac{s_2}{2N}\right) \cdot \log\left(\frac{s_2}{d-c}\right) \\&
-\sum_{(s_1,s_2) \in (\mathbf{Z}/2N\mathbf{Z})^2 \atop (d-c)s_1+(d+c)s_2 \not\equiv 0 \text{ (modulo } N\text{)}} (-1)^{s_1+s_2} \B_1\left(\frac{s_1}{2N}\right)\B_1\left(\frac{s_2}{2N}\right) \cdot \log((d-c)s_1+(d+c)s_2)) \text{ .}
\end{align*}
This is independent of the choice of $c$ and $d$. Let $F_{1,p}([1:1]) = 0$. We have: 
$$m_1^+ \equiv \sum_{x \in \mathbf{P}^1(\mathbf{Z}/N\mathbf{Z})} F_{1,p}(x) \cdot \xi_{\Gamma_0(N)}(x) \text{ in } \left(M_+/p^t\cdot M_+\right)/\mathbf{Z}\cdot m_0^+$$
\end{thm}
\begin{proof}
Let $\beta : J/J^2 \rightarrow \mathbf{Z}/p^{t}\mathbf{Z}$ be given by $[x]-1 \mapsto \log(x)$ for $x \in (\mathbf{Z}/N\mathbf{Z})^{\times}$. This induces a map $\beta_* : J\cdot H_1(X_1(N), \cusps, \mathbf{Z}[\frac{1}{2N}][(\mathbf{Z}/N\mathbf{Z})^{\times}/\pm 1]) \rightarrow H_1(X_1(N), \cusps, \mathbf{Z}/p^{t}\mathbf{Z})$. 

By (\ref{even_modSymb_crucial_Hecke}) and Lemma \ref{even_modSymb_reecriture_m_0^+}, we have:
\begin{equation}\label{even_modSymb_crucial_Hecke_2}
(T_{\ell}-\ell-1)(\beta_*(\mathcal{U})) \equiv  6 \cdot \frac{\ell-1}{2}\cdot \log(\ell) \cdot \mathcal{E}' \text{ (modulo }p^v\text{)} \text{ .}
\end{equation}
 
Let $F_{1,p}' : \mathbf{P}^1(\mathbf{Z}/N\mathbf{Z}) \rightarrow \mathbf{Z}/p^t\mathbf{Z}$ be defined by
\begin{align*}
12 \cdot F_{1,p}'(x) &= \sum_{(s_1,s_2) \in (\mathbf{Z}/2N\mathbf{Z})^2 \atop (d-c)s_1+(d+c)s_2 \equiv 0 \text{ (modulo } N\text{)}} (-1)^{s_1+s_2} \B_1\left(\frac{s_1}{2N}\right)\B_1\left(\frac{s_2}{2N}\right) \cdot \log((b-a)s_1+(b+a)s_2) \\&
- \sum_{(s_1,s_2) \in (\mathbf{Z}/2N\mathbf{Z})^2 \atop (d-c)s_1+(d+c)s_2 \not\equiv 0 \text{ (modulo } N\text{)}} (-1)^{s_1+s_2} \B_1\left(\frac{s_1}{2N}\right)\B_1\left(\frac{s_2}{2N}\right) \cdot \log((d-c)s_1+(d+c)s_2))
\end{align*}
where as above $a$ and $b$ are such that $\begin{pmatrix} a & b \\ c & d \end{pmatrix} \in \SL_2(\mathbf{Z}/N\mathbf{Z})$. This expression does not depend on the choice of $c$ and $d$ by Lemma \ref{even_modSymb_reecriture_m_0^+}. This also does not depend on the choice of $a$ and $b$. 

For all $[c:d] \in \mathbf{P}^1(\mathbf{Z}/N\mathbf{Z})$, we have $F_{1,p}'([-c:d]) = F_{1,p}'([c:d])$. Thus, we have $\sum_{x \in \mathbf{P}^1(\mathbf{Z}/N\mathbf{Z})} F_{1,p}'(x)\cdot \xi_{\Gamma_0(N)}(x) \in H_1(X_0(N), \cusps, \mathbf{Z}/p^t\mathbf{Z})_+$. The element $\frac{1}{6}\cdot \beta_*(\mathcal{U})$ is the image of $\sum_{x \in \mathbf{P}^1(\mathbf{Z}/N\mathbf{Z})} F_{1,p}'(x)\cdot \xi_{\Gamma_0(N)}(x)$ via the pull-back map $H_1(X_0(N), \cusps, \mathbf{Z}/p^t\mathbf{Z})_+ \rightarrow H_1(X_1(N), \cusps, \mathbf{Z}/p^t\mathbf{Z})_+$. By (\ref{even_modSymb_crucial_Hecke_2}) and Proposition \ref{even_modSymb_injectivity_trace} (ii), for any prime $\ell$ not dividing $N$ we have in $H_1(X_0(N), \cusps, \mathbf{Z}/p^t\mathbf{Z})_+$:
\begin{equation}\label{even_modSymb_identity_p>3_hecke_superieur}
(T_{\ell}-\ell-1)\left(\sum_{x \in \mathbf{P}^1(\mathbf{Z}/N\mathbf{Z})} F_{1,p}'(x)\cdot \xi_{\Gamma_0(N)}(x) \right) = \frac{\ell-1}{2}\cdot \log(\ell) \cdot m_0^+ \text{ .}
\end{equation}

If $\begin{pmatrix}
 a & b \\
 c & d
  \end{pmatrix}  \in \SL_2(\mathbf{Z})$ and $(s_1, s_2) \in (\mathbf{Z}/2N\mathbf{Z})^2$ are such that $(d-c)s_1 + (d+c)s_2 \equiv 0 \text{ (modulo }N\text{)}$ and $d \not\equiv c \text{ (modulo }N\text{)}$, then we have $(b-a)s_1+(b+a)s_2 \equiv \frac{2}{d-c}\cdot s_2  \text{ (modulo }N\text{)}$. By Lemma \ref{even_modSymb_reecriture_m_0^+}, we have:
$$12\cdot \sum_{x \in \mathbf{P}^1(\mathbf{Z}/N\mathbf{Z})} F_{1,p}(x)\cdot \xi_{\Gamma_0(N)}(x) = 12\cdot \sum_{x \in \mathbf{P}^1(\mathbf{Z}/N\mathbf{Z})} F_{1,p}'(x)\cdot \xi_{\Gamma_0(N)}(x) - 6\cdot \log(2)\cdot m_0^+ \text{ .}$$
By (\ref{even_modSymb_identity_p>3_hecke_superieur}), we have:
$$
(T_{\ell}-\ell-1)\left(\sum_{x \in \mathbf{P}^1(\mathbf{Z}/N\mathbf{Z})} F_{1,p}(x)\cdot \xi_{\Gamma_0(N)}(x) \right) = \frac{\ell-1}{2}\cdot \log(\ell) \cdot m_0^+ \text{ .}
$$
This concludes the proof of Theorem \ref{thm_Introduction_w_1^+>3_p^t}. 
 \end{proof}

\subsection{A few identities in $(\mathbf{Z}/N\mathbf{Z})^{\times}$ and in  $\left((\mathbf{Z}/N\mathbf{Z})^{\times}\right)^{\otimes 2}$}\label{even_modSymb_few_identities}
In this section, we establish a few identities which will be useful to determine $m_1^+$ when $p \in \{2,3\}$ and also in chapter \ref{Section_comparison}. We do not impose any restriction on $p$. If $i$ is a non-negative integer, we let $$\mathcal{F}_i = \sum_{k=1}^{\frac{N-1}{2}} \log(k)^i \in \mathbf{Z}/p^v\mathbf{Z} \text{ .}$$
We have $\mathcal{F}_0 = \mathcal{F}_1=0$ if $p>2$ and $2\cdot \mathcal{F}_0 = 4\cdot \mathcal{F}_1=0$ if $p=2$. We have $\mathcal{F}_2 = 0$ if $p>3$, $3\cdot \mathcal{F}_2 = 0$ if $p=3$ and $4\cdot \mathcal{F}_2 = 0$ if $p=2$. If $0 \leq i \leq 2$, we easily check that $\mathcal{F}_i = \sum_{k=\frac{N+1}{2}}^{N-1} \log(k)^i$ (we use the fact that $\frac{N-1}{2}$ is even if $p=2$).

\begin{lem}\label{even_modSymb_computation_square}
We have, in $\mathbf{Z}/p^{v}\mathbf{Z}$:
$$4\cdot \sum_{k=1}^{\frac{N-1}{2}} k \cdot \log(k) = -3\cdot \sum_{k=1}^{N-1} k^2 \cdot \log(k) - \log(2)\cdot \frac{N-1}{6} - \mathcal{F}_1 \text{ .}$$
\end{lem}
\begin{proof}
Let $f : \mathbf{R} \rightarrow \mathbf{R}$ be defined by $f(x) = (x-\lfloor x \rfloor)^2$ and let $g : \mathbf{R} \rightarrow \mathbf{R}$ given by $g(x) = f(2x) - 4f(x)$. If $x \in [0, \frac{1}{2}[$, we have $g(x) = 0$ and if $x \in [\frac{1}{2}, 1[$, we have $g(x) = -4x+1$. We thus have, in $\mathbf{Z}/p^v\mathbf{Z}$:
\begin{align*}
\sum_{k \in (\mathbf{Z}/N\mathbf{Z})^{\times}} g(\frac{k}{N}) \cdot \log(k) &= -4 \cdot \sum_{k=\frac{N+1}{2}}^{N-1} k \cdot \log(k) + \log\left((-1)^{\frac{N-1}{2}} \cdot \left(\frac{N-1}{2}\right)! \right) \\& 
= -4 \cdot \sum_{k=1}^{\frac{N-1}{2}}(N-k) \cdot \log(-k) + \log\left((-1)^{\frac{N-1}{2}} \cdot \left(\frac{N-1}{2}\right)! \right) \\&
=4 \cdot \sum_{k=1}^{\frac{N-1}{2}} k \cdot \log(k) + \mathcal{F}_1 \text{ .}
\end{align*}
On the other hand, we have in $\mathbf{Z}/p^v\mathbf{Z}$:
\begin{align*}
\sum_{k \in (\mathbf{Z}/N\mathbf{Z})^{\times}} g(\frac{k}{N}) \cdot \log(k)  &= \sum_{k \in (\mathbf{Z}/N\mathbf{Z})^{\times}}( f(\frac{2k}{N})- 4\cdot f(\frac{k}{N})) \cdot \log(k) \\&= -3\cdot  \sum_{k \in (\mathbf{Z}/N\mathbf{Z})^{\times}} f(\frac{k}{N}) \cdot \log(k) - \log(2)\cdot \sum_{k \in (\mathbf{Z}/N\mathbf{Z})^{\times}} f(\frac{k}{N})  \\& = -3 \cdot \sum_{k=1}^{N-1}k^2\cdot \log(k) - \log(2) \cdot \frac{N-1}{6} \text{ .}
\end{align*}
\end{proof}

We have the following variant of lemma \ref{even_modSymb_computation_square}.
\begin{lem}\label{even_modSymb_computation_square_2}
We have, in $\mathbf{Z}/p^{v}\mathbf{Z}$:
$$4\cdot \sum_{k=1}^{\frac{N-1}{2}} k \cdot \log(k)^2 =  -3 \cdot \sum_{k=1}^{N-1} k^2 \cdot \log(k)^2 + \log(2)^2 \cdot \frac{N-1}{6} - 2 \cdot \log(2) \cdot \sum_{k=1}^{N-1} k^2\cdot \log(k) + 3\cdot \mathcal{F}_2\text{ .}$$
\end{lem}
\begin{proof}
Let $f,g : \mathbf{R} \rightarrow \mathbf{R}$ be as in the proof of Lemma \ref{even_modSymb_computation_square_2}. We have:
\begin{align*}
\sum_{k \in (\mathbf{Z}/N\mathbf{Z})^{\times}} g(\frac{k}{N}) \cdot \log(k)^2 &= -4\cdot  \sum_{k=\frac{N+1}{2}}^{N-1} k \cdot \log(k)^2 + \mathcal{F}_2  \\&
= -4 \cdot \sum_{k=1}^{\frac{N-1}{2}} (N-k) \cdot \log(-k)^2 + \mathcal{F}_2  \\&
= 4\cdot  \sum_{k=1}^{\frac{N-1}{2}} k \cdot \log(k)^2- 3\cdot \mathcal{F}_2\text{ .}
\end{align*}
On the other hand, we have in $\mathbf{Z}/p^{v}\mathbf{Z}$:

\begin{align*}
\sum_{k \in (\mathbf{Z}/N\mathbf{Z})^{\times}} g(\frac{k}{N}) \cdot \log(k)^2  &= \sum_{k \in (\mathbf{Z}/N\mathbf{Z})^{\times}}\left( f(\frac{2k}{N})- 4\cdot f(\frac{k}{N})\right) \cdot \log(k)^2 \\&
=\sum_{k \in (\mathbf{Z}/N\mathbf{Z})^{\times}} (-4)\cdot f(\frac{k}{N}) \cdot \log(k)^2 + f(\frac{2k}{N})\cdot \left(\log(2k)-\log(2)\right)^2 \\&
= -3 \sum_{k \in (\mathbf{Z}/N\mathbf{Z})^{\times}} f(\frac{k}{N}) \cdot \log(k)^2 +  \log(2)^2\cdot \sum_{k \in (\mathbf{Z}/N\mathbf{Z})^{\times}} f(\frac{2k}{N}) \\& - 2 \cdot \log(2) \cdot \sum_{k \in (\mathbf{Z}/N\mathbf{Z})^{\times}} f(\frac{2k}{N}) \cdot \log(2k)  \\&= -3 \cdot \sum_{k=1}^{N-1} k^2 \cdot \log(k)^2 + \log(2)^2 \cdot \frac{N-1}{6} - 2 \cdot \log(2) \cdot \sum_{k=1}^{N-1} k^2\cdot \log(k) \text{ .}
\end{align*} 
\end{proof}

\begin{lem}\label{even_modSymb_somme_bizarre}
We have, in $\mathbf{Z}/p^{v}\mathbf{Z}$:
$$\sum_{t_1,\text{ }t_2=1\atop t_1\neq t_2}^{\frac{N-1}{2}} \log(t_1-t_2) = -2 \cdot \sum_{k=1}^{\frac{N-1}{2}}k \cdot \log(k) $$
and
$$\sum_{t_1, t_2=1}^{\frac{N-1}{2}} \log(t_1+t_2) = 2\cdot \sum_{k=1}^{\frac{N-1}{2}}k \cdot \log(k) -\mathcal{F}_1 \text{ .}$$
\end{lem}
\begin{proof}
We first compute $S_1 := \sum_{t_1,\text{ }t_2=1\atop t_1\neq t_2}^{\frac{N-1}{2}} \log(t_1-t_2) \in \mathbf{Z}/p^{v}\mathbf{Z}$.
When $t_1$ and $t_2$ vary in $\{1, ..., \frac{N-1}{2}\}$, the quantity $t_1-t_2$ varies in $X:=\{-\frac{N-1}{2}+1, ..., \frac{N-1}{2}-1\}$. If $k\neq 0 \in X$, then the number of such $t_1$ and $t_2$ such that $k = t_1-t_2$ is $\text{min}(\frac{N-1}{2}-k, \frac{N-1}{2})-\text{max}(1-k,1)+1$. If $1 \leq k \leq \frac{N-1}{2}-1$, this number is $\frac{N-1}{2}-k$. If $-\frac{N-1}{2}+1 \leq k \leq -1$, this number is $\frac{N-1}{2}+k$. Thus, we have
\begin{align*}
S_1 &=\sum_{k=-\frac{N-1}{2}}^{-1} \left(k+\frac{N-1}{2}\right)\cdot \log(k) + \sum_{k=1}^{\frac{N-1}{2}} \left(-k+\frac{N-1}{2}\right)\cdot \log(k)  \\&
=2\cdot\sum_{k=1}^{\frac{N-1}{2}} \left(-k+\frac{N-1}{2}\right)\cdot \log(k)
\\& = -2 \cdot \sum_{k=1}^{\frac{N-1}{2}}k \cdot \log(k) \text{ .}
\end{align*}
The proof of the second equality is similar and is left to the reader.
\end{proof}
In a similar way, we have the following result (whose proof is left to the reader).
\begin{lem}\label{even_modSymb_somme_bizarre_2}
We have, in $\mathbf{Z}/p^{v}\mathbf{Z}$:
$$\sum_{t_1,\text{ }t_2=1\atop t_1\neq t_2}^{\frac{N-1}{2}} \log(t_1-t_2)^2 = -2 \cdot \sum_{k=1}^{\frac{N-1}{2}}k \cdot \log(k)^2 $$
and
$$\sum_{t_1, t_2=1}^{\frac{N-1}{2}} \log(t_1+t_2)^2 = 2\cdot \sum_{k=1}^{\frac{N-1}{2}}k \cdot \log(k)^2 -\mathcal{F}_2\text{ .}$$
\end{lem}

\begin{lem}\label{even_modSymb_computation_D2}
We have, in $\mathbf{Z}/p^{v}\mathbf{Z}$:
\begin{align*}
\sum_{(t_1, t_2) \in (\mathbf{Z}/N\mathbf{Z})^{2} \atop t_1 \neq t_2} D_2\left(\frac{t_1}{N}\right)\cdot D_2\left(\frac{t_2}{N}\right)\cdot \log(t_1-t_2) &= -\sum_{(t_1, t_2) \in (\mathbf{Z}/N\mathbf{Z})^{2} \atop t_1 \neq -t_2} D_2\left(\frac{t_1}{N}\right)\cdot D_2\left(\frac{t_2}{N}\right)\cdot \log(t_1+t_2) \\&= -8 \cdot \sum_{k=1}^{\frac{N-1}{2}} k\cdot \log(k) + 2\cdot \mathcal{F}_1\\&= 6\cdot  \sum_{k=1}^{N-1} k^2\cdot \log(k) + \log(2) \cdot \frac{N-1}{3}\text{ .}
\end{align*}
\end{lem}
\begin{proof}
The first equality of Lemma \ref{even_modSymb_computation_D2} follows from the change of variable $t_2 \mapsto -t_2$, using $D_2(-x) = -D_2(x)$ for all $x \in \mathbf{R}$.
We have, in $\mathbf{Z}/p^{v}\mathbf{Z}$:
\begin{align*}
\sum_{(t_1, t_2) \in (\mathbf{Z}/N\mathbf{Z})^{2} \atop t_1 \neq t_2} D_2\left(\frac{t_1}{N}\right)D_2\left(\frac{t_2}{N}\right)\log(t_1-t_2) &= 2 \sum_{t_1,\text{ }t_2=1\atop t_1\neq t_2}^{\frac{N-1}{2}}  \log(t_1-t_2)  - 2\sum_{t_1,t_2=1}^{\frac{N-1}{2}}  \log(t_1+t_2) 
\\& = -8  \sum_{k=1}^{\frac{N-1}{2}}k \cdot \log(k) +2\cdot \mathcal{F}_1\text{ .}
\end{align*}
where in the last equality, we have used Lemma \ref{even_modSymb_somme_bizarre}. This shows the second equality of Lemma \ref{even_modSymb_computation_D2}. The third equality follows from Lemma \ref{even_modSymb_computation_square} and $4\cdot \mathcal{F}_1=0$.
\end{proof}

Similarly, using Lemmas \ref{even_modSymb_computation_square_2} and \ref{even_modSymb_somme_bizarre_2} we get the following result (which will be used in chapter 6).
\begin{lem}\label{even_modSymb_computation_D2_square}
We have, in $\mathbf{Z}/p^{v}\mathbf{Z}$:
\begin{align*}
\sum_{(t_1, t_2) \in (\mathbf{Z}/N\mathbf{Z})^{2} \atop t_1 \neq t_2} D_2\left(\frac{t_1}{N}\right)\cdot D_2\left(\frac{t_2}{N}\right)\cdot \log(t_1-t_2)^2 &= -\sum_{(t_1, t_2) \in (\mathbf{Z}/N\mathbf{Z})^{2} \atop t_1 \neq -t_2} D_2\left(\frac{t_1}{N}\right)\cdot D_2\left(\frac{t_2}{N}\right)\cdot \log(t_1+t_2)^2 \\&= -8 \cdot \sum_{k=1}^{\frac{N-1}{2}} k\cdot \log(k)^2 + 2\cdot \mathcal{F}_2 \\& = 6 \cdot \sum_{k=1}^{N-1}k^2 \cdot \log(k)^2 - \log(2)^2\cdot \frac{N-1}{3} \\&+ 4 \cdot \log(2) \cdot \sum_{k=1}^{N-1} k^2\cdot \log(k) -4\cdot \mathcal{F}_2 \text{ .}
\end{align*}
\end{lem}

The following identity will be useful in chapter \ref{Section_comparison}.
\begin{lem}\label{Bernardi_lemma>3}
For any $a \in (\mathbf{Z}/N\mathbf{Z})^{\times}$, we have in $\mathbf{Z}/p^{v}\mathbf{Z}$:
$$\sum_{k \in (\mathbf{Z}/N\mathbf{Z})^{\times} \atop k \neq a} \log(k-a)\cdot \log(k)=- \log(a)^2 + \log(-1)\cdot \log(a)+ \mathcal{F}_2 \text{ .}$$
\end{lem}
\begin{proof}
We make the change of variable $k = a \cdot s$. We get, in $\mathbf{Z}/p^v\mathbf{Z}$:
\begin{align*}
\sum_{k \in (\mathbf{Z}/N\mathbf{Z})^{\times} \atop k \neq a} \log(k-a)\cdot \log(k)&=\sum_{s \in (\mathbf{Z}/N\mathbf{Z})^{\times} \atop s\neq 1 } \log(s-1)\cdot \log(s) -\log(a)^2+\log(a) \cdot \sum_{s \in (\mathbf{Z}/N\mathbf{Z})^{\times} \atop s\neq 1} \log(s) \\& + \log(a) \cdot \sum_{s \in (\mathbf{Z}/N\mathbf{Z})^{\times} \atop s \neq 1} \log(s-1)  \\& =\sum_{s \in (\mathbf{Z}/N\mathbf{Z})^{\times} \atop s\neq 1 } \log(s-1)\cdot \log(s) -\log(a)^2+2 \cdot \log(a) \cdot \sum_{s \in (\mathbf{Z}/N\mathbf{Z})^{\times} } \log(s) \\& - \log(-1)\cdot \log(a) \\& =
-\log(a)^2 + \log(-1)\cdot \log(a) + \sum_{s \in (\mathbf{Z}/N\mathbf{Z})^{\times} \atop s\neq 1 } \log(s-1)\cdot \log(s)\text{ .}
\end{align*}
Since $2\cdot  \sum_{s \in (\mathbf{Z}/N\mathbf{Z})^{\times}} \log(s) =  \log((N-1)!^2) = 0$, we have in $\mathbf{Z}/p^v\mathbf{Z}$:
\begin{equation}\label{even_modSymb_sum_k-a_k_k-1_k_eq}
\sum_{k \in (\mathbf{Z}/N\mathbf{Z})^{\times} \atop k \neq a} \log(k-a)\cdot \log(k) = -\log(a)^2 + \log(-1)\cdot \log(a) + \sum_{s \in (\mathbf{Z}/N\mathbf{Z})^{\times} \atop s\neq 1 } \log(s-1)\cdot \log(s)\text{ .}
\end{equation}

\begin{lem}\label{Bernardi_Merel_Lemma}
We have, in $\mathbf{Z}/p^v\mathbf{Z}$:
$$\sum_{s \in (\mathbf{Z}/N\mathbf{Z})^{\times} \atop s\neq 1 } \log(s-1)\cdot \log(s) = \mathcal{F}_2 \text{ .}$$
\end{lem}
\begin{proof}
We treat the cases $p=2$ and $p>2$ separately. Assume first that $p>2$. In this case, we have $\mathcal{F}_2=0$. On the other hand, the left-hand side is easily seen to be zero, using the change of variable $s \mapsto \frac{1}{s}$. During the rest of the proof, we assume that $p=2$ (so $N \equiv 1 \text{ (modulo }8\text{)}$).
We define three equivalence relations $\sim_1$, $\sim_2$ and $\sim_3$ in $(\mathbf{Z}/N\mathbf{Z})^{\times} \backslash \{\overline{1}, \overline{-1}\}$, characterized by $x \sim_1 -x$, $x\sim_2 \frac{1}{x}$, $x\sim_3 \frac{1}{x}$ and $x \sim_3 -x$ for all $x \in (\mathbf{Z}/N\mathbf{Z})^{\times}$. For $i \in \{1, 2, 3\}$, let $R_i \subset (\mathbf{Z}/N\mathbf{Z})^{\times}$ be a set of representative for $\sim_i$. We can and do choose $R_1$, $R_2$ and $R_3$ so that $R_3 \subset R_1 \cap R_2$. We denote by $\overline{R}_i$ the complement of $R_i$ in $(\mathbf{Z}/N\mathbf{Z})^{\times} \backslash \{\overline{1}, \overline{-1}\}$. Let $\zeta_4 \in R_2$ be the unique element of order $4$. If $x \in R_3$ and $x \neq \zeta_4$, there is a unique element $[x]$ of $R_2$ such that $[x] \neq x$ and $[x] \sim_3 x$. We have $[x] = -x$ or $[x] = -\frac{1}{x}$. We get a partition $$R_2 = \{\zeta_4\}\bigsqcup_{x \in R_3\atop x \neq \zeta_4} \{x, [x]\} \text{ .}$$

We have, in $\mathbf{Z}/p^v\mathbf{Z}$:
\begin{align*}
\sum_{s \in (\mathbf{Z}/N\mathbf{Z})^{\times} \atop s\neq 1 } \log(s-1)\cdot \log(s) &=  \sum_{s \in (\mathbf{Z}/N\mathbf{Z})^{\times} \atop s\neq 1, -1 } \log(s-1)\cdot \log(s) \\& =  \sum_{s \in R_2} \log(s-1)\cdot \log(s)  +  \sum_{s \in \overline{R}_2} \log(s-1)\cdot \log(s) \\&= \sum_{s \in R_2} \log(s-1)\cdot \log(s) - \log\left(\frac{1}{s}-1\right)\cdot \log(s) \text{ .}
\end{align*}
In the first equality, we have used the fact that $\log(-1)\cdot\log(-2)=0$ since $\log(2) \equiv 0 \text{ (modulo }2\text{)}$ by the quadratic reciprocity law (recall that $N \equiv 1 \text{ (modulo }8\text{)}$).
Thus, we have:
\begin{equation}\label{Bernardi_Merel_Lemma_eq1}
\sum_{s \in (\mathbf{Z}/N\mathbf{Z})^{\times} \atop s\neq 1 } \log(s-1)\cdot \log(s) = \sum_{s \in R_2} \log(s)^2 - \log(-1)\cdot \log(s) \text{ .}
\end{equation}
We have:
$$\sum_{s \in R_2} \log(s) \equiv \log(\zeta_4) + \sum_{s \in R_3 \atop s \neq \zeta_4} \log(s)+\log([s]) \text{ (modulo }2\text{).}$$
We have $\log(s) + \log([s]) \equiv 0 \text{ (modulo }2\text{)}$, and $\log(\zeta_4) \equiv 0 \text{ (modulo }2\text{)}$.
Thus, we have:
\begin{equation}\label{Bernardi_Merel_Lemma_eq2}
\sum_{s \in R_2}\log(-1)\cdot \log(s) = 0 \text{ .}
\end{equation}
By (\ref{Bernardi_Merel_Lemma_eq1}) and (\ref{Bernardi_Merel_Lemma_eq2}), have:
\begin{equation}\label{Bernardi_Merel_Lemma_eq3}
\sum_{s \in (\mathbf{Z}/N\mathbf{Z})^{\times} \atop s\neq 1 } \log(s-1)\cdot \log(s) = \sum_{s \in R_2} \log(s)^2 \text{ .}
\end{equation}
We have:
\begin{align*}
\mathcal{F}_2 &= \sum_{s \in R_1} \log(s)^2  \\& = \log(\zeta_4)^2 + 2\cdot \sum_{s \in R_3 \atop s\neq \zeta_4} \log(s)^2 \\&   =\log(\zeta_4)^2 +  \sum_{s \in R_3 \atop s\neq \zeta_4} \log(s)^2 + \log([s])^2= \sum_{s \in R_2} \log(s)^2 \text{ .}
\end{align*}
Combining the latter equality with (\ref{Bernardi_Merel_Lemma_eq3}), this concludes the proof of Lemma \ref{Bernardi_Merel_Lemma}.
\end{proof}
Lemma \ref{Bernardi_lemma>3} follows from (\ref{even_modSymb_sum_k-a_k_k-1_k_eq}) and Lemma \ref{Bernardi_Merel_Lemma}.
\end{proof}

The following identity will be useful to compute $m_1^+$ when $p=2$ (Theorem \ref{thm_Introduction_w_1^+_2}).

\begin{lem}\label{even_modSymb_magical_identity_Gauss}
Assume that $p=2$, so that $N \equiv 1 \text{ (modulo }8\text{)}$. 
For all $x \in (\mathbf{Z}/N\mathbf{Z})^{\times} \backslash \{\overline{1}, \overline{-1}\}$, we have in $\mathbf{Z}/2\mathbf{Z}$:
\begin{align*}
\sum_{s_1,s_2=1 \atop (1-x)s_1+(1+x)s_2\equiv 0 \text{ (modulo }N\text{)}}^{\frac{N-1}{2}} 1 &=  \log\left(\frac{x+1}{x-1}\right)\\&= \sum_{s_1,s_2=1 \atop (1-x)s_1+(1+x)s_2\equiv 0 \text{ (modulo }N\text{)}}^{\frac{N-1}{2}} \log\left(\frac{2}{1-x}\cdot s_2\right) \\& + \sum_{s_1,s_2=1 \atop (1-x)s_1+(1+x)s_2\not\equiv 0 \text{ (modulo }N\text{)}}^{\frac{N-1}{2}} \log\left( (1-x)s_1+(1+x)s_2\right)\text{ .} 
\end{align*}
\end{lem}
\begin{proof}
The integer
$$\sum_{s_1,s_2=1 \atop (1-x)s_1+(1+x)s_2\equiv 0 \text{ (modulo }N\text{)}}^{\frac{N-1}{2}} 1$$
is the number of $s_2 \in \{1, 2, ..., \frac{N-1}{2}\}$ such that the representative of $\frac{x+1}{x-1} \cdot s_2 \in \mathbf{Z}/N\mathbf{Z}$ in $\{1, 2, ...., N-1\}$ is in $\{1, 2, ..., \frac{N-1}{2}\}$. By Gauss's Lemma \cite[p. 52]{Ireland_Rosen}, this number is congruent to $\frac{N-1}{2}-\log\left(\frac{x+1}{x-1}\right)$ modulo $2$. Since $\frac{N-1}{2}$ is even, we get in $\mathbf{Z}/2\mathbf{Z}$:
 $$\sum_{s_1,s_2=1 \atop (1-x)s_1+(1+x)s_2\equiv 0 \text{ (modulo }N\text{)}}^{\frac{N-1}{2}} 1 =  \log\left(\frac{x+1}{x-1}\right) \text{ .}$$
 To conclude the proof of Lemma \ref{even_modSymb_magical_identity_Gauss}, it suffices to prove the following equality in $\mathbf{Z}/2\mathbf{Z}$:
\begin{equation}\label{even_modSymb_identity_log_log_Gauss}
\sum_{s_1,s_2=1 \atop (1-x)s_1+(1+x)s_2\equiv 0 \text{ (modulo }N\text{)}}^{\frac{N-1}{2}} \log\left(\frac{2}{1-x}\cdot s_2\right) + \sum_{s_1,s_2=1 \atop (1-x)s_1+(1+x)s_2\not\equiv 0 \text{ (modulo }N\text{)}}^{\frac{N-1}{2}} \log\left( (1-x)s_1+(1+x)s_2\right)= \log\left(\frac{x+1}{x-1}\right)\text{ .} 
\end{equation}
We denote by $S$ the left hand side of (\ref{even_modSymb_identity_log_log_Gauss}). Since $N \equiv 1 \text{ (modulo }8\text{)}$, the class of $2$ in $(\mathbf{Z}/N\mathbf{Z})^{\times}$ is a square \ie $\log(2) \equiv 0 \text{ (modulo }2\text{)}$, and we have $\log(-1)\equiv 0 \text{ (modulo }4\text{)}$. 

We have, in $\mathbf{Z}/2\mathbf{Z}$ (using the first equality):
\begin{align*}
S &= \log\left(\frac{x+1}{x-1}\right) \cdot \log(x-1) + \sum_{s_1,s_2=1 \atop (1-x)s_1+(1+x)s_2\equiv 0 \text{ (modulo }N\text{)}}^{\frac{N-1}{2}} \log(s_2) \\& + \sum_{s_1,s_2=1 \atop (1-x)s_1+(1+x)s_2\not\equiv 0 \text{ (modulo }N\text{)}}^{\frac{N-1}{2}} \log\left( (1-x)s_1+(1+x)s_2\right) \\& =  \log\left(\frac{x+1}{x-1}\right) \cdot \log(x-1) + \sum_{s_1,s_2=1}^{\frac{N-1}{2}} \log(s_2) + \sum_{s_1,s_2=1 \atop (1-x)s_1+(1+x)s_2\not\equiv 0 \text{ (modulo }N\text{)}}^{\frac{N-1}{2}} \log\left( (1-x)\cdot \frac{s_1}{s_2}+(1+x)\right) \\&=
\log\left(\frac{x+1}{x-1}\right) \cdot \log(x-1) + \sum_{s_1,s_2=1 \atop (1-x)s_1+(1+x)s_2\not\equiv 0 \text{ (modulo }N\text{)}}^{\frac{N-1}{2}} \log\left( (1-x)\cdot \frac{s_1}{s_2}+(1+x)\right) 
\\&=\sum_{s_1,s_2=1 \atop (1-x)s_1+(1+x)s_2\not\equiv 0 \text{ (modulo }N\text{)}}^{\frac{N-1}{2}} \log\left( \frac{s_1}{s_2}-\frac{x+1}{x-1}\right) 
\text{ .} 
\end{align*}
In the last equality, we have used:
$$\sum_{s_1,s_2=1 \atop (1-x)s_1+(1+x)s_2\not\equiv 0 \text{ (modulo }N\text{)}}^{\frac{N-1}{2}}  1 \equiv \sum_{s_1,s_2=1 \atop (1-x)s_1+(1+x)s_2\equiv 0 \text{ (modulo }N\text{)}}^{\frac{N-1}{2}}  1 \equiv \log\left(\frac{x+1}{x-1} \right)\text{ (modulo }2\text{),}$$
which follows from:
$$\sum_{s_1,s_2=1}^{\frac{N-1}{2}} 1 = \left(\frac{N-1}{2}\right)^2 \equiv 0 \text{ (modulo }2\text{).}$$
Let $f : (\mathbf{Z}/N\mathbf{Z})^{\times} \rightarrow \mathbf{Z}$ be such that if $y \in (\mathbf{Z}/N\mathbf{Z})^{\times}$, $f(y)$ is the number of elements $(s_1,s_2) \in \{1, 2, ..., \frac{N-1}{2}\}^2$ such that $\frac{s_1}{s_2} \equiv y \text{ (modulo }N\text{)}$. 
We have shown that we have, in $\mathbf{Z}/2\mathbf{Z}$:
\begin{equation}\label{even_modSymb_identity_log(y-x-1)}
S = \sum_{y \in (\mathbf{Z}/N\mathbf{Z})^{\times}\atop y \neq \frac{x+1}{x-1}} \log\left(y-\frac{x+1}{x-1}\right)\cdot f(y) \text{ .}
\end{equation}
As above, by Gauss's lemma and the fact that $N \equiv 1 \text{ (modulo }4\text{)}$, we have $f(y) \equiv \log(y) \text{ (modulo }2\text{)}$. By (\ref{even_modSymb_identity_log(y-x-1)}), we have in $\mathbf{Z}/2\mathbf{Z}$:
$$
S = \sum_{y \in (\mathbf{Z}/N\mathbf{Z})^{\times}\atop y \neq \frac{x+1}{x-1}} \log\left(y-\frac{x+1}{x-1}\right)\cdot \log(y) \text{ .}
$$

By Lemma \ref{Bernardi_lemma>3} and the fact that $N \equiv 1 \text{ (modulo }8\text{)}$, we have in $\mathbf{Z}/2\mathbf{Z}$:
$$S = -\log\left(\frac{x+1}{x-1}\right)^2+\frac{N-1}{12} = \log\left(\frac{x+1}{x-1}\right) \text{ .}$$
This concludes the proof of Lemma \ref{even_modSymb_magical_identity_Gauss}.
\end{proof}

\subsection{The case $p=3$}\label{even_modSymb_the_case_p=3}

In this section, we focus on the case $p=3$. We determine the image of $m_1^+$ in $\left(M_+/3^t\cdot M_+\right)/\mathbf{Z}\cdot m_0^+$. It suffices to determine the image of $6 \cdot m_1^+$ in $\left( M_+/3^{t+1}\cdot M_+\right)/\mathbf{Z}\cdot 3\cdot m_0^+$. The formula given is a minor variation of Theorem \ref{thm_Introduction_w_1^+>3_p^t}. Recall that we lift $\log$ to a group homomorphism $(\mathbf{Z}/N\mathbf{Z})^{\times} \rightarrow \mathbf{Z}/3^{t+1}\mathbf{Z}$ (still denoted by $\log$).

\begin{thm}\label{thm_Introduction_w_1^+_3}
Assume that $p=3$.
We have, in $\left( M_+/3^{t+1}\cdot M_+\right)/\mathbf{Z}\cdot 3\cdot m_0^+$:
$$6 \cdot m_1^+  \equiv \log(2) \cdot \tilde{m}_0^+ + \sum_{x \in \mathbf{P}^1(\mathbf{Z}/N\mathbf{Z})} F_{1,3}(x) \cdot \xi_{\Gamma_0(N)}(x)  \text{ ,}$$
where $F_{1,3} : \mathbf{P}^1(\mathbf{Z}/N\mathbf{Z}) \rightarrow \mathbf{Z}/3^{t+1}\mathbf{Z}$ is defined by
\begin{align*}
F_{1,3}([c:d]) &= \frac{1}{2}\sum_{(s_1,s_2) \in (\mathbf{Z}/2N\mathbf{Z})^2 \atop (d-c)s_1+(d+c)s_2 \equiv 0 \text{ (modulo } N\text{)}} (-1)^{s_1+s_2} \B_1\left(\frac{s_1}{2N}\right)\B_1\left(\frac{s_2}{2N}\right) \cdot \log\left(\frac{s_2}{d-c}\right) 
\\&
- \frac{1}{2}\sum_{(s_1,s_2) \in (\mathbf{Z}/2N\mathbf{Z})^2 \atop (d-c)s_1+(d+c)s_2 \not\equiv 0 \text{ (modulo } N\text{)}} (-1)^{s_1+s_2} \B_1\left(\frac{s_1}{2N}\right)\B_1\left(\frac{s_2}{2N}\right) \cdot \log((d-c)s_1+(d+c)s_2))  
\end{align*}
if $[c:d] \neq [1:1]$ and $F_{1,3}([1:1])=0$.
\end{thm}
\begin{proof}
Let $\beta : J/J^2 \rightarrow \mathbf{Z}/3^{t+1}\mathbf{Z}$ be given by $[x]-1 \mapsto \log(x)$ for $x \in (\mathbf{Z}/N\mathbf{Z})^{\times}$. This induces a map $\beta_* : J\cdot H_1(X_1(N), \cusps, \mathbf{Z}[\frac{1}{2N}][(\mathbf{Z}/N\mathbf{Z})^{\times}/\pm 1]) \rightarrow H_1(X_1(N), \cusps, \mathbf{Z}/3^{t+1}\mathbf{Z})$. 

Let $F_{1,3}': \mathbf{P}^1(\mathbf{Z}/N\mathbf{Z}) \rightarrow \mathbf{Z}/3^{t+1}\mathbf{Z}$ be defined by:
\begin{align*}
F_{1,3}'([c:d]) &= \frac{1}{2}\sum_{(s_1,s_2) \in (\mathbf{Z}/2N\mathbf{Z})^2 \atop (d-c)s_1+(d+c)s_2 \equiv 0 \text{ (modulo } N\text{)}} (-1)^{s_1+s_2} \B_1\left(\frac{s_1}{2N}\right)\B_1\left(\frac{s_2}{2N}\right) \cdot \log((b-a)s_1+(b+a)s_2)  \\&
- \frac{1}{2}\sum_{(s_1,s_2) \in (\mathbf{Z}/2N\mathbf{Z})^2 \atop (d-c)s_1+(d+c)s_2 \not\equiv 0 \text{ (modulo } N\text{)}} (-1)^{s_1+s_2} \B_1\left(\frac{s_1}{2N}\right)\B_1\left(\frac{s_2}{2N}\right) \cdot \log((d-c)s_1+(d+c)s_2)) 
\end{align*}
where $\begin{pmatrix} a&b\\ c&d\end{pmatrix} \in \SL_2(\mathbf{Z})$. By Lemma \ref{even_modSymb_reecriture_m_0^+}, this does not depend on the choice of  $\begin{pmatrix} a&b\\ c&d\end{pmatrix}$. We also have, for all $[c:d] \in \mathbf{P}^1(\mathbf{Z}/N\mathbf{Z})$:
\begin{align*}
F_{1,3}'([c:d]) &= \frac{1}{8}\sum_{(s_1,s_2) \in (\mathbf{Z}/N\mathbf{Z})^2 \atop (d-c)s_1+(d+c)s_2 \equiv 0 \text{ (modulo } N\text{)}} D_2\left(\frac{s_1}{N}\right)D_2\left(\frac{s_2}{N}\right)\cdot \log((b-a)s_1+(b+a)s_2)  \\&
- \frac{1}{8}\sum_{(s_1,s_2) \in (\mathbf{Z}/N\mathbf{Z})^2 \atop (d-c)s_1+(d+c)s_2 \not\equiv 0 \text{ (modulo } N\text{)}} D_2\left(\frac{s_1}{N}\right)D_2\left(\frac{s_2}{N}\right) \cdot \log((d-c)s_1+(d+c)s_2)) \text{ .}
\end{align*}

For all $[c:d] \in \mathbf{P}^1(\mathbf{Z}/N\mathbf{Z})$, we have $F_{1,3}'([-c:d]) = F_{1,3}'([c:d])$. Thus, we have $$\sum_{x \in \mathbf{P}^1(\mathbf{Z}/N\mathbf{Z}) } F_{1,3}'(x)\cdot \xi_{\Gamma_0(N)}(x) \in H_1(X_0(N), \cusps, \mathbf{Z}/3^{t+1}\mathbf{Z})_+ \text{ .}$$ By construction, the pull-back of  $\sum_{x\in \mathbf{P}^1(\mathbf{Z}/N\mathbf{Z}) } F_{1,3}'(x)\cdot \xi_{\Gamma_0(N)}(x)$ in $H_1(X_1(N), \cusps, \mathbf{Z}/3^{t+1}\mathbf{Z})_+$ is $\beta_*(\mathcal{U})$. 

By (\ref{even_modSymb_crucial_Hecke_2}), Lemma \ref{even_modSymb_reecriture_m_0^+} and Proposition \ref{even_modSymb_injectivity_trace} (ii), for all prime $\ell$ not dividing $N$ we have in $M_+/3^{t+1}\cdot M_+$:
$$(T_{\ell}-\ell-1)\left(\sum_{x \in \mathbf{P}^1(\mathbf{Z}/N\mathbf{Z})} F_{1,3}'(x) \cdot \xi_{\Gamma_0(N)}(x)\right) = 6\cdot \frac{\ell-1}{2}\cdot \log(\ell)\cdot m_0^+ \text{ .}$$ 

Thus, there exists $K_3 \in \mathbf{Z}/3^{t+1}\mathbf{Z}$ such that we have:
$$6 \cdot m_1^+  \equiv K_3 \cdot m_0^+ + \sum_{x \in \mathbf{P}^1(\mathbf{Z}/N\mathbf{Z})} F_{1,3}'(x) \cdot \xi_{\Gamma_0(N)}(x) \text{ in } \left(M_+/3^{t+1}\cdot M_+\right)/\mathbf{Z}\cdot 3\cdot m_0^+ \text{ .}$$
Note that $K_3$ is only uniquely defined modulo $3$. We have, in $\mathbf{Z}/3^{t+1}\mathbf{Z}$:
$$6 \cdot m_1^+ \bullet m_0^-  = K_3 \cdot (\tilde{m}_0^+ \bullet \tilde{m}_0^-)-F_{1,3}'([0:1])+F_{1,3}'([1:0]) \text{ .}$$
Recall that $\tilde{m}_0^+ \bullet \tilde{m}_0^- = \frac{N-1}{12}$.
We also have in $\mathbf{Z}/3^t\mathbf{Z}$:
$$m_1^+ \bullet m_0^-  = m_0^+ \bullet m_1^- =  -\frac{1}{4} \cdot \left(\frac{N-1}{6}\cdot \log(2) + \sum_{k=1}^{N-1}k^2\cdot \log(k) \right) \text{ .}$$
The first equality follows from Corollary \ref{Comparison_corr_comparison} and the second equality follows from Theorem \ref{Comparison_Merel_sqrt_u} (these results essentially come from \cite{Merel_accouplement} and are thus independents of the results of this chapter).

We have, using Lemma \ref{even_modSymb_computation_D2}: 
\begin{align*}
F_{1,3}'([0:1]) &= -\frac{1}{8}\sum_{(s_1,s_2) \in (\mathbf{Z}/N\mathbf{Z})^2 \atop s_1 + s_2 \not\equiv 0 \text{ (modulo }N\text{)}} D_2\left(\frac{s_1}{N}\right)D_2\left(\frac{s_2}{N}\right)\cdot \log(s_2+s_1)  \\&=\frac{3}{4}\sum_{k=1}^{N-1} k^2\cdot \log(k) + \frac{\log(2)}{2}\cdot \frac{N-1}{12}
\end{align*}
Note also that we have $F_{1,3}'([1:0]) = -F_{1,3}'([0:1])$.
We thus get, in $\mathbf{Z}/3^{t+1}\mathbf{Z}$:
$$-\frac{1}{2}\cdot \left( \frac{N-1}{2}\cdot \log(2) + 3  \sum_{k=1}^{N-1} k^2\cdot \log(k)\right) =  (K_3-\log(2))\cdot \frac{N-1}{12}-\frac{3}{2}\sum_{k=1}^{N-1}k^2\cdot \log(k) \text{ .}$$
We thus have $K_3 \equiv \log(2) \text{ (modulo }3\text{)}$. 

If $\begin{pmatrix}
 a & b \\
 c & d
  \end{pmatrix}  \in \Gamma(2)$ and $(s_1, s_2) \in (\mathbf{Z}/2N\mathbf{Z})^2$ are such that $(d-c)s_1 + (d+c)s_2 \equiv 0 \text{ (modulo }N\text{)}$ and $d \not\equiv c \text{ (modulo }N\text{)}$, then we have $(b-a)s_1+(b+a)s_2 \equiv \frac{2}{d-c}\cdot s_2  \text{ (modulo }N\text{)}$. By Lemma \ref{even_modSymb_reecriture_m_0^+}, we have in $M_+/3^{t+1}\cdot M_+$:
 $$\sum_{x \in \mathbf{P}^1(\mathbf{Z}/N\mathbf{Z})} F_{1,3}(x)\cdot \xi_{\Gamma_0(N)}(x) = \sum_{x \in \mathbf{P}^1(\mathbf{Z}/N\mathbf{Z})} F_{1,3}'(x)\cdot \xi_{\Gamma_0(N)}(x)  - 3\cdot \log(2) \cdot m_0^+ \text{ .}$$
This concludes the proof of Theorem \ref{thm_Introduction_w_1^+_3}.
\end{proof}

\subsection{The case $p=2$}\label{even_modSymb_the_case_p=2}
We now study the case $p=2$ using a similar method, although our result is only partial. Recall that $\mathcal{R}$ is the set of equivalence classes in $\mathbf{P}^1(\mathbf{Z}/N\mathbf{Z})$ for the equivalence relation $[c:d] \sim [-d:c]$.

We first give a formula for $\tilde{m}_0^+$ in terms of Manin symbols.

\begin{thm}\label{thm_Introduction_w_0^+_2}
Assume that $p=2$. Let $F_{0,2} : \mathbf{P}^1(\mathbf{Z}/N\mathbf{Z}) \rightarrow \mathbf{Z}_p$ be given by
$$F_{0,2}([c:d]) = -\frac{N-1}{12}+\frac{1}{3}\cdot \sum_{s_1,s_2=1 \atop (d-c)s_1+(d+c)s_2 \equiv 0 \text{ (modulo }N\text{)}}^{\frac{N-1}{2}} 1 \text{ .}$$ 
If $[c:d] \in  \mathbf{P}^1(\mathbf{Z}/N\mathbf{Z}) \backslash \{[1:1], [-1:1]\}$, we have $F_{0,2}([-d:c]) = -F_{0,2}([c:d])$. Thus, for all $x \in  \mathbf{P}^1(\mathbf{Z}/N\mathbf{Z})$ the element $F_{0,2}(x) \cdot \xi_{\Gamma_0(N)}(x)$ of $M_+$ only depends on the class of $x$ in $R$. We have in $M_+$:
$$\tilde{m}_0^+ = \sum_{x \in \mathcal{R}} F_{0,2}(x) \cdot \xi_{\Gamma_0(N)}(x) \text{ .}$$
\end{thm}
\begin{proof}
Recall \cite[Lemme 3, Corollaire 4]{Merel_accouplement} that we have in $H_1(X_0(N), \cusps, \mathbf{Q}_p)_+$:
\begin{equation}\label{even_modSymb_p=2_m_0^+_E}
12 \cdot \tilde{m}_0^+ = \mathcal{E}
\end{equation}
where $$\mathcal{E} = \sum_{x \in \mathbf{P}^1(\mathbf{Z}/N\mathbf{Z}) } F_{0,2}'(x) \cdot \xi_{\Gamma_0(N)}(x) $$
and
$$F_{0,2}'([c:d]) =\frac{1}{2}\cdot  \sum_{(s_1,s_2) \in (\mathbf{Z}/N\mathbf{Z})^2 \atop (d-c)s_1+(d+c)s_2 \equiv 0 \text{ (modulo }N\text{)} } D_2\left(\frac{s_1}{N}\right)\cdot D_2\left(\frac{s_2}{N}\right) \text{ .}$$
Assume that $c\neq \pm d$. We have, in $\mathbf{Z}$:
\begin{align*}
F_{0,2}'([c:d]) &= \sum_{s_1,s_2=1  \atop (d-c)s_1+(d+c)s_2 \equiv 0 \text{ (modulo }N\text{)} }^{\frac{N-1}{2}} 1 -\sum_{s_1=1}^{\frac{N-1}{2}} \sum_{s_2=\frac{N+1}{2}  \atop (d-c)s_1+(d+c)s_2 \equiv 0 \text{ (modulo }N\text{)} }^{N-1} 1 \\& = 2 \cdot  \sum_{s_1,s_2=1  \atop (d-c)s_1+(d+c)s_2 \equiv 0 \text{ (modulo }N\text{)} }^{\frac{N-1}{2}} 1 - \sum_{s_1=1}^{\frac{N-1}{2}} \sum_{s_2=1 \atop (d-c)s_1+(d+c)s_2 \equiv 0 \text{ (modulo }N\text{)} }^{N-1} 1 \\&= -\frac{N-1}{2} + 2 \cdot  \sum_{s_1,s_2=1  \atop (d-c)s_1+(d+c)s_2 \equiv 0 \text{ (modulo }N\text{)} }^{\frac{N-1}{2}} 1 \text{ .}
\end{align*}
Thus, for all $[c:d] \neq [1:\pm 1] \in \mathbf{P}^1(\mathbf{Z}/N\mathbf{Z})$ we have
\begin{equation}\label{even_modSymb_F_0_2_F_0,2'}
F_{0,2}'([c:d]) = 6\cdot F_{0,2}([c:d]) \text{ .}
\end{equation}
Thus, we have $$\mathcal{E} = 6\cdot\sum_{x \in \mathbf{P}^1(\mathbf{Z}/N\mathbf{Z}) } F_{0,2}(x) \cdot \xi_{\Gamma_0(N)}(x) \text{ .}$$
By (\ref{even_modSymb_p=2_m_0^+_E}), we get:
\begin{equation}\label{even_modSymb_p=2_m_0^+_E_2}
\tilde{m}_0^+ = \frac{1}{2}\cdot \sum_{x \in \mathbf{P}^1(\mathbf{Z}/N\mathbf{Z}) } F_{0,2}(x) \cdot \xi_{\Gamma_0(N)}(x) \text{ .}
\end{equation}
For all $[c:d] \in \mathbf{P}^1(\mathbf{Z}/N\mathbf{Z})$, we have 
\begin{align*}
F_{0,2}'([-d:c]) &= \frac{1}{2}\cdot  \sum_{(s_1,s_2) \in (\mathbf{Z}/N\mathbf{Z})^2 \atop (c+d)s_1+(c-d)s_2 \equiv 0 \text{ (modulo }N\text{)} } D_2\left(\frac{s_1}{N}\right)\cdot D_2\left(\frac{s_2}{N}\right) 
\\& =  \frac{1}{2}\cdot  \sum_{(s_1,s_2) \in (\mathbf{Z}/N\mathbf{Z})^2 \atop (c+d)s_1+(d-c)s_2 \equiv 0 \text{ (modulo }N\text{)} } D_2\left(\frac{s_1}{N}\right)\cdot D_2\left(-\frac{s_2}{N}\right)  \\&=- \frac{1}{2}\cdot  \sum_{(s_1,s_2) \in (\mathbf{Z}/N\mathbf{Z})^2 \atop (c+d)s_1+(d-c)s_2 \equiv 0 \text{ (modulo }N\text{)} } D_2\left(\frac{s_1}{N}\right)\cdot D_2\left(\frac{s_2}{N}\right) \\&=-F_{0,2}'([c,d]) \text{ .}
\end{align*}
By (\ref{even_modSymb_p=2_m_0^+_E_2}) and (\ref{even_modSymb_F_0_2_F_0,2'}), for all $[c:d] \neq [1:\pm1] \in \mathbf{P}^1(\mathbf{Z}/N\mathbf{Z})$, we have $F_{0,2}([-d:c]) = -F_{0,2}([c:d])$. Since $\xi_{\Gamma_0(N)}([-1:1]) = \xi_{\Gamma_0(N)}([1:1])=0$, we have in $M_+$:
$$\tilde{m}_0^+ = \sum_{x \in \mathcal{R}} F_{0,2}(x) \cdot \xi_{\Gamma_0(N)}(x) \text{ .}$$
\end{proof}
\begin{rem}
By combining Theorem \ref{thm_Introduction_w_0^+_2}, Lemma \ref{even_modSymb_magical_identity_Gauss} and Proposition \ref{even_modSymb_injectivity_trace} (iii), we get a new proof of Proposition \ref{even_modSymb_injectivity_trace} (i) when $p=2$ and $r=1$.
\end{rem}

\begin{thm}\label{thm_Introduction_w_1^+_2}
Assume that $p=2$ and that $t \geq 2$, \ie $N \equiv 1 \text{ (modulo }16\text{)}$.
We have:
$$6 \cdot m_1^+ \equiv m_0^+ + \sum_{x \in \mathcal{R}} F_{1,2}(x) \cdot \xi_{\Gamma_0(N)}(x) \text{ in } \left(M_+/2^t\cdot M_+\right)/\mathbf{Z}\cdot 2\cdot m_0^+$$
where $F_{1,2} : \mathbf{P}^1(\mathbf{Z}/N\mathbf{Z}) \rightarrow \mathbf{Z}/2^{t+1}\mathbf{Z}$ is defined by
\begin{align*}
F_{1,2}([c:d]) &= \sum_{s_1, s_2=1 \atop (d-c)s_1+(d+c)s_2 \equiv 0 \text{ (modulo }N\text{)}}^{\frac{N-1}{2}}   \log\left(\frac{2}{d-c}\cdot s_2\right)  \\& - \sum_{s_1, s_2=1 \atop (d-c)s_1+(d+c)s_2 \neq 0 \text{ (modulo }N\text{)}}^{\frac{N-1}{2}}  \log((d-c)s_1+(d+c)s_2) 
\end{align*}
if $[c:d]\neq [1:1]$ and $F_{1,2}([1:1])=0$ (this does not depend on the choice of $c$ and $d$). 
\end{thm}

\begin{proof}
Let $\gamma=\begin{pmatrix}
 a & b \\
 c & d
  \end{pmatrix} \in \PSL_2(\mathbf{Z})$ with $c \neq \pm d$. We have
\begin{align*}
G_{\infty}(\Gamma_1(N) \cdot \gamma) &= \frac{1}{4}\sum_{(s_1, s_2) \in (\mathbf{Z}/N\mathbf{Z})^2 \atop (d-c)s_1+(d+c)s_2 \not\equiv 0 \text{ (modulo }N\text{)}} D_2\left(\frac{s_1}{N}\right)D_2\left(\frac{s_2}{N}\right) \cdot [(d-c)s_1+(d+c)s_2]^{-1} \\&
=\frac{1}{2}\sum_{s_1,s_2=1 \atop (d-c)s_1+(d+c)s_2 \not\equiv 0 \text{ (modulo }N\text{)}}^{\frac{N-1}{2}}  [(d-c)s_1+(d+c)s_2]^{-1} \\& - \frac{1}{2}\sum_{s_1=\frac{N+1}{2}}^{N-1}\sum_{s_2=1 \atop (d-c)s_1+(d+c)s_2 \not\equiv 0 \text{ (modulo }N\text{)}}^{\frac{N-1}{2}}  [(d-c)s_1+(d+c)s_2]^{-1} \\&= 
\sum_{s_1,s_2=1 \atop (d-c)s_1+(d+c)s_2 \not\equiv 0 \text{ (modulo }N\text{)}}^{\frac{N-1}{2}}  [(d-c)s_1+(d+c)s_2]^{-1} \\& - \frac{1}{2}\sum_{s_1=1}^{N-1}\sum_{s_2=1 \atop (d-c)s_1+(d+c)s_2 \not\equiv 0 \text{ (modulo }N\text{)}}^{\frac{N-1}{2}}  [(d-c)s_1+(d+c)s_2]^{-1} 
\\& =-\frac{N-1}{2}\cdot \delta + \frac{1}{2}\cdot \sum_{s_2=1}^{\frac{N-1}{2}} [(d+c)s_2]^{-1}  + \sum_{s_1,s_2=1 \atop (d-c)s_1+(d+c)s_2 \not\equiv 0 \text{ (modulo }N\text{)}}^{\frac{N-1}{2}}  [(d-c)s_1+(d+c)s_2]^{-1} 
  \end{align*}
 where $\delta = \sum_{x \in (\mathbf{Z}/N\mathbf{Z})^{\times}/\pm1}[x]$. We thus get
\begin{align*}
G_{\infty}(\Gamma_1(N) \cdot \gamma) &=\left(\frac{1}{2}-\frac{N-1}{2}\right)\cdot \delta   + \sum_{s_1,s_2=1 \atop (d-c)s_1+(d+c)s_2 \not\equiv 0 \text{ (modulo }N\text{)}}^{\frac{N-1}{2}}  [(d-c)s_1+(d+c)s_2]^{-1} \text{ .}
\end{align*}

Similarly, noting $(d-c)s_1+(d+c)s_2 \equiv 0 \text{ (modulo }N\text{)}$ implies $(b-a)s_1+(b+a)s_2 \equiv \frac{2\cdot s_2}{d-c} \text{ (modulo }N\text{)}$, we get:
\begin{align*}
G_0(\Gamma_1(N) \cdot \gamma) = -\frac{1}{2} \cdot \delta +  \sum_{s_1, s_2=1 \atop (d-c)s_1+(d+c)s_2 \equiv 0 \text{ (modulo }N\text{)}}^{\frac{N-1}{2}}  \left[\frac{2}{d-c}\cdot s_2\right] \text{ .}
 \end{align*}

Note that for all $[c,d] \in (\mathbf{Z}/N\mathbf{Z})^2 \backslash \{(0,0)\} / \pm1$, we have $G_{\infty}([-d,c]) = -G_{\infty}([c,d])$ and $G_0([-d,c]) = -G_0([c,d])$. Recall also that $\xi_{\Gamma_1(N)}([-d,c]) = - \xi_{\Gamma_1(N)}([c,d])$ and that $\xi_{\Gamma_1(N)}([1,1])=0$. Let $R_1$ be the set of equivalence classes in $(\mathbf{Z}/N\mathbf{Z})^2 \backslash \{(0,0)\} / \pm1$ under the equivalence relation $[c,d] \sim [-d,c]$. We thus have:
\begin{align*}
\mathcal{U} &= \sum_{[c,d] \in R_1 \atop [c,d]\not\sim [1,1]} \xi_{\Gamma_1(N)}([c,d]) \cdot \\& ( -\frac{N-1}{2}\cdot \delta + \sum_{s_1, s_2=1 \atop (d-c)s_1+(d+c)s_2 \neq 0 \text{ (modulo }N\text{)}}^{\frac{N-1}{2}}  [(d-c)s_1+(d+c)s_2]^{-1} \\& + \sum_{s_1, s_2=1 \atop (d-c)s_1+(d+c)s_2 \equiv 0 \text{ (modulo }N\text{)}}^{\frac{N-1}{2}}  \left[\frac{2}{d-c}\cdot s_2\right] )  \text{ .}
\end{align*}

By Lemma \ref{even_modSymb_reecriture_m_0^+} and the definition of $\mathcal{U}$, we have $\mathcal{U} \in J\cdot H_1(X_1(N), \cusps, \mathbf{Z}[(\mathbf{Z}/N\mathbf{Z})^{\times}/\pm 1])$, where as above $J$ is the augmentation ideal of $\mathbf{Z}[(\mathbf{Z}/N\mathbf{Z})^{\times}/\pm 1]$. Let $\beta : J/J^2 \rightarrow \mathbf{Z}/2^{t+1}\mathbf{Z}$ be the group homomorphism given by $[x]-[1]\mapsto \log(x)$. Let $$\beta_* : J\cdot H_1(X_1(N), \cusps, \mathbf{Z}[(\mathbf{Z}/N\mathbf{Z})^{\times}/\pm 1]) \rightarrow H_1(X_1(N), \cusps, \mathbf{Z}/2^{t+1}\mathbf{Z})$$ be the map induced by $\beta$. We have:

\begin{equation}\label{beta_U_2}
\beta_*(\mathcal{U})= \sum_{[c,d] \in R_1 \atop [c,d]\not\sim [1,1]} F_{1,2}'([c,d]) \cdot \xi_{\Gamma_1(N)}([c,d]) \text{ ,}
\end{equation}
where 
$F_{1,2}': (\mathbf{Z}/N\mathbf{Z})^2 \backslash \{(0,0)\}/\pm 1 \rightarrow \mathbf{Z}/2^{t+1}\mathbf{Z}$ is defined by

\begin{align*}
F_{1,2}'([c,d]) &= \sum_{s_1, s_2=1 \atop (d-c)s_1+(d+c)s_2 \equiv 0 \text{ (modulo }N\text{)}}^{\frac{N-1}{2}}   \log\left(\frac{2}{d-c}\cdot s_2\right)  \\& - \sum_{s_1, s_2=1 \atop (d-c)s_1+(d+c)s_2 \not\equiv 0 \text{ (modulo }N\text{)}}^{\frac{N-1}{2}}  \log((d-c)s_1+(d+c)s_2) \text{ .}
\end{align*}

Formula (\ref{beta_U_2}) makes sense because 
\begin{equation}\label{even_modSymb_inversion_F_1,2'}
F_{1,2}'([-d,c])=-F_{1,2}'([c,d]) \text{, } 
\end{equation}
so $F_{1,2}'([c,d]) \cdot \xi_{\Gamma_1(N)}([c,d])$ does not depend on the choice of $[c,d]$ in its equivalence class in $R_1$. By (\ref{even_modSymb_crucial_Hecke}), for all prime $\ell$ not dividing $2\cdot N$ we have in $H_1(X_1(N), \cusps, \mathbf{Z}/2^{t+1}\mathbf{Z})$:
\begin{equation}\label{even_modSymb_crucial_Hecke_p=2_U}
(T_{\ell}-\ell-1)(\beta_*(\mathcal{U})) = 6 \cdot \frac{\ell-1}{2} \cdot \log(\ell) \cdot \mathcal{E}' \text{ .}
\end{equation}
Note that for all $\lambda \in (\mathbf{Z}/N\mathbf{Z})^{\times}$ and $[c,d] \in (\mathbf{Z}/N\mathbf{Z})^2 \backslash \{(0,0)\}/\pm 1$, we have 
\begin{equation}\label{even_modSymb_scaling_F_1,2'}
F_{1,2}'([\lambda \cdot c, \lambda \cdot d]) = F_{1,2}'([c,d]) \text{ .}
\end{equation}
Thus $F_{1,2}'$ induces the map $F_{1,2} : \mathbf{P}^1(\mathbf{Z}/N\mathbf{Z}) \rightarrow \mathbf{Z}/2^{t+1}\mathbf{Z}$ defined in the statement of Theorem \ref{thm_Introduction_w_1^+_2}. Let $$\mathcal{V} =  \sum_{x \in \mathcal{R}} F_{1,2}(x) \cdot \xi_{\Gamma_0(N)}(x) \in H_1(X_0(N), \cusps, \mathbf{Z}/2^{t+1}\mathbf{Z}) \text{ .}$$
One easily sees that for all  $[c,d] \in (\mathbf{Z}/N\mathbf{Z})^2 \backslash \{(0,0)\}/\pm 1$, we have 
\begin{equation}\label{even_modSymb_parity_F_1,2'}
F_{1,2}'([-c,d]) = F_{1,2}'([c,d]) \text{ .}
\end{equation}
Thus, we have $\mathcal{V} \in  H_1(X_0(N), \cusps, \mathbf{Z}/2^{t+1}\mathbf{Z})_+$.

Let $\zeta_4$ be a primitive fourth root in $(\mathbf{Z}/N\mathbf{Z})^{\times}$. For all $\lambda,x \in (\mathbf{Z}/N\mathbf{Z})^{\times}$, we have (using the Manin relations) in $H_1(X_1(N), \cusps, \mathbf{Z})$:
\begin{equation}\label{even_modSymb_Manin_zeta_4_Gamma_1_1}
\xi_{\Gamma_1(N)}([x, x\cdot \zeta_4]) = -\xi_{\Gamma_1(N)}([x\cdot \zeta_4, -x]) 
\end{equation}
and
\begin{equation}\label{even_modSymb_Manin_zeta_4_Gamma_1_2}
\xi_{\Gamma_1(N)}([x, -x\cdot \zeta_4]) = -\xi_{\Gamma_1(N)}([x\cdot \zeta_4, x]) \text{ .}
\end{equation}
We also have, in $H_1(X_0(N), \cusps, \mathbf{Z})$:
\begin{equation}\label{even_modSymb_Manin_zeta_4_Gamma_0}
\xi_{\Gamma_0(N)}([\zeta_4]) = \xi_{\Gamma_0(N)}([-\zeta_4]) = 0 \text{ .}
\end{equation}
By (\ref{even_modSymb_Manin_zeta_4_Gamma_1_1}) and (\ref{even_modSymb_Manin_zeta_4_Gamma_1_2}), the element 
$$\alpha := \sum_{x \in (\mathbf{Z}/N\mathbf{Z})^{\times}/\langle \zeta_4 \rangle} \xi_{\Gamma_1(N)}([x, x\cdot \zeta_4])+ \xi_{\Gamma_1(N)}([x, -x \cdot \zeta_4])$$
is well-defined in $H_1(X_1(N), \cusps, \mathbf{Z}/2\mathbf{Z})$, where $\langle \zeta_4 \rangle$ is the subgroup generated by $\zeta_4$. Thus, the element $2^t\cdot \alpha$ is well-defined in $H_1(X_1(N), \cusps, \mathbf{Z}/2^{t+1}\mathbf{Z})$. We have $$2^t\cdot \alpha \in (1+c)\cdot H_1(X_1(N), \cusps, \mathbf{Z}/2^{t+1}\mathbf{Z}) \text{ , }$$ where $c$ is the complex conjugation. 

By (\ref{even_modSymb_inversion_F_1,2'}), we have $F_{1,2}'([1, \zeta_4])=-F_{1,2}'([\zeta_4,-1])$. Since $[\zeta_4,-1] = [\zeta_4\cdot 1, \zeta_4 \cdot \zeta_4]$, by (\ref{even_modSymb_scaling_F_1,2'}) we also have $F_{1,2}'([\zeta_4,-1]) = F_{1,2}'([1, \zeta_4])$. Thus, we have $2 \cdot F_{1,2}'([1, \zeta_4]) = 0$, \ie $F_{1,2}'([1, \zeta_4]) \equiv 0 \text{ (modulo }2^t \text{)}$. Thus, there exists $K_2 \in \mathbf{Z}/2\mathbf{Z}$ such that for all $x \in (\mathbf{Z}/N\mathbf{Z})^{\times}$, we have 
\begin{equation}\label{even_modSymb_F_12_zeta_4_1}
F_{1,2}'([x, x\cdot \zeta_4]) = K_2 \cdot 2^t \text{ .}
\end{equation}
By (\ref{even_modSymb_parity_F_1,2'}), for all $x \in (\mathbf{Z}/N\mathbf{Z})^{\times}$, we have 
\begin{equation}\label{even_modSymb_F_12_zeta_4_2}
F_{1,2}'([x, -x\cdot \zeta_4]) = K_2 \cdot 2^t \text{ .}
\end{equation}

For any integer $r \geq 1$, let $\pi_r^* : H_1(X_0(N), \cusps, \mathbf{Z}/2^{r}\mathbf{Z}) \rightarrow H_1(X_1(N), \cusps, \mathbf{Z}/2^{r}\mathbf{Z})$ be the pull-back map. By construction and by (\ref{even_modSymb_Manin_zeta_4_Gamma_0}), we have
\begin{equation}\label{even_modSymb_pull_back_U_computation}
\pi_{t+1}^*(\mathcal{V}) = \beta_*(\mathcal{U})+K_2 \cdot 2^t\cdot \alpha \text{ .}
\end{equation}

\begin{lem}\label{even_modSymb_alpha_Trace}
We have, in $H_1(X_1(N), \mathbf{Z}/2^{t+1}\mathbf{Z})$:
$$2^t\cdot \alpha = \pi_{t+1}^*\left(\sum_{x \in \mathcal{R} \atop x \not\sim 1} \log\left(\frac{x+1}{x-1}\right) \cdot \xi_{\Gamma_0(N)}(x)\right) \text{ .} $$
\end{lem}
\begin{proof}
Recall the notation of Lemma \ref{even_modSymb_duality_lemma_RUI}, applied to $\Gamma = \Gamma_1(N)$. We identify as before $\Gamma_1(N) \backslash \PSL_2(\mathbf{Z})$ with $(\mathbf{Z}/N\mathbf{Z})^2\backslash \{(0,0)\}/\pm1$. We let $\mathcal{R}_1 \subset \mathcal{R}_{\Gamma_1(N)}$ be the subset of equivalence classes of elements $[c,d]$ with $c \cdot d \not\equiv 0 \text{ (modulo }N\text{)}$. We have, in $H_1(X_1(N), \cusps, \mathbf{Z}/2^{t+1}\mathbf{Z})$:
\begin{equation}\label{even_modSymb_alpha_Trace_eq1}
\pi_{t+1}^*\left(\sum_{x \in \mathcal{R} \atop x \not\sim 1} \log\left(\frac{x+1}{x-1}\right) \cdot \xi_{\Gamma_0(N)}(x)\right) = 2^t\cdot \alpha +  \sum_{[c,d] \in \mathcal{R}_1} \log\left(\frac{d-c}{d+c}\right) \cdot \xi_{\Gamma_1(N)}([c,d]) \text{ .} 
\end{equation}
To conclude the proof of Lemma \ref{even_modSymb_alpha_Trace}, it suffices to prove the following result.
\begin{lem}\label{even_modSymb_alpha_Trace_vanishing}
We have, in $H_1(X_1(N), \mathbf{Z}/2^{t+1}\mathbf{Z})$:
$$\sum_{[c,d] \in \mathcal{R}_1} \log\left(\frac{d-c}{d+c}\right) \cdot \xi_{\Gamma_1(N)}([c,d]) = 0 \text{ .}$$
\end{lem}
\begin{proof}
Consider the morphism $h : \mathbf{Z}[\Gamma_1(N) \backslash \PSL_2(\mathbf{Z})]\rightarrow \mathbf{Z}/2^{t+1}\mathbf{Z}$ given by $[c,d] \mapsto \log\left(\frac{c}{d}\right)$ if $c \cdot d \not\equiv 0 \text{ (modulo }N\text{)}$ and $[c,d]\mapsto 0$ else. This factors through $\xi_{\Gamma_1(N)}$, and we let $g : H_1(X_1(N),\cusps, \mathbf{Z}) \rightarrow \mathbf{Z}/2^{t+1}\mathbf{Z}$ be the induced map. Let $f : H_1(X_1(N),\mathbf{Z}) \rightarrow \mathbf{Z}/2^{t+1}\mathbf{Z}$ be the restriction of $g$ to $H_1(X_1(N),\mathbf{Z})$, and $\hat{f} \in H_1(X_1(N), \mathbf{Z}/2^{t+1}\mathbf{Z})$ be the element corresponding to $f$ by intersection duality. By Lemma \ref{even_modSymb_duality_lemma_RUI}, we have:
\begin{align*}
\hat{f} &= \sum_{[c,d] \in \mathcal{R}_1} \left( \log\left(\frac{c}{d} \right)+ \frac{2}{3} \cdot \log\left( \frac{d}{-(c+d)}\right) + \frac{1}{3} \cdot \log\left( \frac{-(c+d)}{c}\right) - \frac{2}{3} \cdot \log\left( \frac{c}{d-c}\right)-\frac{1}{3} \cdot \log\left( \frac{c-d}{d}\right) \right) \\& \cdot \xi_{\Gamma_1(N)}([c,d]) \\& 
=\frac{1}{3}\cdot  \sum_{[c,d] \in \mathcal{R}_1}  \log\left(\frac{d-c}{d+c}\right) \cdot \xi_{\Gamma_1(N)}([c,d]) \text{ .}
\end{align*}

To prove that $\hat{f}=0$, it suffices to prove that $f = 0$. Let $$\sum_{[c,d] \in (\mathbf{Z}/N\mathbf{Z})^2 \backslash \{(0,0)\}/\pm 1 \atop c\cdot d \not\equiv 0 \text{ (modulo }N\text{)}} \lambda_{[c,d]} \cdot \xi_{\Gamma_1(N)}([c,d]) \in H_1(X_1(N), \mathbf{Z}) \text{ .}$$ We have (\cf Section \ref{odd_modSymb_section_hida}), in $\mathbf{Z}[(\mathbf{Z}/N\mathbf{Z})^{\times}/\pm 1]$:
$$\sum_{[c,d] \in (\mathbf{Z}/N\mathbf{Z})^2 \backslash \{(0,0)\}/\pm 1 \atop c\cdot d \not\equiv 0 \text{ (modulo }N\text{)}} \lambda_{[c,d]}\cdot ([c]-[d]) = 0 \text{ .}$$
Thus, we have
$$f\left( \sum_{[c,d] \in (\mathbf{Z}/N\mathbf{Z})^2 \backslash \{(0,0)\}/\pm 1 \atop c\cdot d \not\equiv 0 \text{ (modulo }N\text{)}} \lambda_{[c,d]} \cdot \xi_{\Gamma_1(N)}([c,d])\right) =  \sum_{[c,d] \in (\mathbf{Z}/N\mathbf{Z})^2 \backslash \{(0,0)\}/\pm 1 \atop c\cdot d \not\equiv 0 \text{ (modulo }N\text{)}} \lambda_{[c,d]}\cdot \left( \log(c)-\log(d) \right)= 0 \text{ .}$$
This concludes the proof of Lemma \ref{even_modSymb_alpha_Trace_vanishing}.
\end{proof}
This concludes the proof of Lemma \ref{even_modSymb_alpha_Trace}.
\end{proof}
By (\ref{even_modSymb_pull_back_U_computation}), Lemma \ref{even_modSymb_alpha_Trace} and (\ref{even_modSymb_crucial_Hecke_p=2_U}), for all prime $\ell$ not dividing $2\cdot N$, we have in $H_1(X_1(N), \cusps, \mathbf{Z}/2^{t+1}\mathbf{Z})$:
$$
\pi_{t+1}^* \left( (T_{\ell}-\ell-1)\left(\mathcal{V} - K_2\cdot \sum_{x \in \mathcal{R} \atop x \not\sim 1} \log\left(\frac{x+1}{x-1}\right) \cdot \xi_{\Gamma_0(N)}(x) \right)- 6\cdot \frac{\ell-1}{2}\cdot \log(\ell)\cdot m_0^+  \right) = 0 \text{ .}
$$ 
Thus, we have in $H_1(X_1(N), \cusps, \mathbf{Z}/2^{t+1}\mathbf{Z})$:
\begin{equation}\label{even_modSymb_crucial_Hecke_trace_gamma_V_U}
\pi_{t+1}^*\left( (T_{\ell}-\ell-1)\left(\mathcal{V} - 6\cdot m_1^+ - K_2\cdot \sum_{x \in \mathcal{R} \atop x \not\sim 1, 0} \log\left(\frac{x+1}{x-1}\right) \cdot \xi_{\Gamma_0(N)}(x)  \right) \right) =0 \text{ .}
\end{equation}
By (\ref{even_modSymb_crucial_Hecke_trace_gamma_V_U}) and Proposition \ref{even_modSymb_injectivity_trace} (i), for all prime $\ell$ not dividing $2\cdot N$, we have in $H_1(X_0(N), \cusps, \mathbf{Z}/2^{t+1}\mathbf{Z})$:
$$
(1+c)\cdot (T_{\ell}-\ell-1)\left(\mathcal{V} - 6\cdot m_1^+\right) = 0 \text{ ,}
$$
where $c$ is the complex conjugation. By (\ref{even_modSymb_parity_F_1,2'}), we have $(1-c) \cdot (T_{\ell}-\ell-1)\left(\mathcal{V} - 6\cdot m_1^+\right) = 0$. Thus, for all prime $\ell$ not dividing $2\cdot N$, we have in $H_1(X_0(N), \cusps, \mathbf{Z}/2^t\mathbf{Z})$:
$$(T_{\ell}-\ell-1)\left(\mathcal{V} - 6\cdot m_1^+\right) = 0 \text{ .}$$
Thus, there exists a constant $C_2$, uniquely defined modulo $2$, such that we have:
$$
6 \cdot m_1^+ \equiv C_2\cdot m_0^+ + \mathcal{V} \text{ in } \left(M_+/2^{t}\cdot M_+\right)/\mathbf{Z}\cdot 2\cdot m_0^+ \text{ .}
$$
To conclude the proof of Theorem \ref{thm_Introduction_w_1^+_2}, it suffices to prove that $m_0^+ \equiv \mathcal{V}$ modulo $2$. This follows from Theorem \ref{thm_Introduction_w_1^+_2} and Lemma \ref{even_modSymb_magical_identity_Gauss}.
\end{proof}

\section{Interplay between higher Eisenstein elements and connection with Galois deformations}\label{Section_comparison}
In this chapter, we relate the various Hecke modules (and their higher Eisenstein elements) that were considered in chapters \ref{Section_Supersingular}, \ref{Section_odd_modSymb} and \ref{Section_even_modSymb}. We assume as usual that $p$ divides the numerator of $\frac{N-1}{12}$ (we allow $p=2$ and $p=3$). We keep the notation of chapters \ref{Section_introduction}, \ref{Section_Formalism}, \ref{Section_odd_modSymb} and \ref{Section_even_modSymb}. In particular $M=\mathbf{Z}_p[S]$ is the supersingular module, $M_+ = H_1(X_0(N), \cusps, \mathbf{Z}_p)_+$ and $M^- = H_1(Y_0(N), \mathbf{Z}_p)^-$. Recall also that we let $\nu = \gcd(N-1,12)$. We fix an integer $r$ such that $1 \leq r \leq t$. We denote by $f_0$, $f_1$, ..., $f_{n(r,p)}$ the higher Eisenstein elements if $\mathcal{M}/p^r\cdot \mathcal{M}$ (where $\mathcal{M}$ was defined in section \ref{even_modSymbidea}). The $q$-expansion of an element $f \in \mathcal{M}$ at the cusp $\infty$ is denoted as usual by $\sum_{n \geq 0} a_n(f) \cdot q^n$.

\subsection{Higher Eisenstein elements in the module $\mathcal{M}$ of modular forms}

\begin{prop}\label{Comparison_T_0}
\begin{enumerate}
\item We have, for all $m\in M$:
$$T_0(m)= \frac{12}{\nu} \cdot (m\bullet \tilde{e}_0)\cdot \tilde{e}_0 \text{ .}$$
\item We have, for all $m_+\in M_+$:
$$T_0(m_+)=\frac{12}{\nu}\cdot (m_+\bullet \tilde{m}_0^-)\cdot \tilde{m}_0^+ \text{ .}$$
\item We have, for all $m^-\in M^-$:
$$T_0(m^-)=\frac{12}{\nu}\cdot (\tilde{m}_0^+\bullet m^-)\cdot \tilde{m}_0^- \text{ .}$$
\end{enumerate}
\end{prop}
\begin{proof}
We have
\begin{equation}\label{comparison_T_0_I=0}
\tilde{I} \cdot T_0=0
\end{equation}
since the Hecke operators of $\tilde{I} \cdot T_0$ annihilate all the modular forms of weight $2$ and level $\Gamma_0(N)$.

We first prove (i). By (\ref{comparison_T_0_I=0}), for all $m \in M$ the element $T_0(m)$ is annihilated by $\tilde{I}$, so is proportional to $\tilde{e}_0$. Thus, for all $m \in M$ there exists $C_m \in \mathbf{Z}_p$ such that
$$T_0(m) = C_m \cdot \tilde{e}_0 \text{ .}$$
We thus have, in $\mathbf{Z}_p$:
\begin{equation}\label{comparison_equation_C_m}
C_m \cdot (\tilde{e}_0 \bullet \tilde{e}_0) = T_0(m) \bullet \tilde{e}_0 = m \bullet T_0(\tilde{e}_0) \text{ .}
\end{equation}

By (\ref{comparison_T_0-n in I}), we have $T_0(\tilde{e}_0) = \frac{N-1}{\nu} \cdot \tilde{e}_0$. Recall that by Eichler mass formula, we have $\tilde{e}_0 \bullet \tilde{e}_0 = \frac{N-1}{12}$. 
By (\ref{comparison_equation_C_m}), we get:
$$C_m = \frac{12}{\nu} \cdot (m \bullet \tilde{e}_0) \text{ .}$$

We now prove (ii). As above, for all $m_+ \in M_+$ there exists $C_{m_+} \in \mathbf{Z}_p$ such that 
$$T_0(m_+) = C_{m_+} \cdot \tilde{m}_0^+ \text{ .}$$
We thus have, in $\mathbf{Z}_p$:
\begin{equation}\label{comparison_equation_C_m^+}
C_{m_+} \cdot (\tilde{m}_0^+ \bullet \tilde{m}_0^-) = T_0(m_+) \bullet \tilde{m}_0^- = m_+ \bullet T_0(\tilde{m}_0^-) \text{ .}
\end{equation}
By (\ref{comparison_T_0-n in I}), we have $T_0(\tilde{m}_0^-) = \frac{N-1}{\nu} \cdot \tilde{m}_0^-$. Recall that we have $\tilde{m}_0^+ \bullet \tilde{m}_0^- = \frac{N-1}{12}$. 
By (\ref{comparison_equation_C_m^+}), we get:
$$C_{m_+} = \frac{12}{\nu} \cdot (m_+\bullet \tilde{m}_0^-) \text{ .}$$
The proof of (iii) is similar.
\end{proof}

Let $M_1$ and $M_2$ be $\tilde{\mathbb{T}}$-modules equipped with a $\tilde{\mathbb{T}}$-equivariant bilinear pairing $\bullet: M_1 \times M_2 \rightarrow \mathbf{Z}_p$. By (\ref{comparison_relation_a_1_a_0_T_0}), for any $m_1 \in M_1$ and $m_2 \in M_2$ we have an element of $\mathcal{M}$ whose $q$-expansion at the cusp $\infty$ is
\begin{equation}\label{comparison_pairing_modular_form}
\frac{\nu}{24}\cdot \left(m_1\bullet T_0(m_2)\right) + \sum_{n \geq 1} \left(m_1\bullet T_n(m_2)\right) \cdot q^n \text{ .}
\end{equation}
We apply this remark to the cases $(M_1,M_2)=(M,M)$ and $(M_1,M_2)=(M_+,M^-)$. We choose the pairing $\bullet$ already described in the previous chapters for these couples.

\begin{prop}\label{Comparison_comparison_tensor}
Let $x \in M$ (resp. $x_+ \in M_+$, resp. $x^- \in M^-$) be such that $\tilde{e}_0 \bullet x = 1$ (resp. $x_+ \bullet \tilde{m}_0^-=1$, resp. $\tilde{m}_0^+ \bullet x^- = 1$). For all integers $i$ such that $0 \leq i \leq n(r,p)$, let $$E_i = \frac{e_i\bullet e_0}{2} +\sum_{n\geq 1} (e_i \bullet T_n(x)) \cdot q^n \text{, }$$ $$E_i^+ = \frac{m_i^+ \bullet m_0^-}{2} + \sum_{n\geq 1} (m_i^+\bullet T_n(x^-) )\cdot q^n $$ and $$E_i^- =\frac{m_0^+ \bullet m_i^-}{2}+ \sum_{n\geq 1} (T_n(x_+) \bullet m_i^-)\cdot q^n$$
in $\left(\frac{1}{2}\mathbf{Z}\oplus q\cdot \mathbf{Z}[[q]]\right) \otimes_{\mathbf{Z}} \mathbf{Z}/p^r\mathbf{Z}$. Note that the image of $E_i$ in $$\left(\left(\frac{1}{2}\mathbf{Z}\oplus q\cdot \mathbf{Z}[[q]]\right) \otimes_{\mathbf{Z}} \mathbf{Z}/p^r\mathbf{Z}\right)/\left(\mathbf{Z}\cdot E_1+ ... + \mathbf{Z}\cdot E_{i-1}\right)$$ is uniquely determined, and similarly for $E_i^+$ and $E_i^-$.

For all integers $i$ such that $0 \leq i \leq n(r,p)$, we have:
$$\sum_{n \geq 0} a_n(f_i) \cdot q^n \equiv E_i \equiv E_i^+ \equiv E_i^- \text{ in } \left(\left(\frac{1}{2}\mathbf{Z}\oplus q\cdot \mathbf{Z}[[q]]\right) \otimes_{\mathbf{Z}} \mathbf{Z}/p^r\mathbf{Z}\right)/\left(\mathbf{Z}\cdot E_1+ ... + \mathbf{Z}\cdot E_{i-1}\right) \text{ .}$$
\end{prop}
\begin{proof}
By (\ref{comparison_relation_a_1_a_0_T_0}), the image of the group homomorphism $\mathcal{M} \rightarrow \mathbf{Q}_p$ given by $f \mapsto a_0(f)$ is contained in $\frac{\nu}{24}\cdot \mathbf{Z}$. Thus, the $q$-expansion at the cusp $\infty$ gives an injective group homomorphism $\iota : \mathcal{M}/p^r\cdot \mathcal{M} \hookrightarrow \left(\frac{1}{2}\mathbf{Z}\oplus q\cdot \mathbf{Z}[[q]]\right) \otimes_{\mathbf{Z}} \mathbf{Z}/p^r\mathbf{Z}$.  The fact that $E_i$, $E_i^+$ and $E_i^-$ lie in $\iota\left( \mathcal{M}/p^r\cdot \mathcal{M}\right)$ follows from Proposition \ref{Comparison_T_0} and (\ref{comparison_pairing_modular_form}). We abuse notation and consider $E_i$, $E_i^+$ and $E_i^-$ as elements of $\mathcal{M}/p^r\cdot \mathcal{M}$.

One easily sees that we have $E_0 = E_0^+ = E_0^- = f_0$. Furthermore, for all prime $\ell$ not dividing $2\cdot N$ and all $1 \leq i \leq n(r,p)$, we have $$(T_{\ell}-\ell-1)(E_i) \equiv \frac{\ell-1}{2}\cdot \log(\ell) \cdot E_{i-1} \text{ in  } \left( \mathcal{M}/p^r\cdot \mathcal{M} \right)/\left( \mathbf{Z}\cdot E_0 + ... + \mathbf{Z}\cdot E_{i-1} \right) \text{ .}$$
Proposition \ref{Comparison_comparison_tensor} follows immediately by induction on $i$.
\end{proof}

By comparing the constant coefficients of the various modular forms of Proposition \ref{Comparison_comparison_tensor}, we get the following comparison result, which is Theorem \ref{thm_Introduction_comparison_pairings} if $r=1$.

\begin{corr}\label{Comparison_corr_comparison}
For all $0 \leq i,j\leq n(r,p)$ such that $i+j\leq n(r,p)$, we have in $\mathbf{Z}/p^r\mathbf{Z}$:
$$e_i \bullet e_j = m_i^+ \bullet m_j^{-} \text{.}$$
\end{corr}

By Corollary \ref{Formalism_criterion_pairing}, the common quantity of Corollary \ref{Comparison_corr_comparison} depends only on $i+j$, is $0$ if $i+j<n(r,p)$ and is non zero if $i+j=n(r,p)$.

\subsection{Computation of $m_i^+ \bullet m_j^-$ when $i+j=1$}\label{comparison_section_computation_i+j=1}

By Corollary \ref{Comparison_corr_comparison}, we have in $\mathbf{Z}/p^r\mathbf{Z}$:
\begin{equation}\label{comparison_eq_e_0_e_1=m_0_m_1}
e_0 \bullet e_1 = e_1 \bullet e_0 = m_1^+ \bullet m_0^- = m_0^+ \bullet m_1^- \text{ .}
\end{equation}

Theorem \ref{Supersingular_w_1_S_5} shows that if $p\geq 5$ then $e_1\bullet e_0 = \frac{1}{3} \cdot \sum_{k=1}^{\frac{N-1}{2}} k \cdot \log(k)$. In fact, one can compute $m_0^+ \bullet m_1^-$ directly even for $p \leq 3$, using Merel's work.

\begin{thm}\label{Comparison_Merel_sqrt_u}
We have 
$$m_0^+ \bullet m_1^- = -\frac{1}{12}\cdot \log\left(\epsilon\cdot \zeta \cdot \prod_{k=1}^{\frac{N-1}{2}} k^{-4\cdot k} \right) \in \mathbf{Z}/p^r\mathbf{Z}$$
where $\epsilon=1$ if $N \not\in 1+8\mathbf{Z}$, $\epsilon=-1$ if $N \in 1+8\mathbf{Z}$, $\zeta=1$ if $N \not\in 1+3\mathbf{Z}$ and $\zeta = 2^{\frac{N-1}{3}}$ if $N \in 1+3\mathbf{Z}$.

If $p=2$, $\epsilon\cdot \zeta \cdot \prod_{k=1}^{\frac{N-1}{2}} k^{-4\cdot k} \equiv x^4 \text{ (modulo }N\text{)}$ is a $4$th power of $(\mathbf{Z}/N\mathbf{Z})^{\times}$, and the meaning of the right-hand side is $-\frac{1}{3}\cdot \log(x)$, which is
$$-\frac{1}{3}\cdot (2^{t-1}-\sum_{k=1}^{\frac{N-1}{2}} k\cdot \log(k)) \text{ .}$$
If $p=3$, $\epsilon\cdot \zeta \cdot \prod_{k=1}^{\frac{N-1}{2}} k^{-4\cdot k} \equiv x^3 \text{ (modulo }N\text{)}$ is a $3$rd power of $(\mathbf{Z}/N\mathbf{Z})^{\times}$, and the meaning of the right-hand side is $-\frac{1}{4}\cdot \log(x)$, which is
$$-\frac{1}{4}\sum_{k=1}^{N-1} k^2\cdot \log(k) \text{ .}$$
\end{thm}
\begin{proof}
This follows from \cite[Th\'eor\`eme 1]{Merel_accouplement}. The last assertion in the case $p=3$ follows from Lemma \ref{even_modSymb_computation_square}.
\end{proof}

\begin{corr}\label{comparison_corr_g_p_2}
Assume $p=2$. We have $n(r,p) \geq 2$ if and only if $\sum_{k=1}^{\frac{N-1}{2}} k\cdot \log(k) \equiv 2^{t-1} \text{ (modulo }2^r\text{)}$.
\end{corr}

\begin{rem}\label{comparison_CE_p=2_g_2}
In the case $r=1$, $g_2=n(1,p)$ was determined completely in terms of the class group of $\mathbf{Q}(\sqrt{-N})$ in \cite[Theorem 1.1 and p.133]{Calegari_Emerton}: we have $g_2=2^{m-1}-1$ where $m$ is the $2$-adic valuation of the order of this  class group.
\end{rem}

\begin{corr}\label{comparison_corr_g_p_3}
Assume $p=3$. We have $n(r,p) \geq 2$ if and only if $\sum_{k=1}^{N-1} k^2\cdot \log(k) \equiv 0 \text{ (modulo }3^r\text{)}$.
\end{corr}

\subsection{Computation of $m_i^+ \bullet m_j^-$ when $i+j \leq 3$ and $p \geq 5$}\label{comparison_section_various_pairings_>5}
In all this section, we assume $p\geq 5$. Theorem \ref{thm_Introduction_g_p>2} is a consequence of the following result.

\begin{thm}\label{even_modSymb_computation_pairing_>3}
Assume that we have $1 \leq r \leq t$ and $n(r,p) \geq 2$ (\ie $\sum_{k=1}^{\frac{N-1}{2}} k \cdot \log(k) \equiv 0 \text{ (modulo }p^r\text{)}$). We have:
$$m_1^+ \bullet m_1^- \equiv \frac{1}{6}\sum_{k=1}^{\frac{N-1}{2}} k \cdot \log(k)^2 \text{ (modulo }p^r\text{).}$$
\end{thm}
\begin{proof}
By Theorems \ref{thm_Introduction_w_1^+>3_p^t} and \ref{odd_even_modSymb_determination_m_1^-}, we have in $\mathbf{Z}/p^r\mathbf{Z}$:
\begin{align*}
24\cdot  m_1^+ \bullet m_1^- &= \sum_{x \in (\mathbf{Z}/N\mathbf{Z})^{\times} \atop x\neq 1 }\log(x) \cdot \left(\sum_{(s_1,s_2)\in (\mathbf{Z}/2N\mathbf{Z})^2 \atop  (1-x)s_1 + (1+x)s_2 \equiv 0 \text{ (modulo }N\text{)}} (-1)^{s_1+s_2} \B_1\left(\frac{s_1}{2N}\right)\B_1\left(\frac{s_2}{2N}\right) \cdot \log\left(\frac{s_2}{1-x}\right) \right)  
\\& - \log(x) \cdot \left( \sum_{(s_1,s_2)\in (\mathbf{Z}/2N\mathbf{Z})^2 \atop  (1-x)s_1 + (1+x)s_2 \not\equiv 0 \text{ (modulo }N\text{)}} (-1)^{s_1+s_2} \B_1\left(\frac{s_1}{2N}\right)\B_1\left(\frac{s_2}{2N}\right) \cdot \log((1-x)s_1 + (1+x)s_2) \right)
\\& = \frac{1}{4}\sum_{x \in (\mathbf{Z}/N\mathbf{Z})^{\times} \atop x\neq 1 } \log(x) \cdot \left(\sum_{(s_1,s_2)\in (\mathbf{Z}/N\mathbf{Z})^2 \atop  (1-x)s_1 + (1+x)s_2 = 0} D_2\left(\frac{s_1}{N}\right)D_2\left(\frac{s_2}{N}\right)\cdot \log\left(\frac{s_2}{1-x}\right) \right)   
\\& -\log(x) \cdot \left(\sum_{(s_1,s_2)\in (\mathbf{Z}/N\mathbf{Z})^2 \atop  (1-x)s_1 + (1+x)s_2 \neq 0} D_2\left(\frac{s_1}{N}\right)D_2\left(\frac{s_2}{N}\right) \cdot \log((1-x)s_1 + (1+x)s_2) \right)  
\\& =
\frac{1}{4}\sum_{(s_1,s_2) \in (\mathbf{Z}/N\mathbf{Z})^2 \atop s_1 \neq \pm s_2} D_2\left(\frac{s_1}{N}\right)D_2\left(\frac{s_2}{N}\right)\cdot (\log\left(\frac{s_1+s_2}{s_1-s_2}\right) \cdot \log\left(\frac{s_2-s_1}{2}\right) \\&- \sum_{x \neq 0, \frac{s_1+s_2}{s_1-s_2}} \log(x)\cdot \log((x-1)s_1+(x+1)s_2) ) \text{.}
\end{align*}
Since $n(r,p) \geq 2$, we have $\sum_{k=1}^{\frac{N-1}{2}} k \cdot \log(k) \equiv 0 \text{ (modulo }p^r\text{)}$. By Lemma \ref{even_modSymb_computation_D2}, we have:
$$\sum_{(s_1,s_2) \in (\mathbf{Z}/N\mathbf{Z})^2 \atop s_1 \neq \pm s_2} D_2\left(\frac{s_1}{N}\right)D_2\left(\frac{s_2}{N}\right)\cdot \log\left(\frac{s_1+s_2}{s_1-s_2}\right)  \equiv 0 \text{ (modulo }p^r\text{).}$$ 
Thus, we have:
\begin{align*}
96\cdot  m_1^+ \bullet m_1^- &= \sum_{(s_1,s_2) \in (\mathbf{Z}/N\mathbf{Z})^2 \atop s_1 \neq \pm s_2} D_2\left(\frac{s_1}{N}\right)D_2\left(\frac{s_2}{N}\right)\cdot (\log\left(\frac{s_1+s_2}{s_1-s_2}\right) \cdot \log(s_2-s_1) \\&- \sum_{x \neq 0, \frac{s_1+s_2}{s_1-s_2}} \log(x)\cdot \log((1-x)s_1+(1+x)s_2) ) \text{.}
\end{align*}

Lemma \ref{Bernardi_lemma>3} shows that 
\begin{align*}
 \sum_{x \neq 0, \frac{s_1+s_2}{s_1-s_2}} \log(x)\cdot \log((1-x)s_1+(1+x)s_2) = -\log(s_2-s_1)\cdot \log\left(\frac{s_1+s_2}{s_1-s_2}\right)  - \log\left(\frac{s_1+s_2}{s_1-s_2}\right) ^2 \text{ .}
\end{align*}
We conclude that
\begin{align*}
96 \cdot m_1^+\bullet m_1^- &= \sum_{(s_1,s_2) \in (\mathbf{Z}/N\mathbf{Z})^2 \atop s_1 \neq \pm s_2} D_2\left(\frac{s_1}{N}\right)D_2\left(\frac{s_2}{N}\right) \cdot \\& \left( \log\left(\frac{s_1+s_2}{s_1-s_2}\right) \log(s_2-s_1) + \log(s_2-s_1)\log\left(\frac{s_1+s_2}{s_1-s_2}\right)  + \log\left(\frac{s_1+s_2}{s_1-s_2}\right) ^2\right)
\\& =  \sum_{(s_1,s_2) \in (\mathbf{Z}/N\mathbf{Z})^2 \atop s_1 \neq \pm s_2} D_2\left(\frac{s_1}{N}\right)D_2\left(\frac{s_2}{N}\right) \cdot   \left( \log(s_1+s_2)^2 - \log(s_1-s_2)^2 \right) \\& =16 \cdot \sum_{k=1}^{\frac{N-1}{2}} k\cdot \log(k)^2
\end{align*}
where the last equality follows from Lemma \ref{even_modSymb_computation_D2_square}.
\end{proof}

Combining Corollary \ref{Comparison_corr_comparison}, Theorem \ref{Supersingular_w_1_S_5} (iii) and Theorem \ref{even_modSymb_computation_pairing_>3}, we get the following identity.
\begin{corr}\label{comparison_higher_Eichler_quadratic_formula}
Assume $n(r,p) \geq 2$, \ie $\sum_{k=1}^{\frac{N-1}{2}} k \cdot \log(k) \equiv 0 \text{ (modulo }p^r\text{)}$. We have, in $\mathbf{Z}/p^r\mathbf{Z}$:
$$\sum_{\lambda \in L} 3\cdot \log(H'(\lambda))^2 - 4\cdot \log(\lambda)^2 = 12\cdot \sum_{k=1}^{\frac{N-1}{2}} k \cdot \log(k)^2\text{ .}$$
\end{corr}

We now compute the pairings $m_0^+ \bullet m_2^-$ and $m_1^+ \bullet m_2^-$ using the formula for $m_2^-$ given in Theorem \ref{odd_modSymb_main_theorem_>5}. Recall that we defined function $F_{0,p}, F_{1,p} : \mathbf{P}^1(\mathbf{Z}/N\mathbf{Z}) \rightarrow \mathbf{Z}/p^r\mathbf{Z}$ in Theorems \ref{thm_Introduction_w_0^+} and \ref{thm_Introduction_w_1^+>3_p^t} ($F_{0,p}$ really takes values in $\mathbf{Z}_p$, but we abuse notation and consider it modulo $p^r$). Recall the group isomorphism $\delta_r : \mathbf{Z}/p^r\mathbf{Z} \rightarrow J\cdot \mathcal{K}_r/J^2\cdot \mathcal{K}_r$ defined in Section \ref{odd_modSymb_section_class_group}.

\begin{thm}\label{even_modSymb_computation_m_0_m_2}
Assume that Conjecture \ref{odd_modSymb_Sharifi_conj} holds. 
Assume that $n(r,p) \geq 2$, \ie that $\sum_{k=1}^{\frac{N-1}{2}} k \cdot \log(k) \equiv 0 \text{ (modulo }p^r\text{)}$. Let $g \in (\mathbf{Z}/N\mathbf{Z})^{\times}$ be a generator such that $\log(g) \equiv 1 \text{ (modulo }p^r\text{)}$.
We have:
\begin{align*}
\delta_r(m_0^+ \bullet m_2^-) &= \frac{1}{4}\sum_{i=1}^{N-1}  \left(\sum_{j=i}^{N-1} F_{0,p}(g^j) \right) \cdot \left(1-\zeta_N^{g^{i-1}}, \frac{1-\zeta_N}{1-\zeta_N^{g^{-1}}}\right)_r \text{ .}
\end{align*}
\end{thm}
\begin{proof}
We have, in $H_1(X_0(N), \cusps, \mathbf{Z}/p^r\mathbf{Z})$:
$$m_0^+ = \sum_{i=1}^{N-1} F_{0,p}(g^i)\cdot \xi_{\Gamma_0(N)}([g^i:1]) \text{ .}$$
The hypothesis $n(r,p) \geq 2$ can be rewritten as  
\begin{equation}\label{Comparison_F_0_p_sum_n(r,p)}
\sum_{i=1}^{N-1} i\cdot F_{0,p}(g^i) \equiv 0 \text{ (modulo }p^r\text{).}
\end{equation}
Thus, we have in $H_1(X_0(N), \cusps, \mathbf{Z}/p^r\mathbf{Z})$:
$$m_0^+ = \sum_{i=1}^{N-1} F_{0,p}(g^i) \cdot \left(  \xi_{\Gamma_0(N)}([g^i:1]) - i\cdot   \xi_{\Gamma_0(N)}([g:1]) \right) \text{ .}$$
We have 
\begin{align*}
\xi_{\Gamma_0(N)}([g^i:1]) - i\cdot   \xi_{\Gamma_0(N)}([g:1]) &=  \left( \xi_{\Gamma_0(N)}([g^i:1])  -  \xi_{\Gamma_0(N)}([g^{i-1}:1])- \xi_{\Gamma_0(N)}([g:1])\right) \\& +  \left( \xi_{\Gamma_0(N)}([g^{i-1}:1]) -  (i-1)\cdot  \xi_{\Gamma_0(N)}([g:1]) \right) \text{ .}
\end{align*}
By induction on $i$, we get (using $\xi_{\Gamma_0(N)}(\overline{1})=0$):
$$ \xi_{\Gamma_0(N)}([g^i:1]) - i\cdot   \xi_{\Gamma_0(N)}([g:1])  = \sum_{j=1}^i  \left( \xi_{\Gamma_0(N)}(g^j)  -  \xi_{\Gamma_0(N)}(g^{j-1})- \xi_{\Gamma_0(N)}([g:1])\right) \text{ .}$$

Thus, we have in $H_1(X_0(N), \cusps, \mathbf{Z}/p^r\mathbf{Z})_+$:
$$2\cdot m_0^+ = \sum_{i=1}^{N-1} \left(\sum_{j=i}^{N-1} F_{0,p}(g^j) \right) \cdot (1+c)\cdot \left(  \xi_{\Gamma_0(N)}([g^i:1]) -   \xi_{\Gamma_0(N)}([g^{i-1}:1]) - \xi_{\Gamma_0(N)}([g:1]) \right) \text{ ,}$$
where $c$ is the complex conjugation.
By Theorem \ref{odd_modSymb_main_theorem_>5}, we have:
\begin{align*}
(1+c)\cdot &\left(  \xi_{\Gamma_0(N)}([g^i:1]) -   \xi_{\Gamma_0(N)}([g^{i-1}:1]) - \xi_{\Gamma_0(N)}([g:1]) \right) \bullet m_2^- = \\& \frac{1}{2}\cdot \left((1-\zeta_N^{g^{i-1}}, 1-\zeta_N)_r-(1-\zeta_N^{g^{i-1}}, 1-\zeta_N^{g^{-1}})_r - (1-\zeta_N^{g^{-1}}, 1-\zeta_N)_r \right)
\end{align*}
By the bilinearity of $(\cdot, \cdot)_r$, we have: $$(1-\zeta_N^{g^{i-1}}, 1-\zeta_N)_r-(1-\zeta_N^{g^{i-1}}, 1-\zeta_N^{g^{-1}})_r = \left(1-\zeta_N^{g^{i-1}}, \frac{1-\zeta_N}{1-\zeta_N^{g^{-1}}}\right)_r \text{ .}$$
Using (\ref{Comparison_F_0_p_sum_n(r,p)}), this concludes the proof of Theorem \ref{even_modSymb_computation_m_0_m_2}.
\end{proof}

Since $m_1^+\bullet m_1^- = m_0^+ \bullet m_2^-$, Theorems \ref{even_modSymb_computation_pairing_>3} and \ref{even_modSymb_computation_m_0_m_2} give us the following identity.

\begin{corr}
Assume that Conjecture \ref{odd_modSymb_Sharifi_conj} holds. 
Assume that $n(r,p) \geq 2$, \ie that $\sum_{k=1}^{\frac{N-1}{2}} k \cdot \log(k) \equiv 0 \text{ (modulo }p^r\text{)}$. Let $g \in (\mathbf{Z}/N\mathbf{Z})^{\times}$ be a generator such that $\log(g) \equiv 1 \text{ (modulo }p^r\text{)}$. We have the following equality in $J\cdot \mathcal{K}_r/J^2\cdot  \mathcal{K}_r$:
$$
\delta_r\left( \sum_{k=1}^{\frac{N-1}{2}} k \cdot \log(k)^2\right) = \frac{3}{2} \cdot \sum_{i=1}^{N-1}  \left(\sum_{j=i}^{N-1} F_{0,p}(g^j) \right)\cdot \left(1-\zeta_N^{g^{i-1}}, \frac{1-\zeta_N}{1-\zeta_N^{g^{-1}}}\right)_r  \text{ .}
$$
\end{corr}

\begin{rem}
We have not been able to prove directly this identity, without using the theory of higher Eisenstein elements.
\end{rem}

Similar computations give us the following result.
\begin{thm}\label{even_modSymb_computation_m_1_m_2}
Assume that Conjecture \ref{odd_modSymb_Sharifi_conj} holds. 
Assume that $n(r,p) \geq 3$, \ie that $\sum_{k=1}^{\frac{N-1}{2}} k \cdot \log(k) \equiv \sum_{k=1}^{\frac{N-1}{2}} k \cdot \log(k)^2 \equiv 0 \text{ (modulo }p^r\text{)}$. Let $g \in (\mathbf{Z}/N\mathbf{Z})^{\times}$ be a generator such that $\log(g) \equiv 1 \text{ (modulo }p^r\text{)}$.
We have:
\begin{align*}
\delta_r(m_1^+ \bullet m_2^-) &= \frac{1}{4}\sum_{i=1}^{N-1}  \left(\sum_{j=i}^{N-1} F_{1,p}(g^j) \right)\cdot \left( 1-\zeta_N^{g^{i-1}}, \frac{1-\zeta_N}{1-\zeta_N^{g^{-1}}}\right)_r \text{ .}
\end{align*}
\end{thm}

As a corollary, we get our most advanced numerical criterion for bounding $n(r,p)$ from below.

\begin{corr}
Assume that Conjecture \ref{odd_modSymb_Sharifi_conj} holds. 
Assume that $n(r,p) \geq 3$, \ie that $\sum_{k=1}^{\frac{N-1}{2}} k \cdot \log(k) \equiv \sum_{k=1}^{\frac{N-1}{2}} k \cdot \log(k)^2 \equiv 0 \text{ (modulo }p^r\text{)}$. Let $g \in (\mathbf{Z}/N\mathbf{Z})^{\times}$ be a generator such that $\log(g) \equiv 1 \text{ (modulo }p^r\text{)}$. We have $n(r,p) \geq 4$ if and only if we have, in $J\cdot \mathcal{K}_r/J^2\cdot \mathcal{K}_r$:
$$\sum_{i=1}^{N-1} \left(\sum_{j=i}^{N-1} F_{1,p}(g^j) \right)\cdot  \left( 1-\zeta_N^{g^{i-1}}, \frac{1-\zeta_N}{1-\zeta_N^{g^{-1}}}\right)_r  = 0 \text{ .}$$
\end{corr}

\subsection{Computation of $m_1^+ \bullet m_1^-$ when $p=3$}\label{Comparison_case_p=3}
In this section, we assume $p=3$. 

We give a formula for $m_1^+ \bullet m_1^-$ modulo $3^t$. We are only able to simplify the formula for this intersection product modulo $3^{t-1}$. Note that we do not have an explicit formula for $m_2^-$, but nevertheless it is possible to compute $m_0^+ \bullet m_2^-=m_1^+ \bullet m_1^-$.

For $\overline{a}$ and $\overline{b}$ in $(\mathbf{Z}/N\mathbf{Z})^{\times}$, we let $[\overline{a},\overline{b}]$ be the reduction modulo $N$ of the interval $[a,b]$, where $a$ and $b$ are representatives of $a$ and $b$ in $\{-N,..., -1\}$ and $\{1,2,...,N\}$ respectively.

Let $\mu_3'$ the set of cubic primitive roots of unity in $(\mathbf{Z}/N\mathbf{Z})^{\times}$. Recall that $\sigma = \begin{pmatrix}
0 & -1 \\
1 & 0
\end{pmatrix} $ and $\tau = \begin{pmatrix}
0 & -1 \\
1 & -1
\end{pmatrix}$ $\in \SL_2(\mathbf{Z})$. There is a right action of $\SL_2(\mathbf{Z})$ on $\mathbf{P}^1(\mathbf{Z}/N\mathbf{Z})$ given by $x \cdot \begin{pmatrix}
a&b \\
c &d
\end{pmatrix} = \frac{ax+c}{bx+d}$ if $x \in \mathbf{P}^1(\mathbf{Z}/N\mathbf{Z})$. The set  $(\mathbf{Z}/N\mathbf{Z})^{\times} \backslash \{\overline{1}, \overline{-1}\} \subset \mathbf{P}^1(\mathbf{Z}/N\mathbf{Z})$ is stable under $\sigma$.

Let $\log_3 : \mathbf{Z}[ (\mathbf{Z}/N\mathbf{Z})^{\times} \backslash \{\overline{1}, \overline{-1}\}]\rightarrow \mathbf{Z}/3^{t+1}\mathbf{Z}$ be given by
$$\log_3\left(\sum_{x \in  (\mathbf{Z}/N\mathbf{Z})^{\times} \backslash \{\overline{1}, \overline{-1}\} } \lambda_x \cdot [x]\right) = \sum_{x_3 \in \mu_3'} \sum_{y \in [-x_3,x_3]} (\lambda_{y\sigma} - \lambda_y)\cdot \log(x_3) \text{ .}$$

The following lemma summarizes the properties $\log_3$ we will need.
\begin{lem}[Merel]\label{comparison_properties_log_3}
\begin{enumerate}
\item If $a \in (\mathbf{Z}/N\mathbf{Z})^{\times} \backslash \{\overline{1}, \overline{-1}\}$ is fixed by $\tau$, then $\log_3([a]) \equiv \log(a) \text{ (modulo }3^{t+1}\text{)}$.
\item If $a \in  (\mathbf{Z}/N\mathbf{Z})^{\times} \backslash \{\overline{1}, \overline{-1}\}$ is fixed by $\sigma$, then $\log_3(a) \equiv 0 \text{ (modulo }3^{t+1}\text{)}$.
\end{enumerate}
\end{lem}
\begin{proof}
Point (i) (resp. point (ii)) follows from \cite[Lemme 6]{Merel_accouplement} (resp. \cite[Lemme 7]{Merel_accouplement}).
\end{proof}

By Lemma \ref{comparison_properties_log_3}, the group homomorphism $\mathbf{Z}[(\mathbf{Z}/N\mathbf{Z})^{\times} \backslash \{\overline{1}, \overline{-1}\}] \rightarrow \mathbf{Z}/3^{t+1}\mathbf{Z}$ given by $[x] \mapsto \log(x) - \log_3(x)$ annihilates the Manin relations, and thus induces a group homomorphism $$\varphi_3 : H_1(X_0(N), \mathbf{Z}/3^{t+1}\mathbf{Z}) \rightarrow \mathbf{Z}/3^{t+1}\mathbf{Z} \text{ .}$$

\begin{thm}\label{comparison_pairing_m_1^+_m_1^-_3}
Assume $p=3$. Assume $n(r,p) \geq 2$, \ie $\sum_{k=1}^{N-1} k^2\cdot \log(k) \equiv 0 \text{ (modulo }3^r\text{)}$. Then we have in $\mathbf{Z}/3^{r+1}\mathbf{Z}$:
$$12 \cdot m_1^+ \bullet m_1^- = \log(2) \cdot \varphi_3(m_0^+) + \varphi_3\left(  \sum_{x \in (\mathbf{Z}/N\mathbf{Z})^{\times} \backslash \{\overline{1}, \overline{-1}\} } F_{1,3}(x) \cdot \xi_{\Gamma_0(N)}([x:1])\right) $$
where $F_{1,3} : \mathbf{P}^1(\mathbf{Z}/N\mathbf{Z}) \rightarrow \mathbf{Z}/3^{t+1}\mathbf{Z}$ is defined in Theorem \ref{thm_Introduction_w_1^+_3}.
\end{thm}
\begin{proof}
By Theorem \ref{thm_Introduction_w_1^+_3}, we have $H_1(X_0(N), \cusps, \mathbf{Z}/3^{t+1}\mathbf{Z})_+$:
$$6\cdot m_1^+ = \log(2) \cdot m_0^+ + \sum_{x \in (\mathbf{Z}/N\mathbf{Z})^{\times} \backslash \{\overline{1}, \overline{-1}\}} F_{1,3}(x) \cdot \xi_{\Gamma_0(N)}([x:1]) \text{ .}$$
Let $G_0 : (\mathbf{Z}/N\mathbf{Z})^{\times} \backslash \{\overline{1}, \overline{-1}\} \rightarrow \mathbf{Z}/3^t\mathbf{Z}$ be such that 
$$m_0^+ = \sum_{x \in (\mathbf{Z}/N\mathbf{Z})^{\times} \backslash \{\overline{1}, \overline{-1}\}} G_0(x) \cdot \xi_{\Gamma_0(N)}([x:1]) \text{ .}$$
Let $G_1 : (\mathbf{Z}/N\mathbf{Z})^{\times} \backslash \{\overline{1}, \overline{-1}\} \rightarrow \mathbf{Z}/3^t\mathbf{Z}$ be such that 
$$m_1^+ = \sum_{x \in (\mathbf{Z}/N\mathbf{Z})^{\times} \backslash \{\overline{1}, \overline{-1}\} } G_1(x) \cdot \xi_{\Gamma_0(N)}([x:1]) \text{ .}$$

By Theorem \ref{thm_Introduction_w_1^+_3} and \cite[Proposition 3]{Merel_accouplement}, we have the following equality in $(\mathbf{Z}/3^{t+1}\mathbf{Z})[(\mathbf{Z}/N\mathbf{Z})^{\times} \backslash \{\overline{1}, \overline{-1}\} ]$:
\begin{align*}
6\cdot \sum_{x \in (\mathbf{Z}/N\mathbf{Z})^{\times} \backslash \{\overline{1}, \overline{-1}\} } G_1(x) \cdot [x]  &= \log(2) \cdot \sum_{x \in (\mathbf{Z}/N\mathbf{Z})^{\times} \backslash \{\overline{1}, \overline{-1}\} } G_0(x) \cdot [x]   \\ +  \sum_{x \in (\mathbf{Z}/N\mathbf{Z})^{\times} \backslash \{\overline{1}, \overline{-1}\} } F_{1,3}(x) \cdot [x] + a_{\sigma}+a_{\tau}
\end{align*}
where $a_{\sigma}$ (resp. $a_{\tau}$) is an element of $(\mathbf{Z}/3^{t+1}\mathbf{Z})[(\mathbf{Z}/N\mathbf{Z})^{\times} \backslash \{\overline{1}, \overline{-1}\} ]$ fixed by $\sigma$ (resp. by $\tau$). By definition, we have in $\mathbf{Z}/3^r\mathbf{Z}$:
$$m_1^+ \bullet m_1^- = \frac{1}{2}\cdot  \sum_{x \in (\mathbf{Z}/N\mathbf{Z})^{\times} \backslash \{\overline{1}, \overline{-1}\} } G_1(x) \cdot \log(x) \text{ .}$$
By Lemma \ref{comparison_properties_log_3} (i), we get in $\mathbf{Z}/3^{r+1}\mathbf{Z}$:
$$12 \cdot m_1^+ \bullet m_1^- = \log(2) \cdot \varphi_3(m_0^+) + \varphi_3\left(  \sum_{x \in (\mathbf{Z}/N\mathbf{Z})^{\times} \backslash \{\overline{1}, \overline{-1}\} } F_{1,3}(x) \cdot \xi_{\Gamma_0(N)}([x:1])\right) $$ 
which concludes the proof of Theorem \ref{comparison_pairing_m_1^+_m_1^-_3}.
\end{proof}

\begin{corr}\label{comparison_pairing_m_1^+_m_1^-_3}
Assume $t \geq 2$, \ie $N \equiv 1 \text{ (modulo }27\text{)}$. Assume $1 \leq r \leq t-1$. If $n(r,p) \geq 2$, \ie if $\sum_{k=1}^{N-1} k^2 \cdot \log(k) \equiv 0 \text{ (modulo } 3^r\text{)}$, then we have in $\mathbf{Z}/3^r\mathbf{Z}$:
$$m_1^+ \bullet m_1^- = -\frac{1}{4}\cdot \sum_{k=1}^{N-1}k^2\cdot \log(k)^2 \text{ .}$$
\end{corr}
\begin{proof}
Since $r \leq t-1$, the map $\log_3$ is zero modulo $3^{r+1}$. Thus, for all $x \in (\mathbf{Z}/N\mathbf{Z})^{\times} \backslash \{\overline{1}, \overline{-1}\}$ we have $\varphi_3(\xi_{\Gamma_0(N)}(x)) \equiv \log(x) \text{ (modulo }3^{r+1}\text{)}$. In particular, we have 
$$\varphi_3(m_0^+) \equiv 2 \cdot m_0^+ \bullet m_1^- \text{ (modulo }3^{r+1}\text{).} $$
We get, by Theorem \ref{Comparison_Merel_sqrt_u}:
$$\varphi_3(m_0^+) \equiv - \frac{1}{2} \cdot \sum_{k=1}^{N-1} k^2\cdot \log(k) \text{ (modulo }3^{r+1}\text{).} $$

Thus, we have in $\mathbf{Z}/3^{r+1}\mathbf{Z}$:
\begin{equation}\label{even_modSymb_calcul_p_3_r<t}
6 \cdot m_1^+ \bullet m_1^- = -\frac{1}{2} \cdot \log(2) \cdot  \sum_{k=1}^{N-1} k^2\cdot \log(k) +  \sum_{x \in (\mathbf{Z}/N\mathbf{Z})^{\times} \backslash \{\overline{1}, \overline{-1}\} } F_{1,3}(x) \cdot \log(x) \text{ .}
\end{equation}

Recall that $F_{1,3} : \mathbf{P}^1(\mathbf{Z}/N\mathbf{Z}) \rightarrow \mathbf{Z}/3^{t+1}\mathbf{Z}$ is defined by
\begin{align*}
F_{1,3}([c:d]) &= \frac{1}{2}\sum_{(s_1,s_2) \in (\mathbf{Z}/2N\mathbf{Z})^2 \atop (d-c)s_1+(d+c)s_2 \equiv 0 \text{ (modulo } N\text{)}} (-1)^{s_1+s_2} \B_1\left(\frac{s_1}{2N}\right)\B_1\left(\frac{s_2}{2N}\right) \cdot \log\left(\frac{s_2}{d-c}\right)  \\&
-  \frac{1}{2}\sum_{(s_1,s_2) \in (\mathbf{Z}/2N\mathbf{Z})^2 \atop (d-c)s_1+(d+c)s_2 \not\equiv 0 \text{ (modulo } N\text{)}} (-1)^{s_1+s_2} \B_1\left(\frac{s_1}{2N}\right)\B_1\left(\frac{s_2}{2N}\right) \cdot \log((d-c)s_1+(d+c)s_2)) 
\end{align*}
if $[c:d] \neq [1:1]$ and $F_{1,3}([1:1])=0$. We have, in $\mathbf{Z}/3^{t+1}\mathbf{Z}$:

\begin{align*}
2\cdot & \sum_{x \in (\mathbf{Z}/N\mathbf{Z})^{\times} \backslash \{\overline{1}, \overline{-1}\} } F_{1,3}(x) \cdot \log(x) \\& = \sum_{x \in (\mathbf{Z}/N\mathbf{Z})^{\times} \atop x\neq 1 }\log(x) \cdot \left(\sum_{(s_1,s_2)\in (\mathbf{Z}/2N\mathbf{Z})^2 \atop  (1-x)s_1 + (1+x)s_2 \equiv 0 \text{ (modulo }N\text{)}} (-1)^{s_1+s_2} \B_1\left(\frac{s_1}{2N}\right)\B_1\left(\frac{s_2}{2N}\right) \cdot \log\left(\frac{s_2}{1-x}\right) \right)  \\& - \log(x) \cdot \left( \sum_{(s_1,s_2)\in (\mathbf{Z}/2N\mathbf{Z})^2 \atop  (1-x)s_1 + (1+x)s_2 \not\equiv 0 \text{ (modulo }N\text{)}} (-1)^{s_1+s_2} \B_1\left(\frac{s_1}{2N}\right)\B_1\left(\frac{s_2}{2N}\right) \cdot \log((1-x)s_1 + (1+x)s_2) \right)
\\& = \frac{1}{4}\sum_{x \in (\mathbf{Z}/N\mathbf{Z})^{\times} \atop x\neq 1 } \log(x) \cdot \left(\sum_{(s_1,s_2)\in (\mathbf{Z}/N\mathbf{Z})^2 \atop  (1-x)s_1 + (1+x)s_2 \equiv 0 \text{ (modulo }N\text{)}} D_2\left(\frac{s_1}{N}\right)D_2\left(\frac{s_2}{N}\right)\cdot \log\left(\frac{s_2}{1-x}\right) \right)   \\& -\log(x) \cdot \left(\sum_{(s_1,s_2)\in (\mathbf{Z}/N\mathbf{Z})^2 \atop  (1-x)s_1 + (1+x)s_2 \not\equiv 0 \text{ (modulo }N\text{)}} D_2\left(\frac{s_1}{N}\right)D_2\left(\frac{s_2}{N}\right) \cdot \log((1-x)s_1 + (1+x)s_2) \right)  \\& =
\frac{1}{4}\sum_{(s_1,s_2) \in (\mathbf{Z}/N\mathbf{Z})^2 \atop s_1 \neq \pm s_2} D_2\left(\frac{s_1}{N}\right)D_2\left(\frac{s_2}{N}\right)\cdot (\log\left(\frac{s_1+s_2}{s_1-s_2}\right) \cdot \log\left(\frac{s_2-s_1}{2}\right) \\&- \sum_{x \neq \frac{s_1+s_2}{s_1-s_2}} \log(x)\log((1-x)s_1+(1+x)s_2) ) \text{.}
\end{align*}

Since $r+1 \leq t$, Lemma \ref{Bernardi_lemma>3} shows that we have, in $\mathbf{Z}/3^{r+1}\mathbf{Z}$:
\begin{align*}
 \sum_{x \neq \frac{s_1+s_2}{s_1-s_2}} \log(x)\log((1-x)s_1+(1+x)s_2) = - \log(s_2-s_1)\log\left(\frac{s_1+s_2}{s_1-s_2}\right)  - \log\left(\frac{s_1+s_2}{s_1-s_2}\right) ^2 \text{ .}
\end{align*}
Thus we have, in $\mathbf{Z}/3^{r+1}\mathbf{Z}$:

\begin{align*}
2\sum_{x \in (\mathbf{Z}/N\mathbf{Z})^{\times} \backslash \{\overline{1}, \overline{-1}\} } F_{1,3}(x) \cdot \log(x) & =  \frac{1}{4}\sum_{(s_1,s_2) \in (\mathbf{Z}/N\mathbf{Z})^2 \atop s_1 \neq \pm s_2} D_2\left(\frac{s_1}{N}\right)D_2\left(\frac{s_2}{N}\right)\cdot  \\& \left( \log(s_1+s_2)^2-\log(s_1-s_2)^2 - \log(2) \cdot \log\left(\frac{s_1+s_2}{s_1-s_2}\right) \right) 
\end{align*}
Using Lemmas \ref{even_modSymb_computation_D2} and \ref{even_modSymb_computation_D2_square} and the assumption $r+1 \leq t$, we get in $\mathbf{Z}/3^{r+1}\mathbf{Z}$:
$$\sum_{x \in (\mathbf{Z}/N\mathbf{Z})^{\times} \backslash \{\overline{1}, \overline{-1}\} } F_{1,3}(x) \cdot \log(x) = -\frac{3}{2}\cdot \sum_{k=1}^{N-1}k^2\cdot \log(k)^2 +\frac{1}{2}\cdot \log(2) \cdot \sum_{k=1}^{N-1} k^2\cdot \log(k) \text{ .}$$
By (\ref{even_modSymb_calcul_p_3_r<t}), we have in $\mathbf{Z}/3^{r+1}\mathbf{Z}$:
$$6 \cdot m_1^+ \bullet m_1^- = -\frac{3}{2} \cdot \sum_{k=1}^{N-1}k^2 \cdot \log(k)^2 \text{ .}$$
This concludes the proof of Theorem \ref{comparison_pairing_m_1^+_m_1^-_3}.
\end{proof}

\begin{corr}\label{even_modSymb_g_3_corr}
Assume that $t \geq 2$, \ie that $N \equiv 1 \text{ (modulo }27\text{)}$. Assume that $1 \leq r \leq t-1$. The following assertions are equivalent:
\begin{enumerate}
\item $n(r,p) \geq 3$
\item $\sum_{k=1}^{N-1} k^2 \cdot \log(k) \equiv \sum_{k=1}^{N-1} k^2 \cdot \log(k)^2 \equiv 0 \text{ (modulo }3^r \text{)}$.
\end{enumerate}
\end{corr}

\begin{rem}
Corollary \ref{even_modSymb_g_3_corr} does not hold in general if $N \not\equiv 1 \text{ (modulo }27 \text{)}$. For instance, if $N = 1279$ and $r=1$ then $g_3=2$ and $\sum_{k=1}^{N-1} k^2 \cdot \log(k) \equiv \sum_{k=1}^{N-1} k^2 \cdot \log(k)^2 \equiv 0 \text{ (modulo }3 \text{)}$. If 
$N=1747$ and $r=1$ then $g_3=3$ and $\sum_{k=1}^{N-1} k^2 \cdot \log(k)^2 \not\equiv 0 \text{ (modulo }3 \text{)}$.
\end{rem}

\subsection{Computation of $m_1^+ \bullet m_1^-$ when $p =2$}\label{Comparison_case_p=2}
In this section, we assume that $p=2$. We give a formula for $m_1^+ \bullet m_1^-$ modulo $2^{t-1}$. Note that we do not have an explicit formula for $m_2^-$ modulo $2^{t-1}$, but nevertheless it is possible to compute $m_0^+ \bullet m_2^- = m_1^+ \bullet m_1^-$ modulo $2^{t-1}$.

\begin{thm}\label{comparison_pairing_m_1^+_m_1^-_2}
Assume that $t\geq 2$, \ie that $N \equiv 1 \text{ (modulo }16 \text{)}$. Let $r$ be an integer such that $1 \leq r \leq t-1$ and $n(r,p) \geq 2$ (\ie $\sum_{k=1}^{\frac{N-1}{2}} k \cdot \log(k) \equiv 0 \text{ (modulo }2^r\text{)}$ by Corollary \ref{comparison_corr_g_p_2}). We have, in $\mathbf{Z}/2^{r+1}\mathbf{Z}$:
$$6\cdot m_1^+ \bullet m_1^- =\sum_{k=1}^{\frac{N-1}{2}} k \cdot \log(k) +  \sum_{k=1}^{\frac{N-1}{2}} k \cdot \log(k)^2 \text{ .}$$
\end{thm}
\begin{proof}
Recall that we have normalized $\tilde{m}_0^-$, and thus $m_1^-$, so that for all $x \in \mathbf{P}^1(\mathbf{Z}/N\mathbf{Z}) \backslash \{0, \infty\}$ we have, in $\mathbf{Z}/2^{t}\mathbf{Z}$:
$$(1+c) \cdot \xi_{\Gamma_0(N)}(x) \bullet m_1^- = \log(x) \text{ .}$$
For all $x \in (\mathbf{Z}/N\mathbf{Z})^{\times}$, we have in $\mathbf{Z}/2^{t+1}\mathbf{Z}$:
\begin{equation}\label{comparison_parity_F_1_2}
F_{1,2}(x) = F_{1,2}(-x) \text{ .}
\end{equation}
By Theorem \ref{thm_Introduction_w_1^+_2}, Theorem \ref{Comparison_Merel_sqrt_u} and (\ref{comparison_parity_F_1_2}), we have in $\mathbf{Z}/2^{r+2}\mathbf{Z}$:
\begin{equation}\label{comparison_pairing_m_1^+_m_1^-_2_eq1}
12\cdot m_1^+ \bullet m_1^- = -\frac{2}{3}\cdot \left(2^{t-1} - \sum_{k=1}^{\frac{N-1}{2}} k \cdot \log(k) \right) + \sum_{x \in R'} F_{1,2}(x) \cdot \log(x) \text{ ,}
\end{equation}
where $R'$ is any set of representative in $(\mathbf{Z}/N\mathbf{Z})^{\times} \backslash \{\overline{1}, \overline{-1}\}$ for the equivalence relation $x \sim -\frac{1}{x}$.

Let $\tilde{F}_{1,2} : (\mathbf{Z}/N\mathbf{Z})^{\times} \rightarrow \mathbf{Z}/2^{t+2}\mathbf{Z}$ be defined by
\begin{align*}
\tilde{F}_{1,2}(x) &= \sum_{s_1, s_2=1 \atop (1-x)s_1+(1+x)s_2 \equiv 0 \text{ (modulo }N\text{)}}^{\frac{N-1}{2}}  \log\left(\frac{2}{1-x}\cdot s_2\right)  \\& - \sum_{s_1, s_2=1 \atop (1-x)s_1+(1+x)s_2 \neq 0 \text{ (modulo }N\text{)}}^{\frac{N-1}{2}}  \log((1-x)s_1+(1+x)s_2) \text{ .}
\end{align*}
By definition, for all $x\in (\mathbf{Z}/N\mathbf{Z})^{\times}$ we have 
\begin{equation}\label{comparison_tilde_1,2_eq}
F_{1,2}(x) \equiv \tilde{F}_{1,2}(x) \text{ (modulo }2^{t+1}\text{).}
\end{equation}

\begin{lem}\label{comparison_horror_lemma}
For all $x\in (\mathbf{Z}/N\mathbf{Z})^{\times}$, we have in $\mathbf{Z}/2^{t+2}\mathbf{Z}$:
$$\tilde{F}_{1,2}\left(-\frac{1}{x}\right) = -\tilde{F}_{1,2}(x) + \frac{N-1}{2}+\frac{N-1}{2}\cdot \log\left(\frac{x-1}{x+1}\right) \text{ .}$$
\end{lem}
\begin{proof}
Let $x\in (\mathbf{Z}/N\mathbf{Z})^{\times}$. We have, in $\mathbf{Z}/2^{t+1}\mathbf{Z}$:
\begin{align*}
\tilde{F}_{1,2}(x)+\tilde{F}_{1,2}\left(-\frac{1}{x}\right) &= \sum_{s_2=1}^{\frac{N-1}{2}}\sum_{s_1=1 \atop (1-x)s_1+(1+x)s_2 \equiv 0 \text{ (modulo }N\text{)}}^{\frac{N-1}{2}}\log\left(\frac{2}{1-x}\cdot s_2 \right) \\& +  \sum_{s_2=1}^{\frac{N-1}{2}}\sum_{s_1=\frac{N+1}{2}  \atop (1-x)s_1+(1+x)s_2 \equiv 0 \text{ (modulo }N\text{)} }^{N-1} \log\left(\frac{2\cdot x}{1-x}\cdot s_2 \right) \\& - \sum_{s_2=1}^{\frac{N-1}{2}} \sum_{s_1=1  \atop (1-x)s_1+(1+x)s_2 \not\equiv 0 \text{ (modulo }N\text{)}}^{\frac{N-1}{2}} \log\left((1-x)s_1+(1+x)s_2\right) \\&- \sum_{s_2=1}^{\frac{N-1}{2}} \sum_{s_1=\frac{N+1}{2}  \atop (1-x)s_1+(1+x)s_2 \not\equiv 0 \text{ (modulo }N\text{)}}^{N-1} \log\left(\frac{1}{x}\cdot\left((1-x)s_1+(1+x)s_2 \right)\right)  \\&= \sum_{s_2=1}^{\frac{N-1}{2}}\sum_{s_1=1 \atop (1-x)s_1+(1+x)s_2 \equiv 0 \text{ (modulo }N\text{)}}^{N-1}\log\left(\frac{2}{1-x}\cdot s_2 \right) + \log(x) \cdot \mathcal{N}(x) \\&  - \sum_{s_2=1}^{\frac{N-1}{2}} \sum_{s_1=1  \atop (1-x)s_1+(1+x)s_2 \not\equiv 0 \text{ (modulo }N\text{)}}^{N-1} \log\left((1-x)s_1+(1+x)s_2\right) \\&+ \log(x) \cdot \left( \left(\frac{N-1}{2}\right)^2-\mathcal{N}(x)\right)\text{ ,}
\end{align*}
where $\mathcal{N}(x)$ is the number of $s \in \{1, 2, ..., \frac{N-1}{2}\}$ such that the representative of $s \cdot \frac{x-1}{x+1}$ in $\{1, ..., N-1\}$ is in $\{\frac{N+1}{2}, ..., N-1\}$. Thus, we have in $\mathbf{Z}/2^{t+2}\mathbf{Z}$:
\begin{align*}
\tilde{F}_{1,2}(x)+\tilde{F}_{1,2}\left(-\frac{1}{x}\right) &= \sum_{s_2=1}^{\frac{N-1}{2}}\sum_{s_1=1 \atop (1-x)s_1+(1+x)s_2 \equiv 0 \text{ (modulo }N\text{)}}^{N-1}\log\left(\frac{2}{1-x}\cdot s_2 \right)  \\&  - \sum_{s_2=1}^{\frac{N-1}{2}} \sum_{s_1=1  \atop (1-x)s_1+(1+x)s_2 \not\equiv 0 \text{ (modulo }N\text{)}}^{N-1} \log\left((1-x)s_1+(1+x)s_2\right) \\& 
= \frac{N-1}{2}\cdot \log\left(\frac{2}{1-x}\right) + \log\left(\left(\frac{N-1}{2}\right)! \right) \\& - \frac{N-1}{2} \cdot (N-2)\cdot\log(1+x) -\sum_{s_2=1}^{\frac{N-1}{2}}\sum_{s_1=1  \atop (1-x)s_1+(1+x)s_2 \not\equiv 0 \text{ (modulo }N\text{)}}^{N-1} \log\left(s_2 - \frac{x-1}{x+1}\cdot s_1\right) \text{ .}
\end{align*}
Since $N \equiv 1 \text{ (modulo }8\text{)}$, we have $\log(2) \equiv 0 \text{ (modulo }2\text{)}$ by the quadratic reciprocity law. Thus we have in $\mathbf{Z}/2^{t+2}\mathbf{Z}$:
\begin{align*}
\tilde{F}_{1,2}(x)+\tilde{F}_{1,2}\left(-\frac{1}{x}\right) &= \frac{N-1}{2}\cdot \log\left(\frac{x+1}{x-1}\right) + \log\left(\left(\frac{N-1}{2}\right)! \right) - \sum_{s_2=1}^{\frac{N-1}{2}} \left(\log\left((N-1)!\right) - \log(s_2) \right) \\&=  \frac{N-1}{2}\cdot \log\left(\frac{x+1}{x-1}\right) + 2\cdot \log\left(\left(\frac{N-1}{2}\right)! \right) \\& = \frac{N-1}{2}+ \frac{N-1}{2}\cdot \log\left(\frac{x+1}{x-1}\right)\text{ .}
\end{align*}
This concludes the proof of Lemma \ref{comparison_horror_lemma}.
\end{proof}

Let $\zeta_4$ be an element of order $4$ in $(\mathbf{Z}/N\mathbf{Z})^{\times}$. Since $t\geq 2$, we have $2^t \cdot\log( \zeta_4 )= 0$ in $\mathbf{Z}/2^{t+2}\mathbf{Z}$. Since $F_{1,2}(\zeta_4)\equiv 0 \text{ (modulo }2^t\text{)}$, we have in $\mathbf{Z}/2^{t+2}\mathbf{Z}$:
\begin{equation}\label{comparison_tilde_1_2_zeta_4_eq}
\tilde{F}_{1,2}(\zeta_4)\cdot \log(\zeta_4) = 0
\end{equation}
Similarly, we have in $\mathbf{Z}/2^{t+2}\mathbf{Z}$:
\begin{equation}\label{comparison_tilde_1_2_zeta_4_eq2}
\tilde{F}_{1,2}(-1)\cdot \log(-1) = 0
\end{equation}
By Lemma \ref{comparison_horror_lemma}, (\ref{comparison_tilde_1_2_zeta_4_eq}) and (\ref{comparison_tilde_1_2_zeta_4_eq2}), we have in $\mathbf{Z}/2^{t+2}\mathbf{Z}$:
\begin{align*}
\sum_{x \in (\mathbf{Z}/N\mathbf{Z})^{\times}} \tilde{F}_{1,2}(x) \cdot \log(x) &= \sum_{x \in R'} \tilde{F}_{1,2}(x) \cdot \log(x) + \tilde{F}_{1,2}\left(-\frac{1}{x}\right) \cdot \log\left(-\frac{1}{x}\right) \\& 
= 2\cdot \sum_{x \in R'} \tilde{F}_{1,2}(x) \cdot \log(x) +  \sum_{x \in R'} \frac{N-1}{2}\cdot \log(x) + \frac{N-1}{2}\cdot \log(x) \cdot \log\left(\frac{x-1}{x+1}\right) \\& + \log(-1)\cdot \sum_{x \in R'} \tilde{F}_{1,2}(x) \text{ .}
\end{align*}
By (\ref{comparison_parity_F_1_2}), we have $\sum_{x \in R'} F_{1,2}(x) \equiv 0 \text{ (modulo }2\text{)}$. By Lemma \ref{Bernardi_lemma>3} and the fact that $\log(-1)\cdot \log(-2) \equiv 0 \text{ (modulo }4\text{)}$, we have 
$$2\cdot \sum_{x \in R'} \log(x) \cdot \log\left(\frac{x-1}{x+1}\right) \equiv \sum_{x \in (\mathbf{Z}/N\mathbf{Z})^{\times} \backslash \{\overline{1}, \overline{-1}\}}  \log(x) \cdot \log\left(\frac{x-1}{x+1}\right) \equiv 0 \text{ (modulo }4\text{).}$$
Thus, we have in $\mathbf{Z}/2^{t+2}\mathbf{Z}$:
\begin{equation}\label{comparison_double_tilde_1,2,R_eq}
\sum_{x \in (\mathbf{Z}/N\mathbf{Z})^{\times}} \tilde{F}_{1,2}(x) \cdot \log(x) = 2\cdot  \sum_{x \in R'} \tilde{F}_{1,2}(x) \cdot \log(x)  \text{ .}
\end{equation}

By (\ref{comparison_pairing_m_1^+_m_1^-_2_eq1}) and (\ref{comparison_double_tilde_1,2,R_eq}), we have in $\mathbf{Z}/2^{t+2}\mathbf{Z}$:
\begin{equation}\label{comparison_pairing_m_1^+_m_1^-_2_eq2}
24\cdot m_1^+ \bullet m_1^- = -\frac{4}{3}\cdot \left(2^{t-1} - \sum_{k=1}^{\frac{N-1}{2}} k \cdot \log(k) \right) + \sum_{x \in (\mathbf{Z}/N\mathbf{Z})^{\times}} \tilde{F}_{1,2}(x) \cdot \log(x) \text{ .}
\end{equation}
We have, in $\mathbf{Z}/2^{t+2}\mathbf{Z}$:
\begin{align*}
\sum_{x \in (\mathbf{Z}/N\mathbf{Z})^{\times}} \tilde{F}_{1,2}(x) \cdot \log(x) &=\sum_{x \in (\mathbf{Z}/N\mathbf{Z})^{\times}} \sum_{s_1,s_2=1\atop (1-x)s_1+(1+x)s_2 \equiv 0 \text{ (modulo }N\text{)}}^{\frac{N-1}{2}} \log\left( \frac{2}{1-x} \cdot s_2\right)\cdot \log(x) \\&-\sum_{x \in (\mathbf{Z}/N\mathbf{Z})^{\times}} \sum_{s_1,s_2=1 \atop (1-x)s_1+(1+x)s_2 \not\equiv 0 \text{ (modulo }N\text{)}}^{\frac{N-1}{2}} \log\left((1-x)s_1+(1+x)s_2\right)\cdot \log(x) 
\\& = \sum_{s_1,s_2=1 \atop s_1 \neq s_2}^{\frac{N-1}{2}} \log\left(s_2-s_1\right)\cdot \log\left(\frac{s_1+s_2}{s_1-s_2} \right) \\&-\sum_{s_1,s_2=1}^{\frac{N-1}{2}} \sum_{x \in (\mathbf{Z}/N\mathbf{Z})^{\times}\atop x\not\equiv \frac{s_1+s_2}{s_1-s_2}\text{ (modulo }N\text{)}}  \log\left((1-x)s_1+(1+x)s_2\right)\cdot \log(x) 
\\& =\sum_{s_1,s_2=1 \atop s_1 \neq s_2}^{\frac{N-1}{2}} \log\left(s_2-s_1\right)\cdot \log\left(\frac{s_1+s_2}{s_1-s_2} \right)  \\& - \sum_{s_1,s_2=1 \atop s_1 \neq s_2}^{\frac{N-1}{2}} \log(s_2-s_1)\cdot \left(\log((N-1)!)-\log\left(\frac{s_1+s_2}{s_1-s_2}\right)\right) \\& +  \sum_{s_1,s_2=1 \atop s_1 \neq s_2}^{\frac{N-1}{2}}  \sum_{x \in (\mathbf{Z}/N\mathbf{Z})^{\times}\atop x\not\equiv \frac{s_1+s_2}{s_1-s_2}\text{ (modulo }N\text{)}} \log\left(x - \frac{s_1+s_2}{s_1-s_2}\right)\cdot\log(x) 
\\& =2\cdot \sum_{s_1,s_2=1 \atop s_1 \neq s_2}^{\frac{N-1}{2}} \log\left(s_2-s_1\right)\cdot \log\left(\frac{s_1+s_2}{s_1-s_2} \right)  + \log(-1)\cdot \sum_{s_1,s_2=1 \atop s_1 \neq s_2}^{\frac{N-1}{2}} \log(s_2-s_1) \\&+   \sum_{s_1,s_2=1 \atop s_1 \neq s_2}^{\frac{N-1}{2}}\sum_{x \in (\mathbf{Z}/N\mathbf{Z})^{\times}\atop x\not\equiv \frac{s_1+s_2}{s_1-s_2}\text{ (modulo }N\text{)}} \log\left(x - \frac{s_1+s_2}{s_1-s_2}\right)\cdot\log(x) \text{ .}
\end{align*}

By Lemmas \ref{even_modSymb_somme_bizarre} and  \ref{Bernardi_lemma>3}, we have in $\mathbf{Z}/2^{t+2}\mathbf{Z}$:
\begin{align*}
\sum_{x \in (\mathbf{Z}/N\mathbf{Z})^{\times}} \tilde{F}_{1,2}(x) \cdot \log(x) &=2\cdot \sum_{s_1,s_2=1 \atop s_1 \neq s_2}^{\frac{N-1}{2}} \log\left(s_1-s_2\right)\cdot \log\left(\frac{s_1+s_2}{s_1-s_2} \right)  \\& + \sum_{s_1,s_2=1 \atop s_1 \neq s_1}^{\frac{N-1}{2}} \left( \log(-1)\cdot \log\left(\frac{s_1+s_2}{s_1-s_2}\right) - \log\left(\frac{s_1+s_2}{s_1-s_2}\right)^2+  \mathcal{F}_2 \right) \text{ .}
\end{align*}
Since $N \equiv 1 \text{ (modulo }8\text{)}$, $4\cdot \mathcal{F}_2 = 4 \cdot \mathcal{F}_1 = 0$ and $\log(-1) \equiv 0 \text{ (modulo }4\text{)}$, we have in $\mathbf{Z}/2^{t+2}\mathbf{Z}$:
\begin{equation}\label{comparison_F_tilde_log_equa}
\sum_{x \in (\mathbf{Z}/N\mathbf{Z})^{\times}} \tilde{F}_{1,2}(x) \cdot \log(x) =2\cdot \sum_{s_1,s_2=1 \atop s_1 \neq s_2}^{\frac{N-1}{2}} \log\left(s_1-s_2\right)\cdot \log\left(\frac{s_1+s_2}{s_1-s_2} \right) - \sum_{s_1,s_2=1 \atop s_1 \neq s_1}^{\frac{N-1}{2}} \log\left(\frac{s_1+s_2}{s_1-s_2}\right)^2  \text{ .}
\end{equation}
We have in $\mathbf{Z}/2^{t+2}\mathbf{Z}$:
\begin{align*}
\sum_{s_1,s_2=1 \atop s_1 \neq s_2}^{\frac{N-1}{2}} \log(s_1+s_2)^2 &=\sum_{s_1,s_2=1}^{\frac{N-1}{2}} \log(s_1+s_2)^2 - \sum_{s=1}^{\frac{N-1}{2}} \log(2\cdot s)^2 \\& = \sum_{s_1,s_2=1}^{\frac{N-1}{2}} \log(s_1+s_2)^2 - \mathcal{F}_2\text{ .}
\end{align*}
By Lemma \ref{even_modSymb_somme_bizarre_2} and the fact that $2\cdot \mathcal{F}_2 = \frac{N-1}{2}$, we have in $\mathbf{Z}/2^{t+2}\mathbf{Z}$:
\begin{equation}\label{comparison_variante_Bernardi_eq}
\sum_{s_1,s_2=1 \atop s_1 \neq s_2}^{\frac{N-1}{2}} \log(s_1+s_2)^2 =\frac{N-1}{2} +  2\cdot \sum_{k=1}^{\frac{N-1}{2}} k\cdot \log(k)^2  
\end{equation}
and
\begin{equation}\label{comparison_variante_Bernardi_eq2}
\sum_{s_1,s_2=1 \atop s_1 \neq s_2}^{\frac{N-1}{2}} \log(s_1-s_2)^2 =-2\cdot \sum_{k=1}^{\frac{N-1}{2}} k\cdot \log(k)^2  \text{ .}
\end{equation}
By (\ref{comparison_F_tilde_log_equa}), (\ref{comparison_variante_Bernardi_eq}) and (\ref{comparison_variante_Bernardi_eq2}), we have in $\mathbf{Z}/2^{t+2}\mathbf{Z}$:
\begin{equation}\label{even_modSymb_p=2_pairing_computation_m-1^+m-1^-_eq}
\sum_{x \in (\mathbf{Z}/N\mathbf{Z})^{\times}} \tilde{F}_{1,2}(x) \cdot \log(x) = \frac{N-1}{2}+4\cdot\sum_{k=1}^{\frac{N-1}{2}} k\cdot \log(k)^2   + 4\cdot \sum_{s_1,s_2=1 \atop s_1 \neq s_1}^{\frac{N-1}{2}} \log(s_1+s_2)\cdot \log(s_1-s_2)  \text{ .}
\end{equation}
We have, in $\mathbf{Z}/2^{t+2}\mathbf{Z}$:
\begin{align*}
2\cdot \sum_{s_1,s_2=1 \atop s_1 \neq s_1}^{\frac{N-1}{2}} \log(s_1+s_2)\cdot \log(s_1-s_2)  &= \sum_{s_1=1}^{\frac{N-1}{2}}\sum_{s_2=1 \atop s_2 \neq s_1, N-s_1}^{N-1}\log(s_1+s_2)\cdot \log(s_1-s_2) \\& = \sum_{s_1=1}^{\frac{N-1}{2}} \sum_{s_2 \in (\mathbf{Z}/N\mathbf{Z})^{\times}\atop s_2 \neq \pm s_1} \log(s_1+s_2)\cdot \log(s_1-s_2) \\& = \sum_{s_1=1}^{\frac{N-1}{2}} \sum_{s_2 \in (\mathbf{Z}/N\mathbf{Z})^{\times}\atop s_2 \neq 2\cdot s_1, s_1} \log(s_2)\cdot \log(2\cdot s_1-s_2) \\& =  \sum_{s_1=1}^{\frac{N-1}{2}} \sum_{s_2 \in (\mathbf{Z}/N\mathbf{Z})^{\times}\atop s_2 \neq 2\cdot s_1} \log(s_2)\cdot \log(s_2-2\cdot s_1)  \\& +\log(-1)\cdot  \sum_{s_1=1}^{\frac{N-1}{2}} \log(2\cdot s_1) - \sum_{s_1=1}^{\frac{N-1}{2}} \log(s_1)^2\\& = -\mathcal{F}_2+ \sum_{s_1=1}^{\frac{N-1}{2}} \sum_{s_2 \in (\mathbf{Z}/N\mathbf{Z})^{\times}\atop s_2 \neq 2\cdot s_1} \log(s_2)\cdot \log(s_2-2\cdot s_1)  \text{ .}
\end{align*}
In the last equality, we have used the fact that $\log(2) \equiv 0 \text{ (modulo }2\text{)}$ by the quadratic reciprocity law, since $N \equiv 1 \text{ (modulo }8\text{)}$. 
\begin{align*}
2\cdot \sum_{s_1,s_2=1 \atop s_1 \neq s_1}^{\frac{N-1}{2}} \log(s_1+s_2)\cdot \log(s_1-s_2)  &= -\mathcal{F}_2+ \sum_{s_1=1}^{\frac{N-1}{2}} \left( \log(-1)\cdot \log(2\cdot s_1) -\log(2\cdot s_1)^2 + \mathcal{F}_2\right) \\&=-\mathcal{F}_2 -\sum_{s_1=1}^{\frac{N-1}{2}} \log(2\cdot s_1)^2 \\& = -2\cdot \mathcal{F}_2\text{ .}
\end{align*}
By (\ref{even_modSymb_p=2_pairing_computation_m-1^+m-1^-_eq}), we have in $\mathbf{Z}/2^{t+2}\mathbf{Z}$:
$$
\sum_{x \in (\mathbf{Z}/N\mathbf{Z})^{\times}} F_{1,2}(x) \cdot \log(x) = \frac{N-1}{2}+4\cdot \sum_{k=1}^{\frac{N-1}{2}} k \cdot \log(k)^2 \text{ .}
$$
By (\ref{comparison_pairing_m_1^+_m_1^-_2_eq2}), we have in $\mathbf{Z}/2^{r+3}\mathbf{Z}$:
\begin{align*}
24\cdot m_1^+\bullet m_1^- &= \frac{4}{3}\cdot \sum_{k=1}^{\frac{N-1}{2}}k \cdot \log(k) + 4\cdot \sum_{k=1}^{\frac{N-1}{2}} k \cdot \log(k)^2 \\& = 4\cdot  \sum_{k=1}^{\frac{N-1}{2}}k \cdot \log(k) + 4\cdot \sum_{k=1}^{\frac{N-1}{2}} k \cdot \log(k)^2 \text{ .}
\end{align*}
In the last equality, we have used the fact that $\sum_{k=1}^{\frac{N-1}{2}}k \cdot \log(k) \equiv 0 \text{ (modulo }2^{r}\text{)}$.
This concludes the proof of Theorem \ref{comparison_pairing_m_1^+_m_1^-_2}.
\end{proof}

\begin{corr}\label{even_modSymb_g_2_corr}
Assume that $t \geq 2$, \ie that $N \equiv 1 \text{ (modulo }16\text{)}$. Let $r$ be an integer such that $1 \leq r \leq t-1$. Assume that $n(r,p) \geq 2$, (\ie $\sum_{k=1}^{\frac{N-1}{2}} k \cdot \log(k) \equiv 0 \text{ (modulo }2^r\text{)}$ by Corollary \ref{comparison_corr_g_p_2}). The following assertions are equivalent:
\begin{enumerate}
\item We have $n(r,p) \geq 3$.
\item We have $$\sum_{k=1}^{\frac{N-1}{2}} k \cdot \log(k) +  \sum_{k=1}^{\frac{N-1}{2}} k \cdot \log(k)^2 \equiv 0 \text{ (modulo }2^{r+1} \text{)}$$ 
\end{enumerate}
\end{corr}

\begin{rem}
By Remark (\ref{comparison_CE_p=2_g_2}), the integer $n(1,p)$ is odd when $p=2$. In particular, if $n(1,p) \geq 2$ then $n(1,p) \geq 3$. If $N \equiv 1 \text{ (modulo }16\text{)}$, this means by Corollary \ref{even_modSymb_g_2_corr} that $\sum_{k=1}^{\frac{N-1}{2}} k \cdot \log(k) \equiv 0 \text{ (modulo }2\text{)}$ implies $\sum_{k=1}^{\frac{N-1}{2}} k \cdot \log(k) +  \sum_{k=1}^{\frac{N-1}{2}} k \cdot \log(k)^2 \equiv 0 \text{ (modulo }4\text{)}$. We have not found an elementary proof of this fact.
\end{rem}

\subsection{The $q$-expansion of the second higher Eisenstein element $f_2$ via Galois deformations}\label{Section_q_expansion_higher_Eis_series}
Recall that we keep the notation of chapter \ref{Section_odd_modSymb}, and in particular those of section \ref{odd_modSymb_section_class_group}. In this section, we assume $p \geq 5$, $r=1$ and $n(r,p)  \geq 2$. Theorem \ref{main_thm_deformations_f_2} below aims at giving a rather explicit description of the higher Eisenstein element $f_2$.

Our proof relies on the results of \cite{Calegari_Emerton} and \cite{Lecouturier_class_group}, some of which we now recall. If $i \in \mathbf{Z}$, let $K_{(i)}$ be the number field defined in \cite[Proposition 5.4]{Calegari_Emerton}. In particular, $K_{(i)}$ is an abelian Galois extension of $\mathbf{Q}(\zeta_p)$ of exponent $p$, unramified outside $N$ and $p$ over $\mathbf{Q}$ and such that $\Gal(\mathbf{Q}(\zeta_p)/\mathbf{Q})$ acts by multiplication by $\omega_p^{i}$ on $\Gal(K_{(i)}/\mathbf{Q}(\zeta_p))$. We have $K_{(1)} = \mathbf{Q}(N^{\frac{1}{p}}, \zeta_p)$ and $K_{(0)} = L_1$. If $F$ is a number field, we let $\mathcal{C}_{F} = \Pic(\mathcal{O}_F) \otimes_{\mathbf{Z}} \mathbf{Z}/p\mathbf{Z}$, where $\Pic(\mathcal{O}_F)$ is the class group of $F$.
The following result follows easily from the proof of \cite[Proposition 5.4]{Calegari_Emerton}.
\begin{prop}\label{comparison_canonical_CFT}
Let $i \in \{-1,0,1\}$, so that in particular we have $\left(\mathcal{C}_{\mathbf{Q}(\zeta_p)}\right)_{(i)} = 0$. Then $\Gal(K_{(i)}/\mathbf{Q}(\zeta_p))$ is cyclic of order $p$, and global class field theory gives a canonical $\Gal(\mathbf{Q}(\zeta_p)/\mathbf{Q})$-equivariant group isomorphism
$$ \beta_{(i)} : \Gal(K_{(i)}/\mathbf{Q}(\zeta_p)) \xrightarrow{\sim} U^{\otimes 1-i} \otimes \mu_{p}^{\otimes i} \text{ ,}$$
where $U := (\mathbf{Z}/N\mathbf{Z})^{\times}/\left((\mathbf{Z}/N\mathbf{Z})^{\times}\right)^p$ is equipped with the trivial action of $\Gal(\mathbf{Q}(\zeta_p)/\mathbf{Q})$ (the group law of $U$ is denoted multiplicatively).
\end{prop}
\begin{proof}
As in the proof of \cite[Proposition 5.4]{Calegari_Emerton}, global class field theory gives a canonical $\Gal(\mathbf{Q}(\zeta_p)/\mathbf{Q})$-equivariant group isomorphism
$$\alpha_{(i)} : \left( \left( \mathbf{Z}[\zeta_p]/(N) \right)^{\times} \otimes_{\mathbf{Z}} \mathbf{Z}/p\mathbf{Z} \right)_{(i)} \xrightarrow{\sim}  \Gal(K_{(i)}/\mathbf{Q}(\zeta_p)) \text{ .}$$
Let $S$ be the set of prime ideals above $N$ in $\mathbf{Z}[\zeta_p]$. The Chinese remainder theorem shows that there is a canonical group isomorphism
$$f : \left( \mathbf{Z}[\zeta_p]/(N) \right)^{\times} \xrightarrow{\sim} \prod_{\mathfrak{n} \in S} \left( \mathbf{Z}[\zeta_p]/\mathfrak{n} \right)^{\times} \text{ .}$$
We let $\Gal(\mathbf{Q}(\zeta_p)/\mathbf{Q})$ acts on $\prod_{\mathfrak{n} \in S} \left( \mathbf{Z}[\zeta_p]/\mathfrak{n} \right)^{\times}$, via the formula
$$g \cdot (x_{\mathfrak{n}})_{\mathfrak{n} \in S} =  \left(g(x_{g^{-1}(\mathfrak{n})})\right)_{\mathfrak{n} \in S} \text{ ,}$$
where $g \in \Gal(\mathbf{Q}(\zeta_p)/\mathbf{Q})$. The isomorphism $f$ is then $\Gal(\mathbf{Q}(\zeta_p)/\mathbf{Q})$-equivariant.

To conclude the proof of Proposition \ref{comparison_canonical_CFT}, it suffices to construct a canonical $\Gal(\mathbf{Q}(\zeta_p)/\mathbf{Q})$-equivariant group isomorphism 
$$\gamma_{(i)} :\left( \prod_{\mathfrak{n} \in S} \left( \mathbf{Z}[\zeta_p]/\mathfrak{n} \right)^{\times} \otimes_{\mathbf{Z}} \mathbf{Z}/p\mathbf{Z} \right)_{(i)} \xrightarrow{\sim} U^{\otimes 1-i} \otimes \mu_p^{\otimes i} \text{ .}$$ 
We then let $\beta_{(i)} := \gamma_{(i)} \circ \alpha_{(i)}^{-1}$.
There is a canonical group isomorphism $$\left( \mathbf{Z}[\zeta_p]/\mathfrak{n} \right)^{\times}  \otimes_{\mathbf{Z}} \mathbf{Z}/p\mathbf{Z} \xrightarrow{\sim} \left( \mathbf{Z}[\zeta_p]/\mathfrak{n} \right)^{\times}[p]$$ given by 
$$x \otimes 1 \mapsto x^{\frac{N-1}{p}} \text{ .}$$
We thus get a canonical group isomorphism
$$\left( \prod_{\mathfrak{n} \in S} \left( \mathbf{Z}[\zeta_p]/\mathfrak{n} \right)^{\times} \otimes_{\mathbf{Z}} \mathbf{Z}/p\mathbf{Z} \right)_{(i)}  \xrightarrow{\sim} \left( \prod_{\mathfrak{n} \in S} \left( \mathbf{Z}[\zeta_p]/\mathfrak{n} \right)^{\times} [p]\right)_{(i)} \text{ .}$$
Note that there is a canonical group isomorphism $U^{\otimes 1-i} \otimes \mu_p^{\otimes i} \xrightarrow{\sim} V \otimes \left(\mu_p \otimes \hat{V}\right)^{\otimes i}$ where $V = (\mathbf{Z}/N\mathbf{Z})^{\times}[p]$ and $\hat{V} = \Hom(V, \mathbf{Z}/p\mathbf{Z})$. 
Thus, it suffices to construct a canonical group isomorphism
$$ \gamma_{(i)}' : \left( \prod_{\mathfrak{n} \in S} \left( \mathbf{Z}[\zeta_p]/\mathfrak{n} \right)^{\times} [p]\right)_{(i)}  \xrightarrow{\sim} V \otimes (\mu_p \otimes \hat{V})^{\otimes i} \text{ .}$$
Let $\zeta_p \in \mu_p$. If $\mathfrak{n}\in S$, let $\zeta_{\mathfrak{n}} \in V$ be the reduction of $\zeta_p$ modulo $\mathfrak{n}$ and $\hat{\zeta}_{\mathfrak{n}} \in \hat{V}$ be given by $\hat{\zeta}_{\mathfrak{n}}(\zeta_{\mathfrak{n}})=1$. The element $\zeta_p \otimes \hat{\zeta}_{\mathfrak{n}} \in \mu_p \otimes \hat{V}$ only depends on $\mathfrak{n}$, and not on the choice of $\zeta_p$. We let
$$ \gamma_{(i)}'\left((x_{\mathfrak{n}})_{\mathfrak{n} \in S} \right) = \sum_{\mathfrak{n} \in S} x_{\mathfrak{n}} \otimes (\zeta_p \otimes \hat{\zeta}_{\mathfrak{n}})^{\otimes i} \text{ .}$$
This is independent of the choice of $\zeta_p$, so this is canonical. This concludes the proof of Proposition \ref{comparison_canonical_CFT}.
\end{proof}

Let $\ell$ be a rational prime not dividing $N$ such that $\ell \equiv 1 \text{ (modulo }p\text{)}$. Let $\lambda$ be any prime above $(\ell)$ in $\mathbf{Z}[\zeta_p]$. We define 
$$\beta_{\ell} = \beta_{(1)}(\Frob_{\lambda}) \cdot \beta_{(-1)}(\Frob_{\lambda}) \in \mu_p \otimes \left( U^{\otimes 2} \otimes \mu_p^{\otimes -1} \right)= U^{\otimes 2} \text{ .}$$

This does not depend on the choice of $\lambda$ dividing $(\ell)$.

Let $K = \mathbf{Q}(N^{\frac{1}{p}})$. Genus theory shows that $\mathcal{C}_K$ is non-zero \cite[Section 5]{Lecouturier_class_group}. By \cite[Theorem $1.3$ (ii)]{Calegari_Emerton}, if $g_p \geq 2$ then the rank of the $\mathbf{F}_p$-vector space $\mathcal{C}_K$ is $\geq 2$. The group $\mathcal{C}_{K_{(1)}}$ has an action of $\Gal(K_{(1)}/\mathbf{Q}(\zeta_p))$ (a cyclic group of order $p$) and of $\Gal(K_{(1)}/K) = \Gal(\mathbf{Q}(\zeta_p)/\mathbf{Q}) =  (\mathbf{Z}/p\mathbf{Z})^{\times}$ (via the Teichm\"uller character).  Let $J$ be the augmentation ideal of $(\mathbf{Z}/p\mathbf{Z})[\Gal(K_{(1)}/\mathbf{Q}(\zeta_p))]$. If $\chi : \Gal(K_{(1)}/K) \rightarrow (\mathbf{Z}/p\mathbf{Z})^{\times}$ is a character, let $$e_{\chi} = \frac{1}{p-1}\cdot\sum_{g \in \Gal(K_{(1)}/K)} \chi^{-1}(g)\cdot [g] \in (\mathbf{Z}/p\mathbf{Z})[\Gal(K_{(1)}/K)]$$ be the idempotent associated to $\chi$. Let $\chi_0$ be the trivial character and $\omega_p$ be the Teichm\"uller character (considered as a character $\Gal(K_{(1)}/K)$ via the canonical restriction isomorphism $\Gal(K_{(1)}/K) \xrightarrow{\sim} \Gal(\mathbf{Q}(\zeta_p)/\mathbf{Q})$). If $V$ is a $(\mathbf{Z}/p\mathbf{Z})[\Gal(K_{(1)}/K)]$-module, we let $V^{(\chi)} = e_{\chi}\cdot V \subset V$ and $V_{(\chi)} = V/\left(\bigoplus_{\chi' \neq \chi} e_{\chi'}\cdot V \right)$. The map $V^{(\chi)} \rightarrow V_{(\chi)}$ is a group isomorphism.

Since $[K_{(1)}:K] = p-1$ is prime to $p$, the natural map $\mathcal{C}_K  \rightarrow \mathcal{C}_{K_{(1)}}$ is injective, and its image is $(\mathcal{C}_{K_{(1)}})^{(\chi_0)}$.  There exists a generator $\Delta$ of $J$ such that for all character $\chi$ as above, we have \cite[Lemma 2.1]{Lecouturier_class_group}:
 $$\Delta \cdot e_{\chi} = e_{\chi\cdot \omega_p} \cdot \Delta \text{ .}$$
Thus, the multiplication by $\Delta$ induces a natural surjective group homomorphism
$$(\mathcal{C}_{K_{(1)}}/\Delta \cdot \mathcal{C}_{K_{(1)}})^{(\omega_p^{-1})} \rightarrow (\Delta \cdot \mathcal{C}_{K_{(1)}}/\Delta^2 \cdot \mathcal{C}_{K_{(1)}})^{(\chi_0)}$$ 
which gives a surjective group homomorphism:
\begin{equation}\label{delta_map_deformation}
\delta : (\mathcal{C}_{K_{(1)}}/\Delta \cdot \mathcal{C}_{K_{(1)}})_{(\omega_p^{-1})} \rightarrow (\Delta \cdot \mathcal{C}_{K_{(1)}}/\Delta^2 \cdot \mathcal{C}_{K_{(1)}})_{(\chi_0)} \text{ .}
\end{equation}
The group $(\mathcal{C}_{K_{(1)}}/\Delta \cdot \mathcal{C}_{K_{(1)}})_{(\omega_p^{-1})}$ is cyclic of order $p$.

It was proven in \cite[Section 5]{Lecouturier_class_group} that the above map is an isomorphism if and only if $g_p \geq 2$ (which is assumed in this section). Thus, we have a map
$$\mathcal{C}_K = (\mathcal{C}_{K_{(1)}})_{(\chi_0)} \rightarrow (\mathcal{C}_{K_{(1)}}/\Delta^2 \cdot \mathcal{C}_{K_{(1)}})_{(\chi_0)} $$
and an exact sequence
$$0 \rightarrow (\Delta \cdot \mathcal{C}_{K_{(1)}}/\Delta^2 \cdot \mathcal{C}_{K_{(1)}})_{(\chi_0)}\rightarrow (\mathcal{C}_{K_{(1)}}/\Delta^2 \cdot \mathcal{C}_{K_{(1)}})_{(\chi_0)} \rightarrow (\mathcal{C}_{K_{(1)}}/\Delta \cdot \mathcal{C}_{K_{(1)}})_{(\chi_0)} \rightarrow 0 $$
with $$(\Delta \cdot \mathcal{C}_{K_{(1)}}/\Delta^2 \cdot \mathcal{C}_{K_{(1)}})_{(\chi_0)} \simeq (\mathcal{C}_{K_{(1)}}/\Delta \cdot \mathcal{C}_{K_{(1)}})_{(\chi_0)} \simeq \mathbf{Z}/p\mathbf{Z} \text{ .}$$

If $\ell$ is a prime such that $\ell \not\equiv 1 \text{ (modulo }p\text{)}$, let $\mathfrak{p}_{\ell}$ be the unique prime of $\mathcal{O}_K$ above $\ell$ with residual degree $1$ and $\overline{\mathfrak{p}}_{\ell}$ be the image of the class of $\mathfrak{p}_{\ell}$ in $(\mathcal{C}_{K_{(1)}}/\Delta^2 \cdot \mathcal{C}_{K_{(1)}})_{(\chi_0)}$. There is a unique group isomorphism $\alpha : (\mathcal{C}_{K_{(1)}}/\Delta \cdot \mathcal{C}_{K_{(1)}})_{(\chi_0)} \xrightarrow{\sim} U$ such that for any prime $\ell \not\equiv 1 \text{ (modulo }p\text{)}$, we have 
\begin{equation}\label{definition_alpha_deformations}
\alpha(\overline{\mathfrak{p}}_{\ell}) = \overline{\ell}^{\frac{1}{2}} \text{ ,}
\end{equation}
where $\overline{\ell}$ is the image of $\ell$ in $U$.
We let $$\pi :  (\mathcal{C}_{K_{(1)}}/\Delta^2 \cdot \mathcal{C}_{K_{(1)}})_{(\chi_0)} \rightarrow U$$ be the morphism induced by $\alpha$.

\begin{thm}\label{main_thm_deformations_f_2}
Assume that $g_p \geq 2$, so that $f_2 \in \mathcal{M}/p\cdot \mathcal{M}$ exists (and is uniquely determined modulo the subgroup generated by $f_0$ and $f_1$).

Let $\mathcal{P}$ be the set of prime numbers different from $N$. Let $A_0 = (\ell+1)_{\ell \in \mathcal{P}}$ and $A_1=\left(\frac{\ell-1}{2}\cdot \log(\ell)\right)_{\ell \in \mathcal{P}}$ in $(\mathbf{Z}/p\mathbf{Z})^{(\mathcal{P})}$.

There exists $C \in (\mathbf{Z}/p\mathbf{Z})^{\times}$ and a group homomorphism $\pi' :  (\mathcal{C}_{K_{(1)}}/\Delta^2 \cdot \mathcal{C}_{K_{(1)}})_{(\chi_0)} \rightarrow U^{\otimes 2}$ whose image in $\Hom( (\mathcal{C}_{K_{(1)}}/\Delta^2 \cdot \mathcal{C}_{K_{(1)}})_{(\chi_0)}, U^{\otimes 2})/ \pi \otimes U$ is non-zero and uniquely determined, such that we have, in $(U^{\otimes 2})^{(\mathcal{P})}$ modulo the subgroup generated by $A_0 \otimes U^{\otimes 2}$ and $A_1 \otimes U^{\otimes 2}$: 
\begin{equation}\label{deformations_description_f_2_thm}
(a_{\ell}(f_2) \otimes \gamma^{\otimes 2})_{\ell \in \mathcal{P}} = \left((\overline{\ell}^{\otimes 2})^{\frac{1}{4}}\cdot \epsilon_{\ell}\right)_{\ell \in \mathcal{P}}
\end{equation}
where $\gamma \in U$ is such that $\log(\gamma) \equiv 1 \text{ (modulo }p\text{)}$ and $\epsilon_{\ell} \in U^{\otimes 2}$ is defined as follows. 
\begin{itemize}
\item If $\ell \equiv 1 \text{ (modulo }p\text{)}$, then we let $\epsilon_{\ell} = \beta_{\ell}^{C}$.
\item If $\ell \not\equiv 1 \text{ (modulo }p \text{)}$, then we let 
$$\epsilon_{\ell} = \left( \pi'(\overline{\mathfrak{p}}_{\ell})\cdot (\overline{\ell}^{\otimes 2})^{\frac{1}{8}}\right)^{\ell-1} \text{ .}$$
\end{itemize}
\end{thm}
\begin{proof}
By \cite[Corollary $1.6$]{Calegari_Emerton}, if $g_p \geq 2$ then there exists a Galois representation:
$$\rho : \Gal(\overline{\mathbf{Q}}/\mathbf{Q}) \rightarrow \text{GL}_2\left((\mathbf{Z}/p\mathbf{Z})[x]/x^3\right)$$
which satisfies certain deformation conditions defined in \cite[p. 100]{Calegari_Emerton}. We denote by $F$ the number field cut out by the kernel of $\rho$. By \cite[Proposition $5.5$]{Calegari_Emerton}, one can assume that $\rho$ has the following form:
$$
\rho =  \begin{pmatrix}
\omega_p\cdot (1 + a \cdot x + a'\cdot x^2) & x \cdot \beta_1 + x^2\cdot b\\
x\cdot \beta_{-1} + x^2\cdot c & 1 +d\cdot x + d'\cdot x^2
\end{pmatrix} \text{ .}
$$
where $a$, $a'$, $b$, $c$, $d$, $d'$ and $\beta_{-1}$ are maps $\Gal(\overline{\mathbf{Q}}/\mathbf{Q}) \rightarrow \mathbf{Z}/p\mathbf{Z}$.

We can choose $\rho$ such that for all $\ell \neq N$, if $\Frob_{\ell}$ is any Frobenius element at $\ell$, we have:
\begin{equation}\label{deformations_equation_trace_x^2_3}
a(\Frob_{\ell}) = \frac{\log(\ell)}{2} \text{ .}
\end{equation}
The cocycle $\beta_{-1} : \Gal(\overline{\mathbf{Q}}/\mathbf{Q}) \rightarrow \mathbf{Z}/p\mathbf{Z}$ satisfies, for all $g$, $g'$ $\in \Gal(\overline{\mathbf{Q}}/\mathbf{Q})$: $$\beta_{-1}(g\cdot g') = \beta_{-1}(g)+ \omega_p(g)^{-1}\cdot \beta_{-1}(g') \text{ .}$$
Furthermore, the restriction of $\beta_{-1}$ to $\Gal(\overline{\mathbf{Q}}/\mathbf{Q}(\zeta_p))$ is uniquely determined since we have fixed the choices of $\beta_{1}$, $a$ and the conjugacy class of $\rho$.

Since $\text{det}(\rho) = \omega_p$, we have:
\begin{equation}\label{deformations_equation_trace_x^2_1}
a+d=0
\end{equation}
and
\begin{equation}\label{deformations_equation_trace_x^2_2}
a'+d' =\omega_p^{-1}\cdot \beta_1\cdot \beta_{-1} - a\cdot d \text{ .}
\end{equation}

Furthermore, as in \cite[Section 5.2]{Calegari_Emerton} the subgroup $\rho(\Gal(F/\mathbf{Q}))$ of $\text{GL}_2((\mathbf{Z}/p\mathbf{Z})[x]/x^3)$ is generated by $$\left\{\begin{pmatrix} \alpha & 0 \\ 0 & 1 \end{pmatrix}, \alpha \in (\mathbf{Z}/p\mathbf{Z})^{\times} \right\} \cup \Ker\left(\SL_2((\mathbf{Z}/p\mathbf{Z})[x]/x^3) \rightarrow \SL_2(\mathbf{Z}/p\mathbf{Z})\right) \text{ .}$$ 

Since $\rho$ is a group homomorphism, we have for all $g$, $g'$ $\in \text{Gal}(\overline{\mathbf{Q}}/\mathbf{Q})$:
$$a'(gg') = a'(g)+a'(g')+a(g)\cdot a(g')+\omega_p(gg')^{-1}\cdot \beta_1(g)\cdot \beta_1(g') \text{ .}$$
Since the restriction of $\beta_1$ to $\Gal(\overline{\mathbf{Q}}/K)$ is trivial, the restriction of $a'':=a'-\frac{a^2}{2}$ to $\Gal(\overline{\mathbf{Q}}/K)$ is a group homomorphism.

\begin{lem}\label{deformations_admitted_lemma_commutator}
The commutator subgroup of $\rho(\Gal(F/K))$ is generated by the matrices $\begin{pmatrix} 1 & x^2 \\ 0 & 1\end{pmatrix}$, $\begin{pmatrix} 1 & 0 \\ x & 1\end{pmatrix}$ and $\begin{pmatrix} 1 & 0 \\ x^2 & 1\end{pmatrix}$, and the abelianization of $\rho(\Gal(F/K))$ is isomorphic to $(\mathbf{Z}/p\mathbf{Z})^2 \times (\mathbf{Z}/p\mathbf{Z})^{\times}$.
\end{lem}
\begin{proof}
The subgroup $\rho(\Gal(F/K))$ of $\rho(\Gal(F/\mathbf{Q}))$ consists of those matrices in $\rho(\Gal(F/\mathbf{Q}))$ whose upper-right coefficient is $0$ modulo $x^2$. Let $A = \begin{pmatrix} a_1 & 0 \\ a_3 & a_4 \end{pmatrix}$, $B=\begin{pmatrix} b_1 & b_2 \\ b_3 & b_4 \end{pmatrix}$, $C=\begin{pmatrix} c_1 & 0 \\ c_3 & c_4 \end{pmatrix}$ and $D = \begin{pmatrix} d_1 & d_2 \\ d_3 & d_4 \end{pmatrix}$  be in $M_2(\mathbf{Z}/p\mathbf{Z})$ and let $\alpha$, $\beta$ in $\mathbf{Z}/p\mathbf{Z}^{\times}$. Let $\Delta_{\alpha} = \begin{pmatrix} \alpha & 0 \\ 0 & 1 \end{pmatrix}$ and $\Delta_{\beta} =   \begin{pmatrix} \beta & 0 \\ 0 & 1 \end{pmatrix}$. A formal computation shows that we have, in $\text{GL}_2((\mathbf{Z}/p\mathbf{Z})[x]/x^3)$:
\begin{align*}
&\left(\Delta_{\alpha} + x\cdot A + x^2\cdot B\right)\cdot \left(\Delta_{\beta} + x\cdot C + x^2\cdot D\right) \cdot \left(\Delta_{\alpha} + x\cdot A + x^2\cdot B\right)^{-1}\cdot \left(\Delta_{\beta} + x\cdot C + x^2\cdot D\right)^{-1}  \\&= \begin{pmatrix} 1 & 0 \\ 0 & 1 \end{pmatrix} + x\cdot \begin{pmatrix} 0 & 0 \\ \frac{\beta-1}{\alpha\beta}\cdot a_3+ \frac{1-\alpha}{\alpha\beta}\cdot c_3 & 0 \end{pmatrix}+x^2\cdot  \begin{pmatrix} 0 & (1-\beta)\cdot b_2 + (\alpha-1)\cdot d_3 \\ X & 0 \end{pmatrix}
\end{align*}
where 
$$X = \frac{1}{\alpha\beta}\cdot \left(\frac{1-\beta}{\alpha}\cdot a_1a_3 + \frac{1}{\beta}\cdot a_3c_1 - \frac{1}{\alpha}\cdot a_1c_3 + a_4c_3+\frac{\alpha-1}{\beta}\cdot c_1c_3-a_3c_4+ (\beta-1)\cdot b_3 + (1-\alpha)\cdot d_3 \right)  \text{ .}$$
Thus, the commutator of $\rho(\Gal(F/K))$ is 
$$\left\{\begin{pmatrix}1 & u \cdot x^2 \\ v\cdot x + w\cdot x^2 & 1 \end{pmatrix}\text{, }u,v,w \in \mathbf{Z}/p\mathbf{Z}\right\} \text{ .}$$
Thus,  the commutator of $\rho(\Gal(F/K))$ has order $p^3$ and is generated by $\begin{pmatrix} 1 & x^2 \\ 0 & 1\end{pmatrix}$, $\begin{pmatrix} 1 & 0 \\ x & 1\end{pmatrix}$ and $\begin{pmatrix} 1 & 0 \\ x^2 & 1\end{pmatrix}$. The second claim follows immediately. 
\end{proof}

By Lemma \ref{deformations_admitted_lemma_commutator}, the kernel of the couple $(a,a'')$ cuts out an abelian extension $H$ of $K$ such that $\Gal(H/K) \simeq (\mathbf{Z}/p\mathbf{Z})^2$. By \cite[proof of Lemma $5.13$]{Calegari_Emerton}, $H$ is unramified everywhere over $K$. Furthermore, as a subquotient of the image of $\rho$, $\Gal(H/K)$ is generated by the image of the matrices $\begin{pmatrix}
 1+x^2 & 0 \\
  0 & 1-x^2
  \end{pmatrix}$ and $\begin{pmatrix}
 1+x & 0 \\
  0 & (1+x)^{-1}
  \end{pmatrix}$. 

Let $\mathcal{H}$ be the maximal abelian extension of $K_{(1)}$ unramified everywhere such that $\Gal(\mathcal{H}/K_1)$ is a group of exponent $p$, and let $\mathcal{G} = \Gal(\mathcal{H}/K_{(1)})$. Global class field theory gives us a canonical $\Gal(K_1/\mathbf{Q})$-equivariant group isomorphism $\mathcal{C}_{K_{(1)}} \simeq \mathcal{G}$. We claim that the compositum $K_{(0)}K_{(1)}$ is unramified everywhere over $K_{(1)}$. One only needs to check that $K_{(0)}K_{(1)}$ is unramified at the primes above $N$ in $K_{(1)}$. This follows from the fact that $\rho(I_N)$ is cyclic of order $p$, where $I_N \subset \Gal(\overline{\mathbf{Q}}/\mathbf{Q})$ is the inertia at $N$, and that $K_{(1)}$ totally ramified of ramification index $p$ at the primes above $N$ in $\mathbf{Q}(\zeta_p)$. Thus, we have $K_{(0)}K_{(1)} \subset \mathcal{H}$.

\begin{lem}\label{Section_deformations_comparison_CE_class_group}
We have the following inclusions of subgroups of $\mathcal{G}$:
$$\Gal(\mathcal{H}/K_{(0)}K_{(1)}) = e_{\chi_0}(\Delta \cdot \mathcal{G}) \cdot \prod_{\chi \neq \chi_0} e_{\chi}(\mathcal{G})$$
and
$$\Gal(\mathcal{H}/H(\zeta_p)) = e_{\chi_0}(\Delta^2 \cdot \mathcal{G}) \cdot \prod_{\chi \neq \chi_0} e_{\chi}(\mathcal{G}) \text{ .}$$
Here, $\chi$ runs trough the set of characters $\Gal(K_{(1)}/K) \rightarrow (\mathbf{Z}/p\mathbf{Z})^{\times}$.
\end{lem}
\begin{proof}
The first equality comes from genus theory: $K_{(0)}K_{(1)}$ is the largest extension of $K_{(1)}$ which is unramified everywhere,  abelian over $\mathbf{Q}(\zeta_p)$ and such that $\Gal(K_{(1)}/K) \simeq \Gal(\mathbf{Q}(\zeta_p)/\mathbf{Q})$ acts by $\chi_0$. 

Since $\Gal(K_{(1)}/K)$ acts trivially on $\Gal(H(\zeta_p)/K_{(1)})$, we have
$$\prod_{\chi \neq \chi_0} e_{\chi}(\mathcal{G}) \subset \Gal(\mathcal{H}/H(\zeta_p)) \text{ .}$$
Since $\Gal(\mathcal{H}/H(\zeta_p))$ is a subgroup of index $p$ of $\Gal(\mathcal{H}/K_{(0)}K_{(1)})$, it suffices to prove the following inclusion:
$$
e_{\chi_0}(\Delta^2 \cdot \mathcal{G}) \cdot \prod_{\chi \neq \chi_0} e_{\chi}(\mathcal{G})  \subset \Gal(\mathcal{H}/H(\zeta_p)) \text{ .}
$$
Since $\Delta \cdot e_{\omega_p^{-1}} = e_{\chi_0} \cdot \Delta$, it suffices to show the following lemma. 
\begin{lem}
The compositum $K_{(1)}K_{(-1)}$ contained in $\mathcal{H}$.
We have $$\Gal(\mathcal{H}/K_{(1)}K_{(-1)}) = e_{\omega_p^{-1}}(\Delta \cdot \mathcal{G}) \cdot \prod_{\chi \neq \omega_p^{-1}} e_{\chi}(\mathcal{G})$$
and
$$\Delta \cdot \Gal(\mathcal{H}/K_{(1)}K_{(-1)}) \subset \Gal(\mathcal{H}/H(\zeta_p)) \text{ .}$$
\end{lem}
\begin{proof}
The compositum $K_{(1)}K_{(-1)}$ is unramified everywhere over $K_{(1)}$ \cite[Section 5.2]{Calegari_Emerton}, so is contained in $\mathcal{H}$
The first equality comes from genus theory: $K_{(-1)}$ is the largest extension of $K_{(1)}$ which is unramified everywhere,  abelian over $\mathbf{Q}(\zeta_p)$ and such that $\Gal(\mathbf{Q}(\zeta_p)/\mathbf{Q})$ acts by $\omega_p^{-1}$. 

Let $M = K_{(1)} \cdot K_{(-1)} \cdot H$; this is an abelian extension of $K_{(1)}$ contained in $F$. The group $\rho(\Gal(F/M))$ is the set of matrices of the form $\text{Id} + x^2\cdot  \begin{pmatrix}
 0& a \\
  b &0
  \end{pmatrix}$
  where $a$ and $b$ are in $\mathbf{Z}/p\mathbf{Z}$. Thus, $\Gal(F/M)$ is a normal subgroup of $\Gal(F/\mathbf{Q}(\zeta_p))$, which means that $M$ is Galois over $\mathbf{Q}(\zeta_p)$. There is an action of $\Gal(K_{(1)}/\mathbf{Q})$ on $\Gal(M/\mathbf{Q}(\zeta_p))$. It suffices to show that we have
  $$\Delta \cdot \Gal(M/K_{(1)}K_{(-1)}) \subset \Gal(M/H(\zeta_p))\text{ .}$$
In fact, we shall prove that we have $\Delta \cdot \Gal(M/K_{(1)}K_{(-1)})  = 0$. The group $\rho\left( \Gal(M/K_{(1)}K_{(-1)})\right)$ is generated by the images of $\begin{pmatrix}
 1+x^2 & 0 \\
  0 & 1-x^2
  \end{pmatrix}$ and $\begin{pmatrix}
 1+x & 0 \\
  0 & (1+x)^{-1}
  \end{pmatrix}$. This follows indeed from the description of $\rho\left(\Gal(H/K)\right)$ given above.  
 Let $s$ be the image of $\begin{pmatrix}
 1& x \\
  0 & 1
  \end{pmatrix}$ in $\rho\left(\Gal(M/\mathbf{Q}(\zeta_p))\right)$. The image of $s$ in $\rho\left(\Gal(K_{(1)}/\mathbf{Q}(\zeta_p)) \right)$ generator of $\rho\left(\Gal(K_{(1)}/\mathbf{Q}(\zeta_p)) \right)$. We have to show that the commutators $$\begin{pmatrix}
 1& x \\
  0 & 1
  \end{pmatrix} \cdot \begin{pmatrix}
 1+x^2 & 0 \\
  0 & 1-x^2
  \end{pmatrix} \cdot \begin{pmatrix}
 1& x \\
  0 & 1
  \end{pmatrix}^{-1}\cdot \begin{pmatrix}
 1+x^2 & 0 \\
  0 & 1-x^2
  \end{pmatrix}^{-1}$$
  and
  $$\begin{pmatrix}
 1& x \\
  0 & 1
  \end{pmatrix} \cdot \begin{pmatrix}
 1+x & 0 \\
  0 & (1+x)^{-1}
  \end{pmatrix} \cdot \begin{pmatrix}
 1& x \\
  0 & 1
  \end{pmatrix}^{-1}\cdot \begin{pmatrix}
 1+x & 0 \\
  0 & (1+x)^{-1}
  \end{pmatrix}^{-1}$$
  are in $\rho\left(\Gal(F/M)\right)$.
 
If $A$, $B$, $C$ and $D$ are in $\text{M}_2(\mathbf{Z}/p\mathbf{Z})$, the commutator of $1+x\cdot A + x^2\cdot B$ and $1+x\cdot C+x^2\cdot D$ is $1+x^2\cdot (AC-CA)$ (in $\text{GL}_2((\mathbf{Z}/p\mathbf{Z})[x]/x^3)$).
Thus, these two commutators are respectively $0$ and $\begin{pmatrix}
 1 & -2x^2 \\
  0 & 1
  \end{pmatrix}$. These matrices are in $\rho(\Gal(F/M))$, which concludes the proof of Lemma \ref{Section_deformations_comparison_CE_class_group}.
  \end{proof}
\end{proof}

By Lemma \ref{Section_deformations_comparison_CE_class_group}, Global class field theory gives a canonical group isomorphism 
$$
\varpi : (\mathcal{C}_{K_{(1)}}/\Delta^2 \cdot \mathcal{C}_{K_{(1)}})_{(\chi_0)} \xrightarrow{\sim} \Gal(H/K) \text{ .}
$$
The group homomorphism $\log\circ \pi : (\mathcal{C}_{K_{(1)}}/\Delta^2 \cdot \mathcal{C}_{K_{(1)}})_{(\chi_0)} \rightarrow \mathbf{Z}/p\mathbf{Z}$ is the composition of the restriction of $a$ to $\Gal(H/K)$ with $\varpi$.
We let $$\pi' : (\mathcal{C}_{K_{(1)}}/\Delta^2 \cdot \mathcal{C}_{K_{(1)}})_{(\chi_0)} \rightarrow U^{\otimes 2}$$
be the group homomorphism such that $\log^{\otimes 2}\circ \pi'$ is the composition of the restriction of $a''$ to $\Gal(H/K)$ with $\varpi$, where $\log^{\otimes 2} : U^{\otimes 2} \xrightarrow{\sim} \mathbf{Z}/p\mathbf{Z}$ is defined by $\log^{\otimes 2}(a\otimes b) = \log(a)\cdot \log(b)$. The group homomorphisms $\pi'$ and $\pi \otimes \gamma$ are not proportionals (recall that $\gamma$ is a generator of $U$ such that $\log(\gamma)=1$).

\begin{lem}\label{deformations_coeff_f_2}
For each prime $\ell \neq N$, let $a_{\ell} \in \mathbf{Z}/p\mathbf{Z}$ be defined by the following equality in $(\mathbf{Z}/p\mathbf{Z})[x]/x^3$:
$$\Tr\left(\rho\left(\Frob_{\ell}\right)\right) = \ell+1 + x\cdot \frac{\ell-1}{2}\cdot \log(\ell) + x^2\cdot a_{\ell} \text{ ,}$$
where $\Frob_{\ell}$ is any Frobenius substitution at $\ell$ in $\Gal(F/\mathbf{Q})$.

We have, in $(\mathbf{Z}/p\mathbf{Z})^{(\mathcal{P})}$ modulo the subgroup generated by $A_0:=(a_{\ell}(f_0))_{\ell \in  \mathcal{P }}$ and $A_1:=(a_{\ell}(f_1))_{\ell \in \mathcal{P}}$:
$$(a_{\ell}(f_2))_{\ell \in \mathcal{P}} = (a_{\ell})_{\ell \in \mathcal{P}} \text{ .}$$
\end{lem}
\begin{proof}
By \cite[Theorem 1.5]{Calegari_Emerton}, the Galois representation $\rho$ corresponds to a ring homomorphism $\varphi : \tilde{\mathbf{T}} \rightarrow (\mathbf{Z}/p\mathbf{Z})[x]/x^3$ where we recall that $\tilde{\mathbf{T}}$ is the completion of the full Hecke algebra of weight $2$ and level $\Gamma_0(N)$ at the $p$-maximal Eisenstein ideal. Furthermore, we have by construction $\varphi(T_{\ell}) = \Tr(\rho(\Frob_{\ell}))$. We have normalized $\rho$ such that $\Tr(\rho(\Frob_{\ell})) \equiv \ell+1 + x\cdot \frac{\ell-1}{2} \cdot \log(\ell) \text{ (modulo }x^2\text{)}$. Thus, we have 
\begin{equation}\label{deformations_coeff_f_2_normalization}
\varphi(T_{\ell}-\ell-1) \equiv x \cdot \frac{\ell-1}{2}\cdot \log(\ell) \text{ (modulo }x^2 \text{)} \text{ .}
\end{equation}
The morphism $\varphi$ corresponds a to a modular form $F$ in $\text{M}_2(\Gamma_0(N), (\mathbf{Z}/p\mathbf{Z})[x]/x^3)$. For each $n\geq 1$, let $a_n(F)$ be the $n$th Fourier coefficient of $F$ at the cusp $\infty$. We write $$a_n(F) = a_n^0(F) + x\cdot a_n^1(F)+x^2\cdot a_n^2(F)$$ in $(\mathbf{Z}/p\mathbf{Z})[x]/x^3$ where $a_n^i(F) \in \mathbf{Z}/p\mathbf{Z}$ for $i \in \{0,1,2\}$. The $q$-expansion $\sum_{n \geq 0} a_n^0(F) \cdot q^n$ in $\mathbf{Z}/p\mathbf{Z}[[q]]$ is the $q$-expansion of the (normalized) Eisenstein series $f_0$ of weight $2$ and level $\Gamma_0(N)$. Thus, $x \cdot \sum_{n \geq 0} (a_n^1(F)+x\cdot a_n^2(F))\cdot q^n $ is the $q$-expansion of a modular form (namely $F-f_0$). By the $q$-expansion principle \cite[Corollary 1.6.2]{Katz_properties}, $\sum_{n \geq 0} (a_n^1(F)+x\cdot a_n^2(F))\cdot q^n$ is the $q$-expansion at $\infty$ of a modular form in $M_2\left(\Gamma_0(N), (\mathbf{Z}/p\mathbf{Z})[x]/x^2\right)$. Thus, by specializing at $x=0$, we get that $\sum_{n \geq 0} a_n^1(F) \cdot q^n$ is the $q$-expansion at $\infty$ of a modular form in $M_2(\Gamma_0(N), \mathbf{Z}/p\mathbf{Z})$, which we call $F_1$. Again, by the $q$-expansion principle, $\sum_{n \geq 2 }a_n^2(F)\cdot  q^n$ is the $q$-expansion at $\infty$ of a modular form in $M_2(\Gamma_0(N), \mathbf{Z}/p\mathbf{Z})$, which we call $F_2$.

We have, for every prime $\ell \neq N$:
$$(T_{\ell}-\ell-1)(F) = (T_{\ell}-\ell-1)(f_0)+x\cdot (T_{\ell}-\ell-1)(F_1)+x^2\cdot (T_{\ell}-\ell-1)(F_2) = x\cdot (T_{\ell}-\ell-1)(F_1)+x^2\cdot (T_{\ell}-\ell-1)(F_2) 
\text{ .}$$   
Thus, we have:
$$(T_{\ell}-\ell-1)(F) \equiv x \cdot (T_{\ell}-\ell-1)(F_1) \text{ (modulo }x^2\text{)} \text{ .}$$
On the other hand, we have $(T_{\ell}-\ell-1)(F) = \varphi(T_{\ell}-\ell-1)\cdot F$. By (\ref{deformations_coeff_f_2_normalization}), we get:
$$(T_{\ell}-\ell-1)(F_1) = \frac{\ell-1}{2} \cdot \log(\ell) \cdot f_0\text{ .}$$ 
Thus, $F_1$ is the first higher Eisenstein element $f_1$ (modulo the subgroup generated by $f_0$). Similarly, $F_2$ is the second higher Eisenstein element $f_2$ (modulo the subgroup generated by $f_0$ and $f_1$), which concludes the proof of Lemma \ref{deformations_coeff_f_2}.
\end{proof}

By Lemma \ref{deformations_coeff_f_2} and (\ref{deformations_equation_trace_x^2_2}), we have in $(\mathbf{Z}/p\mathbf{Z})^{(\mathcal{P})}$ modulo the subgroup generated by $A_0$ and $A_1$:
\begin{align*}
\left(a_{\ell}(f_2)\right)_{\ell \in \mathcal{P}} &= \left( \ell \cdot a'(\Frob_{\ell})+d'(\Frob_{\ell}) \right)_{\ell \in \mathcal{P}} \\&
= \left((\ell-1)\cdot a'(\Frob_{\ell}) + \frac{\beta_1(\Frob_{\ell}) \cdot \beta_{-1}(\Frob_{\ell})}{\ell} + a(\Frob_{\ell})^2\right)_{\ell \in \mathcal{P}} \\&
= \left((\ell-1)\cdot (a''(\Frob_{\ell}) + \frac{\log(\ell)^2}{8}) + \frac{\beta_1(\Frob_{\ell}) \cdot \beta_{-1}(\Frob_{\ell})}{\ell} + \frac{\log(\ell)^2}{4}\right)_{\ell \in \mathcal{P}}
\end{align*}
where in the second equality we used (\ref{deformations_equation_trace_x^2_1}) and (\ref{deformations_equation_trace_x^2_2}) and in the last equality we used (\ref{deformations_equation_trace_x^2_3}) and the definition of $a''$. 

If $\ell \not\equiv 1 \text{ (modulo }p\text{)}$, then we may choose $\Frob_{\ell}$ in $\Gal(F/K)$. With this choice, we have $\beta_1(\Frob_{\ell}) = 0$ and $a''(\Frob_{\ell}) = \log^{\otimes 2}\left(\pi'(\overline{\mathfrak{p}}_{\ell})\right)$. If we change $\pi'$ by a multiple of $\pi$, we only change the right-hand side of (\ref{deformations_description_f_2_thm}) by a multiple of $A_1$. Thus, the choice of $\pi'$ is only well-defined up to a multiple of $\pi$. 

If $\ell \equiv 1 \text{ (modulo }p\text{)}$, we have $\Frob_{\ell} \in \Gal(F/\mathbf{Q}(\zeta_p))$. The restriction of $\beta_1$ to $\Gal(F/\mathbf{Q}(\zeta_p))$ factors through $\Gal(K_{(1)}/\mathbf{Q}(\zeta_p))$, thus giving a group isomorphism $\beta_{1}' : \Gal(K_{(1)}/\mathbf{Q}(\zeta_p)) \xrightarrow{\sim} \mathbf{Z}/p\mathbf{Z}$. Fix a group isomorphism $\log_{\zeta} : \mu_p \xrightarrow{\sim} \mathbf{Z}/p\mathbf{Z}$. Then $\beta_1'$ and $\log_{\zeta} \circ \beta_{(1)}$ are proportionals. Similarly, $\beta_{-1}$ gives a group isomorphism $\beta_{-1}' : \Gal(K_{(-1)}/\mathbf{Q}(\zeta_p)) \xrightarrow{\sim} \mathbf{Z}/p\mathbf{Z}$ which is proportional to $\left( \log^{\otimes 2} \otimes \log_{\zeta}^{\otimes -1}\right) \circ \beta_{(-1)}$, where $\log_{\zeta}^{\otimes -1} : \mu_p^{\otimes -1} \xrightarrow{\sim} \mathbf{Z}/p\mathbf{Z}$ is induced by $\log_{\zeta}$. Thus, there exists $C \in (\mathbf{Z}/p\mathbf{Z})^{\times}$ (independant of $\ell$) such that $\beta_{1}(\Frob_{\ell}) \cdot \beta_{-1}(\Frob_{\ell}) =  C \cdot \log^{\otimes 2}(\beta_{\ell})$.

This concludes the proof of Theorem \ref{main_thm_deformations_f_2}.
\end{proof}

\begin{rem}
The constant $C$ is an invariant of the conjugacy class of the modular Galois representation $\rho : \Gal(\overline{\mathbf{Q}}/\mathbf{Q}) \rightarrow \GL_2\left(\mathbf{T}/\left(p\cdot\mathbf{T}+I^2\right)\right)$ coming from $J_0(N)[I^2+(p)](\overline{\mathbf{Q}})$. The forthcoming thesis of Jun Wang (a Phd student of Sharifi) seems to imply that, at least conditionnally on a conjecture of Sharifi (which seems related to our Conjecture \ref{odd_modSymb_Sharifi_conj}), we have $C=1$.
\end{rem}

\begin{corr}
Let $\mathcal{P}'$ be the set of prime numbers $\ell$ satisfying the following conditions.
\begin{itemize}
\item We have $\ell \neq N,p$
\item We have $\ell \equiv 1 \text{ (modulo }p\text{).}$
\item The prime $N$ is a $p$th power modulo $\ell$.
\end{itemize}
The element $(a_{\ell}(f_2))_{\ell \in \mathcal{P}'}$ of $(\mathbf{Z}/p\mathbf{Z})^{(\mathcal{P}')}$ is well-defined modulo the element $X=(\ell+1)_{\ell \in \mathcal{P}'}$. We have the following equality modulo $(\mathbf{Z}/p\mathbf{Z}) \cdot X$ in $(\mathbf{Z}/p\mathbf{Z})^{(\mathcal{P}')}$:
$$(a_{\ell}(f_2))_{\ell \in \mathcal{P}'} = \left(\frac{\log(\ell)^2}{4}\right)_{\ell \in \mathcal{P}'} \text{ .}$$
\end{corr}
\begin{proof}
This follows from Theorem \ref{main_thm_deformations_f_2} and the fact that if $\ell\equiv 1 \text{ (modulo }p\text{)}$, the following assertions are equivalent:
\begin{itemize}
\item We have $\beta_{(1)}(\Frob_{\ell}) = 1$. 
\item The prime $\ell$ splits completely in $\mathbf{Q}(\zeta_p, N^{\frac{1}{p}})$ 
\item We have $\ell \equiv 1 \text{ (modulo }p\text{)}$ and $N$ is a $p$th power modulo $\ell$.
\end{itemize}
\end{proof}
\newpage

\section{Tables and summary of our results}

The following table for $g_p$ was extracted from the data of \cite{Bartosz} (Naskrecki extended his computations to $N <13000$, and kindly sent the result to us). We give the $5$-uples $(N, p, t, g_p, m)$, where $N$, $p$, $t$ and $g_p$ were defined in the article, and $m$ is the number of conjugacy class of newforms which are congruent to the Eisenstein series modulo $p$. The range is $N < 13000$, and we only display the data where $p\geq 5$ and $g_p \geq 3$.

\begin{table}[htb]
\centering
\label{table_g_p}
\caption {Numerical data for $g_p$}
\begin{tabular}{|l|l|l|l|l|}
\hline
$N$  & $p$  & $t$ & $g_p$ & $m$ \\ \hline
181   & 5  & 1 & 3 & 1 \\ \hline
1571  & 5  & 1 & 3 & 1 \\ \hline
2621  & 5  & 1 & 3 & 1 \\ \hline
3001  & 5  & 3 & 6 & 3 \\ \hline
3671  & 5  & 1 & 5 & 1 \\ \hline
4931  & 5  & 1 & 3 & 1 \\ \hline
5381  & 5  & 1 & 3 & 1 \\ \hline
5651  & 5  & 2 & 4 & 2 \\ \hline
5861  & 5  & 1 & 4 & 1 \\ \hline
6451  & 5  & 2 & 3 & 2 \\ \hline
9001  & 5  & 3 & 4 & 2 \\ \hline
9521  & 5  & 1 & 3 & 1 \\ \hline
10061 & 5  & 1 & 3 & 1 \\ \hline
11321 & 5  & 1 & 3 & 1 \\ \hline
\end{tabular}
\quad
\begin{tabular}{|l|l|l|l|l|}
\hline
$N$  & $p$  & $t$ & $g_p$ & $m$ \\ \hline
12101 & 5  & 2 & 4 & 2 \\ \hline
12301 & 5  & 2 & 3 & 2 \\ \hline
12541 & 5  & 1 & 3 & 1 \\ \hline
12641 & 5  & 1 & 4 & 1 \\ \hline
12791 & 5  & 1 & 3 & 1 \\ \hline
4159  & 7  & 1 & 4 & 1 \\ \hline
4229  & 7  & 1 & 3 & 1 \\ \hline
4957  & 7  & 1 & 3 & 1 \\ \hline
7673  & 7  & 1 & 3 & 1 \\ \hline
10627 & 7  & 1 & 3 & 1 \\ \hline
11159 & 7  & 1 & 3 & 1 \\ \hline
1321  & 11 & 1 & 3 & 1 \\ \hline
6761  & 13 & 2 & 3 & 2 \\ \hline
1381  & 23 & 1 & 3 & 1 \\ \hline
\end{tabular}
\end{table}

\begin{landscape}
Let $r$ be an integer such that $1 \leq r \leq t = v_p\left(\frac{N-1}{12}\right)$. The following tables summarize our results about the integer $n(r,p)$ using the three $\tilde{\mathbb{T}}$-modules $M$, $M^-$ and $M_+$ studied in this paper (the equalities below take place in $\mathbf{Z}/p^r\mathbf{Z}$ unless explicitly stated otherwise). 
\begin{table}[htb]
 \centering
    \caption{The case $p \geq 5$}
 \begin{tabular}{c|c|c|c|l}
\cline{2-4}
\multicolumn{1}{l|}{}       & $n(r,p) \geq 2$                                  & $n(r,p) \geq 3$                                                                                                                                                                                                                               & $n(r,p) \geq 4$                                                                                                                                                                                                                                &  \\ \cline{1-4}
\multicolumn{1}{|c|}{$M$}   & $\sum_{\lambda \in L} \log(H'(\lambda)) = 0$     & $\sum_{\lambda \in L} 3\cdot \log(H'(\lambda))^2 - 4\cdot \log(\lambda)^2 = 0$                                                                                                                                                                & ?                                                                                                                                                                                                                                              &  \\ \cline{1-4}
\multicolumn{1}{|c|}{$M_+$} & $\sum_{k=1}^{\frac{N-1}{2}} k \cdot \log(k) = 0$ & $\sum_{k=1}^{\frac{N-1}{2}} k \cdot \log(k)^2 = 0$                                                                                                                                                                                            & ?                                                                                                                                                                                                                                              &  \\ \cline{1-4}
\multicolumn{1}{|c|}{$M^-$} & $\sum_{k=1}^{\frac{N-1}{2}} k \cdot \log(k) = 0$ & \begin{tabular}[c]{@{}c@{}}$\sum_{i=1}^{N-1},\left(\sum_{j=i}^{N-1} F_{0,p}(g^j) \right)\cdot \left(1-\zeta_N^{g^{i-1}}, \frac{1-\zeta_N}{1-\zeta_N^{g^{-1}}}\right)_r= 0$ \\ in $J\cdot \mathcal{K}_r/J^2 \cdot \mathcal{K}_r$\end{tabular} & \begin{tabular}[c]{@{}c@{}}$\sum_{i=1}^{N-1} \left(\sum_{j=i}^{N-1} F_{1,p}(g^j) \right)\cdot\left( 1-\zeta_N^{g^{i-1}}, \frac{1-\zeta_N}{1-\zeta_N^{g^{-1}}}\right)_r = 0$ \\ in $J\cdot \mathcal{K}_r/J^2 \cdot \mathcal{K}_r$\end{tabular} &  \\ \cline{1-4}
\end{tabular}
\end{table}

\begin{table}[htb]
      \centering
        \caption{The case $p = 3$}
               \begin{tabular}{l|c|c|}
\cline{2-3}
                           & $n(r,p) \geq 2$                          & $n(r,p) \geq 3$, $t\geq 2$ and $r \leq t-1$ \\ \hline
\multicolumn{1}{|c|}{$M_+$} & $\sum_{k=1}^{N-1} k^2 \cdot \log(k) = 0$ & $\sum_{k=1}^{N-1} k^2 \cdot \log(k)^2 = 0$  \\ \hline
\end{tabular}    
\end{table}

\begin{table}[htb]
      \centering
        \caption{The even modular symbols $M_+$, $p = 2$}
              \begin{tabular}{l|c|c|}
\cline{2-3}
                           & $n(r,p) \geq 2$                               & $n(r,p) \geq 3$, $t \geq 2$ and $r \leq t-1$                                                                    \\ \hline
\multicolumn{1}{|c|}{$M_+$} & $2^{t-1}-\sum_{k=1}^{N-1} k\cdot \log(k) = 0$ & $\sum_{k=1}^{\frac{N-1}{2}} k \cdot \left( \log(k) + \log(k)^2\right) \equiv 0 \text{ (modulo }2^{r+1}\text{)}$ \\ \hline
\end{tabular}   
\end{table}
\end{landscape}
\newpage

\bibliography{biblio}
\bibliographystyle{plain}
\newpage

\end{document}